%% file: main.tex
\documentclass{article}
\usepackage{graphicx} 
\usepackage{amsfonts,amssymb,amsmath,amsthm,dsfont,txfonts}
\usepackage{caption}
\usepackage{hyperref}
\usepackage{orcidlink}
\usepackage[left=3cm,top=4cm,right=3cm,bottom=4cm]{geometry}
\numberwithin{equation}{section}
\numberwithin{figure}{section}
\numberwithin{table}{section}
\newtheorem{theorem}{Theorem}[section]

\theoremstyle{definition}

\newtheorem{remark}{Remark}
\numberwithin{remark}{section}
\usepackage{float} 
\usepackage{hhline}

\usepackage[
    backend=biber,
    style=numeric,
    maxcitenames=999,
    mincitenames=1,
    maxbibnames=999,
    minbibnames=1,
    sorting=none,
]{biblatex}

\DeclareFieldFormat*{title}{#1}

\DeclareFieldFormat{journaltitle}{\mkbibitalic{#1}}
\DeclareFieldFormat{booktitle}{\mkbibitalic{#1}}

\DeclareFieldFormat[article,incollection,inproceedings]{volume}{#1}
\DeclareFieldFormat[article,incollection,inproceedings]{number}{\mkbibparens{#1}}

\DeclareFieldFormat[misc]{volume}{#1}
\DeclareFieldFormat[misc]{number}{\mkbibitalic{#1}}

\renewbibmacro{in:}{}

\DefineBibliographyStrings{english}{%
  page = {},
  pages = {},
}

\DeclareBibliographyDriver{article}{%
  \usebibmacro{bibindex}%
  \usebibmacro{begentry}%
  \usebibmacro{author/translator+others}%
  \setunit{\labelnamepunct}\newblock
  \printfield{title}%
  \newunit\newblock
  \printfield{journaltitle}%
  \setunit{\addcomma\space}
  \printfield{volume}%
  \printfield{number}%
  \setunit{\addcolon\space}
  \printfield{pages}%
  \setunit{\addcomma\space}
  \printfield{year}%
  \usebibmacro{finentry}%
}

\DeclareBibliographyDriver{incollection}{%
  \usebibmacro{bibindex}%
  \usebibmacro{begentry}%
  \usebibmacro{author/translator+others}%
  \setunit{\labelnamepunct}\newblock
  \printfield{title}%
  \newunit\newblock
  \printfield{booktitle}%
  \setunit{\addcomma\space}
  \printfield{volume}%
  \printfield{number}%
  \setunit{\addcolon\space}
  \printfield{pages}%
  \setunit{\addcomma\space}
  \printfield{year}%
  \usebibmacro{finentry}%
}

\DeclareBibliographyDriver{inproceedings}{%
  \usebibmacro{bibindex}%
  \usebibmacro{begentry}%
  \usebibmacro{author/translator+others}%
  \setunit{\labelnamepunct}\newblock
  \printfield{title}%
  \newunit\newblock
  \printfield{booktitle} 
  \setunit{\addcomma\space}
  \printfield{volume}%
  \printfield{number}%
  \setunit{\addcolon\space}
  \printfield{pages}%
  \setunit{\addcomma\space}
  \printfield{year}%
  \usebibmacro{finentry}%
}

\DeclareBibliographyDriver{misc}{%
  \usebibmacro{bibindex}%
  \usebibmacro{begentry}%
  \usebibmacro{author/translator+others}%
  \setunit{\labelnamepunct}\newblock
  \printfield{title}%
  \newunit\newblock
  \printfield{journaltitle}%
  \setunit{\addcomma\space}
  \printfield{volume}%
  \setunit{\addcomma\space}
  \printfield{number}%
  \setunit{\addcolon\space}
  \printfield{pages}%
  \setunit{\addcomma\space}
  \printfield{year}%
  \usebibmacro{finentry}%
}

\DeclareBibliographyDriver{software}{%
  \usebibmacro{bibindex}%
  \usebibmacro{begentry}%
  \usebibmacro{author}%
  \setunit{\labelnamepunct}\newblock
  \printfield{title}%
  \setunit{\addcomma\space}
  \printfield{version}%
  \setunit{\addcomma\space}
  \printfield{year}%
  \setunit{\addperiod\space}
  \printfield{url}%
  \usebibmacro{finentry}%
}

\usepackage{authblk}

\title{On Hyperbolic Stochastic Galerkin Projections of Shallow Water Linearised Moment Equations}
\author[1,2]{Safiere Kuijpers \orcidlink{0009-0007-8251-4601}}
\author[1,2,3]{Julian Koellermeier \orcidlink{0000-0002-8822-461X}}
\affil[1]{Bernoulli Institute, University of Groningen, Groningen, The Netherlands (s.a.kuijpers@rug.nl).}
\affil[2]{CogniGron (Groningen Cognitive Systems and Materials Centre), University of Groningen, Groningen, The Netherlands.}
\affil[3]{Department of Mathematics, Computer Science and Statistics, Ghent University, Ghent, Belgium.}
\date{\today}

\begin{document}

\maketitle

\section*{Abstract}
\input{Sections/00_Abstract}

\vspace{1em}
\noindent\textbf{Keywords} Uncertainty quantification $\cdot$ Free-surface flow $\cdot$ Hyperbolic regularisation $\cdot$ Generalised polynomial chaos $\cdot$ Pseudospectral product\\

\vspace{0.5em}
\noindent\textbf{Mathematics Subject Classification (2020)} 35L60 $\cdot$ 35Q35 $\cdot$ 35R60 $\cdot$ 76M12

\section{Introduction}\label{section:intro}
\input{Sections/01_Introduction}

\section{Extended Shallow Water Model Equations}\label{section:SWME}
\input{Sections/02_Shallow_Water_Moment_Equations}

\section{Stochastic Galerkin Projection}\label{section:SG}
\input{Sections/03_Stochastic_Galerkin}

\section{Stochastic Galerkin Projection of the SWLME}\label{section:SGSWME}
\input{Sections/04_SG_For_SWLME}

\section{Analysis}\label{section:ANAL}
\input{Sections/05_Analysis}

\section{Numerics}\label{section:NUM}
\input{Sections/06_Numerics}

\section{Results}\label{section:RES}
\input{Sections/07_Results}

\section{Conclusion}\label{section:CON}
\input{Sections/08_Conclusion}

\section*{Data Availability Statement}
The datasets generated and analysed during this study are available in a Github repository \cite{kuijpers_sgswlme_2026}.

\section*{Author Contributions}
Safiere Kuijpers: Conceptualisation, methodology, analysis, investigation, software, data curation, visualisation, manuscript - initial draft;
Julian Koellermeier: Conceptualisation, methodology, analysis, funding acquisition, supervision, manuscript - review;
All authors read and approved the final manuscript.

\section*{Funding}
The authors would like to acknowledge the financial support of the CogniGron research
centre and the Ubbo Emmius Funds (University of Groningen). JK is supported by the Dutch Research Council (NWO) through the ENW Vidi project HiWAVE with file number VI.Vidi.233.066.

\appendix
\section{Convergence of Monte Carlo Sampling}\label{section:appendix}
\input{Sections/A_Convergence_MC}

\printbibliography

\end{document}

%% file: Sections/00_Abstract.tex
In this work, we present an intrusive stochastic Galerkin formulation of the one-dimensional shallow water linearised moment equations expressed in conservative variables, using the pseudospectral product for generalised polynomial chaos expansions. The shallow water linearised moment equations constitute a hyperbolic system of partial differential equations with an arbitrary number of equations that enhance the accuracy of the standard shallow water equations. Without loss of generality, we assume for both the theoretical analysis and the simulations that the uncertain parameter is the friction coefficient.
For the new stochastic Galerkin shallow water linearised moment equations, we derive an energy equation, analyse the hyperbolicity - since this property is not preserved by the stochastic Galerkin projection - and introduce a regularisation to ensure hyperbolicity for the linear case of the shallow water linearised moment equations. Through numerical tests, we demonstrate the accuracy of the new stochastic Galerkin formulation in comparison with a non-intrusive Monte Carlo method, showing that the stochastic Galerkin approach achieves comparable accuracy with significantly faster run times. 

%% file: Sections/01_Introduction.tex
The Shallow Water Equations (SWE) are a simplification of the Navier-Stokes equations \cite{kowalski_moment_2019}. They can be used in cases where the horizontal length scale is much greater than the vertical length scale, for example, in weather forecasting and simulation-based natural hazard assessment. The SWE can be derived from integrating the Navier-Stokes equations over the depth, also called depth-averaging. As a consequence, the SWE use a single-variable simplification of the underlying vertical velocity profile, which is the depth-averaged velocity. Due to this simplification, the SWE are, of course, less accurate in describing the motion of viscous fluid substances than the original Navier-Stokes equations. However, the reduced model complexity brings about the benefit of computational efficiency.

As applications of the SWE are limited due to the information loss from depth-averaging, multiple attempts have been made to (partially) keep the computational efficiency of the SWE, while increasing accuracy. One of these approaches is a multilayer shallow water model, which yields a piecewise-constant velocity profile to model changes in the vertical direction \cite{fernandez-nieto_multilayer_2016}. However, the discontinuous vertical velocity complicates coupling the interfaces and ensuring the hyperbolicity of the resulting system \cite{fernandez-nieto_multilayer_2016}. Another approach to improving the accuracy of the SWE is inspired by the moment method used in kinetic theory \cite{struchtrup_macroscopic_2005}. This method applies an expansion around the depth-averaged velocity in a series of polynomials. The coefficients of the polynomials in this series are the so-called moments. Evolution equations for these additional moments are derived by a Galerkin projection of the momentum equation of the Navier-Stokes equations. These Shallow Water Moment Equations (SWME) were introduced in \cite{kowalski_moment_2019}, while hyperbolicity (after different regularisations) was shown in \cite{koellermeier_analysis_2020} and \cite{koellermeier_steady_2022}, resulting in the Hyperbolic Shallow Water Moment Equations (HSWME) and Shallow Water Linearised Moment Equations (SWLME), respectively.

The friction effects in the SWE and SWME can be modelled using different approaches. In \cite{kowalski_moment_2019}, a Newtonian slip friction model parametrised by the friction coefficient and slip length was used, whereas other friction models exist \cite{garres-diaz_shallow_2021}. 

To perform simulations using SWE and SWME, some parameters, such as the friction coefficient, gravitational constant, slip length, and bottom topography, need to be known a priori. However, in practice, it is often the case that one or more of these parameters are not exactly known due to natural variability, measurement errors, or incomplete knowledge of the system. Instead of ruling out uncertainty in these parameters, it is therefore necessary to take this uncertainty into account. (Forward) uncertainty quantification is the process of quantifying how uncertainty in parameters propagates to uncertainty in simulation results. In this way, one can determine how likely different outcomes are if some aspects of the system are not exactly known.

Uncertainty quantification can be performed using two distinct approaches: (1) intrusive methods and (2) non-intrusive methods. Intrusive methods derive a new deterministic system of equations that incorporates uncertainty, while non-intrusive methods reuse existing deterministic equations and obtain error bounds via sampling techniques, such as Monte Carlo or stochastic collocation. The deterministic system derived using an intrusive method has the advantage of requiring only a single solve to obtain the desired results. In contrast, non-intrusive methods require solving the original system multiple times. Additionally, intrusive methods provide knowledge of the entire random space, whereas non-intrusive methods only yield solutions at specific samples or quadrature points \cite{gerster_entropies_2020}. This allows for adaptivity in the stochastic space \cite{meyer_posteriori_2020, burger_hybrid_2017} and the construction of numerical schemes that are well-balanced at any location in the random space \cite{jin_well-balanced_2016}. Non-intrusive methods can enforce this property only at the collocation nodes. Once derived, intrusive methods are generally faster and more computationally efficient, although non-intrusive methods can also be accelerated using techniques such as low-rank multilevel Monte Carlo \cite{patwardhan_dynamical_2026}. On the other hand, non-intrusive methods do not require deriving a new deterministic system, allowing reuse of existing code frameworks. Consequently, implementation is typically more costly for intrusive methods. Another key difference is that intrusive methods allow for analysis of the resulting system, enabling analytical derivation of its properties. This is not possible for non-intrusive approaches.

For the SWE, which consist of a hyperbolic system of two partial differential equations (PDEs), an intrusive stochastic Galerkin projection has been derived in several studies. This intrusive method relies on generalised polynomial chaos expansions \cite{field_jr_accuracy_2004}. For instance, \cite{bender_entropy-conservative_2024} developed an entropy-conservative discontinuous Galerkin method to solve the stochastic shallow water equations within a stochastic Galerkin framework using a Roe variable transformation, while \cite{dai_hyperbolicity-preserving_2021} designed a second-order, energy-stable, and well-balanced scheme with a specific projection step incorporated into the Galerkin projection. A comparison of these methods is provided in \cite{offner_high-order_2026}. In general, the stochastic Galerkin projection does not preserve hyperbolicity \cite{despres_robust_2013}. Other methods to address this loss of hyperbolicity - aside from the Roe variable transformation combined with Haar wavelets, projection techniques, and high-order quadrature rules - include limiting strategies \cite{schlachter_hyperbolicity-preserving_2018,durrwachter_hyperbolicity-preserving_2020}, filtering techniques \cite{kusch_filtered_2020}, linearisation methods \cite{wu_stochastic_2017}, and entropic variable representations \cite{poette_uncertainty_2009}.

This paper extends the work of \cite{dai_hyperbolicity-preserving_2021} by applying an intrusive stochastic Galerkin projection to the SWLME, which consist of a system of an arbitrarily large number of PDEs. We derive the stochastic Galerkin projection of the SWLME, derive an energy equation, and investigate its hyperbolicity. Additionally, we numerically test the resulting system and compare it with a reference Monte Carlo simulation, allowing us to assess its accuracy and demonstrate its faster run time compared with non-intrusive Monte Carlo simulations. Without loss of generality, we focus on the case of an uncertain friction coefficient, one of the most influential parameters in the Newtonian model derived in \cite{kowalski_moment_2019}. The extension to other uncertain parameters is straightforward.

The rest of this paper is structured as follows: in Section \ref{section:SWME}, we recall the SWE and the derivation of the SWME and SWLME. Since \cite{dai_hyperbolicity-preserving_2021} works in convective variables, as a first contribution, we will rewrite the SWME and SWLME systems, their eigenvalues and eigenvectors in convective variables. This allows us to remain within the framework of \cite{dai_hyperbolicity-preserving_2021} and extend the analysis to the more complex case of the SWLME. The special attention to the eigenvalues and -vectors of the systems is due to their later importance in analysing hyperbolicity. In Section \ref{section:SG}, we introduce terminology and notation of the stochastic Galerkin projection. The core of this work is Section \ref{section:SGSWME}, where a stochastic Galerkin projection is performed on the SWLME, including an analysis with respect to hyperbolicity and energy of the resulting system in Section \ref{section:ANAL}. After a brief description of the numerical scheme used in Section \ref{section:NUM}, we show numerical experiments with the SWE and SWLME including uncertainty in Section \ref{section:RES}.

%% file: Sections/02_Shallow_Water_Moment_Equations.tex
In this section, we recall the model equations used later in the paper. These form the basis for uncertainty quantification using the stochastic Galerkin projection in the next two sections. We start with the SWE in Section \ref{subsection:SWE}, then extend this to the SWME in Section \ref{subsection:SWME}, before considering as a special case the SWLME in Section \ref{subsection:SWLME}.

\subsection{Shallow Water Equations}\label{subsection:SWE}
We first recall the shallow water equations (SWE) in one dimension and for a flat bottom. The extension to two dimensions can be done similar to \cite{dai_hyperbolicity-preserving_2022} as well as the inclusion of non-flat bottom bathymetries in one dimension \cite{dai_hyperbolicity-preserving_2021} and two dimensions \cite{dai_hyperbolicity-preserving_2022}. For a flat bottom, the one-dimensional SWE for the height of the water $h=h(t,x)$ and the mean velocity $u_m=u_m(t,x)$ are given by
\begin{subequations}
    \begin{alignat}{1}
        \partial_th+\partial_x(hu_m)&=0,\label{SWME-eq:SWE_a}\\
        \partial_t(hu_m)+\partial_x\left(\frac{gh^2}{2}+hu_m^2\right) &= -\frac{\nu}{\lambda}u_m,\label{SWME-eq:SWE_b}
    \end{alignat}
\end{subequations}
where $g$ is the gravitational constant, $\nu$ is the friction coefficient, and $\lambda$ is the slip length, modelling the friction at the bottom proportional to the velocity slip at the bottom.

A system of first-order partial differential equations of the form 
\begin{equation}\label{SWME-eq:hyp}
    \frac{\partial \boldsymbol{W}(t,x)}{\partial t}+A(\boldsymbol{W})\frac{\partial \boldsymbol{W}(t,x)}{\partial x}=0 \quad\text{for }t\in\mathbb{R}_{>0},x\in\mathbb{R},
\end{equation}
with $\boldsymbol{W}\colon \mathbb{R}_{\geq 0}\times\mathbb{R}\rightarrow\mathbb{R}^D$ the vector of unknowns of length $D$ and the system matrix $A\colon \mathbb{R}_{\geq 0}\times\mathbb{R}\rightarrow\mathbb{R}^{D\times D}$ is said to be globally hyperbolic if $A(\boldsymbol{W})$ is diagonalisable with real eigenvalues for all states $\boldsymbol{W}$.

The solutions to hyperbolic equations are ``wavelike" with finite propagation speeds given by the eigenvalues of $A(\boldsymbol{W})$. Furthermore, hyperbolic systems that are symmetrisable can be shown to be well-posed, and hyperbolic systems can be readily solved with a variety of well-established numerical methods. Hyperbolicity is therefore a desirable property for time-dependent first-order PDEs.

Rewriting the SWE \eqref{SWME-eq:SWE_a}-\eqref{SWME-eq:SWE_b} in the form of equation \eqref{SWME-eq:hyp} using the water discharge $q_0\coloneqq hu_m$, we find
\begin{equation}\label{SWME-eq:SWE_matrix}
    \partial_t
    \begin{pmatrix}
        h \\ q_0
    \end{pmatrix}
    + \underbrace{
    \begin{pmatrix}
        0 & 1\\
        gh-\left(\frac{q_0}{h}\right)^2 & 2\frac{q_0}{h}\\
    \end{pmatrix}}_{\eqqcolon A_{\text{SWE}}}
    \partial_x
    \begin{pmatrix}
        h \\ q_0
    \end{pmatrix}
    = -\frac{\nu}{\lambda}
    \begin{pmatrix}
        0 \\ \frac{q_0}{h}
    \end{pmatrix},
\end{equation}
for $\boldsymbol{W}_0(t,x)=(h(t,x),q_0(t,x))^T\in\mathbb{R}^2$. The eigenvalues of the system matrix $A_{\text{SWE}}$ are
\begin{equation}\label{SWE-eq:SWE_ev}
    \sigma\left(A_{\text{SWE}}\right)=\left\{\frac{q_0}{h}-\sqrt{gh},\:\frac{q_0}{h}+\sqrt{gh}\right\}.
\end{equation}
Since these are two distinct real eigenvalues if the height of the water $h$ is positive, it follows that the system \eqref{SWME-eq:SWE_matrix} is diagonalisable and therefore hyperbolic. The eigenvectors of the system matrix $A_{\text{SWE}}$ in \eqref{SWME-eq:SWE_matrix} are given by the columns of the following matrix
\begin{equation}
    \begin{pmatrix}
        1 & 1\\
        \frac{q_0}{h}-\sqrt{gh} & \frac{q_0}{h}+\sqrt{gh}\\
    \end{pmatrix},
\end{equation}
where we see the eigenvalues of \eqref{SWE-eq:SWE_ev} appear in the second row.

\subsection{Shallow Water Moment Equations}\label{subsection:SWME}
The simplification for the velocity profile in the $z$-direction of the SWE, which results from integrating the Navier-Stokes equations over the depth, is not observed in nature. Often, the water speed is slower at the bottom, for example, due to friction with sand, stones, and plants. A way to reintroduce an explicit reconstruction of the velocity profile in the $z$-direction, while still largely benefitting from the computational efficiency of the SWE, is described in \cite{kowalski_moment_2019}. There, the so-called Shallow Water Moment Equations (SWME) are introduced, which rely on a polynomial expansion of the velocity profile in the $z$-direction:
\begin{equation}\label{SWME-eq:expansion}
    u(t,x,z)=u_m(t,x)+\sum^N_{j=1}u_j(t,x)\phi_j\left(\frac{z}{h(t,x)}\right).
\end{equation}
Here $\phi_j\colon[0,1]\rightarrow\mathbb{R}$ is the Legendre polynomial of degree $j$ scaled to the interval $[0,1]$ and $u_j\colon[0,T]\times\mathbb{R}\rightarrow\mathbb{R}$ are the corresponding basis coefficients, also called moments. Since Legendre polynomials $\left\{\phi_j\right\}_{j=0}^N$ form an orthogonal set of polynomials on $[0,1]$, any polynomial velocity profile of degree $N$ can be written in the form of \eqref{SWME-eq:expansion}. $N$ is called the order of the model.

Additional evolution equations for the coefficients, the so-called moment equations, can be derived by inserting the expansion of \eqref{SWME-eq:expansion} into the momentum balance equation of the Navier-Stokes system, multiplying by $\phi_i$ and integrating the resulting expression over the interval $[0,1]$ for every $i=0,\ldots,N$. Using the orthogonality property of the Legendre polynomials, the resulting equations can be greatly simplified. As shown in \cite{kowalski_moment_2019}, for general order $N$, the SWME system of $N+2$ equations reads
\begin{subequations}
    \begin{alignat}{1}
        \partial_th+\partial_x(hu_m)&=0,\label{SWME-eq:SWME_N_a}\\
        \partial_t(hu_m)+\partial_x\left(\frac{gh^2}{2}+hu_m^2+h\sum^N_{k=1}\frac{u_k^2}{2k+1}\right)&=-\frac{\nu}{\lambda}\left(u_m+\sum^N_{k=1}u_k\right),\label{SWME-eq:SWME_N_b}\\
        \partial_t(hu_j)+\partial_x\left(2hu_mu_j+h\sum_{k,l=1}^NA_{jkl}u_ku_l\right)&=u_m\partial_x(hu_j)-\sum_{k,l=1}^NB_{jkl}u_l\partial_x(hu_k)\nonumber\\&\quad-\frac{(2j+1)\nu}{\lambda}\left(u_m+\sum^N_{k=1}\left(1+\frac{\lambda}{h}C_{jk}\right)u_k\right)\label{SWME-eq:SWME_N_c},
    \end{alignat}
\end{subequations}
for $j=1,\ldots,N$, where
\begin{subequations}
    \begin{alignat}{1}
        A_{jkl}&=(2j+1)\int_0^1\phi_j\phi_k\phi_ld\zeta,\label{SWME-eq:SWME_coef_a}\\
        B_{jkl}&=(2j+1)\int_0^1\partial_{\zeta}\phi_j\left(\int^{\zeta}_0\phi_kd\hat{\zeta}\right)\phi_ld\zeta,\label{SWME-eq:SWME_coef_b}\\
        C_{jk}&=\int^1_0\partial_\zeta\phi_j\partial_\zeta\phi_kd\zeta\label{SWME-eq:SWME_coef_c}.
    \end{alignat}
\end{subequations}

Note the emergence of the non-conservative terms $u_m\partial_x(hu_j)-\sum_{k,l=1}^NB_{jkl}u_l\partial_x(hu_k)$ on the right-hand side of \eqref{SWME-eq:SWME_N_c}. While this looks non-standard, we note that in typical shallow water settings even without friction, already a non-flat bottom would lead to a non-conservative product term, meaning that momentum is transferred between the bottom and the fluid. In the field of SWE, non-conservative products are, therefore, nothing unusual.

Multiple extensions exist for the SWME, for example, using Manning friction \cite{garres-diaz_shallow_2021} or different modifications for larger $N$ \cite{koellermeier_analysis_2020,koellermeier_steady_2022}.

For investigating hyperbolicity in Section \ref{section:ANAL}, we will need the eigenvalues and eigenvectors of the system of equations \eqref{SWME-eq:SWME_N_a}-\eqref{SWME-eq:SWME_N_c}. However, for the SWME they cannot be given in explicit form for arbitrary order $N$. Therefore, we display the case of order $N=1$ individually. We will also write the system of equations, eigenvalues and -vectors in convective variables to remain in the framework of \cite{dai_hyperbolicity-preserving_2021}. By writing down the shallow water models in these variables, we can also later compare their structure to their stochastic Galerkin projection in Sections \ref{section:SGSWME} and \ref{section:ANAL}.

For $N=1$, i.e., for the linear velocity profile $u(t,x,z)=u_m(t,x)+u_1(t,x)\phi_1\left(\frac{z}{h(t,x)}\right)$, the SWME1 read
\begin{subequations}
    \begin{alignat}{1}
        \partial_th+\partial_x(hu_m)&=0,\label{SWME-eq:SWME_N=1_a}\\
        \partial_t(hu_m)+\partial_x\left(\frac{gh^2}{2}+hu_m^2+\frac{hu_1^2}{3}\right)&=-\frac{\nu}{\lambda}(u_m+u_1),\label{SWME-eq:SWME_N=1_b}\\
        \partial_t(hu_1)+\partial_x(2hu_mu_1)&=u_m\partial_x(hu_1)-\frac{3\nu}{\lambda}(u_m+u_1)-\frac{12\nu u_1}{h}.\label{SWME-eq:SWME_N=1_c}
    \end{alignat}
\end{subequations}

This system of equations \eqref{SWME-eq:SWME_N=1_a}-\eqref{SWME-eq:SWME_N=1_c} can be rewritten in matrix form \eqref{SWME-eq:hyp} as
\begin{equation}\label{SWME-eq:SWME_N=1_matrix1}
    \partial_t
    \begin{pmatrix}
        h \\ hu_m \\ hu_1
    \end{pmatrix}
    +
    \begin{pmatrix}
        0 & 1 & 0\\
        -h-u_m^2+\frac{u_1^2}{3} & 2u_m & \frac{2}{3}u_1\\
        -2u_mu_1 & 2u_1 & 2u_m
    \end{pmatrix}
    \partial_x
    \begin{pmatrix}
        h \\ hu_m \\ hu_1
    \end{pmatrix}
    = \underbrace{
    \begin{pmatrix}
        0 \\ 0 \\ u_m\partial_x(hu_1)
    \end{pmatrix}}_{\eqqcolon \boldsymbol{Q_1}}
    + \underbrace{
    \begin{pmatrix}
        0 \\
        -\frac{\nu}{\lambda}(u_m+u_1) \\
        -\frac{3\nu}{\lambda}(u_m+u_1)-\frac{12\nu u_1}{h}
    \end{pmatrix}}_{\eqqcolon \boldsymbol{S_1}}.
\end{equation}
On the right-hand side we have separated the non-conservative terms $\boldsymbol{Q_1}$ from the friction $\boldsymbol{S_1}$.

Similarly to the SWE, we rewrite the SWME1 \eqref{SWME-eq:SWME_N=1_matrix1} in the form of equation \eqref{SWME-eq:hyp} using $q_1\coloneqq hu_1$. We then find
\begin{equation}\label{SWME-eq:SWME_N=1_matrix3}
    \partial_t
    \begin{pmatrix}
        h \\ q_0\\ q_1
    \end{pmatrix}
    + \underbrace{
    \begin{pmatrix}
        0 & 1 & 0\\
        gh-\left(\frac{q_0}{h}\right)^2-\frac{1}{3}\left(\frac{q_1}{h}\right)^2 & 2\frac{q_0}{h} & \frac{2}{3}\frac{q_1}{h}\\
        -2\frac{q_0q_1}{h^2} & 2\frac{q_1}{h} & \frac{q_0}{h}
    \end{pmatrix}}_{\eqqcolon A_1}
    \partial_x
    \begin{pmatrix}
        h \\ q_0 \\ q_1
    \end{pmatrix}
    = -\frac{\nu}{\lambda}
    \begin{pmatrix}
        0 \\ 
        \frac{q_0+q_1}{h} \\
        3\frac{q_0+q_1}{h} + \frac{12\lambda q_1}{h^2}
    \end{pmatrix}
\end{equation}
for $\boldsymbol{W}_1=(h(t,x),q_0(t,x),q_1(t,x))^T\in\mathbb{R}^3$. The eigenvalues of $A_1$ are
\begin{equation}\label{SWE-eq:SWME_ev}
    \sigma\left(A_1\right)=\left\{\frac{q_0}{h}-\sqrt{gh+\left(\frac{q_1}{h}\right)^2},\:\frac{q_0}{h},\:\frac{q_0}{h}+\sqrt{gh+\left(\frac{q_1}{h}\right)^2}\right\},
\end{equation}
i.e., the system is again hyperbolic if the water height $h$ is positive. The eigenvectors of the system matrix in equation \eqref{SWME-eq:SWME_N=1_matrix3} are given by the columns of the matrix
\begin{equation}\label{SWME-eq:SWME_EV}
    \begin{pmatrix}
        1 & 1 & 1\\
        \frac{q_0}{h}-\sqrt{gh+\left(\frac{q_1}{h}\right)^2} & \frac{q_0}{h} & \frac{q_0}{h}+\sqrt{gh+\left(\frac{q_1}{h}\right)^2}\\
        2\frac{q_1}{h} & \frac{-3gh^3+q_1^2}{2hq_1} & 2\frac{q_1}{h}
    \end{pmatrix},
\end{equation}
where we again see that the eigenvalues \eqref{SWE-eq:SWME_ev} appear in the second row.

For $N\geq2$ the SWME system is not hyperbolic for all states \cite{koellermeier_steady_2022}. In this paper, therefore, we will not extend the SWME to incorporate uncertainty, but one of its regularisations. We choose to work with the Shallow Water Linearised Moment Equations (SWLME) of \cite{koellermeier_steady_2022}, because of its ease to extend it from $N=1$ to the general case of arbitrary $N$.

\subsection{Shallow Water Linearised Moment Equations}\label{subsection:SWLME}
The SWLME are a hyperbolic regularisation of the SWME derived first in \cite{koellermeier_steady_2022}. The SWLME are obtained after a linearisation of terms appearing during the derivation of the moment equations, which effectively results in setting $A_{ijk}=0=B_{ijk}$ \eqref{SWME-eq:SWME_coef_a}-\eqref{SWME-eq:SWME_coef_b} in the SWME \eqref{SWME-eq:SWME_N_a}-\eqref{SWME-eq:SWME_N_c}, while keeping the same friction term as well as the same first two equations from the full SWME. As shown in \cite{koellermeier_steady_2022}, the SWLME allow for an analytical computation of steady states in addition to analytical eigenvalues and eigenvectors of the system matrix even for general $N$. This is the main reason we focus on the study of Stochastic Galerkin models with $N>1$ to the SWLME, as the hyperbolicity of the Stochastic Galerkin version of the SWLME (later called SGSWLME) can be obtained based on the hyperbolicity of the SWLME themselves. We note that the SWLME were widely used in the literature, see \cite{koellermeier_new_2026, caballero-cardenas_semi-implicit_2025, careaga_entropy_2026, fan_well-balanced_2026, cao_flux_2026}.

For general $N$, this leads to the following system, see also \cite{koellermeier_steady_2022}
\begin{equation}\label{SWME-eq:SWLME_N_matrix}
    \partial_t
    \begin{pmatrix}
        h \\ q_0 \\ q_1 \\ \vdots \\ q_N
    \end{pmatrix}
    +\partial_x
    \begin{pmatrix}
        q_0 \\
        \frac{gh^2}{2}+\sum\limits_{j=0}^N\frac{1}{2j+1}\frac{q_j^2}{h}\\
        2\frac{q_0q_1}{h} \\
        \vdots \\
        2\frac{q_0q_N}{h}
    \end{pmatrix}
    = \underbrace{
    \begin{pmatrix}
        0 \\ 0 \\ \frac{q_0}{h}\partial_xq_1 \\ \vdots \\ \frac{q_0}{h}\partial_xq_N
    \end{pmatrix}}_{\eqqcolon \boldsymbol{Q}}+\boldsymbol{S},
\end{equation}
where $q_j\coloneqq hu_j$ and the entries of $\boldsymbol{S}$ are given by
\begin{equation}\label{SWME-eq:S2}
    \begin{split}
        S_0 &= 0,\\
        S_{j+1}&=-\frac{(2j+1)\nu}{\lambda h}\left(\sum_{k=0}^Nq_k\right)-\frac{4(2j+1)\nu}{h^2}\left(\sum_{k=0}^Na_{j,k}q_k\right),\quad\text{for }j=0, \ldots, N,
    \end{split}
\end{equation}
where the constants $a_{j,k}$ are computed by
\begin{equation}\label{SWME-eq:ajk}
    a_{j,k}=
    \begin{cases}
        0 & \text{if }j+k=\text { odd}, \\
        \frac{\min(j,k)(\min(j,k)+1)}{2}\quad & \text{if } j+k=\text { even}.
    \end{cases}
\end{equation}
The vector $\boldsymbol{S}$ of \eqref{SWME-eq:S2} is unchanged with respect to the SWME \eqref{SWME-eq:SWME_N_a}-\eqref{SWME-eq:SWME_N_c}. The system matrix $A_{\text{SWLME}}$ for general $N$ reads
\begin{equation}\label{SWME-eq:SWLME_A}
    A_{\text{SWLME}} =\begin{pmatrix}
    0&1&0&\cdots&0\\
    gh -\left(\frac{q_0}{h}\right)^2-\sum\limits_{j=1}^N\frac{1}{2j+1}\left(\frac{q_j}{h}\right)^2 & 2\frac{q_0}{h} & \frac{2}{3}\frac{q_1}{h} & \cdots & \frac{2}{2N+1}\frac{q_N}{h} \\
    -2\frac{q_0q_1}{h^2} &2\frac{q_1}{h}&\frac{q_0}{h} & &\\
    \vdots &\vdots & & \ddots & \\
    -2\frac{q_0q_N}{h^2} &2\frac{q_N}{h} & & &\frac{q_0}{h}
    \end{pmatrix} \in \mathbb{R}^{(N+2) \times (N+2)}.
\end{equation}

Note that here, for the first time, we write the SWLME equations in convective variables $q_i$, which is not done before in the literature. According to \cite[Theorem 1]{koellermeier_steady_2022}, the eigenvalues of $A_{\text{SWLME}}$ can be explicitly computed to be 
\begin{equation}\label{SWME-eq:SWLME_ev}
    \lambda_{1,2}=\frac{q_0}{h}\pm\sqrt{gh+\sum_{j=1}^N\frac{3}{2j+1}\left(\frac{q_j}{h}\right)^2}\quad\text{ and }\quad\lambda_{j+2}=\frac{q_0}{h},\quad\text{ for }j=1,\ldots,N.
\end{equation}
The corresponding eigenvectors $v_{1}, v_2$, and $v_{j+2}$ for $j=1, \ldots, N$ that will be used later are given in the same reference as 
\begin{equation}
    v_{1,2} = \begin{pmatrix}
    1 \\
    \frac{q_0}{h} \pm \sqrt{gh + \sum\limits_{j=1}^N \frac{3}{2j+1}\left(\frac{q_j}{h}\right)^2} \\
    2\frac{q_1}{h} \\
    \vdots \\
    2\frac{q_N}{h}
    \end{pmatrix}, \quad v_{j+2} = \begin{pmatrix}
    1\\
    \frac{q_0}{h}\\
    \frac{-3gh + \sum\limits_{j=1}^N \frac{3}{2j+1}\left(\frac{q_j}{h}\right)^2}{\frac{6}{2(n+1-j)+1}\frac{q_{n+1-j}}{h}}\delta_{n+3-j,3}  \\
    \vdots \\
    \frac{-3gh + \sum\limits_{j=1}^N \frac{3}{2j+1}\left(\frac{q_j}{h}\right)^2}{\frac{6}{2(n+1-j)+1}\frac{q_{n+1-j}}{h}}\delta_{n+3-j,N} 
    \end{pmatrix}
\end{equation}
using the Kronecker delta $\delta_{i,j}$. We again see that the eigenvalues of \eqref{SWME-eq:SWLME_ev} appear as the second entry of all eigenvectors. Note that for $N=1$ the SWME and SWLME coincide.

For the SWE, uncertainty quantification using stochastic Galerkin projection has already been performed in \cite{dai_hyperbolicity-preserving_2021}. Now we extend this to the SWLME for arbitrary $N$.

%% file: Sections/03_Stochastic_Galerkin.tex
To study the effect of uncertain parameters on the water height and velocity variables in the SWE and SWLME, we use a stochastic Galerkin projection \cite{constantine_primer_2007}. This approach relies heavily on the generalised Polynomial Chaos Expansion (gPCE). For the coefficients of this expansion, new equations are derived. The stochastic Galerkin approach is therefore an intrusive approach, meaning a new system of equations is established to incorporate the uncertainty. Uncertainty quantification tasks such as moment estimation and sensitivity analysis can be performed by simply post-processing the gPCE coefficients \cite{knio_uncertainty_2006, sudret_global_2008,crestaux_polynomial_2009}. The stochastic Galerkin approach also gives a way to represent knowledge of the whole random space, instead of only at particular sample points, and makes it possible to design numerical schemes that are well-balanced at any location in space, instead of only at the collocation nodes \cite{jin_well-balanced_2016}. Stochastic Galerkin methods also allow for adaptivity in the stochastic space \cite{meyer_posteriori_2020, burger_hybrid_2017}. In this section, we will introduce the stochastic Galerkin setting and its notation. In the following section, we will apply it to the SWLME for arbitrary $N$ (including SWME1), following the results of \cite{dai_hyperbolicity-preserving_2021}.

The basis of the stochastic Galerkin approach lies in the generalised Polynomial Chaos Expansion (gPCE).

\begin{theorem}\label{SG-th:gPCE}[modified from \cite{ernst_convergence_2012}]
    Let $\boldsymbol{\xi}=(\xi_1,\ldots,\xi_D)$ be a vector of $D\in\mathbb{N}$ independent random variables and $(\Omega,\sigma(\boldsymbol{\xi}),\mathds{P})$ be a probability space. We assume that $\boldsymbol{\xi}$ satisfies
    \begin{itemize}
        \item The distribution functions $F_{\xi_d}(x):=\mathds{P}(\xi_d\leq x)$ of the basic random variables are continuous,
        \item Each basic random variable $\xi_d$ possesses finite moments of all order, i.e., $\langle|\xi_d|^k\rangle=\int_\mathbb{R}|x|^kdF_{\xi_d}(dx)<\infty$ for all $k\in\mathbb{N}$.
    \end{itemize}
    Let  $\{\psi_j^{(d)}\}_{j\in\mathbb{N}_0}$, $d=1,\ldots,D$ be the sequences of polynomials orthonormal with respect to the distribution of $\xi_d$. Then the set of polynomials is given by
    \begin{equation*}
        \psi_{\boldsymbol{\alpha}}(\boldsymbol{\xi})=\prod^D_{d=1}\psi^{(d)}_{\alpha_d}(\xi_d),\quad\boldsymbol{\alpha}\in\mathbb{N}_0^D,
    \end{equation*}
    is an orthonormal basis of the space $L^2(\Omega,\sigma(\boldsymbol{\xi}),\mathds{P})$ if and only if the distribution function is uniquely determined by the sequence of its moments
    \begin{equation*}
        \mu_{d,k} :=\langle\xi_d^k\rangle=\int_\mathbb{R}x^kF_{\xi_d}(dx),\quad k\in\mathbb{N}_0,
    \end{equation*}
    for each random variable $\xi_d$, $d=1,\ldots,D$. In this case any random variable $\eta\in L^2(\Omega,\sigma(\boldsymbol{\xi}),\mathds{P})$ can be expanded in an abstract Fourier series of multivariate orthonormal polynomials $\psi_{\boldsymbol{\alpha}}$ in the basic random variables $\boldsymbol{\xi}$, the generalised polynomial chaos expansion
    \begin{equation}\label{SG-eq:gPCE}
        \eta=\sum_{\boldsymbol{\alpha}\in\mathbb{N}^D_0}\hat{\eta}_{\boldsymbol{\alpha}}\psi_{\boldsymbol{\alpha}}(\boldsymbol{\xi})\quad\text{with coefficients}\quad \hat{\eta}_{\boldsymbol{\alpha}}=\langle\eta\: \psi_{\boldsymbol{\alpha}}(\boldsymbol{\xi})\rangle.
    \end{equation}
\end{theorem}

\begin{remark}
    Examples of probability distributions for which the moment problem is uniquely solvable are uniform, beta, gamma, and normal distributions. By contrast, the moment problem is not uniquely solvable for the lognormal distribution, so that the sequence of random variables $\{\psi_n(\xi)\}_{n\in\mathbb{N}_0}$ for a lognormal random variable $\boldsymbol{\xi}$ does not constitute a basis of the Hilbert space $L^2(\Omega,\sigma(\boldsymbol{\xi}),\mathds{P})$, and therefore there will be some elements in this space that are not the limit of their generalised polynomial chaos expansion. Other examples of random variables with indeterminate distributions are certain powers of random variables with normal or gamma distribution \cite{ernst_convergence_2012}.
\end{remark}

Taking into account uncertainty, the variables of the SWLME depend not only on time $t$ and space $x$, but also on a random parameter of the system $\xi$. In principle, the effect of multiple random parameters on the system could be investigated, in which case $\xi$ would be a vector. However, in this paper, we will only analyse the effect of a single random parameter and therefore, for the sake of simplicity, we take $\xi$ to be only a one-dimensional scalar. The generalised polynomial chaos expansion of equation \eqref{SG-eq:gPCE} separates the dependence on time and space from the dependence on random parameters assuming that the random variable $\xi$ is such that the assumptions of Theorem \ref{SG-th:gPCE} are satisfied and that the moment problem is uniquely solvable.

If we define $P_K\coloneqq\text{span}\left\{\zeta^k\:|\:k=0,\dots,K\right\}$ for $K\in\mathbb{N}_0$, the space of polynomials of degree less or equal to $K$, then the $P_K$-truncated gPCE of a random variable $\eta$ that also depends on time and space is given by
\begin{equation*}
    \eta(t,x,\cdot)\approx\sum_{k=0}^K\hat{\eta}_{k}(t,x)\psi_{k}(\xi).
\end{equation*}
This $P_K$-truncated gPCE can be used to derive a deterministic system for stochastic Fourier coefficients $\hat{\eta}_k$, as this is not possible for the infinite sum. For square-integrable and $K$-times continuously differentiable, square-integrable functions with a normal random variable, an estimation of the accuracy of the truncated gPCE is given in \cite{field_jr_accuracy_2004}.

For brevity, we introduce a linear projection operator for this $P_K$-trucated gPCE $\mathcal{G}_{K}\colon L^2(\Omega,\sigma(\xi),\mathds{P})\rightarrow P_K$ defined as
\begin{equation*}
    \mathcal{G}_K[\eta](t,x,\xi)=\sum_{k=0}^K\hat{\eta}_{k}(t,x)\psi_{k}(\xi).
\end{equation*}

Polynomial statistics of the gPCE can easily be shown to be equal to
\begin{equation}
    \mathds{E}[\mathcal{G}_K[\eta](t,x,\xi)]=\hat{\eta}_0(t,x),\quad \text{Var}[\mathcal{G}_K[\eta](t,x,\xi)]=\sum_{k=1}^K\hat{\eta}_k^2(t,x),
\end{equation}
where $\mathds{E}$ is the expectation and $\text{Var}$ is the variance.

The SWE and SWLME of equations \eqref{SWME-eq:SWE_a}-\eqref{SWME-eq:SWE_b} and \eqref{SWME-eq:SWLME_N_matrix}, respectively, also contain products and ratios of functions. We therefore need to compute nonlinear gPCE's, which is not straightforward. We follow the (notational) framework of \cite{dai_hyperbolicity-preserving_2021}. Consider the two stochastic processes $\eta$ and $\mu$ with their respective $P_K$-truncated gPCE representations
\begin{equation*}
    \mathcal{G}_K[\eta](t,x,\xi)=\sum_{k=0}^K\hat{\eta}_{k}(t,x)\psi_{k}(\xi),\quad\mathcal{G}_K[\mu](t,x,\xi)=\sum_{k=0}^K\hat{\mu}_{k}(t,x)\psi_{k}(\xi).
\end{equation*}
Since there is no direct way to compute the basis expansion coefficients $\beta_{k}$'s of the product $\beta=\eta\mu$ of these stochastic processes, we assume  the following
\begin{equation}\label{SG-eq:product}
    \mathcal{G}_K[\beta](t,x,\xi) =\sum_{k=0}^K\hat{\beta}_{k}(t,x)\psi_{k}(\xi)  =\sum_{l=0}^K\hat{\eta}_l(t,x)\psi_l(\xi)\sum_{m=0}^K\hat{\mu}_m(t,x)\psi_m(\xi).
\end{equation}
Note that $\beta$ is expressed as a polynomial of the polynomial space $P_K$ with degree $K$, although the product of the gPCE's of $\eta$ and $\mu$ is in $P_{K^2}$ with degree $2K$. The $\hat{\beta}_{k}$ are therefore obtained by a Galerkin projection, which minimises the error of the resulting gPCE representation and the space spanned by the basis functions of $P_K$. Multiplying both sides of equation \eqref{SG-eq:product} by $\psi_{k}$ and taking expectations gives
\begin{equation*}
    \hat{\beta}_{k}(t,x) = \sum_{l=0}^K\sum_{m=0}^K\hat{\eta}_l(t,x)\hat{\mu}_m(t,x)\langle\psi_l\psi_m\psi_k\rangle\quad\text{for all }k=0,\ldots,K
\end{equation*}
using that the $\psi_k$'s are orthonormal. This gives rise to the notation
\begin{equation}
    \mathcal{G}_K[\eta,\mu](t,x,\xi)\coloneqq \sum_{k=0}^K\left(\sum_{l=0}^K\sum_{m=0}^K\hat{\eta}_l(t,x)\hat{\mu}_m(t,x)\langle\psi_l\psi_m\psi_k\rangle\right)\psi_k(\xi).
\end{equation}
This approximation is sometimes called the pseudospectral product.

To shorten the notation, we denote by $\boldsymbol{\hat{\eta}}\in\mathbb{R}^{K+1}$ the vector of gPCE expansion coefficients of $\eta$. Furthermore, we introduce the symmetric matrix $\mathcal{M}_k$ defined as
\begin{equation*}
    \left(\mathcal{M}_k\right)_{lm}=\langle\psi_k\psi_l\psi_m\rangle,\quad\mathcal{M}_k\in\mathbb{R}^{(K+1)\times(K+1)},
\end{equation*}
for each $k\in\{0,\ldots,K\}$ and the linear operator $\mathcal{P}\colon\mathbb{R}^{K+1}\rightarrow\mathbb{R}^{(K+1)\times (K+1)}$
\begin{equation}\label{SG-eq:GalerkinMatrix}
    \mathcal{P}(\boldsymbol{\hat{\eta}})\coloneqq\sum_{k=0}^K\hat{\eta}_k\mathcal{M}_k.
\end{equation}
Then the following properties hold
\begin{equation*}
    \widehat{\mathcal{G}_K[\eta,\mu]}=\mathcal{P}(\boldsymbol{\hat{\eta}})\boldsymbol{\hat{\mu}},\quad\mathcal{P}(\boldsymbol{\hat{\eta}})\boldsymbol{\hat{\mu}}=\mathcal{P}(\boldsymbol{\hat{\mu}})\boldsymbol{\hat{\eta}}.
\end{equation*}

Calculating the ratio of two stochastic processes $\eta$ and $\mu$ is also not so straightforward. If $\mu$ is a single-signed process, we note that the following expression is exact
\begin{equation*}
    \mathcal{G}_K[\eta](t,x,\xi)=\mathcal{G}_K\left[\mu\frac{\eta}{\mu}\right](t,x,\xi).
\end{equation*}
This motivates us to assume
\begin{equation*}
    \mathcal{G}_K[\eta](t,x,\xi)=\mathcal{G}_K\left[\mu,\frac{\eta}{\mu}\right](t,x,\xi),
\end{equation*}
which can be proved to imply
\begin{equation*}
    \mathcal{P}(\boldsymbol{\hat{\mu}})\widehat{\left(\frac{\eta}{\mu}\right)}=\boldsymbol{\hat{\eta}}.
\end{equation*}
This gives rise to the definition of a new operator
\begin{equation}\label{SG-eq:fraction}
    \mathcal{G}_K^\dagger\left[\frac{\eta}{\mu}\right](t,x,\xi)\coloneqq\sum_{k=0}^K\left(\mathcal{P}(\boldsymbol{\hat{\mu}}(t,x))^{-1}\boldsymbol{\hat{\eta}}(t,x)\right)_k\psi_k(\xi),
\end{equation}
assuming $\mathcal{P}(\boldsymbol{\hat{\mu}}(t,x))$ is invertible.

%% file: Sections/04_SG_For_SWLME.tex
Taking into account uncertainty, the SWLME variables $h, q_0,q_1,\ldots,q_N$ no longer depend on time and space alone, but also on a random parameter of the system. The generalised polynomial chaos expansion gives a way to separate the time and space dependence from the dependence on the random parameter. If we assume that the random variable $\xi$ is such that the assumptions of Theorem \ref{SG-th:gPCE} are satisfied and that the moment problem is uniquely solvable, then we can use the gPCE from Theorem \ref{SG-th:gPCE} for the SWLME variables
\begin{equation}\label{SGSWME-eq:gPCE_hq}
        h(t,x,\cdot)\approx\sum_{k=0}^K\hat{h}_k(t,x)\psi_k(\xi),\quad
        q_0(t,x,\cdot)\approx\sum_{k=0}^K\hat{q}_{0,k}(t,x)\psi_k(\xi), \quad\ldots,\quad
        q_N(t,x,\cdot)\approx\sum_{k=0}^K\hat{q}_{N,k}(t,x)\psi_k(\xi),
\end{equation}
if $h(t,x,\cdot),q_0(t,x,\cdot),\ldots,q_N(t,x,\cdot)\in L^2(\Omega,\sigma(\xi),\mathds{P})$. Note that we choose to expand $q_0,q_1,\ldots,q_N$ instead of $u_m,u_1,\ldots,u_N$. In this way, we extend the set of convective variables used in \cite{dai_hyperbolicity-preserving_2021} to stay within their framework. Furthermore, this may have slight benefits in the performance of conservative numerical methods, since in the original SWE case the product $q_0=hu_m$ is a conserved quantity, but not $u_m$ itself.

We now insert the expansions \eqref{SGSWME-eq:gPCE_hq} into the system \eqref{SWME-eq:SWLME_N_matrix}, but first we have to make crucial assumptions about how we approximate the nonlinear terms $\frac{q_j^2}{h}$, $\frac{2q_0q_j}{h}$ and $\frac{q_j}{h^2}$, since this might impact the approximations. In \cite{dai_hyperbolicity-preserving_2021}, the only nonlinear term that occurred was $\frac{q_0^2}{h}$, which was projected as
\begin{equation*}
    \mathcal{G}_K\left[\frac{q_0^2}{h}\right]=\mathcal{G}_K\left[q_0,\mathcal{G}^\dagger_K\left[\frac{q_0}{h}\right]\right]
\end{equation*}
using the definition of the operator $\mathcal{G}_K^\dagger$ from \eqref{SG-eq:fraction}. This led to a symmetric formulation that later ensured hyperbolicity. In a similar fashion, we observe
\begin{equation*}
    \begin{split}
        \mathcal{G}_K\left[q_i,\mathcal{G}^\dagger_K\left[\frac{q_j}{h}\right]\right] = \sum_{k=0}^K\left(\mathcal{P}(\boldsymbol{\hat{q_i}})\mathcal{P}(\boldsymbol{\hat{h}})^{-1}\boldsymbol{\hat{q_j}}\right)_k\psi_k(\xi)&=
        \sum_{k=0}^K\left(\mathcal{P}(\mathcal{P}(\boldsymbol{\hat{h}})^{-1}\boldsymbol{\hat{q_j}})\boldsymbol{\hat{q_i}}\right)_k\psi_k(\xi) = \mathcal{G}_K\left[\mathcal{G}^\dagger_K\left[\frac{q_j}{h}\right],q_i\right]\\
        &\neq \sum_{k=0}^K\left(\mathcal{P}(\boldsymbol{\hat{q_j}})\mathcal{P}(\boldsymbol{\hat{h}})^{-1}\boldsymbol{\hat{q_i}}\right)_k\psi_k(\xi)=\mathcal{G}_K\left[q_j,\mathcal{G}^\dagger_K\left[\frac{q_i}{h}\right]\right],\\
    \end{split}
\end{equation*}
and
\begin{equation*}
    \mathcal{G}_K\left[\mathcal{G}^\dagger_K\left[\frac{1}{h}\right],\mathcal{G}^\dagger_K\left[\frac{q_j}{h}\right]\right] = \sum_{k=0}^K\left(\mathcal{P}(\mathcal{P}(\boldsymbol{\hat{h}})^{-1}\boldsymbol{\hat{1}})\mathcal{P}(\boldsymbol{\hat{h}})^{-1}\boldsymbol{\hat{q_j}}\right)_k\psi_k(\xi) \neq 
    \sum_{k=0}^K\left(\mathcal{P}(\boldsymbol{\hat{h}})^{-1}\mathcal{P}(\boldsymbol{\hat{h}})^{-1}\boldsymbol{\hat{q_j}}\right)_k\psi_k(\xi) = \mathcal{G}^\dagger_K\left[\frac{\mathcal{G}^\dagger_K\left[\frac{q_j}{h}\right]}{h}\right].
\end{equation*}
    
Inspired by the treatment in \cite{dai_hyperbolicity-preserving_2021}, we propose here the use of the following projections:
\begin{equation}\label{SGSWME-eq:conventions}
    \begin{split}
        \frac{q_j^2}{h}=q_j\frac{q_j}{h}\quad&\Rightarrow\quad\mathcal{G}_K\left[\frac{q_j^2}{h}\right]=\mathcal{G}_K\left[q_j,\mathcal{G}^\dagger_K\left[\frac{q_j}{h}\right]\right],\\
        \frac{2q_0q_j}{h}=q_0\frac{q_j}{h}+q_j\frac{q_0}{h}\quad&\Rightarrow\quad\mathcal{G}_K\left[\frac{2q_0q_j}{h}\right]=\mathcal{G}_K\left[q_0,\mathcal{G}^\dagger_K\left[\frac{q_j}{h}\right]\right]+\mathcal{G}_K\left[q_j,\mathcal{G}^\dagger_K\left[\frac{q_0}{h}\right]\right],\\
        \frac{q_j}{h^2}=\frac{\frac{q_j}{h}}{h}\quad&\Rightarrow\quad\mathcal{G}_K\left[\frac{q_j}{h^2}\right]=\mathcal{G}^\dagger_K\left[\frac{\mathcal{G}^\dagger_K\left[\frac{q_j}{h}\right]}{h}\right].
    \end{split}
\end{equation}
The conventions chosen in equation \eqref{SGSWME-eq:conventions} also make it possible to reduce the final system of equations back to the system of equations already derived for the SWE case in \cite{dai_hyperbolicity-preserving_2021}, when setting $N=0$.

To now derive a system of equations for the stochastic Fourier coefficients $\boldsymbol{\hat{h}}, \boldsymbol{\hat{q_0}},\ldots,\boldsymbol{\hat{q_N}}$, we have to choose a parameter from which we will analyse the effect of its uncertainty on the solutions of the SWLME. We choose the friction coefficient $\nu$, which arguably affects both the friction inside the fluid and the friction at the bottom, which is one of the most important sources of uncertainty \cite{kreitmair_effect_2019}. Therefore, the uncertainty in $\nu$ is expected to propagate to the uncertainty in the flow solution. Conveniently, $\nu$ is not in the denominator, which helps to compute the projections of the source terms on the right-hand side. In \cite{dai_hyperbolicity-preserving_2021}, an uncertain bottom topography is used, which is modelled as a finite-dimensional random field. In that case, the random variable is multidimensional, making the resulting system of equations more complicated. Another possible choice for a random variable could be the slip length $\lambda$. However, since the slip length is in the denominator, one cannot rely on the general three-term recurrence relation that exists for every family of orthogonal polynomials:
\begin{equation}\label{SGSWME-eq:recurrence}
    \psi_{k+1}(\xi)=(\alpha_k\xi+\beta_k)\psi_k(\xi)+\gamma_k\psi_{k-1}(\xi),\quad\text{for }k=0,\ldots,K,
\end{equation}
where $\alpha_k\neq0$ for all $k$. This three-term recurrence relation is often used to simplify the right-hand side result of a stochastic Galerkin projection.

Using the conventions of equation \eqref{SGSWME-eq:conventions}, we can insert the gPCEs of $\boldsymbol{\hat{h}}, \boldsymbol{\hat{q_0}},\ldots,\boldsymbol{\hat{q_N}}$ of equation \eqref{SGSWME-eq:gPCE_hq} into the system of equations \eqref{SWME-eq:SWLME_N_matrix} and then take the inner product with test functions $\psi_k$ for $k=0,\ldots,K$. Note that for some combinations of products and fractions of functions taking derivatives and calculating a stochastic Galerkin projection do not commute
\begin{equation*}
    \begin{split}
        \partial_x\left(\frac{q_iq_j}{h}\right)=\frac{q_i}{h}\partial_xq_j+\frac{q_j}{h}\partial_xq_i-\frac{q_iq_j}{h^2}\partial_xh\quad\Rightarrow\quad&\mathcal{P}(\mathcal{P}(\boldsymbol{\hat{h}})^{-1}\boldsymbol{\hat{q_i}})\partial_x\boldsymbol{\hat{q_j}}+\mathcal{P}(\mathcal{P}(\boldsymbol{\hat{h}})^{-1}\boldsymbol{\hat{q_j}})\partial_x\boldsymbol{\hat{q_i}}\\
        &-\mathcal{P}(\mathcal{P}(\mathcal{P}(\boldsymbol{\hat{h}})^{-1}\boldsymbol{\hat{q_i}})\mathcal{P}(\boldsymbol{\hat{h}})^{-1}\boldsymbol{\hat{q_i}})\partial_x\boldsymbol{\hat{h}},\\
        \partial_x\left(\frac{q_iq_j}{h}\right)=\partial_x\left(q_i\frac{q_j}{h}\right)\quad\Rightarrow\quad&\partial_x\left(\mathcal{P}(\boldsymbol{\hat{q_i}})\mathcal{P}(\boldsymbol{\hat{h}})^{-1}\boldsymbol{\hat{q_j}}\right) = \mathcal{P}(\boldsymbol{\hat{q_i}})\mathcal{P}(\boldsymbol{\hat{h}})^{-1}\partial_x\boldsymbol{\hat{q_j}}\\
        &+\mathcal{P}(\mathcal{P}(\boldsymbol{\hat{h}})^{-1}\boldsymbol{\hat{q_j}})\partial_x\boldsymbol{\hat{q_i}}-\mathcal{P}(\boldsymbol{\hat{q_i}})\mathcal{P}(\boldsymbol{\hat{h}})^{-1}\mathcal{P}(\mathcal{P}(\boldsymbol{\hat{h}})^{-1}\boldsymbol{\hat{q_j}})\partial_x\boldsymbol{\hat{h}},
    \end{split}  
\end{equation*}
such that we must make sure to first perform the stochastic Galerkin projection and then calculate the system matrix of the stochastic Galerkin system. 

Assuming normalised orthogonal polynomials $\psi_k$, we find the following system
\begin{equation}\label{SGSWME-eq:almost_final}
    \partial_t
    \begin{pmatrix}
        \boldsymbol{\hat{h}} \\ \boldsymbol{\hat{q_0}} \\ \boldsymbol{\hat{q_1}} \\ \vdots \\ \boldsymbol{\hat{q_N}}
    \end{pmatrix}
    + \partial_x
    \begin{pmatrix}
        \boldsymbol{\hat{q_0}} \\
        \frac{g}{2}\mathcal{P}(\boldsymbol{\hat{h}})\boldsymbol{\hat{h}}+\sum\limits_{j=0}^N\frac{\mathcal{P}(\boldsymbol{\hat{q_j}})\boldsymbol{\hat{u_j}}}{2j+1}\\
        \mathcal{P}(\boldsymbol{\hat{q_0}})\boldsymbol{\hat{u_1}}+\mathcal{P}(\boldsymbol{\hat{q_1}})\boldsymbol{\hat{u_0}}\\
        \vdots\\
        \mathcal{P}(\boldsymbol{\hat{q_0}})\boldsymbol{\hat{u_N}}+\mathcal{P}(\boldsymbol{\hat{q_N}})\boldsymbol{\hat{u_0}}\\
    \end{pmatrix}
    = \boldsymbol{\hat{Q}}\left(\boldsymbol{\hat{h}}, \boldsymbol{\hat{q_0}},\ldots, \boldsymbol{\hat{q_N}}\right) + \boldsymbol{\hat{S}}\left(\boldsymbol{\hat{h}}, \boldsymbol{\hat{q_0}},\ldots,\boldsymbol{\hat{q_N}}\right),
\end{equation}
where $\boldsymbol{\hat{Q}}$ and $\boldsymbol{\hat{S}}$ denote the non-conservative terms and the friction vectors, respectively, just as in equation \eqref{SWME-eq:SWLME_N_matrix}. Here we make use of the notation
\begin{equation*}
    \boldsymbol{\hat{u_0}}\coloneqq\mathcal{P}(\boldsymbol{\hat{h}})^{-1}\boldsymbol{\hat{q_0}},\quad\boldsymbol{\hat{u_1}}\coloneqq\mathcal{P}(\boldsymbol{\hat{h}})^{-1}\boldsymbol{\hat{q_1}},\quad\ldots,\quad\boldsymbol{\hat{u_N}}\coloneqq\mathcal{P}(\boldsymbol{\hat{h}})^{-1}\boldsymbol{\hat{q_N}}.
\end{equation*}
For vector $\boldsymbol{\hat{Q}}\in\mathbb{R}^{(N+2)(K+1)}$ we cannot rely on a conservative formulation and calculate the derivative after the stochastic Galerkin projection. We project the vector $\boldsymbol{\hat{Q}}$ of equation \eqref{SWME-eq:SWLME_N_matrix} as
\begin{equation}
    \boldsymbol{\hat{Q}}\left(\boldsymbol{\hat{h}},\boldsymbol{\hat{q_0}},\ldots, \boldsymbol{\hat{q_N}}\right) = 
    \begin{pmatrix}
        \boldsymbol{0} \\ \boldsymbol{0} \\ \mathcal{P}(\boldsymbol{\hat{u_0}})\partial_x\boldsymbol{\hat{q_1}} \\ \mathcal{P}(\boldsymbol{\hat{u_0}})\partial_x\boldsymbol{\hat{q_2}} \\ \vdots \\ \mathcal{P}(\boldsymbol{\hat{u_0}})\partial_x\boldsymbol{\hat{q_N}}
    \end{pmatrix},
\end{equation}
such that it will cancel out some terms of the stochastic Galerkin SWLME system matrix later. For the friction vector $\boldsymbol{\hat{S}}\in\mathbb{R}^{(N+2)(K+1)}$, where the friction coefficient follows some distribution $\nu\sim\Phi(\mu,\sigma)$, we find
\begin{equation}\label{SGSWME-eq:S2}
    \boldsymbol{\hat{S}}\left(\boldsymbol{\hat{h}}, \boldsymbol{\hat{q_0}},\ldots, \boldsymbol{\hat{q_N}}\right) = 
    \begin{pmatrix}
        \boldsymbol{0} \\
        -\frac{T}{\lambda}\left(\sum\limits_{j=0}^N\boldsymbol{\hat{u_j}}\right) \\
        -\frac{3T}{\lambda}\left(\sum\limits_{j=0}^N\boldsymbol{\hat{u_j}}\right)-12T\left(\sum\limits_{j=1}^Na_{1,j}\mathcal{P}(\boldsymbol{\hat{h}})^{-1}\boldsymbol{\hat{u_j}}\right) \\
        \vdots \\
        -\frac{(2N+1)T}{\lambda}\left(\sum\limits_{j=0}^N\boldsymbol{\hat{u_j}}\right)-4(2N+1)T\left(\sum\limits_{j=1}^Na_{N,j}\mathcal{P}(\boldsymbol{\hat{h}})^{-1}\boldsymbol{\hat{u_j}}\right)
    \end{pmatrix},
\end{equation}
where $T$ is the $(K+1)\times(K+1)$-tridiagonal matrix originating from the recurrence formula \eqref{SGSWME-eq:recurrence} given by
\begin{equation*}
    T =
    \begin{pmatrix}
        \mu-\frac{\sigma\beta_0}{\alpha_0} & -\frac{\sigma\gamma_1}{\alpha_1}\\
        \frac{\sigma}{\alpha_0} & \mu-\frac{\sigma\beta_1}{\alpha_1} & -\frac{\sigma\gamma_2}{\alpha_2}\\
         & \ddots & \ddots & \ddots \\
         & & \frac{\sigma}{\alpha_{K-2}}& \mu-\frac{\sigma\beta_{K-1}}{\alpha_{K-1}} & -\frac{\sigma\gamma_K}{\alpha_K}\\
         & & & \frac{\sigma}{\alpha_{K-1}}& \mu-\frac{\sigma\beta_K}{\alpha_K}
    \end{pmatrix}
\end{equation*}
with the $\alpha_k$, $\beta_k$ and $\gamma_k$ for $k=0,\ldots,K$ as in equation \eqref{SGSWME-eq:recurrence} for the appropriate set of polynomials for the distribution of $\nu$ as in Theorem \ref{SG-th:gPCE} and the $a_{j,k}$ coefficients given by \eqref{SWME-eq:ajk}.
We call the system of $(N+2)(K+1)$ equations of \eqref{SGSWME-eq:almost_final} the Stochastic Galerkin Shallow Water Linearised Moment Equations (SGSWLME). Writing the system in the form of equation \eqref{SWME-eq:hyp}, the system matrix $A_{\text{SGSWLME}}$ of this system of equations with respect to the variable vector $\boldsymbol{\hat{W}}=(\boldsymbol{\hat{h}},\boldsymbol{\hat{q_0}},\boldsymbol{\hat{q_1}},\ldots,\boldsymbol{\hat{q_N}})^T$ is given by
\begin{equation}\label{SGSWME-eq:SGSWLME_flux}
    \resizebox{\textwidth}{!}{$
    \begin{pmatrix}
        O & I & O & \cdots & O\\
        g\mathcal{P}(\boldsymbol{\hat{h}})-\sum\limits_{j=0}^N\frac{\mathcal{P}(\boldsymbol{\hat{q_j}})\mathcal{P}(\boldsymbol{\hat{h}})^{-1}\mathcal{P}(\boldsymbol{\hat{u_j}})}{2j+1}& \mathcal{P}(\boldsymbol{\hat{q_0}})\mathcal{P}(\boldsymbol{\hat{h}})^{-1}+\mathcal{P}(\boldsymbol{\hat{u_0}}) & \frac{\mathcal{P}(\boldsymbol{\hat{q_1}})\mathcal{P}(\boldsymbol{\hat{h}})^{-1}+\mathcal{P}(\boldsymbol{\hat{u_1}})}{3} & \cdots & \frac{\mathcal{P}(\boldsymbol{\hat{q_N}})\mathcal{P}(\boldsymbol{\hat{h}})^{-1}+\mathcal{P}(\boldsymbol{\hat{u_N}})}{2N+1}\\
        -\mathcal{P}(\boldsymbol{\hat{q_0}})\mathcal{P}(\boldsymbol{\hat{h}})^{-1}\mathcal{P}(\boldsymbol{\hat{u_1}})-\mathcal{P}(\boldsymbol{\hat{q_1}})\mathcal{P}(\boldsymbol{\hat{h}})^{-1}\mathcal{P}(\boldsymbol{\hat{u_0}})& \mathcal{P}(\boldsymbol{\hat{q_1}})\mathcal{P}(\boldsymbol{\hat{h}})^{-1} + \mathcal{P}(\boldsymbol{\hat{u_1}}) & \mathcal{P}(\boldsymbol{\hat{q_0}})\mathcal{P}(\boldsymbol{\hat{h}})^{-1} & & \\
        \vdots & \vdots &  & \ddots & \\
        -\mathcal{P}(\boldsymbol{\hat{q_0}})\mathcal{P}(\boldsymbol{\hat{h}})^{-1}\mathcal{P}(\boldsymbol{\hat{u_N}})-\mathcal{P}(\boldsymbol{\hat{q_N}})\mathcal{P}(\boldsymbol{\hat{h}})^{-1}\mathcal{P}(\boldsymbol{\hat{u_0}})& \mathcal{P}(\boldsymbol{\hat{q_N}})\mathcal{P}(\boldsymbol{\hat{h}})^{-1}+\mathcal{P}(\boldsymbol{\hat{u_N}}) &  & & \mathcal{P}(\boldsymbol{\hat{q_0}})\mathcal{P}(\boldsymbol{\hat{h}})^{-1}
    \end{pmatrix}.$}
\end{equation}
For $N=0$ the SWLME reduce to the SWE. Consistently, the newly derived system matrix of the SGSWLME also reduces to the system matrix of the SGSWE given by
\begin{equation}\label{SGSWME-eq:SGSWE_flux}
    A_{\text{SGSWE}}(\boldsymbol{\hat{W}})=
    \begin{pmatrix}
        O & I\\
        g\mathcal{P}(\boldsymbol{\hat{h}})-\mathcal{P}(\boldsymbol{\hat{q_0}})\mathcal{P}(\boldsymbol{\hat{h}})^{-1}\mathcal{P}(\boldsymbol{\hat{u_0}}) & \mathcal{P}(\boldsymbol{\hat{q_0}})\mathcal{P}(\boldsymbol{\hat{h}})^{-1}+\mathcal{P}(\boldsymbol{\hat{u_0}})\\
    \end{pmatrix},
\end{equation}
in \cite{dai_hyperbolicity-preserving_2021}. The final system of equations for the SGSWLME then reads
\begin{equation}\label{SGSWME-eq:final}
    \partial_t
    \begin{pmatrix}
        \boldsymbol{\hat{h}} \\ \boldsymbol{\hat{q_0}} \\ \boldsymbol{\hat{q_1}} \\ \vdots \\ \boldsymbol{\hat{q_N}}
    \end{pmatrix}
    +  A_{\text{SGSWLME}}\left(\boldsymbol{\hat{h}}, \boldsymbol{\hat{q_0}},\boldsymbol{\hat{q_1}},\ldots, \boldsymbol{\hat{q_N}}\right) \partial_x
    \begin{pmatrix}
        \boldsymbol{\hat{h}} \\ \boldsymbol{\hat{q_0}} \\ \boldsymbol{\hat{q_1}} \\ \vdots \\ \boldsymbol{\hat{q_N}}
    \end{pmatrix}
    = \boldsymbol{\hat{S}}\left(\boldsymbol{\hat{h}}, \boldsymbol{\hat{q_0}},\boldsymbol{\hat{q_1}},\ldots, \boldsymbol{\hat{q_N}}\right).
\end{equation}

%% file: Sections/05_Analysis.tex
As an advantage of the intrusive SG formulation, the analytically derived SGSWLME from Section \ref{section:SGSWME} allow for an in-depth analysis with respect to the properties of the resulting system. Here we focus on two important properties: energy preservation in Section \ref{subsection:energy} and hyperbolicity in Section \ref{subsection:hyperbolicity}. In Section \ref{subsection:explicit} we write down the SGSWLME for a specific choice of basis and $K=1$ and observe its similarity to the standard SWLME.

\subsection{Energy Equation}\label{subsection:energy}
For the SWE \eqref{SWME-eq:SWE_matrix}, it is well known that an energy equation can be derived \cite{gassner_well_2016}. In conserved quantities, it is given by
\begin{equation*}
    \partial_t\left(\frac{gh^2}{2}+\frac{q_0^2}{2h}\right)+\partial_x\left(ghq_0+\frac{q_0^3}{2h^2}\right)=-\frac{\nu}{\lambda}\left(\frac{q_0}{h}\right)^2,
\end{equation*}
in which the energy $e=\frac{gh^2}{2}+\frac{q_0^2}{2h}$ is a conserved quantity and also serves as an entropy, while $f=ghq_0+\frac{q_0^3}{2h^2}$ is the entropy flux. The existence of an energy equation allows for tailored numerical methods to preserve energy discretely during numerical simulations \cite{tadmor_numerical_1987, bohm_entropy_2020, ersing_entropy_2024, ersing_entropy_2025, fan_well-balanced_2026, gassner_well_2016}.

The energy equation for the SWLME is derived in \cite{koellermeier_new_2026,caballero-cardenas_semi-implicit_2025,careaga_entropy_2026}. For the SWLME of equation \eqref{SWME-eq:SWLME_N_matrix} the energy equation in conserved quantities is given by
\begin{equation}\label{ANAL-eq:energy_SWLME}
    \partial_t\left(\frac{gh^2}{2}+\frac{1}{2h}\sum_{j=0}^N \frac{q_j^2}{2 j+1}\right)+\partial_x\left(ghq_0+\frac{q_0^3}{2h^2}+\frac{3q_0}{2h^2}\sum_{j=1}^N\frac{q_j^2}{2j+1}\right)=-\frac{\nu}{\lambda h^2}\sum_{i,j=0}^N(2j+1)q_iq_j-\frac{4\nu}{h^3}\sum_{i,j=1}^N(2j+1)a_{i,j}q_iq_j,
\end{equation}
where we denote the total energy $\displaystyle e=\frac{gh^2}{2}+\frac{1}{2h}\sum_{j=0}^N \frac{q_j^2}{2 j+1}$. It can be seen directly that the total energy $e$ is a conserved quantity and an entropy function with entropy flux $\displaystyle f=ghq_0+\frac{q_0^3}{2h^2}+\frac{3q_0}{2h^2}\sum_{j=1}^N\frac{q_j^2}{2j+1}$.

Based on the energy equation of the SWLME \eqref{ANAL-eq:energy_SWLME}, we can now derive the energy equation for the SGSWLME \eqref{SGSWME-eq:final} as stated in the following theorem.
\begin{theorem}\label{ANAL-th:energy_SGSWLME}
    The energy equation for the SGSWLME \eqref{SGSWME-eq:final} reads
    \begin{equation}\label{ANAL-eq:energy_SGSWLME}
        \begin{split}
            \partial_t\left(\frac{g\boldsymbol{\hat{h}}^T\boldsymbol{\hat{h}}}{2}+\frac{1}{2}\sum_{j=0}^N\frac{\boldsymbol{\hat{u_j}}^T\mathcal{P}(\boldsymbol{\hat{h}})\boldsymbol{\hat{u_j}}}{2j+1}\right)
            + \partial_x\left(g\boldsymbol{\hat{h}}^T\mathcal{P}(\boldsymbol{\hat{h}})\boldsymbol{\hat{u_0}}+\frac{1}{2}\sum_{j=0}^N\frac{\boldsymbol{\hat{u_j}}^T\mathcal{P}(\boldsymbol{\hat{q_0}})\boldsymbol{\hat{u_j}}}{2j+1}+\sum_{j=1}^N\frac{\boldsymbol{\hat{u_j}}^T\mathcal{P}(\boldsymbol{\hat{q_j}})\boldsymbol{\hat{u_0}}}{2j+1}\right) \\
            = -\sum_{i,j=0}^N\frac{(2j+1)\boldsymbol{\hat{u_i}}^TT\boldsymbol{\hat{u_j}}}{\lambda}-4\sum_{i,j=1}^N(2j+1)a_{i,j}\boldsymbol{\hat{u_i}}^TT \mathcal{P}(\boldsymbol{\hat{h}})^{-1}\boldsymbol{\hat{u_j}}.
        \end{split}
    \end{equation}
    This means that the total energy is given by $\displaystyle e=\frac{g\boldsymbol{\hat{h}}^T\boldsymbol{\hat{h}}}{2}+\frac{1}{2}\sum_{j=0}^N\frac{\boldsymbol{\hat{u_j}}^T\mathcal{P}(\boldsymbol{\hat{h}})\boldsymbol{\hat{u_j}}}{2j+1}$ and the energy flux is equal to $\displaystyle f=g\boldsymbol{\hat{h}}^T\mathcal{P}(\boldsymbol{\hat{h}})\boldsymbol{\hat{u_0}}+\frac{1}{2}\sum_{j=0}^N\frac{\boldsymbol{\hat{u_j}}^T\mathcal{P}(\boldsymbol{\hat{q_0}})\boldsymbol{\hat{u_j}}}{2j+1}+\sum_{j=1}^N\frac{\boldsymbol{\hat{u_j}}^T\mathcal{P}(\boldsymbol{\hat{q_j}})\boldsymbol{\hat{u_0}}}{2j+1}$. 
\end{theorem}

\begin{proof}
    To derive the energy equation for the SGSWLME, we follow the steps of \cite{koellermeier_new_2026}. We first write the SGSWLME for arbitrary $N$ \eqref{SGSWME-eq:almost_final}.
    \begin{subequations}
        \begin{alignat}{2}
            \text{(C)}:\quad&\partial_t\boldsymbol{\hat{h}}+\partial_x\boldsymbol{\hat{q_0}}&&=\boldsymbol{0},\label{ANAL-eq:SWME_N=1_a}\\
            \text{(M)}:\quad&\partial_t\boldsymbol{\hat{q_0}}+\partial_x\left(\frac{g}{2}\mathcal{P}(\boldsymbol{\hat{h}})\boldsymbol{\hat{h}}+\sum_{j=0}^N\frac{\mathcal{P}(\boldsymbol{\hat{q_j}})\boldsymbol{\hat{u_j}}}{2j+1}\right)&&=-\frac{T}{\lambda}\left(\sum_{j=0}^N\boldsymbol{\hat{u_j}}\right),\label{ANAL-eq:SWME_N=1_b}\\
            \text{($u_i$)}:\quad&\partial_t\boldsymbol{\hat{q_i}}+\partial_x\left(\mathcal{P}(\boldsymbol{\hat{q_0}})\boldsymbol{\hat{u_i}}+\mathcal{P}(\boldsymbol{\hat{q_i}})\boldsymbol{\hat{u_0}}\right)&&=\mathcal{P}\left(\boldsymbol{\hat{u_0}}\right)\partial_x\boldsymbol{\hat{q_i}}-\frac{(2i+1)T}{\lambda}\left(\sum_{j=0}^N\boldsymbol{\hat{u_j}}\right)-4(2i+1)T \left(\sum_{j=1}^Na_{i,j}\mathcal{P}(\boldsymbol{\hat{h}})^{-1}\boldsymbol{\hat{u_j}}\right).\label{ANAL-eq:SWME_N=1_c}
        \end{alignat}
    \end{subequations}
    Here, (C) indicates the continuity equation, (M) the momentum balance, and ($u_i$) the moment coefficient $i$ balance for $i=1,\ldots,N$. The potential energy (P) is now obtained by computing $\text{(P)}=g\boldsymbol{\hat{h}}^T\cdot\text{(C)}$, which results in
    \begin{equation}
        \text{(P)}:\quad\partial_t\left(\frac{1}{2}g\boldsymbol{\hat{h}}^T\boldsymbol{\hat{h}}\right)+g\boldsymbol{\hat{h}}^T\partial_x\left(\mathcal{P}(\boldsymbol{\hat{h}})\boldsymbol{\hat{u_0}}\right)=0.
    \end{equation}
    Secondly, our aim is to derive an equation for the (averaged) kinetic energy of the momentum balance. To find this, we first compute $\text{(A)}=\text{(M)}-\mathcal{P}(\boldsymbol{\hat{u_0}})\cdot\text{(C)}$. 
    \begin{equation}
        \begin{split}
            \underbrace{\partial_t\left(\mathcal{P}(\boldsymbol{\hat{h}})\boldsymbol{\hat{u_0}}\right)-\mathcal{P}(\boldsymbol{\hat{u_0}})\partial_t\boldsymbol{\hat{h}}}+\underbrace{\partial_x\left(\mathcal{P}(\boldsymbol{\hat{q_0}})\boldsymbol{\hat{u_0}}\right)-\mathcal{P}(\boldsymbol{\hat{u_0}})\partial_x\boldsymbol{\hat{q_0}}}+\partial_x\left(\frac{g}{2}\mathcal{P}(\boldsymbol{\hat{h}})\boldsymbol{\hat{h}}+\sum_{j=1}^N\frac{\mathcal{P}(\boldsymbol{\hat{q_j}})\boldsymbol{\hat{u_j}}}{2j+1}\right) &=-\frac{T}{\lambda}\left(\sum_{j=0}^N\boldsymbol{\hat{u_j}}\right),\\
            \text{(A)}:\quad\quad\quad\quad\mathcal{P}(\boldsymbol{\hat{h}})\partial_t\boldsymbol{\hat{u_0}}\quad\quad\quad\quad+\mathcal{P}(\boldsymbol{\hat{q_0}})\partial_x\boldsymbol{\hat{u_0}}\quad\quad\quad+g\mathcal{P}(\boldsymbol{\hat{h}})\partial_x\boldsymbol{\hat{h}}+\partial_x\left(\sum_{j=1}^N\frac{\mathcal{P}(\boldsymbol{\hat{q_j}})\boldsymbol{\hat{u_j}}}{2j+1}\right) &=-\frac{T}{\lambda}\left(\sum_{j=0}^N\boldsymbol{\hat{u_j}}\right).\\
        \end{split}
    \end{equation}
    Here we used the definition of $\boldsymbol{\hat{u_0}}$, the chain rule, and the fact that the following properties hold $\partial_t\left(\mathcal{P}(\boldsymbol{\hat{h}})\right)\boldsymbol{\hat{u_0}}=\mathcal{P}(\boldsymbol{\hat{u_0}})\partial_t\boldsymbol{\hat{h}}$ and $\partial_x\left(\mathcal{P}(\boldsymbol{\hat{q_0}})\right)\boldsymbol{\hat{u_0}}=\mathcal{P}(\boldsymbol{\hat{u_0}})\partial_x\boldsymbol{\hat{q_0}}$.
    
    Next, we average (A) and (M) and multiply by $\boldsymbol{\hat{u_0}}^T$ to obtain an equation for the kinetic energy of the momentum balance $\text{(K)}=\frac{1}{2}\boldsymbol{\hat{u_0}}^T\left[\text{(A)}+\text{(M)}\right]$
    \begin{equation}
        \resizebox{\textwidth}{!}{$
        \begin{aligned}
            \frac{1}{2}\left[\underbrace{\boldsymbol{\hat{u_0}}^T\mathcal{P}(\boldsymbol{\hat{h}})\partial_t\boldsymbol{\hat{u_0}}+\boldsymbol{\hat{u_0}}^T\partial_t\left(\mathcal{P}(\boldsymbol{\hat{h}})\boldsymbol{\hat{u_0}}\right)}+\underbrace{\boldsymbol{\hat{u_0}}^T\mathcal{P}(\boldsymbol{\hat{q_0}})\partial_x\boldsymbol{\hat{u_0}}+\boldsymbol{\hat{u_0}}^T\partial_x\left(\mathcal{P}(\boldsymbol{\hat{q_0}})\boldsymbol{\hat{u_0}}\right)}\right]+g\boldsymbol{\hat{u_0}}^T\mathcal{P}(\boldsymbol{\hat{h}})\partial_x\boldsymbol{\hat{h}}+\boldsymbol{\hat{u_0}}^T\partial_x\left(\sum_{j=1}^N\frac{\mathcal{P}(\boldsymbol{\hat{q_j}})\boldsymbol{\hat{u_j}}}{2j+1}\right)&= -\frac{\boldsymbol{\hat{u_0}}^TT}{\lambda}\left(\sum_{j=0}^N\boldsymbol{\hat{u_j}}\right),\\
            \text{(K)}:\quad\quad\quad\quad\quad\partial_t\left(\frac{\boldsymbol{\hat{u_0}}^T\mathcal{P}(\boldsymbol{\hat{h}})\boldsymbol{\hat{u_0}}}{2}\right)\quad\quad\quad\quad\quad\quad+\partial_x\left(\frac{\boldsymbol{\hat{u_0}}^T\mathcal{P}(\boldsymbol{\hat{q_0}})\boldsymbol{\hat{u_0}}}{2}\right)\quad\quad\quad+g\boldsymbol{\hat{u_0}}^T\mathcal{P}(\boldsymbol{\hat{h}})\partial_x\boldsymbol{\hat{h}}+\boldsymbol{\hat{u_0}}^T\partial_x\left(\sum_{j=1}^N\frac{\mathcal{P}(\boldsymbol{\hat{q_j}})\boldsymbol{\hat{u_j}}}{2j+1}\right)&= -\frac{\boldsymbol{\hat{u_0}}^TT}{\lambda}\left(\sum_{j=0}^N\boldsymbol{\hat{u_j}}\right).\\
        \end{aligned}$}
    \end{equation}
    Third, we compute an (averaged) kinetic energy for the balance of the momentum of coefficient $i$ ($u_i$) by following the same procedure as above. We start by computing $\text{(A$u_i$)}=\text{($u_i$)}-\mathcal{P}(\boldsymbol{\hat{u_i}})\cdot\text{(C)}$
    \begin{equation}
        \resizebox{\textwidth}{!}{$
        \begin{aligned}
            \underbrace{\partial_t\left(\mathcal{P}(\boldsymbol{\hat{h}})\boldsymbol{\hat{u_i}}\right)-\mathcal{P}(\boldsymbol{\hat{u_i}})\partial_t\boldsymbol{\hat{h}}}+\underbrace{\partial_x\left(\mathcal{P}(\boldsymbol{\hat{q_0}})\boldsymbol{\hat{u_i}}\right)-\mathcal{P}(\boldsymbol{\hat{u_i}})\partial_x\boldsymbol{\hat{q_0}}}+\underbrace{\partial_x\left(\mathcal{P}(\boldsymbol{\hat{q_i}})\boldsymbol{\hat{u_0}}\right)-\mathcal{P}\left(\boldsymbol{\hat{u_0}}\right)\partial_x\boldsymbol{\hat{q_i}}}&=-\frac{(2i+1)T}{\lambda}\left(\sum_{j=0}^N\boldsymbol{\hat{u_j}}\right)-4(2i+1)T \left(\sum_{j=1}^Na_{i,j}\mathcal{P}(\boldsymbol{\hat{h}})^{-1}\boldsymbol{\hat{u_j}}\right),\\
            \text{(A$u_i$)}:\quad\quad\mathcal{P}(\boldsymbol{\hat{h}})\partial_t\boldsymbol{\hat{u_i}}\quad\quad\quad\quad\quad+\mathcal{P}(\boldsymbol{\hat{q}})\partial_x\boldsymbol{\hat{u_i}}\quad\quad\quad\quad\quad\quad\quad+\mathcal{P}(\boldsymbol{\hat{q_i}})\partial_x\boldsymbol{\hat{u_0}}\quad\quad\quad&=-\frac{(2i+1)T}{\lambda}\left(\sum_{j=0}^N\boldsymbol{\hat{u_j}}\right)-4(2i+1)T \left(\sum_{j=1}^Na_{i,j}\mathcal{P}(\boldsymbol{\hat{h}})^{-1}\boldsymbol{\hat{u_j}}\right).
        \end{aligned}$}
    \end{equation}
    Then, again, we average $\text{(A$u_i$)}$ and $\text{($u_i$)}$ and multiply by $\boldsymbol{\hat{u_i}}^T$ to obtain an equation for the kinetic energy of the momentum balance of coefficient $i$ $\text{(K$u_i$)}=\frac{1}{2}\boldsymbol{\hat{u_i}}^T\text{[(A$u_i$)+($u_i$)]}$
    \begin{equation}
        \resizebox{\textwidth}{!}{$
        \begin{aligned}
            \frac{1}{2}\left[\underbrace{\boldsymbol{\hat{u_i}}^T\mathcal{P}(\boldsymbol{\hat{h}})\partial_t\boldsymbol{\hat{u_i}}+\boldsymbol{\hat{u_i}}^T\partial_t\left(\mathcal{P}(\boldsymbol{\hat{h}})\boldsymbol{\hat{u_i}}\right)}+\underbrace{\boldsymbol{\hat{u_i}}^T\mathcal{P}(\boldsymbol{\hat{q_0}})\partial_x\boldsymbol{\hat{u_i}}+\boldsymbol{\hat{u_i}}^T\partial_x\left(\mathcal{P}(\boldsymbol{\hat{q_0}})\boldsymbol{\hat{u_i}}\right)}\right.\quad\quad&\\
            \left.+\underbrace{\boldsymbol{\hat{u_i}}^T\mathcal{P}(\boldsymbol{\hat{q_i}})\partial_x\boldsymbol{\hat{u_0}}+\boldsymbol{\hat{u_i}}^T\partial_x\left(\mathcal{P}(\boldsymbol{\hat{q_i}})\boldsymbol{\hat{u_0}}\right)-\boldsymbol{\hat{u_i}}^T\mathcal{P}\left(\boldsymbol{\hat{u_0}}\right)\partial_x\boldsymbol{\hat{q_i}}}\right]&=-\frac{(2i+1)\boldsymbol{\hat{u_i}}^TT}{\lambda}\left(\sum_{j=0}^N\boldsymbol{\hat{u_j}}\right)-4(2i+1)\boldsymbol{\hat{u_i}}^TT \left(\sum_{j=1}^Na_{i,j}\mathcal{P}(\boldsymbol{\hat{h}})^{-1}\boldsymbol{\hat{u_j}}\right)\\
            \partial_t\left(\frac{\boldsymbol{\hat{u_i}}^T\mathcal{P}(\boldsymbol{\hat{h}})\boldsymbol{\hat{u_i}}}{2}\right)\quad\quad\quad\quad+\partial_x\left(\frac{\boldsymbol{\hat{u_i}}^T\mathcal{P}(\boldsymbol{\hat{q_0}})\boldsymbol{\hat{u_i}}}{2}\right)\quad\quad+\boldsymbol{\hat{u_i}}^T\mathcal{P}(\boldsymbol{\hat{q_i}})\partial_x\boldsymbol{\hat{u_0}}&=-\frac{(2i+1)\boldsymbol{\hat{u_i}}^TT}{\lambda}\left(\sum_{j=0}^N\boldsymbol{\hat{u_j}}\right)-4(2i+1)\boldsymbol{\hat{u_i}}^TT \left(\sum_{j=1}^Na_{i,j}\mathcal{P}(\boldsymbol{\hat{h}})^{-1}\boldsymbol{\hat{u_j}}\right),
        \end{aligned}$}
    \end{equation}
    or equivalently
    \begin{equation}
        \begin{split}
             \text{(K$u_i$)}:\quad\partial_t\left(\frac{\boldsymbol{\hat{u_i}}^T\mathcal{P}(\boldsymbol{\hat{h}})\boldsymbol{\hat{u_i}}}{2(2i+1)}\right)+\partial_x\left(\frac{\boldsymbol{\hat{u_i}}^T\mathcal{P}(\boldsymbol{\hat{q_0}})\boldsymbol{\hat{u_i}}}{2(2i+1)}\right)+\frac{\boldsymbol{\hat{u_i}}^T\mathcal{P}(\boldsymbol{\hat{q_i}})\partial_x\boldsymbol{\hat{u_0}}}{2i+1}=&-\frac{(2i+1)\boldsymbol{\hat{u_i}}^TT}{\lambda}\left(\sum\limits_{j=0}^N\boldsymbol{\hat{u_j}}\right)\\
             &\quad\quad-4(2i+1)\boldsymbol{\hat{u_i}}^TT \left(\sum\limits_{j=1}^Na_{i,j}\mathcal{P}(\boldsymbol{\hat{h}})^{-1}\boldsymbol{\hat{u_j}}\right).
         \end{split}
    \end{equation}
    We can then compute an equation for the total kinetic energy (K$u$) using the formula $\text{(K$u$)}=\text{(K)}+\displaystyle \sum_{i=1}^N\text{(K$u_i$)}$
    \begin{equation}
        \resizebox{\textwidth}{!}{$
        \begin{aligned}
            \partial_t\left(\frac{1}{2}\sum_{j=0}^N\frac{\boldsymbol{\hat{u_j}}^T\mathcal{P}(\boldsymbol{\hat{h}})\boldsymbol{\hat{u_j}}}{2j+1}\right)&+\partial_x\left(\frac{1}{2}\sum_{j=0}^N\frac{\boldsymbol{\hat{u_j}}^T\mathcal{P}(\boldsymbol{\hat{q_0}})\boldsymbol{\hat{u_j}}}{2j+1}\right)+g\boldsymbol{\hat{u_0}}^T\mathcal{P}(\boldsymbol{\hat{h}})\partial_x\boldsymbol{\hat{h}}\\
            &+\underbrace{\boldsymbol{\hat{u_0}}^T\partial_x\left(\sum_{j=1}^N\frac{\mathcal{P}(\boldsymbol{\hat{q_j}})\boldsymbol{\hat{u_j}}}{2j+1}\right)+\sum_{j=1}^N\frac{\boldsymbol{\hat{u_j}}^T\mathcal{P}(\boldsymbol{\hat{q_j}})}{2j+1}\partial_x\boldsymbol{\hat{u_0}}}=-\sum_{i,j=0}^N\frac{(2i+1)\boldsymbol{\hat{u_i}}^TT\boldsymbol{\hat{u_j}}}{\lambda}-4\sum_{i,j=1}^N(2i+1)a_{i,j}\boldsymbol{\hat{u_i}}^TT \mathcal{P}(\boldsymbol{\hat{h}})^{-1}\boldsymbol{\hat{u_j}},
        \end{aligned}$}
    \end{equation}
    such that we can write for the total kinetic energy (K$u$)
    \begin{equation}
        \begin{split}
            \partial_t\left(\frac{1}{2}\sum_{j=0}^N\frac{\boldsymbol{\hat{u_j}}^T\mathcal{P}(\boldsymbol{\hat{h}})\boldsymbol{\hat{u_j}}}{2j+1}\right)&+\partial_x\left(\frac{1}{2}\sum_{j=0}^N\frac{\boldsymbol{\hat{u_j}}^T\mathcal{P}(\boldsymbol{\hat{q_0}})\boldsymbol{\hat{u_j}}}{2j+1}+\sum_{j=1}^N\frac{\boldsymbol{\hat{u_j}}^T\mathcal{P}(\boldsymbol{\hat{q_j}})\boldsymbol{\hat{u_0}}}{2j+1}\right)\\
            &+g\boldsymbol{\hat{u_0}}^T\mathcal{P}(\boldsymbol{\hat{h}})\partial_x\boldsymbol{\hat{h}}=-\sum_{i,j=0}^N\frac{(2i+1)\boldsymbol{\hat{u_i}}^TT\boldsymbol{\hat{u_j}}}{\lambda}-4\sum_{i,j=1}^N(2i+1)a_{i,j}\boldsymbol{\hat{u_i}}^TT \mathcal{P}(\boldsymbol{\hat{h}})^{-1}\boldsymbol{\hat{u_j}}.
        \end{split}
    \end{equation}
    By adding the potential energy equation (P) and the total kinetic energy equation (K$u$), we obtain the equation for the total energy $\text{(E)}=\text{(P)}+\text{(K$u$)}$
    \begin{equation}
        \begin{split}
            \partial_t\left(\frac{g\boldsymbol{\hat{h}}^T\boldsymbol{\hat{h}}}{2}\right.&\left.+\frac{1}{2}\sum_{j=0}^N\frac{\boldsymbol{\hat{u_j}}^T\mathcal{P}(\boldsymbol{\hat{h}})\boldsymbol{\hat{u_j}}}{2j+1}\right)+\partial_x\left(\frac{1}{2}\sum_{j=0}^N\frac{\boldsymbol{\hat{u_j}}^T\mathcal{P}(\boldsymbol{\hat{q_0}})\boldsymbol{\hat{u_j}}}{2j+1}+\sum_{j=1}^N\frac{\boldsymbol{\hat{u_j}}^T\mathcal{P}(\boldsymbol{\hat{q_j}})\boldsymbol{\hat{u_0}}}{2j+1}\right)\\
            &+\underbrace{g\boldsymbol{\hat{h}}^T\partial_x\left(\mathcal{P}(\boldsymbol{\hat{h}})\boldsymbol{\hat{u_0}}\right)+g\boldsymbol{\hat{u_0}}^T\mathcal{P}(\boldsymbol{\hat{h}})\partial_x\boldsymbol{\hat{h}}}=-\sum_{i,j=0}^N\frac{(2i+1)\boldsymbol{\hat{u_i}}^TT\boldsymbol{\hat{u_j}}}{\lambda}-4\sum_{i,j=1}^N(2i+1)a_{i,j}\boldsymbol{\hat{u_i}}^TT \mathcal{P}(\boldsymbol{\hat{h}})^{-1}\boldsymbol{\hat{u_j}},
        \end{split}
    \end{equation}
    which can be rewritten as the sought energy equation for the SGSWLME system.
\end{proof}
The explicit derivation of the entropy and entropy flux allows us to compute the entropy variables like in \cite{koellermeier_new_2026}. Furthermore, it can be shown that the energy $e$ of Theorem \ref{ANAL-th:energy_SGSWLME} is convex and that the right-hand side energy source of \eqref{ANAL-eq:energy_SGSWLME} is dissipative, similar to \cite{bender_entropy-conservative_2024,epshteyn_structure-preserving_2026,careaga_entropy_2026}.

As shown, the intrusive SG approach allows the analytical derivation of an energy equation as a useful property. This can be taken into account in numerical schemes for the preservation of energy \cite{tadmor_numerical_1987, bohm_entropy_2020, ersing_entropy_2024, ersing_entropy_2025, fan_well-balanced_2026, gassner_well_2016}.

\subsection{Hyperbolicity}\label{subsection:hyperbolicity}
For the SGSWLME system matrix \eqref{SGSWME-eq:SGSWLME_flux}, hyperbolicity is an important property. If a system matrix is hyperbolic, pre-established numerical methods can be used to solve the system of PDEs. Furthermore, hyperbolicity guarantees finite propagation speeds, which makes sense in the context of shallow water flows. Choosing a set of basis polynomials and deciding the orders $N$ and $K$, the system matrix can be calculated exactly so that the hyperbolicity can be checked.

Ideally, hyperbolicity can be proven for general basis and general $N$ and $K$, just as was done in \cite{dai_hyperbolicity-preserving_2021} for the SWE for general $K$, which corresponds to the case of SWLME with $N=0$. In \cite{dai_hyperbolicity-preserving_2021}, block-matrix versions of the eigenvectors of the system matrix $A_{\text{SGSWE}}(\boldsymbol{\hat{W}})$ \eqref{SGSWME-eq:SGSWE_flux} were used to postulate a matrix $P$, such that $P^{-1}A_{\text{SGSWE}}(\boldsymbol{\hat{W}})P$ is a symmetric matrix. Therefore, the system matrix $A_{\text{SGSWE}}(\boldsymbol{\hat{W}})$ is similar to a diagonalisable matrix with real eigenvalues and thus itself real diagonalisable.

To prove the hyperbolicity of the system matrix of the SGSWLME \eqref{SGSWME-eq:SGSWLME_flux} we opt for the same approach and follow the notation of \cite{dai_hyperbolicity-preserving_2021}. We take order $N=1$ and construct a symmetrizer matrix $P$ for the system matrix of the SGSWLME. Looking at the eigenvectors of the deterministic SWLME in equation \eqref{SWME-eq:SWME_EV}, we define
\begin{equation}\label{ANAL-eq:blocks}
    G\coloneqq\sqrt{g\mathcal{P}(\boldsymbol{\hat{h}})+\mathcal{P}(\boldsymbol{\hat{u_1}})^2}, \quad A\coloneqq\mathcal{P}(\boldsymbol{\hat{u_1}}),\quad B\coloneqq\mathcal{P}(\boldsymbol{\hat{u_0}}),\quad C\coloneqq g\mathcal{P}(\boldsymbol{\hat{u_1}})^{-1}\mathcal{P}(\boldsymbol{\hat{h}}),
\end{equation}
where $\sqrt{M}$ is the (unique) symmetric positive definite square root of a symmetric positive definite matrix $M$, and suggest the following block matrix $P\in\mathbb{R}^{3(K+1)\times3(K+1)}$
\begin{equation}
    P = 
    \begin{pmatrix}
        I & I & I\\
        B-G & B & B+G\\
        2A & \frac{1}{2}A-\frac{3}{2}C & 2A\\
    \end{pmatrix}.
\end{equation}
The inverse of $P$ is then given by
\begin{equation}
    P^{-1}=-\frac{1}{6}(A+C)^{-1}
    \begin{pmatrix}
        A-3C-3(A+C)G^{-1}B & 3(A+C)G^{-1} & -2\\
        -8A & 0 & 4\\
        A-3C+3(A+C)G^{-1}B & -3(A+C)G^{-1} & -2\\
    \end{pmatrix},
\end{equation}
assuming $(A+C)$ is invertible. Ideally, the product $P^{-1}A_{\text{SGSWLME}}(\boldsymbol{\hat{W}})P$ produces a symmetric matrix, such that hyperbolicity is proved as in \cite{dai_hyperbolicity-preserving_2021}. However, this is not the case, unless we make the following four assumptions
\begin{equation}\label{ANAL-eq:assump}
    \begin{split}
        \mathcal{P}(\boldsymbol{\hat{u_0}})=\mathcal{P}(\boldsymbol{\hat{q_0}})\mathcal{P}(\boldsymbol{\hat{h}})^{-1},\quad&\mathcal{P}(\boldsymbol{\hat{u_0}})\mathcal{P}(\boldsymbol{\hat{u_1}})=\mathcal{P}(\boldsymbol{\hat{u_1}})\mathcal{P}(\boldsymbol{\hat{u_0}}),\\
        \mathcal{P}(\boldsymbol{\hat{u_1}})=\mathcal{P}(\boldsymbol{\hat{q_1}})\mathcal{P}(\boldsymbol{\hat{h}})^{-1},\quad&\mathcal{P}(\boldsymbol{\hat{u_1}})^{-1}\mathcal{P}(\boldsymbol{\hat{h}})\mathcal{P}(\boldsymbol{\hat{u_0}})=\mathcal{P}(\boldsymbol{\hat{u_0}})\mathcal{P}(\boldsymbol{\hat{u_1}})^{-1}\mathcal{P}(\boldsymbol{\hat{h}}).
    \end{split}
\end{equation}
Then $P^{-1}A_{\text{SGSWLME}}(\boldsymbol{\hat{W}})P$ can be computed as:
\begin{equation}
    P^{-1}A_{\text{SGSWLME}}P=-\frac{1}{2}(A+C)^{-1}
    \begin{pmatrix}
        A^{-1}G(2G^2-GB-BG) & O & A^{-1}G(BG-GB)\\
        O & -2BA^{-1}G^2 & O\\
        A^{-1}G(BG-GB) & O & -A^{-1}G(2G^2+GB+BG)
    \end{pmatrix},
\end{equation}
which is a symmetric matrix, if the blocks in the block matrix are also symmetric themselves. Thus, with the assumptions of equation \eqref{ANAL-eq:assump} and the assumption that the blocks in the matrix are symmetric, $A_{\text{SGSWLME}}$ is similar to a diagonalisable matrix with real eigenvalues, and so is itself real diagonalisable.

The assumptions of equation \eqref{ANAL-eq:assump} and of the symmetric blocks may seem limiting and, in general, are not automatically satisfied. However, for $K=1$ and $\nu\sim\mathcal{U}(\mu-\sigma,\mu+\sigma)$ (uniform distribution, physical case) or $\nu\sim\mathcal{N}(\mu,\sigma)$ (normal distribution, non-physical case), they can be verified to all be true. This is not the case for $K>1$.

The assumptions of equation \eqref{ANAL-eq:assump} and of the symmetric blocks can be imposed by linearising the system matrix $A_{\text{SGSWLME}}$ for $K>1$ as follows
\begin{equation}\label{ANAL-eq:linearisation}
    \hat{h_j}=0\quad\text{for }j\geq1,\quad\quad\hat{q_{0j}}=0\quad\text{for }j\geq1,\quad\quad\hat{q_{1j}}=0\quad\text{for }j\geq1.
\end{equation}
This condition is checked for $\nu\sim\mathcal{U}(\mu-\sigma,\mu+\sigma)$ and $\nu\sim\mathcal{N}(\mu,\sigma)$ for up to $K=10$, as it is unlikely that $K>10$ since this implies a system of more than 33 equations. The approach of regularising the system matrix to obtain hyperbolicity is inspired by \cite{koellermeier_analysis_2020}, where a similar method is used to make the SWME system matrix hyperbolic.

We thus have found the following theorem:
\begin{theorem}\label{ANAL-th:hyp_N=1}
    The system \eqref{SGSWME-eq:final} is hyperbolic for $\nu\sim\mathcal{U}(\mu-\sigma,\mu+\sigma)$ and $\nu\sim\mathcal{N}(\mu,\sigma)$ and $N=1$ for up to $K=10$ if for $K>1$ the system matrix $A_{\text{SGSWLME}}(\boldsymbol{\hat{W}})$ is linearised according to equation \eqref{ANAL-eq:linearisation}.
\end{theorem}

We expect that Theorem \ref{ANAL-th:hyp_N=1} holds for general $K$, but a general proof is out of scope for now and could be the topic of future work.

\begin{remark}
    The conditions in equation \eqref{ANAL-eq:assump} cannot be avoided by choosing different conventions in equation \eqref{SGSWME-eq:conventions} and/or by defining different $G$, $A$, $B$ and $C$ in equation \eqref{ANAL-eq:blocks} in the sense that writing convex combinations, for example,
    \begin{equation*}
        \frac{q_0}{h}\quad\Rightarrow\quad B\coloneqq \alpha_1\mathcal{P}(\boldsymbol{\hat{q_0}})\mathcal{P}^{-1}(\boldsymbol{\hat{h}})+ (1-\alpha_1)\mathcal{P}(\boldsymbol{\hat{u_0}})\quad\text{for }\alpha_1\in[0,1],
    \end{equation*}
    does not offer a set of appropriate convex scaling constants $\alpha_i$ such that $P^{-1}A_{\text{SGSWLME}}(\boldsymbol{\hat{W}})P$ is symmetric without regularisation.
\end{remark}

\begin{remark}
    Of course, hyperbolicity can also be checked immediately by calculating the eigenvalues of the system matrix $A_{\text{SGSWLME}}(\boldsymbol{\hat{W}})$. However, for larger systems of equations this is a tedious procedure and if a matrix is not hyperbolic regularisation techniques are not immediately clear.
\end{remark}

The approach of this subsection is likely to be extendable to $N>1$ by finding the right analogues to the criteria of \eqref{ANAL-eq:assump}. Due to the increasing complexity in calculating the matrix product $P^{-1}A_{\text{SGSWLME}}P$ for $N>1$, this is left for future work. Instead, in the next section, we consider the case of $K=1$ and general order $N$.

\subsection{Explicit System}\label{subsection:explicit}

In this section, we show that the SGSWLME equations for an uniformly distributed friction coefficient $\nu\sim\mathcal{U}(\mu-\sigma,\mu+\sigma)$ using Legendre polynomials for $K=1$ but general $N$ can be written concisely in explicit form using a set of subsequently transformed variables. 

For that, we start from the general expression of $A_{\text{SGSWLME}}$ \eqref{SGSWME-eq:SGSWLME_flux}.
Firstly, we note that for the special case $K=1$, the matrix $\mathcal{P}(\boldsymbol{\hat{\eta}})$, $\boldsymbol{\hat{\eta}}\in\mathbb{R}^{K+1}$ \eqref{SG-eq:GalerkinMatrix},  e.g., for $\boldsymbol{\hat{\eta}}=\boldsymbol{\hat{h}},\boldsymbol{\hat{q_j}}$, $j=0,\ldots,N$ is a simple two-by-two matrix given by
\begin{equation}\label{ANAL-eq:projectionK1}
    \mathcal{P}(\boldsymbol{\hat{\eta}}) =
    \begin{pmatrix}
        \hat{\eta_0} & \hat{\eta_1} \\
        \hat{\eta_1} & \hat{\eta_0}
    \end{pmatrix},
\end{equation}
with simple eigenvalues $\{\hat{\eta_0}+\hat{\eta_1},\hat{\eta_0}-\hat{\eta_1}\}$ and constant eigenvectors $(1,1)^T$ and $(1,-1)^T$. The existence of constant eigenvectors is crucial for the following derivations, as discussed again in Remark \ref{ANAL-rem:constant_ev}. We use the constant eigenvectors to diagonalise $\mathcal{P}(\boldsymbol{\hat{\eta}})$ using the matrix 
\begin{equation}
    V =
    \begin{pmatrix}
        1 & 1 \\
        1 & -1
    \end{pmatrix}
    ,\quad\text{ with  }\quad V^{-1} =
    \begin{pmatrix}
        \frac{1}{2} & \frac{1}{2} \\
        \frac{1}{2} & -\frac{1}{2}
    \end{pmatrix},
\end{equation}
so that 
\begin{equation}\label{ANAL-eq:diagonalP}
    \mathcal{D}(\boldsymbol{\hat{\eta}})\coloneqq V\mathcal{P}(\boldsymbol{\hat{\eta}})V^{-1}=
    \begin{pmatrix}
        \hat{\eta_0}+\hat{\eta_1} & 0 \\
        0 & \hat{\eta_0}-\hat{\eta_1}
    \end{pmatrix}
    =
    \begin{pmatrix}
        \hat{\eta^+} & 0 \\
        0 & \hat{\eta^-}
    \end{pmatrix},
\end{equation}
when using the new transformed variables
\begin{equation}
    \begin{pmatrix}
        \hat{\eta^+} \\
        \hat{\eta^-}
    \end{pmatrix}
    = V
    \begin{pmatrix}
        \hat{\eta_0}\\
        \hat{\eta_1}
    \end{pmatrix}
    =
    \begin{pmatrix}
        \hat{\eta_0}+\hat{\eta_1}\\
        \hat{\eta_0}-\hat{\eta_1}
    \end{pmatrix},
    \quad\text{ i.e., }\quad
    \begin{pmatrix}
        \hat{\eta_0}\\
        \hat{\eta_1}
    \end{pmatrix}
    = V^{-1}
    \begin{pmatrix}
        \hat{\eta^+}\\
        \hat{\eta^-}
    \end{pmatrix}
    =
    \begin{pmatrix}
        \frac{\hat{\eta^+}+\hat{\eta^-}}{2}  \\
        \frac{\hat{\eta^+}-\hat{\eta^-}}{2}
    \end{pmatrix}.
\end{equation}
Note that this relates to setting
\begin{equation}
    \hat{h^\pm}\coloneqq \hat{h_0}\pm\hat{h_1},\quad\hat{q_j^\pm}\coloneqq \hat{q_{j,0}}\pm\hat{q_{j,1}}.
\end{equation}

To apply the diagonalisation defined by \eqref{ANAL-eq:diagonalP} block-wise to the whole system matrix $A_{\text{SGSWLME}}$ in \eqref{SGSWME-eq:SGSWLME_flux}, we define the block-diagonal matrix $V_{N+2} =\text{diag}(V,\ldots,V)\in\mathbb{R}^{2(N+2)\times2(N+2)}$ that has the same size as $A_{\text{SGSWLME}}$. 

We can then simply compute the transformed matrix $V_{N+2}A_{\text{SGSWLME}}V_{N+2}^{-1}$ block by block. In addition to the simple \eqref{ANAL-eq:diagonalP}, this leads to the following new entries
\begin{equation}
    V\mathcal{P}(\boldsymbol{\hat{q_j}})\mathcal{P}(\boldsymbol{\hat{h}})^{-1}V^{-1}=V\mathcal{P}(\boldsymbol{\hat{q_j}})V^{-1}V\mathcal{P}(\boldsymbol{\hat{h}})^{-1}V^{-1}=\mathcal{D}(\boldsymbol{\hat{q_j}})\mathcal{D}(\boldsymbol{\hat{h}})^{-1}=
    \begin{pmatrix}
        \frac{\hat{q_j^+}}{\hat{h^+}} & 0 \\
        0 & \frac{\hat{q_j^-}}{\hat{h^-}}
    \end{pmatrix}
    =\mathcal{D}(\boldsymbol{\hat{u_j}}),
\end{equation}
where we used the properties of the diagonal matrix $\mathcal{D}(\cdot)$ and the new variables
\begin{equation*}
    \hat{u_j^\pm}\coloneqq \frac{\hat{q_j^\pm}}{\hat{h^\pm}}.
\end{equation*}

In the same fashion, we obtain
\begin{equation}
    V\mathcal{P}(\boldsymbol{\hat{q_j}})\mathcal{P}(\boldsymbol{\hat{h}})^{-1}\mathcal{P}(\boldsymbol{\hat{u_j}})V^{-1}=V\mathcal{P}(\boldsymbol{\hat{q_j}})V^{-1}V\mathcal{P}(\boldsymbol{\hat{h}})^{-1}V^{-1}V\mathcal{P}(\boldsymbol{\hat{u_j}})V^{-1}=\mathcal{D}(\boldsymbol{\hat{q_j}})\mathcal{D}(\boldsymbol{\hat{h}})^{-1}\mathcal{D}(\boldsymbol{\hat{u_j}})=\mathcal{D}(\boldsymbol{\hat{u_j}}^2),
\end{equation}
where we used that
\begin{equation*}
    V\mathcal{P}(\boldsymbol{\hat{u_j}})V^{-1}=V\mathcal{P}(\mathcal{P}(\boldsymbol{\hat{h}})^{-1}\boldsymbol{\hat{q_j}})V^{-1}=V\mathcal{P}\left(
    \begin{pmatrix}
        \frac{\hat{h_0}\hat{q_{j,0}}-\hat{h_1}\hat{q_{j,1}}}{\hat{h_0}^2-\hat{h_1}^2}\\
        \frac{\hat{h_0}\hat{q_{j,1}}-\hat{h_1}\hat{q_{j,0}}}{\hat{h_0}^2-\hat{h_1}^2}
    \end{pmatrix}
    \right)V^{-1} = 
    \begin{pmatrix}
        \frac{(\hat{h_0}-\hat{h_1})(\hat{q_{j,0}}+\hat{q_{j,1}})}{(\hat{h_0}-\hat{h_1})(\hat{h_0}+\hat{h_1})} & 0\\
        0 & \frac{(\hat{h_0}+\hat{h_1})(\hat{q_{j,0}}-\hat{q_{j,1}})}{(\hat{h_0}+\hat{h_1})(\hat{h_0}-\hat{h_1})}
    \end{pmatrix}
    =\mathcal{D}(\boldsymbol{\hat{u_j}}).
\end{equation*}

The transformed system matrix $A^\mathcal{D}_{\text{SGSWLME}} = V_{N+2} A_{\text{SGSWLME}}V_{N+2}^{-1}$ then reads 
\begin{equation}\label{ANAL-eq:A^D_nSGWSLME}
    A^\mathcal{D}_{SGSWLME} =
    \begin{pmatrix}
        O & I & O & \cdots & O\\
        g\mathcal{D}(\boldsymbol{\hat{h}})-\sum\limits_{j=0}^N\frac{\mathcal{D}(\boldsymbol{\hat{u_j}}^2)}{2j+1} & 2\mathcal{D}(\boldsymbol{\hat{u_0}}) & \frac{2\mathcal{D}(\boldsymbol{\hat{u_1}})}{3} & \cdots & \frac{2\mathcal{D}(\boldsymbol{\hat{u_N}})}{2N+1} \\
        -2\mathcal{D}(\boldsymbol{\hat{u_0}}\boldsymbol{\hat{u_1}}) & 2\mathcal{D}(\boldsymbol{\hat{u_1}}) & \mathcal{D}(\boldsymbol{\hat{u_0}}) & &\\
        \vdots &\vdots & & \ddots & \\
        -2\mathcal{D}(\boldsymbol{\hat{u_0}}\boldsymbol{\hat{u_N}}) & 2\mathcal{D}(\boldsymbol{\hat{u_N}}) & & & \mathcal{D}(\boldsymbol{\hat{u_0}})
    \end{pmatrix}.
\end{equation}
Knowing that $A^\mathcal{D}_{\text{SGSWLME}}$ consists of diagonal blocks, we can reorder the odd and even rows and columns, respectively, via a permutation 
resulting in the permuted system matrix
\begin{equation}\label{ANAL-eq:A^P_nSGWSLME}
    A^P_{\text{SGSWLME}} =
    \begin{pmatrix}
        0 & 1 & 0 & \cdots & 0 & \\
        g\hat{h^-}-\sum\limits_{j=0}^N\frac{\hat{u_j^-}^2}{2j+1}& 2\hat{u_0^-} & \frac{2\hat{u_1^-}}{3} & \cdots & \frac{2\hat{u_N^-}}{2N+1} &&&&&\\
        -2\hat{u_0^-}\hat{u_1^-} & 2\hat{u_1^-} & \hat{u_0^-} & & &&&&&\\
        \vdots &\vdots & & \ddots & &&&&&\\
        -2\hat{u_0^-}\hat{u_N^-} & 2\hat{u_N^-} & & & \hat{u_0^-} &&&&& \\
        &&&&& 0 & 1 & 0 & \cdots & 0 \\
        &&&&& g\hat{h^+}-\sum\limits_{j=0}^N\frac{\hat{u_j^+}^2}{2j+1} & 2\hat{u_0^+} & \frac{2\hat{u_1^+}}{3} & \cdots & \frac{2\hat{u_N^+}}{2N+1}\\
        &&&&& -2\hat{u_0^+}\hat{u_1^+} & 2\hat{u_1^+}& \hat{u_0^+} & & \\
        &&&&& \vdots &\vdots & & \ddots & \\
        &&&&& -2\hat{u_0^+}\hat{u_N^+} & 2\hat{u_N^+} & & & \hat{u_0^+} 
    \end{pmatrix}.
\end{equation}
The permuted system matrix $A^P_{\text{SGSWLME}}$ \eqref{ANAL-eq:A^P_nSGWSLME} is simply a block diagonal version of the original SWLME system matrix $A_{\text{SWLME}}$ \eqref{SWME-eq:SWLME_A} depending block-wise only on the new primitive variables $\hat{h^+},\hat{u_j^+}$ and $\hat{h^-},\hat{u_j^-}$, respectively. This means that its eigenvalues are similarly given by the eigenvalues of the original SWLME system matrix \eqref{SWME-eq:SWLME_ev}, evaluated at, respectively, $\hat{h^+},\hat{u_j^+}$ and $\hat{h^-},\hat{u_j^-}$. The eigenvalues are thus all real.

The stochastic Galerkin system for $K=1$ and general $N$ is therefore hyperbolic. 
We remark that this proof appears difficult to generalize to $K>1$ as the projection matrix $\mathcal{P}(\boldsymbol{\hat{\eta}}) \in \mathbb{R}^{(K+1)\times(K+1)}$ is not diagonalisable in a trivial way in that case. 

\begin{remark}\label{ANAL-rem:constant_ev}
    The fact that the eigenvectors of $\mathcal{P}(\boldsymbol{\hat{\eta}})$ for $K=1$ and Legendre polynomials are constant is important, because then we can use the following properties
    \begin{itemize}
        \item $\mathcal{P}(\boldsymbol{\hat{\eta}})$ and $\mathcal{P}(\boldsymbol{\hat{\mu}})$ commute for all $
        \boldsymbol{\hat{\eta}},\boldsymbol{\hat{\mu}}\in\mathbb{R}^{K+1}$ \cite{gerster_entropies_2020},
        \item $\mathcal{P}(\boldsymbol{\hat{\eta}})$ and $\mathcal{P}(\boldsymbol{\hat{\mu}})^{-1}$ commute for all $
        \boldsymbol{\hat{\eta}},\boldsymbol{\hat{\mu}}\in\mathbb{R}^{K+1}$,
        \item$\mathcal{P}(\mathcal{P}(\boldsymbol{\hat{\mu}})\boldsymbol{\hat{\eta}}) = \mathcal{P}(\boldsymbol{\hat{\mu}})\mathcal{P}(\boldsymbol{\hat{\eta}})$ is satisfied for all $
        \boldsymbol{\hat{\eta}},\boldsymbol{\hat{\mu}}\in\mathbb{R}^{K+1}$ \cite{gerster_haar-type_2025},
        \item$\mathcal{P}(\mathcal{P}(\boldsymbol{\hat{\mu}})^{-1}\boldsymbol{\hat{\eta}}) = \mathcal{P}(\boldsymbol{\hat{\mu}})^{-1}\mathcal{P}(\boldsymbol{\hat{\eta}})$ is satisfied for all $
        \boldsymbol{\hat{\eta}},\boldsymbol{\hat{\mu}}\in\mathbb{R}^{K+1}$.
    \end{itemize}
    The found criteria for hyperbolicity for the $N=1$ case in \eqref{ANAL-eq:assump} would, for example, be automatically satisfied for a $\mathcal{P}(\boldsymbol{\hat{\eta}})$ with constant eigenvectors. A way to extend the hyperbolicity proof of this subsection to general $K$ would be to design some regularisation of $\mathcal{P}(\boldsymbol{\hat{\eta}})$ so that its eigenvectors are constant for all $K$.
\end{remark}

%% file: Sections/06_Numerics.tex
The implicit time-splitting path-conservative finite volume scheme used to solve the SGSWLME is described in this section. Following \cite{verbiest_model-adaptive_2025}, the solution to equation \eqref{SGSWME-eq:final} is discretised in its computational domain $\Omega_x$ with grid size $\Delta x$, producing semi-discretised cell-averaged vectors
\begin{equation*}
    \boldsymbol{\hat{W}}^i\coloneqq\left(\boldsymbol{\hat{h}}^i(t),\boldsymbol{\hat{q_0}}^i(t), \boldsymbol{\hat{q_1}}^i(t),\ldots,\boldsymbol{\hat{q_N}}^i(t)\right)^T \in \mathbb{R}^{3(K+1)}
\end{equation*}
in grid cells $\left[x_i-\Delta x / 2, x_i+\Delta x / 2\right]$ with equidistant cell centres $x_i$, $i=1,2, \ldots, N_x$. Additionally, discretising these cell-averaged vectors in time with time step $\Delta t$ produces fully discretised cell-averaged vectors
\begin{equation*}
    \boldsymbol{\hat{W}}^{i,n}=\left(\boldsymbol{\hat{h}}^{i,n},\boldsymbol{\hat{q_0}}^{i,n}, \boldsymbol{\hat{q_1}}^{i,n},\ldots,\boldsymbol{\hat{q_N}}^{i,n}\right)^T \in \mathbb{R}^{3(K+1)}
\end{equation*}
at discrete times $t_n$, $n=0,1, \ldots, N_t$.

The numerical method used to solve equation \eqref{SGSWME-eq:final} is based on a first-order time-splitting scheme introduced in \cite{huang_equilibrium_2022}. Equation \eqref{SGSWME-eq:final} is split into a transport part and a source term part. This results in two subproblems that are solved subsequently in each time step
\begin{subequations}
    \begin{alignat}{2}
        &\partial_t  \boldsymbol{\hat{W}}+A_{\text{SGSWLME}}\left( \boldsymbol{\hat{W}}\right) \partial_x \boldsymbol{\hat{W}} && =\boldsymbol{0}, \label{NUM-eq:splitting_a}\\
        &\partial_t \boldsymbol{\hat{W}} && = \boldsymbol{\hat{S}}\left(\boldsymbol{\hat{W}}\right) \label{NUM-eq:splitting_b} .
    \end{alignat}
\end{subequations}
For the numerical simulation of the transport step of equation \eqref{NUM-eq:splitting_a}, we consider the Polynomial Viscosity Method (PVM). This is common in the field of SWME, because it is generally applicable to non-conservative systems and includes well-known schemes such as the path-conservative PRICE scheme \cite{canestrelli_well-balanced_2009}, which will be used here. For the simulation of the transport step \eqref{NUM-eq:splitting_a} of the SGSWLME the PVM reads
\begin{equation}\label{NUM-eq:PVM}
    \boldsymbol{\hat{W}}^{i,n+1}=\boldsymbol{\hat{W}}^{i,n}-\frac{\Delta t}{\Delta x}\left(\boldsymbol{D}_{i-\frac{1}{2}}^{-}\left(\boldsymbol{\hat{W}}^{i,n}, \boldsymbol{\hat{W}}^{i,n+1}\right)+\boldsymbol{D}_{i+\frac{1}{2}}^{+}\left(\boldsymbol{\hat{W}}^{i,n}, \boldsymbol{\hat{W}}^{i,n+1}\right)\right), \quad \text{for }i=1, \ldots, N_x,
\end{equation}
with fluctuations $\boldsymbol{D}_{k+\frac{1}{2}}^{ \pm}$ given by
\begin{equation*}
    \boldsymbol{D}_{k+\frac{1}{2}}^{\pm}\left(\boldsymbol{\hat{W}}_l,\boldsymbol{\hat{W}}_r\right)=\frac{1}{2}\left(A_{\boldsymbol{\Phi}}\left(\boldsymbol{\hat{W}}_l,\boldsymbol{\hat{W}}_r\right) \left(\boldsymbol{\hat{W}}_r-\boldsymbol{\hat{W}}_l\right) \pm Q_{\boldsymbol{\Phi}}\left(\boldsymbol{\hat{W}}_l,\boldsymbol{\hat{W}}_r\right) \left(\boldsymbol{\hat{W}}_r-\boldsymbol{\hat{W}}_l\right)\right).
\end{equation*}
Here, the generalised Roe matrix $A_{\boldsymbol{\Phi}}\colon\mathbb{R}^{3(K+1)}\times\mathbb{R}^{3(K+1)}\rightarrow\mathbb{R}^{3(K+1)\times 3(K+1)}$ reads
\begin{equation*}
    A_{\boldsymbol{\Phi}}\left(\boldsymbol{\hat{W}}_l,\boldsymbol{\hat{W}}_r\right) \left(\boldsymbol{\hat{W}}_r-\boldsymbol{\hat{W}}_l\right)=\int_0^1 J\left(\boldsymbol{\Phi}\left(s ; \boldsymbol{\hat{W}}_l,\boldsymbol{\hat{W}}_r\right)\right) \frac{\partial \boldsymbol{\Phi}}{\partial s}\left(s ; \boldsymbol{\hat{W}}_l,\boldsymbol{\hat{W}}_r\right) d s
\end{equation*}
and the polynomial viscosity matrix $Q_{\boldsymbol{\Phi}}\colon\mathbb{R}^{3(K+1)}\times\mathbb{R}^{3(K+1)}\rightarrow\mathbb{R}^{3(K+1)\times 3(K+1)}$ 
\begin{equation}\label{NUM-eq:viscosity_matrix}
   Q_{\boldsymbol{\Phi}}\left(\boldsymbol{\hat{W}}_l,\boldsymbol{\hat{W}}_r\right)=R\left(A_{\boldsymbol{\Phi}}\left(\boldsymbol{\hat{W}}_l,\boldsymbol{\hat{W}}_r\right)\right) 
\end{equation}
is determined by a polynomial $R\colon\mathbb{R}^{3(K+1)\times3(K+1)}\rightarrow\mathbb{R}^{3(K+1)\times 3(K+1)}$ of the generalised Roe matrix. The path $\boldsymbol{\Phi}\left(\cdot \:;\boldsymbol{\hat{W}}_l,\boldsymbol{\hat{W}}_r\right)\colon \mathbb{R} \rightarrow \mathbb{R}^{3(K+1)}$ connects the left state $\boldsymbol{\hat{W}}_l \in \mathbb{R}^{3(K+1)}$ to the right state $\boldsymbol{\hat{W}}_r \in \mathbb{R}^{3(K+1)}$ at the cell interface, such that $\boldsymbol{\Phi}\left(0 ; \boldsymbol{\hat{W}}_l, \boldsymbol{\hat{W}}_r\right)=\boldsymbol{\hat{W}}_l$ and $\boldsymbol{\Phi}\left(1 ; \boldsymbol{\hat{W}}_l, \boldsymbol{\hat{W}}_r\right)=\boldsymbol{\hat{W}}_r$. In numerical simulations, we opt for a simple linear path $\boldsymbol{\Phi}\left(\cdot \:;\boldsymbol{\hat{W}}_l,\boldsymbol{\hat{W}}_r\right)$ and calculate the path integral for the matrix $A_{\boldsymbol{\Phi}}$ using the midpoint rule.

For the numerical simulations presented in the next section, the transport step of equation \eqref{NUM-eq:splitting_a} is numerically solved using the PRICE scheme \cite{canestrelli_well-balanced_2009}, which is a PVM scheme \eqref{NUM-eq:PVM} with polynomial viscosity matrix \eqref{NUM-eq:viscosity_matrix}
\begin{equation*}
    Q_{\boldsymbol{\Phi}}\left(\boldsymbol{\hat{W}}_l,\boldsymbol{\hat{W}}_r\right)=\frac{\Delta x}{2 \Delta t} I+\frac{\Delta t}{2 \Delta x} A_{\boldsymbol{\Phi}}^2\left(\boldsymbol{\hat{W}}_l,\boldsymbol{\hat{W}}_r\right) .
\end{equation*}

The second step of the time-splitting scheme, the source term part of equation \eqref{NUM-eq:splitting_b} is solved using implicit Euler, because of its unconditional stability \cite{huang_equilibrium_2022}. The choice of implicit Euler over explicit Euler is based on the fact that the right-hand side of equation \eqref{SGSWME-eq:final} can be stiff due to the factors $\mathcal{P}(\boldsymbol{\hat{h}})^{-1}$ and $\mathcal{P}(\boldsymbol{\hat{h}})^{-2}$.

Similar to \cite{huang_equilibrium_2022}, we can write due to the specific form of the source term in the SGSWLME in equation \eqref{SGSWME-eq:S2},
\begin{equation}
    \boldsymbol{\hat{S}}\left(\boldsymbol{\hat{W}}\right) = \left(\hat{S}_1\left(\boldsymbol{\hat{W}}\right)+\hat{S}_2\left(\boldsymbol{\hat{W}}\right)\right)\boldsymbol{\hat{W}},
\end{equation}
for two matrices $\hat{S}_1,\hat{S}_2\in\mathbb{R}^{(N+2)(K+1)\times(N+2)(K+1)}$ given by
\begin{equation}
    \hat{S}_1\left(\boldsymbol{\hat{W}}\right) = -\frac{T\mathcal{P}(\boldsymbol{\hat{h}})^{-1}}{\lambda}
    \begin{pmatrix}
        O & O & O & \cdots & O\\
        O & I & I & \cdots & I\\
        O & 3I & 3I & \cdots & 3I\\
        \vdots & \vdots & \vdots & \ddots & \vdots\\
        O & (2N+1)I & (2N+1)I & \cdots & (2N+1)I\\
    \end{pmatrix}
\end{equation}
and
\begin{equation}
    \hat{S}_2\left(\boldsymbol{\hat{W}}\right)=-\frac{4T\mathcal{P}(\boldsymbol{\hat{h}})^{-2}}{\lambda}
    \begin{pmatrix}
        O & O & O & \cdots & O\\
        O & O & O & \cdots & O\\
        O & O & 3a_{1,1}I & \cdots & 3a_{1,N}I\\
        \vdots & \vdots & \vdots & \ddots & \vdots\\
        O & O & (2N+1)a_{N,1}I & \cdots & (2N+1)a_{N,N}I\\
    \end{pmatrix}.
\end{equation}
Note that both matrices $\hat{S}_1$ and $\hat{S}_2$ can be assumed to be constant during the second step of the time-splitting scheme \eqref{NUM-eq:splitting_b}, since the water height variables $\boldsymbol{\hat{h}}$ are not changing during the friction step. We obtain
\begin{equation}\label{NUM-eq:exp_imp_Euler}
    \begin{split}
        \frac{\boldsymbol{\hat{W}}^{n+1}-\boldsymbol{\hat{W}}^n}{\Delta t} &= \boldsymbol{\hat{S}}\left(\boldsymbol{\hat{W}}^{n+1}\right)\\
        \Rightarrow\quad\frac{\boldsymbol{\hat{W}}^{n+1}-\boldsymbol{\hat{W}}^n}{\Delta t} &= \left(\hat{S}_1\left(\boldsymbol{\hat{W}}^{n+1}\right)+\hat{S}_2\left(\boldsymbol{\hat{W}}^{n+1}\right)\right)\boldsymbol{\hat{W}}^{n+1}\\
        \Rightarrow\quad\quad\quad\:\:\:\boldsymbol{\hat{W}}^{n+1}&=\left(I-\Delta t\left(\hat{S}_1\left(\boldsymbol{\hat{W}}^n\right)+\hat{S}_2\left(\boldsymbol{\hat{W}}^n\right)\right)\right)^{-1}\boldsymbol{\hat{W}}^n.\\
    \end{split}
\end{equation}
The inverse of $\left(I-\Delta t\left(\hat{S}_1\left(\boldsymbol{\hat{W}}^n\right)+\hat{S}_2\left(\boldsymbol{\hat{W}}^n\right)\right)\right)$ can be precomputed, such that only $\boldsymbol{\hat{h}}^n$ needs to be substituted. This enables efficient evaluation of the implicit time step, avoiding repeated inversion when computing $\boldsymbol{\hat{W}}^{n+1}$. Note that the final source update in \eqref{NUM-eq:exp_imp_Euler} is an explicit, unconditionally stable scheme.

%% file: Sections/07_Results.tex
To numerically analyse the behaviour of the SGSWLME in simulations, we consider two test cases: (1) a dam break and (2) a smooth periodic wave, motivated by \cite{koellermeier_analysis_2020}. These cases differ in their initial conditions for the water height and the velocity profile, as well as potentially in simulation end time and boundary conditions.

For the dam break test case, the initial condition is assumed to be deterministic, meaning all higher stochastic Galerkin orders are zero. The uncertainty arises solely from the friction coefficient $\nu$, which is modelled as a uniformly distributed random variable. The dam is located in the centre of the computational domain ($x=0$), with a discontinuous initial water height at $x=0$. More details are provided in Table \ref{RES-tab:dambreak}.

\begin{table}[H]
    \centering
    \begin{tabular}{|l|l|}
        \hline Slip length & $\lambda = 0.1$ \\
        \hline Friction coefficient & $\nu\sim\mathcal{U}([0.05,0.15])$ \\
        \hline Spatial domain & $x \in[-1,1]$ \\
        \hline Spatial resolution & $\Delta x=0.002$ \\
        \hline Time resolution & $\Delta t=0.0005$ \\
        \hline End time & $t_{\text {end }}=0.2$ \\
        \hline Initial condition & $
        \begin{aligned}
            \mathds{E}[\mathcal{G}_K[h](0,x,\xi)] &= \begin{cases}
                h_d \quad& x<0\\
                1 \quad& x\geq 0
            \end{cases} \\
            \mathds{E}[\mathcal{G}_K[q_0](0,x,\xi)] &= \begin{cases}
                0.25h_d \quad& x<0\\
                0.25 \quad & x\geq 0
            \end{cases} \\
            \mathds{E}[\mathcal{G}_K[q_1](0,x,\xi)] &= \begin{cases}
                -0.25h_d \quad& x<0\\
                -0.25\quad & x\geq 0
            \end{cases} \\
            \mathds{E}[\mathcal{G}_K[q_2](0,x,\xi)] &= 0
        \end{aligned}
        $ \\
        \hline Boundary condition & Inflow-outflow \\
        \hline
    \end{tabular}
    \caption{Simulation setup for dam break test case.}
    \label{RES-tab:dambreak}
\end{table}
\vspace{-1em}
Here, $h_d$ denotes the height of the dam. We set $h_d=1.5$ for a low dam and $h_d=5.0$ for a high dam. This dam break test case represents a challenging test case due to the presence of a shock wave.

For the smooth periodic wave test case, the initial condition is again assumed to be deterministic, with the friction coefficient $\nu$ modelled as a uniformly distributed random variable. More details are provided in Table \ref{RES-tab:smooth_wave}.

\begin{table}[H]
    \centering
    \begin{tabular}{|l|l|}
        \hline Slip length & $\lambda = 0.1$ \\
        \hline Friction coefficient & $\nu\sim\mathcal{U}([0.05,0.15])$ \\
        \hline Spatial domain & $x \in[-1,1]$ \\
        \hline Spatial resolution & $\Delta x=0.002$ \\
        \hline Time resolution & $\Delta t=0.0005$ \\
        \hline End time & $t_{\text {end }}=2.0$ \\
        \hline Initial condition & $
        \begin{aligned}
            \mathds{E}[\mathcal{G}_K[h](0,x,\xi)] &= 1+\exp(3\cos(\pi(x+0.5)))/\exp(4) \\
            \mathds{E}[\mathcal{G}_K[q_0](0,x,\xi)] &= 0.25(1+\exp(3\cos(\pi(x+0.5)))/\exp(4)) \\
            \mathds{E}[\mathcal{G}_K[q_1](0,x,\xi)] &= -0.25(1+\exp(3\cos(\pi(x+0.5)))/\exp(4)) \\
            \mathds{E}[\mathcal{G}_K[q_2](0,x,\xi)] &= 0
        \end{aligned}
        $ \\
        \hline Boundary condition & Periodic \\
        \hline
    \end{tabular}
    \caption{Simulation setup for smooth periodic wave test case.}
    \label{RES-tab:smooth_wave}
\end{table}

The implementation of the SGSWLME solver used in this paper is available in a GitHub repository \cite{kuijpers_sgswlme_2026}.

\subsection{Comparison with Monte Carlo}
We begin by comparing SGSWLME solutions with SWLME Monte Carlo solutions of the same order $N$, which serve as a reference for assessing the accuracy of the SGSWLME across different stochastic Galerkin orders $K$. As demonstrated in Appendix \ref{section:appendix}, convergence is achieved for $S=100-350$ samples in the dam break test cases and for $S=550-900$ samples in the smooth periodic wave test case, meaning that obtaining a converged Monte Carlo solution requires a significant run time. For both the accuracy assessment of our new SGSWLME and the comparison of respective run times, we use converged Monte Carlo reference solutions for benchmarking and error computation. The relative errors in Figures \ref{RES-fig:SG_N1_low}-\ref{RES-fig:SG_N2_wave} are computed as the relative difference between the stochastic Galerkin solutions and the reference Monte Carlo solution.

\begin{figure*}[!htb]
    \centering
    \minipage{0.33\textwidth}
        \centering
        \includegraphics[width=\linewidth]{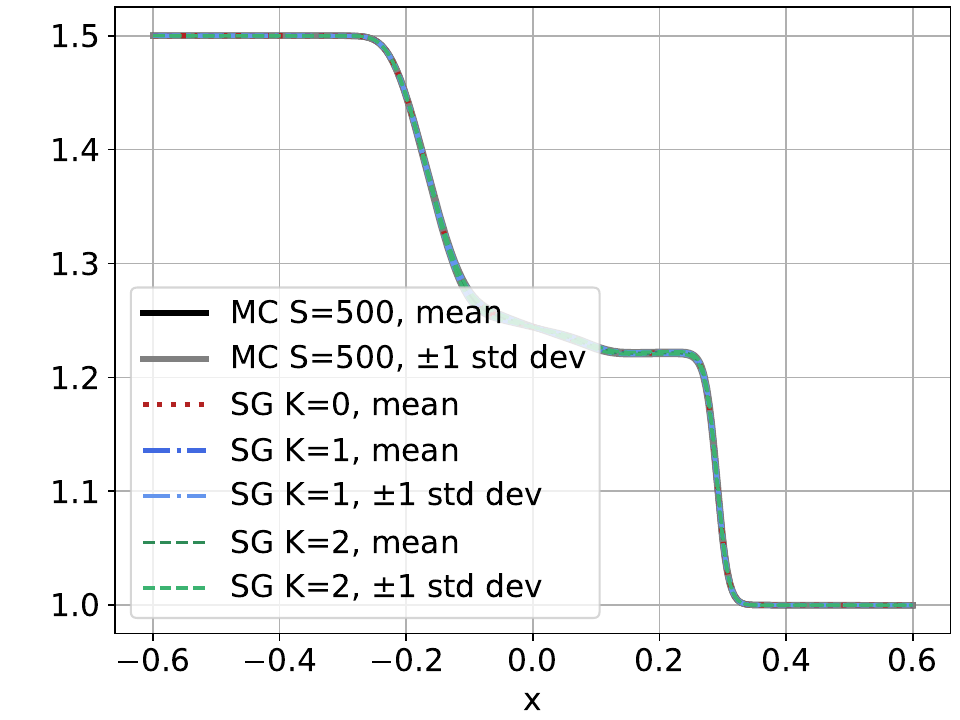}
        \vspace{-2em}
        \caption*{(a) Water height $h$}
        \includegraphics[width=\linewidth]{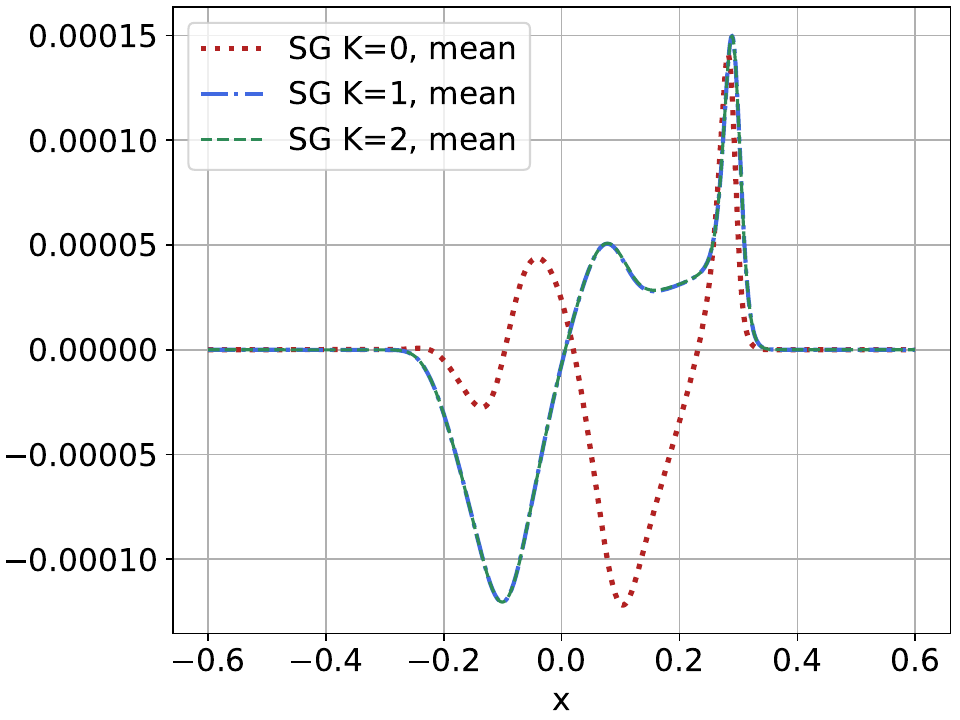}
        \vspace{-2em}
        \caption*{(d) Relative error mean of $h$}
        \includegraphics[width=\linewidth]{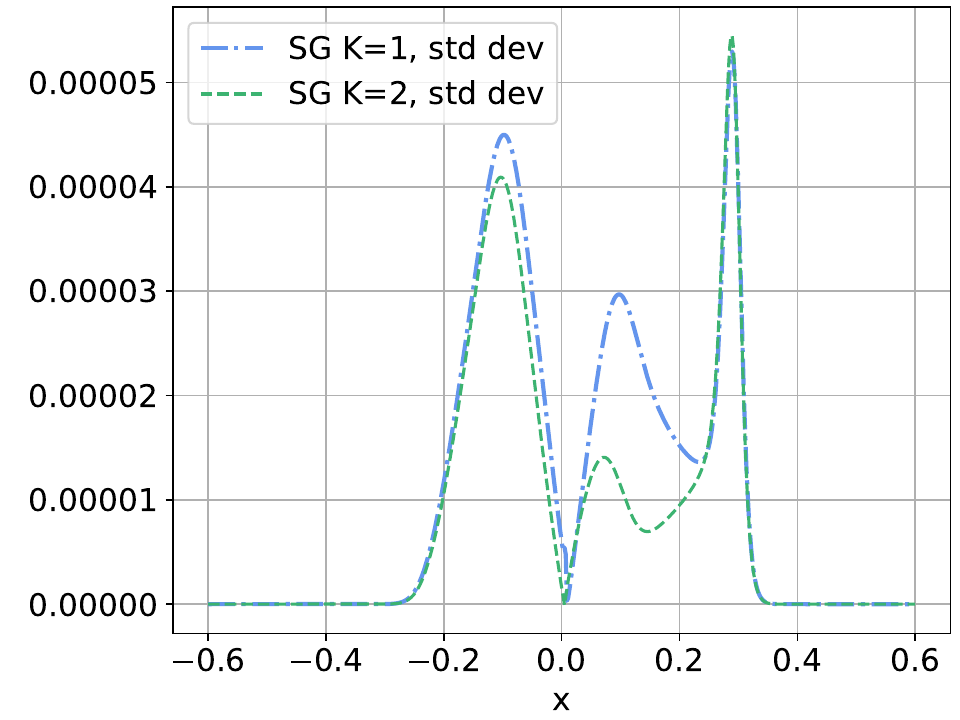}
        \vspace{-2em}
        \caption*{(g) Relative error std. dev. of $h$}
    \endminipage
    \minipage{0.33\textwidth}
        \centering
        \includegraphics[width=\linewidth]{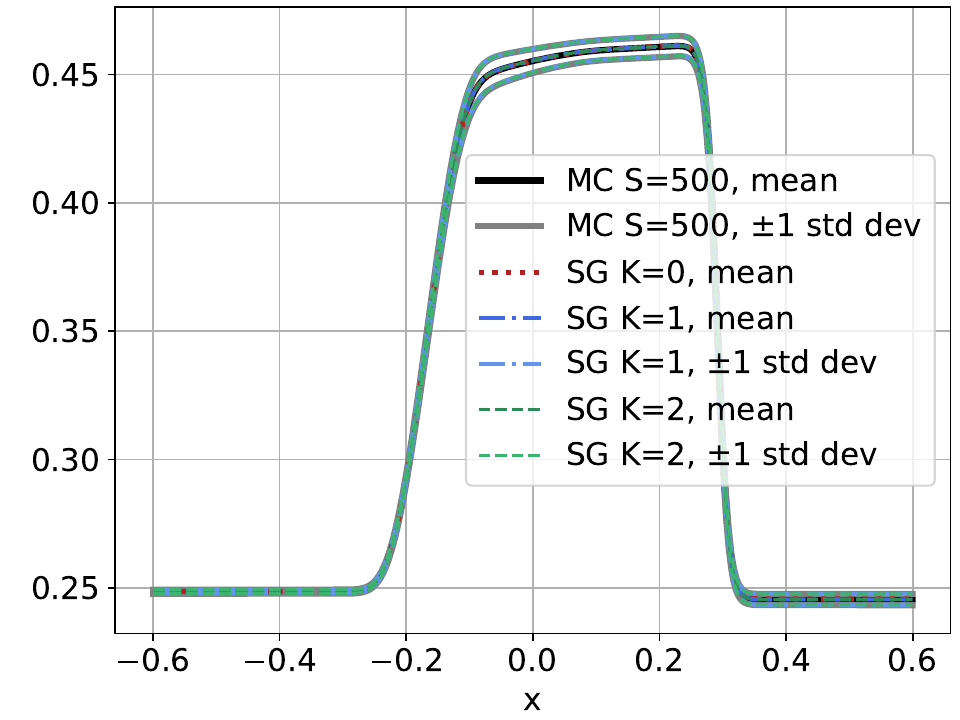}
        \vspace{-2em}
        \caption*{(b) Average velocity $u_m$}
        \includegraphics[width=\linewidth]{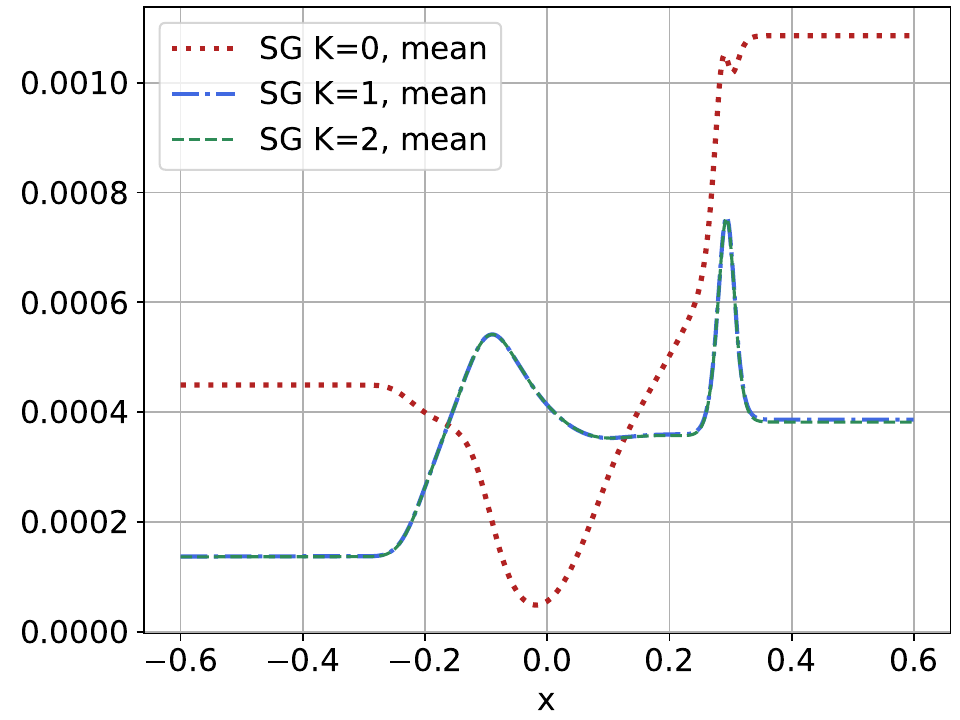}
        \vspace{-2em}
        \caption*{(e) Relative error mean of $u_m$}
        \includegraphics[width=\linewidth]{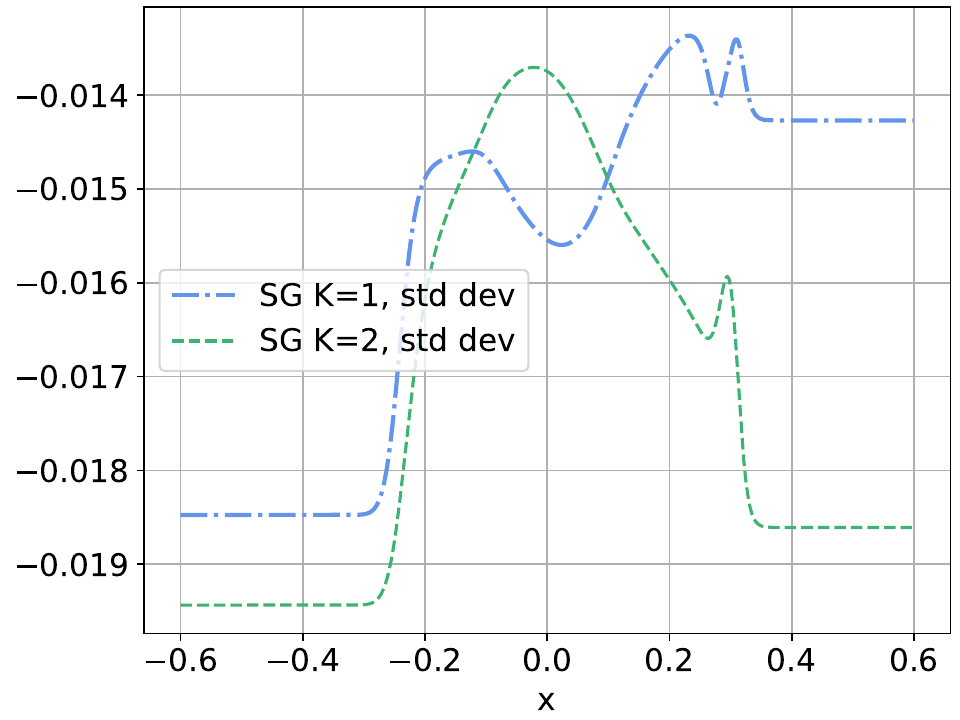}
        \vspace{-2em}
        \caption*{(h) Relative error std. dev. of $u_m$}
    \endminipage
    \minipage{0.33\textwidth}
        \centering
        \includegraphics[width=\linewidth]{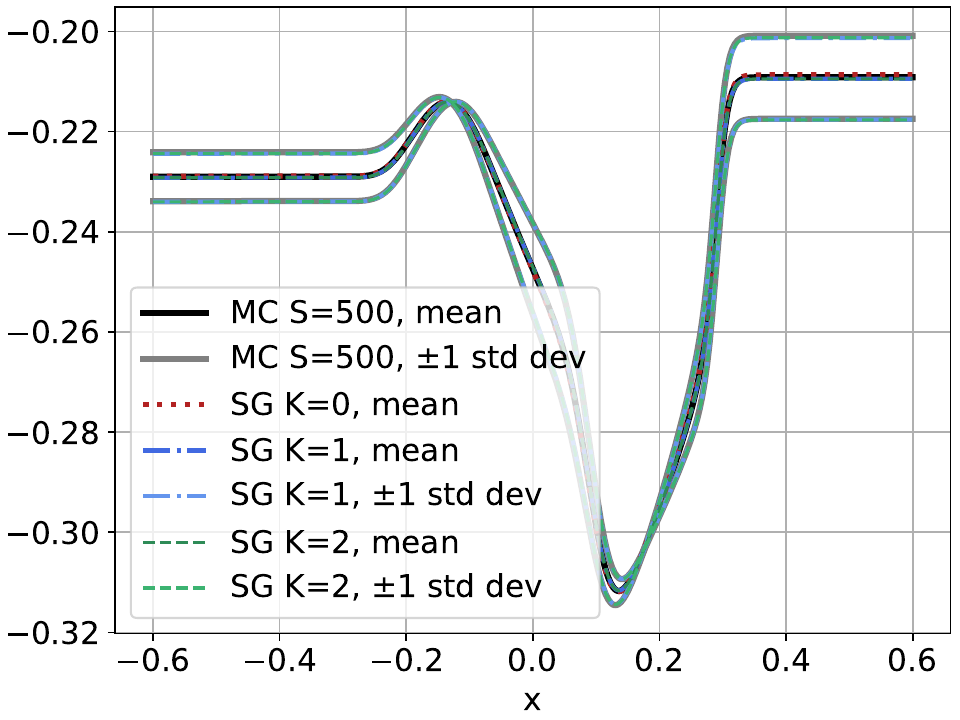}
        \vspace{-2em}
        \caption*{(c) First moment $u_1$}
        \includegraphics[width=\linewidth]{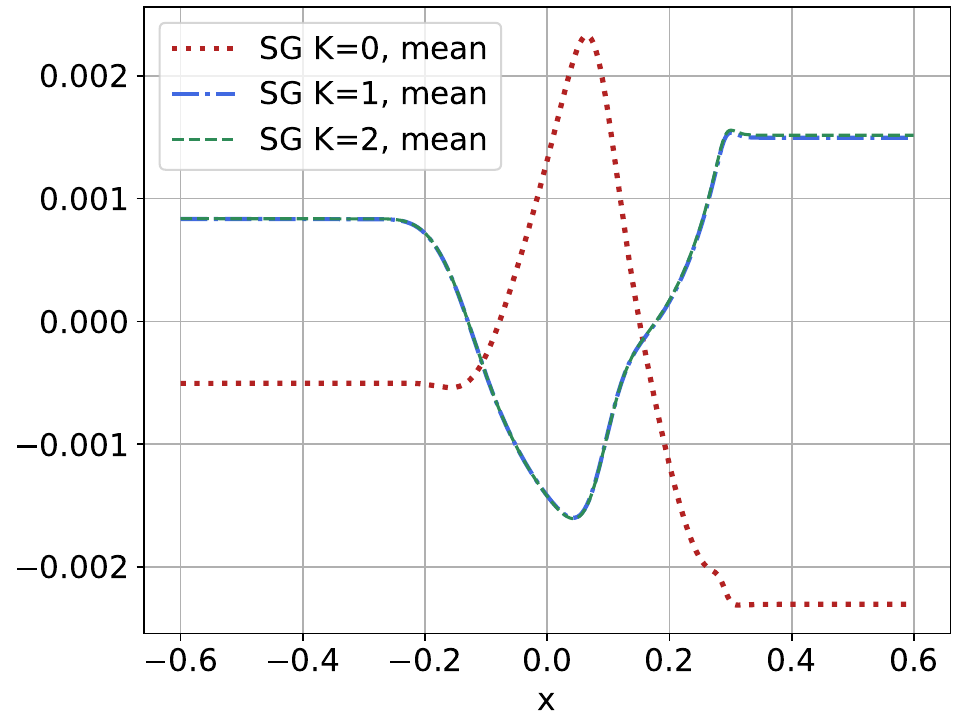}
        \vspace{-2em}
        \caption*{(f) Relative error mean of $u_1$}
        \includegraphics[width=\linewidth]{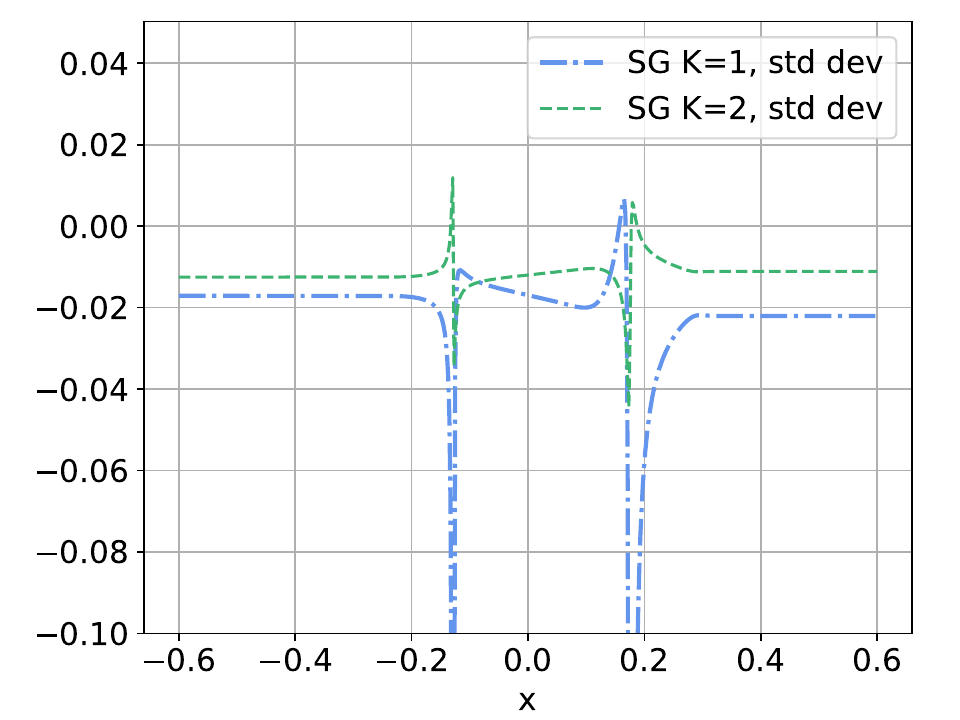}
        \vspace{-2em}
        \caption*{(i) Relative error std. dev. of $u_1$}
    \endminipage
    \caption{Stochastic Galerkin accuracy test for the low dam break test case with (SG)SWLME $N=1$ using different stochastic Galerkin orders $K$. Higher orders improve the accuracy of the solution.}\label{RES-fig:SG_N1_low}
\end{figure*}

In Figure \ref{RES-fig:SG_N1_low}, we compare a converged Monte Carlo solution with stochastic Galerkin solutions of orders $K=0$, $K=1$, and $K=2$ for the low dam break test case of Table \ref{RES-tab:dambreak} for $N=1$. In general, the relative error in the mean decreases as the order $K$ increases from 0 to 1 or 2. In addition, the errors in the mean remain consistently small.

The relative error in the standard deviation decreases slightly when increasing the order from $K=1$ to $K=2$, particularly in the magnitude of the two significant peaks in the first moment $u_1$. In general, errors in the standard deviations are larger than those in the mean, reaching nearly two percent for the average velocity $u_m$ and exhibiting similar magnitudes for the first moment $u_1$ when disregarding the significant peaks. However, these peaks occur at points where the standard deviation is small, so even a small absolute error is amplified in the relative error. 

The error in the water height $h$ is the largest for both the mean and the standard deviation around the shock. This suggests that the uncertainty is primarily concentrated around the shock, where the position of the shock significantly influences the errors.

In Figure \ref{RES-fig:SG_N2_low}, we plot a converged Monte Carlo solution alongside stochastic Galerkin solutions of different orders, and we compute the relative error in the mean and standard deviation for the low dam break case of Table \ref{RES-tab:dambreak}, this time for the $N=2$ systems. Unlike in Figure \ref{RES-fig:SG_N1_low}, we observe in Figure \ref{RES-fig:SG_N2_low} that increasing the order $K$ does not consistently reduce the relative error for all functions. For the average velocity $u_m$ and the second moment $u_2$, the $K=1$ solution is more accurate than the $K=2$ solution.

This behaviour may be attributed to two potential causes. First, there could be a loss of machine precision, as the entries of the system matrix in equation \eqref{SGSWME-eq:SGSWLME_flux} for $N=2$ and $K=2$ include a large number of terms, as does the right hand side friction vector in equation \eqref{SGSWME-eq:S2}. Solving an $8\times8$-system is also less computationally intensive than solving a $12\times12$ system. Second, it could result from a possible lack of hyperbolicity in the transport matrix of equation \eqref{SGSWME-eq:SGSWLME_flux} for $N=2$ and $K=2$.

However, the error remains reasonably small in all cases, disregarding the significant peaks in Figure \ref{RES-fig:SG_N2_low}(i), which were also present in the $N=1$ case. This is particularly notable given that the second moment $u_2$ is of small magnitude in any case.

\begin{figure*}[!htb]
    \centering
    \minipage{0.33\textwidth}
        \centering
        \includegraphics[width=0.9\linewidth]{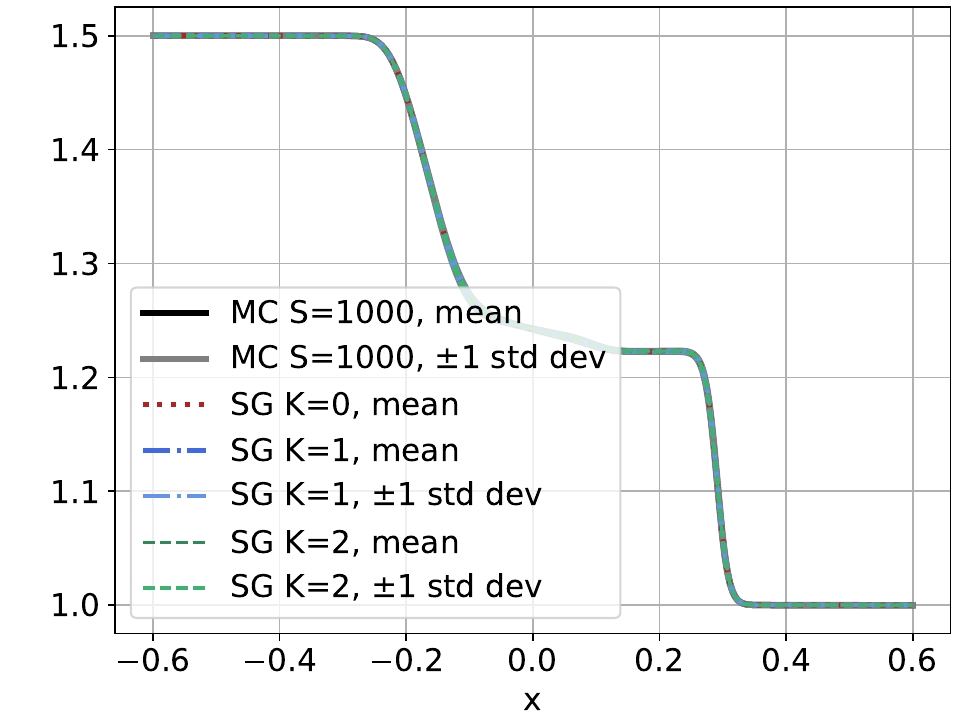}
        \vspace{-1em}
        \caption*{(a) Water height $h$}
        \includegraphics[width=0.9\linewidth]{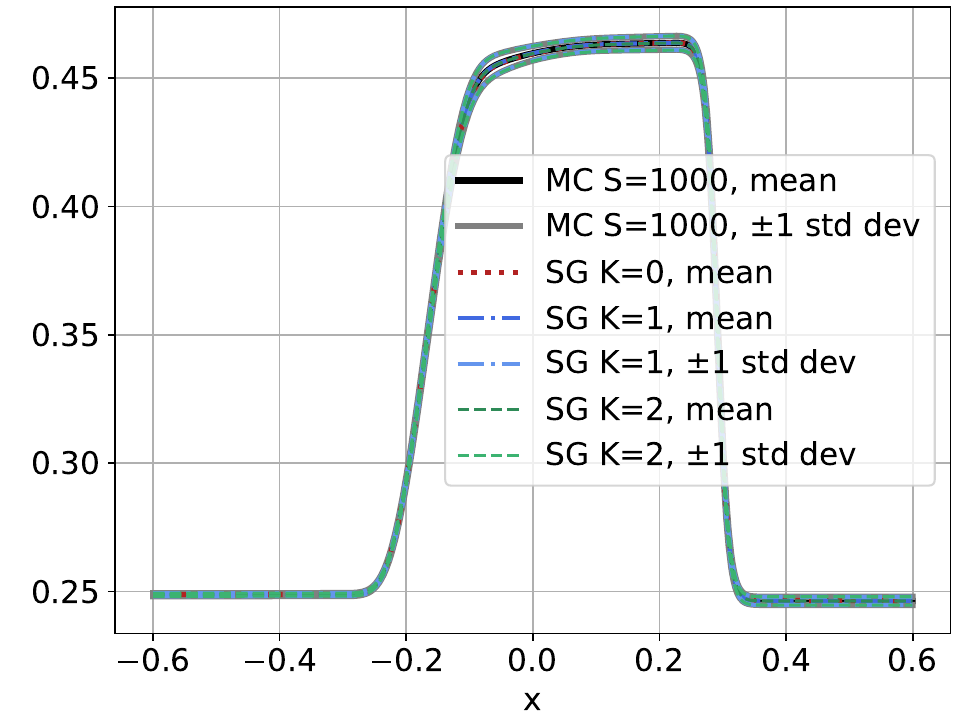}
        \vspace{-1em}
        \caption*{(d) Average velocity $u_m$}
        \includegraphics[width=0.9\linewidth]{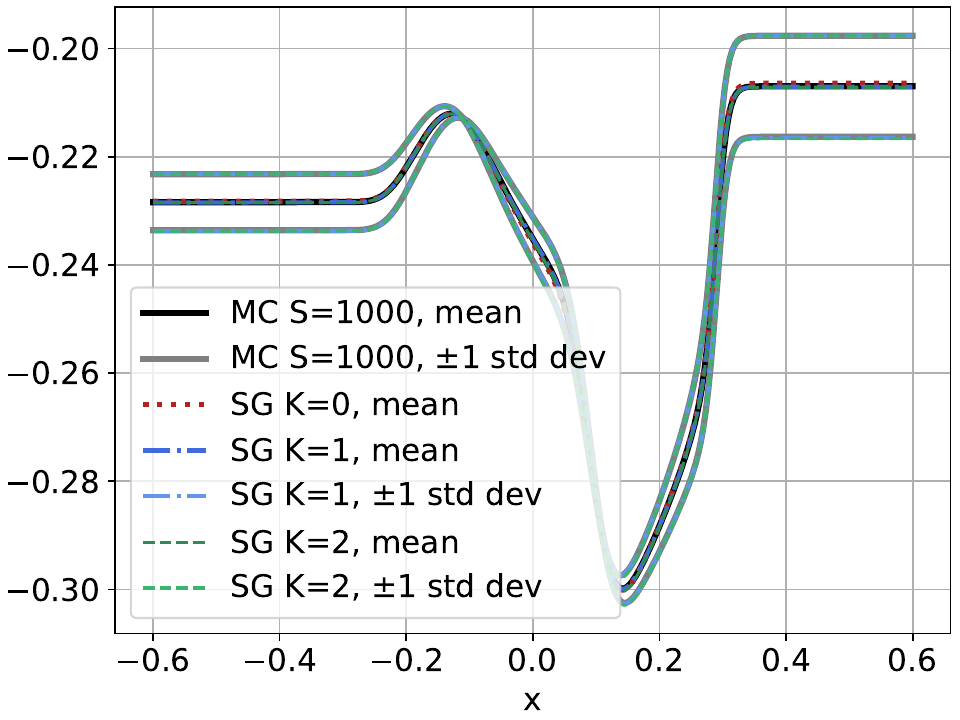}
        \vspace{-1em}
        \caption*{(g) First moment $u_1$}
        \includegraphics[width=0.9\linewidth]{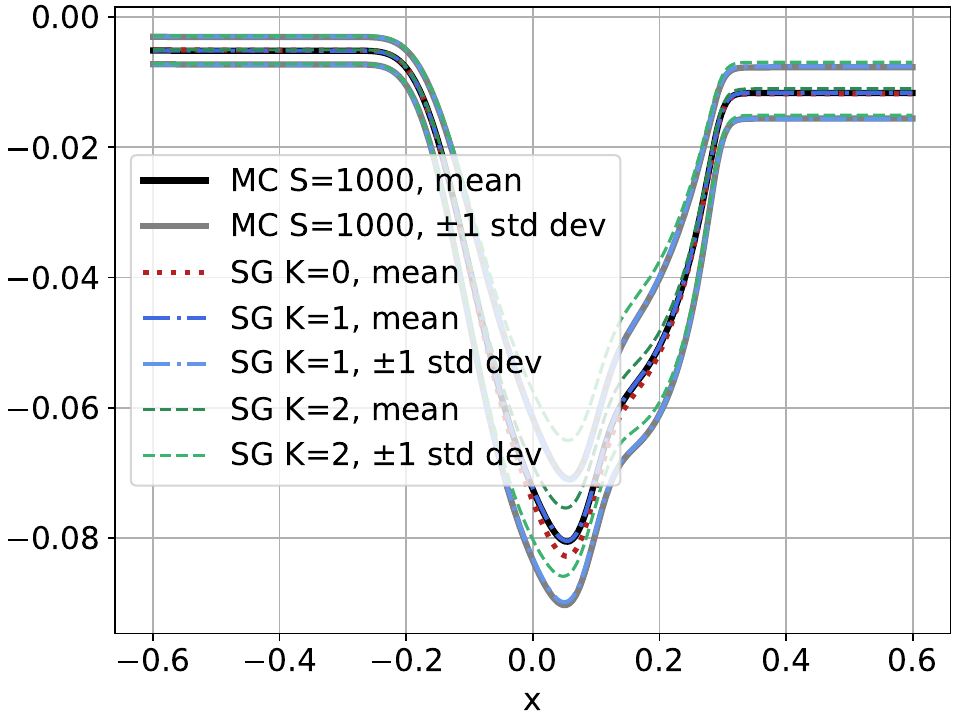}
        \vspace{-1em}
        \caption*{(j) Second moment $u_2$}
    \endminipage
    \minipage{0.33\textwidth}
        \centering
        \includegraphics[width=0.9\linewidth]{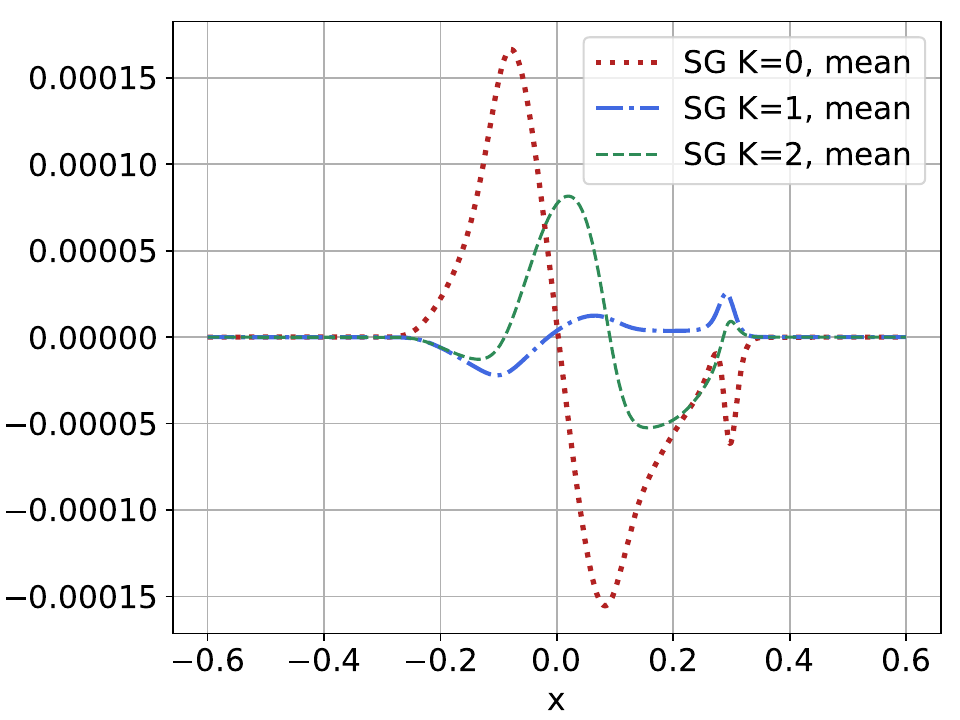}
        \vspace{-1em}
        \caption*{(b) Relative error mean of $h$}
        \includegraphics[width=0.9\linewidth]{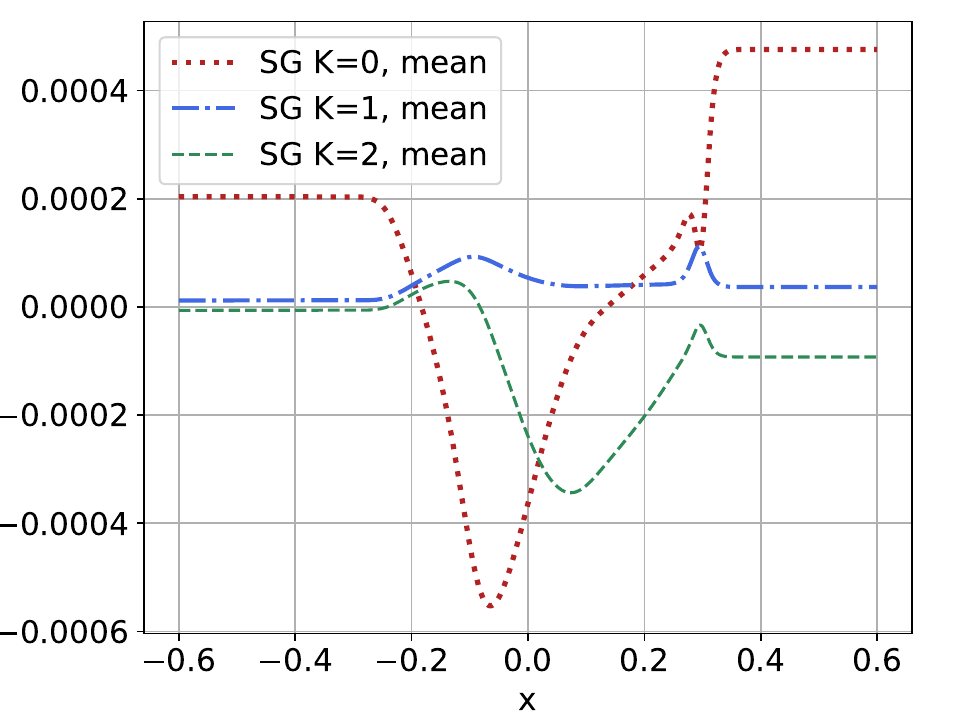}
        \vspace{-1em}
        \caption*{(e) Relative error mean of $u_m$}
        \includegraphics[width=0.9\linewidth]{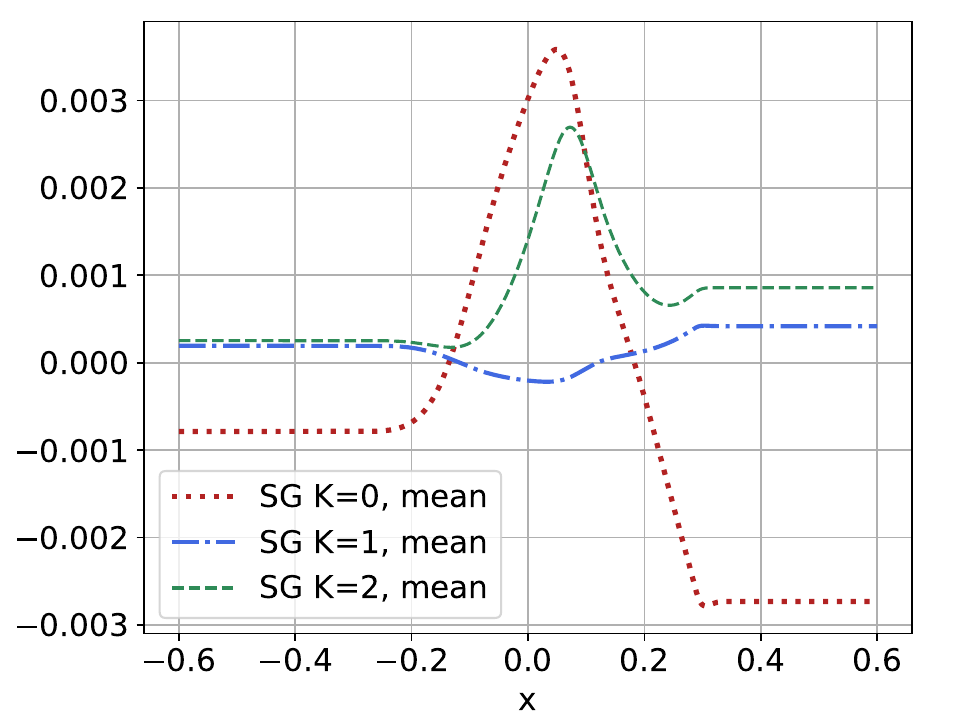}
        \vspace{-1em}
        \caption*{(h) Relative error mean of $u_1$}
        \includegraphics[width=0.9\linewidth]{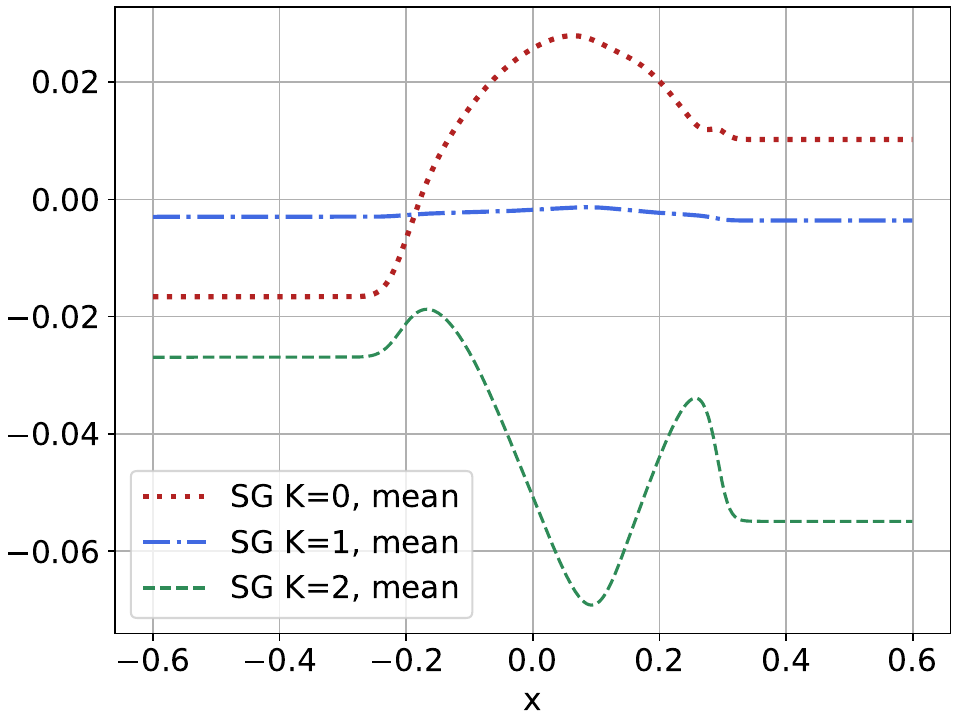}
        \vspace{-1em}
        \caption*{(k) Relative error mean of $u_2$}
    \endminipage
    \minipage{0.33\textwidth}
        \centering
        \includegraphics[width=0.9\linewidth]{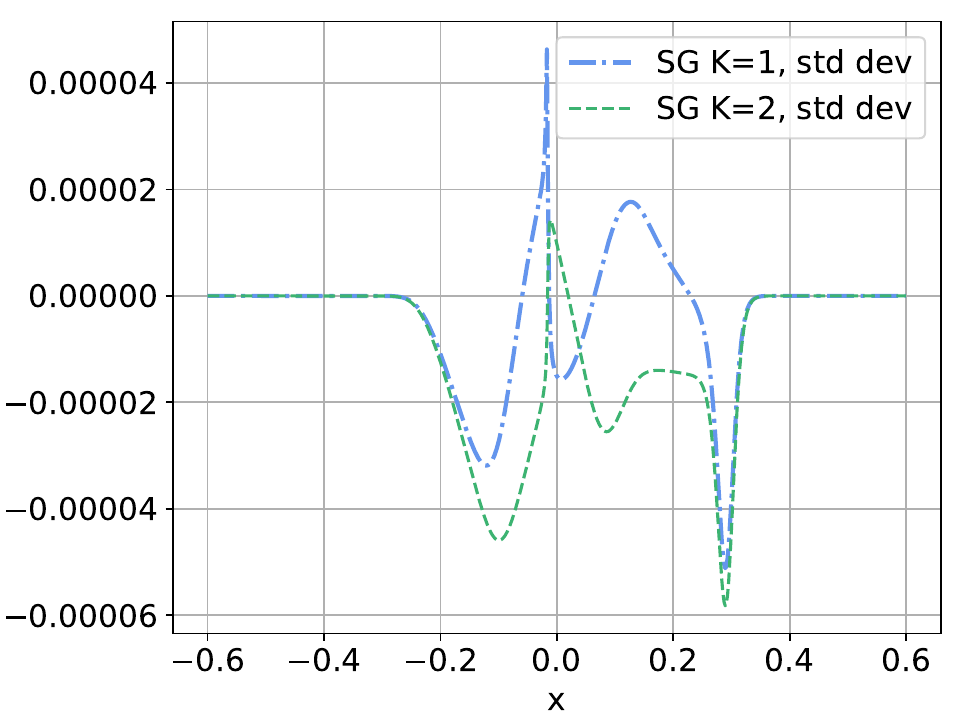}
        \vspace{-1em}
        \caption*{(c) Relative error std. dev. of $h$}
        \includegraphics[width=0.9\linewidth]{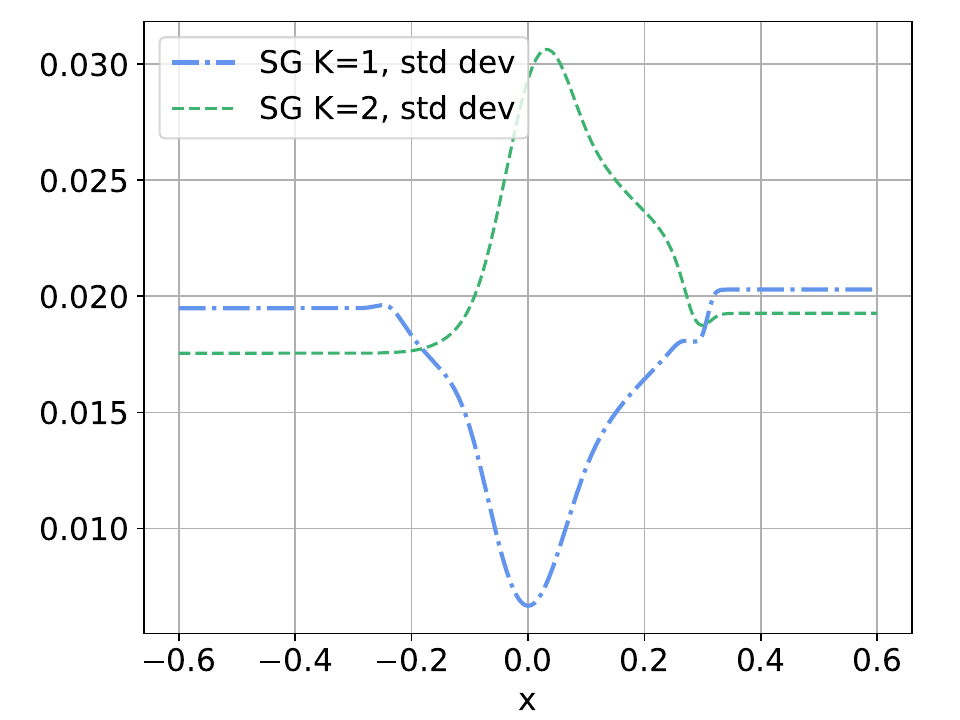}
        \vspace{-1em}
        \caption*{(f) Relative error std. dev. of $u_m$}
        \includegraphics[width=0.9\linewidth]{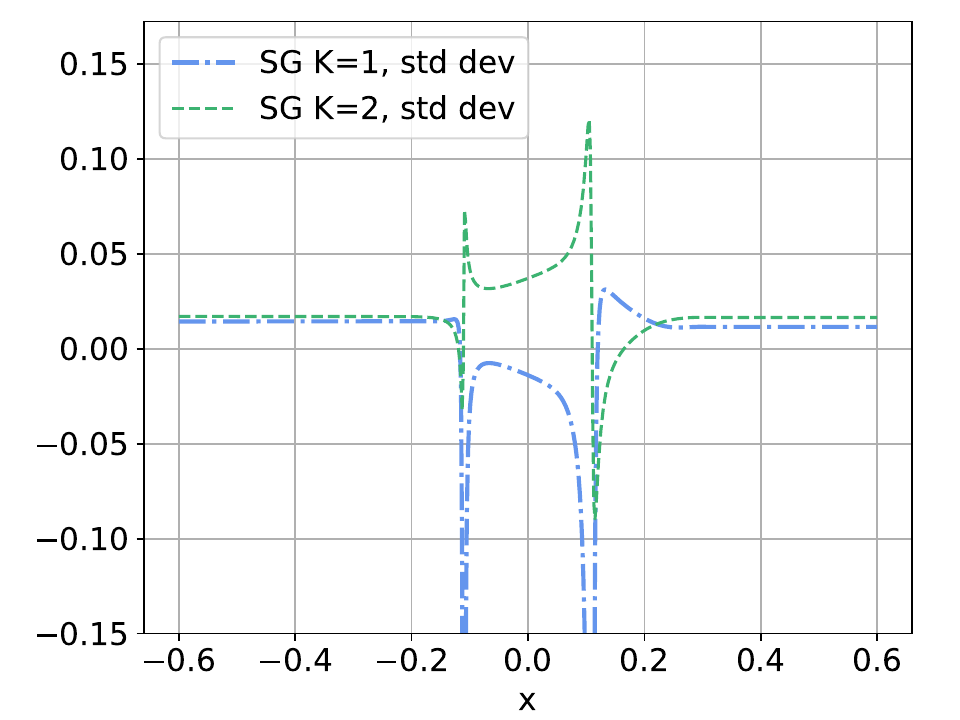}
        \vspace{-1em}
        \caption*{(i) Relative error std. dev. of $u_1$}
        \includegraphics[width=0.9\linewidth]{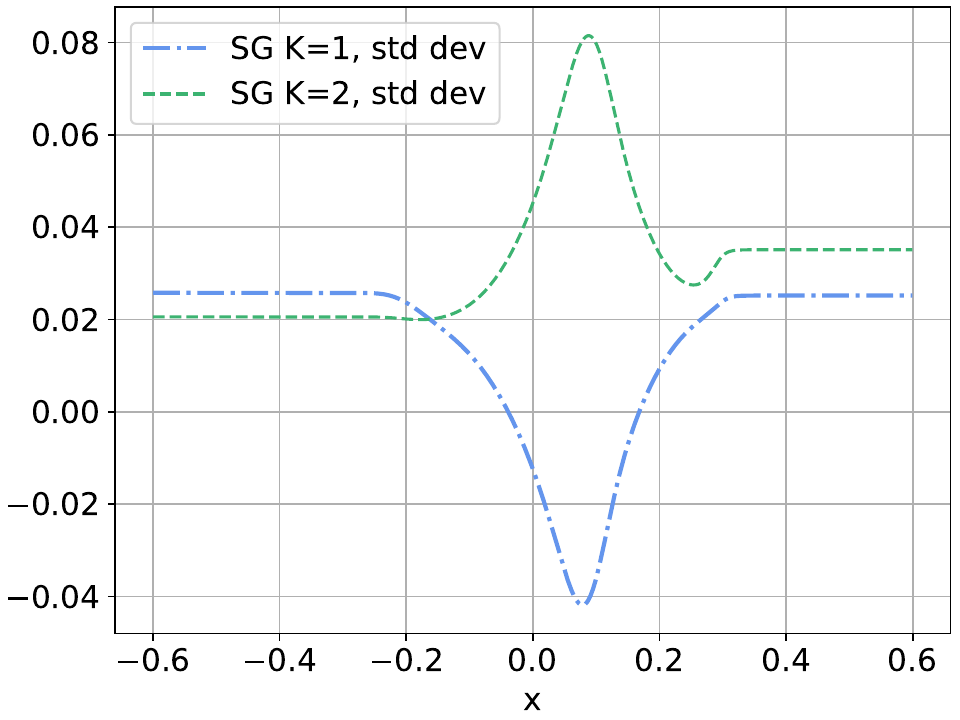}
        \vspace{-1em}
        \caption*{(l) Relative error std. dev. of $u_2$}
    \endminipage
    \caption{Stochastic Galerkin accuracy test for the low dam break test case with (SG)SWLME $N=2$ using different stochastic Galerkin orders $K$. For the average velocity $u_m$ and the second moment $u_2$ the curve corresponding to $K=1$ gives a smaller maximum error than the $K=2$ curve, but overall the errors are small.}\label{RES-fig:SG_N2_low}
\end{figure*}

\newpage
In the case of the higher dam (with $h_d=5.0$ in Table \ref{RES-tab:dambreak}), as shown in Figure \ref{RES-fig:SG_N1_high} for $N=1$ systems, all three variables exhibit a clear peak of inaccuracy at the shock location of the dam break. It is evident that a higher stochastic Galerkin order is required to accurately capture the uncertainty at this point. However, the absolute uncertainty at the shock remains small.

For the average velocity $u_m$ and the first moment $u_1$, the standard deviation is captured less accurately than the mean in comparison with the Monte Carlo solution. Additionally, it is clear that the magnitude of the inaccuracy peaks decreases with increasing order, both for the mean and for the standard deviation.

\begin{figure*}[!htb]
    \centering
    \minipage{0.33\textwidth}
        \centering
        \includegraphics[width=\linewidth]{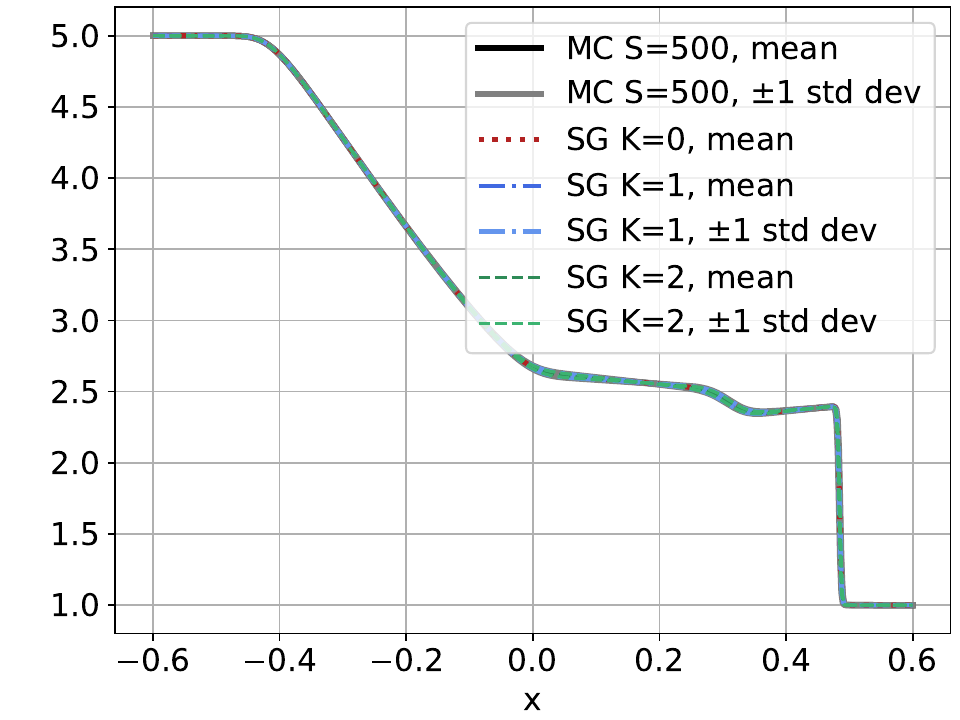}
        \vspace{-2em}
        \caption*{(a) Water height $h$}
        \includegraphics[width=\linewidth]{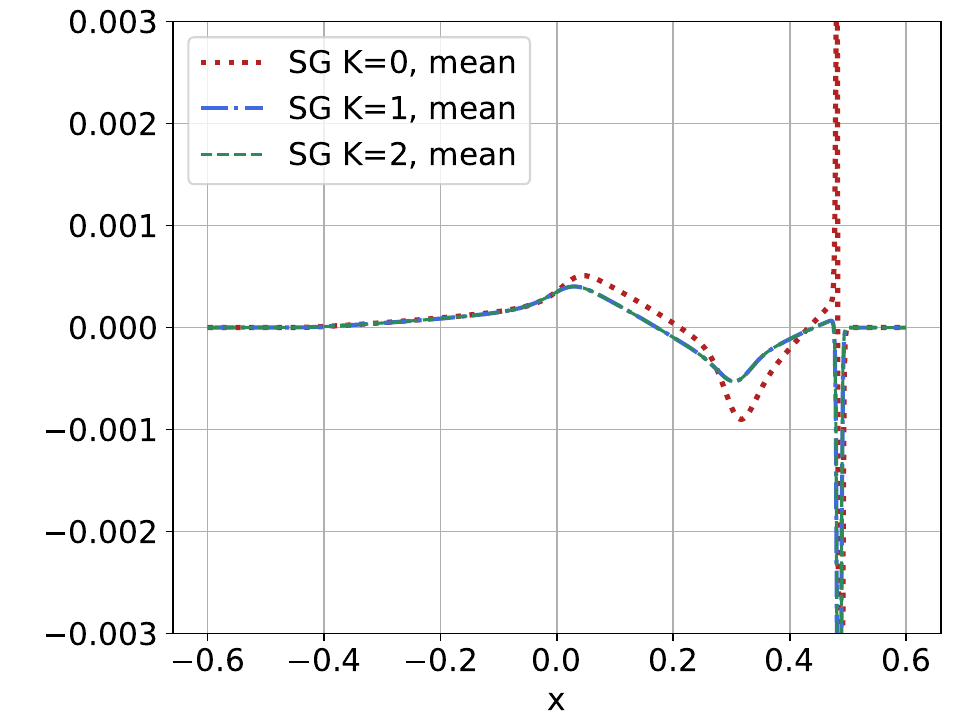}
        \vspace{-2em}
        \caption*{(d) Relative error mean of $h$}
        \includegraphics[width=\linewidth]{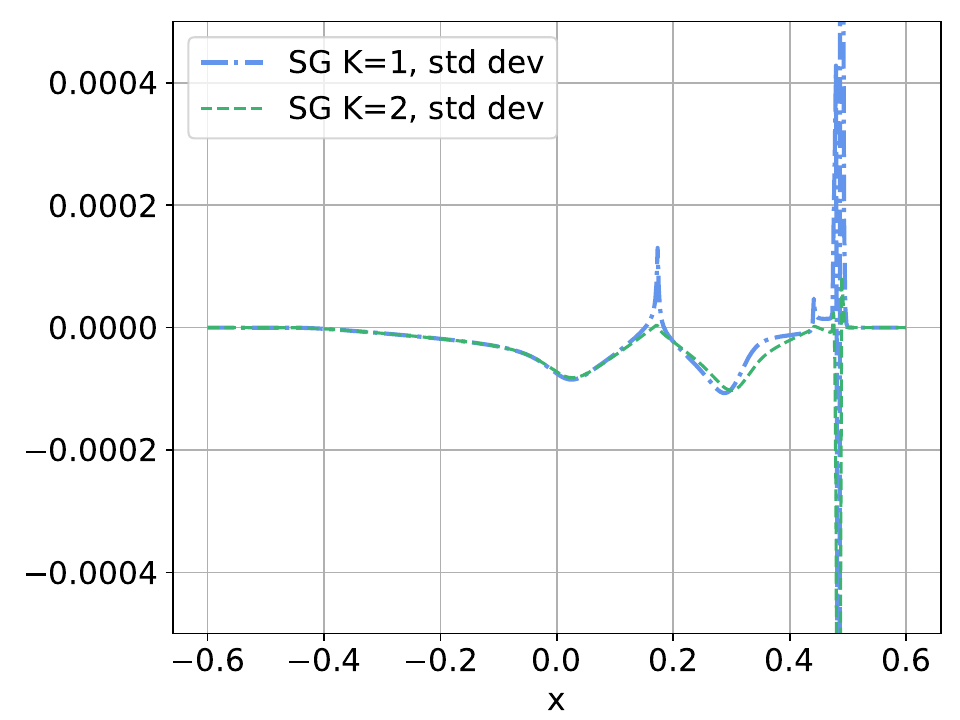}
        \vspace{-2em}
        \caption*{(g) Relative error std. dev. of $h$}
    \endminipage
    \minipage{0.33\textwidth}
        \centering
        \includegraphics[width=\linewidth]{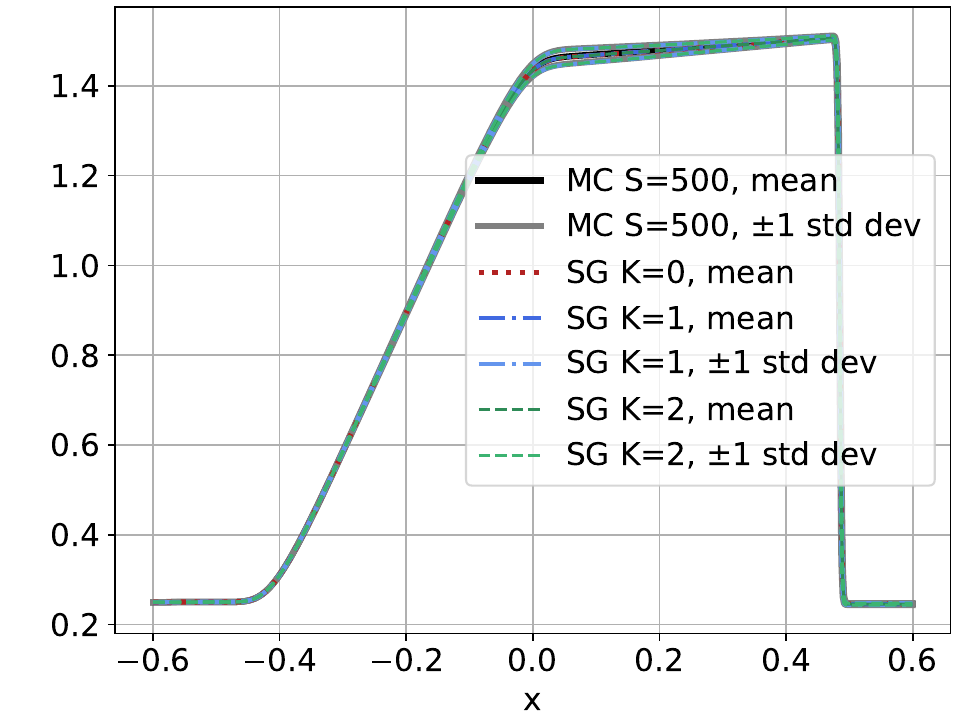}
        \vspace{-2em}
        \caption*{(b) Average velocity $u_m$}
        \includegraphics[width=\linewidth]{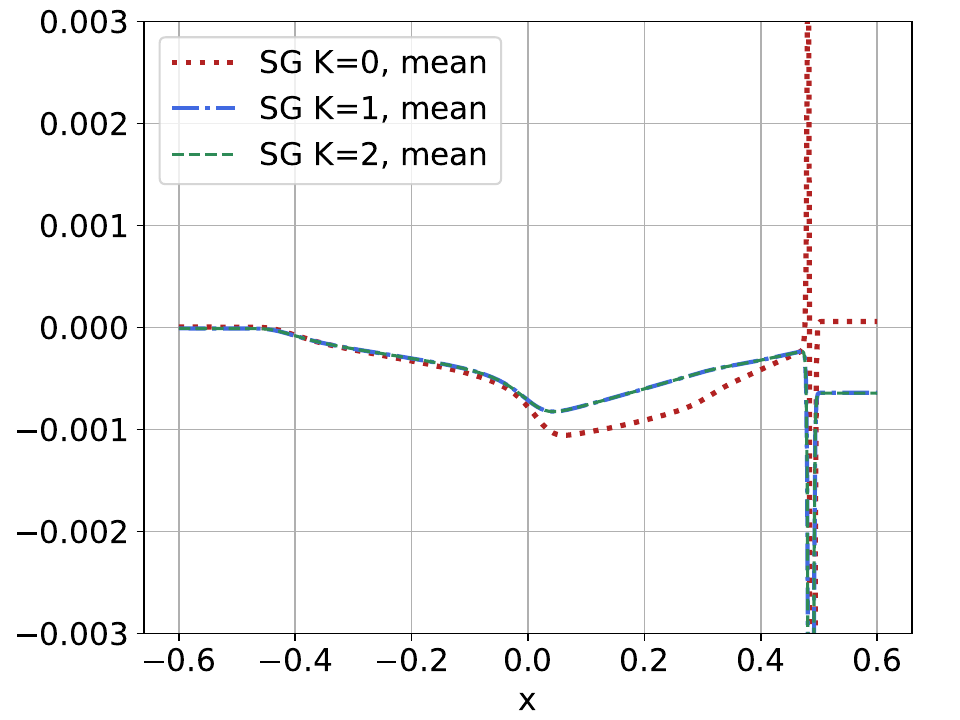}
        \vspace{-2em}
        \caption*{(e) Relative error mean of $u_m$}
        \includegraphics[width=\linewidth]{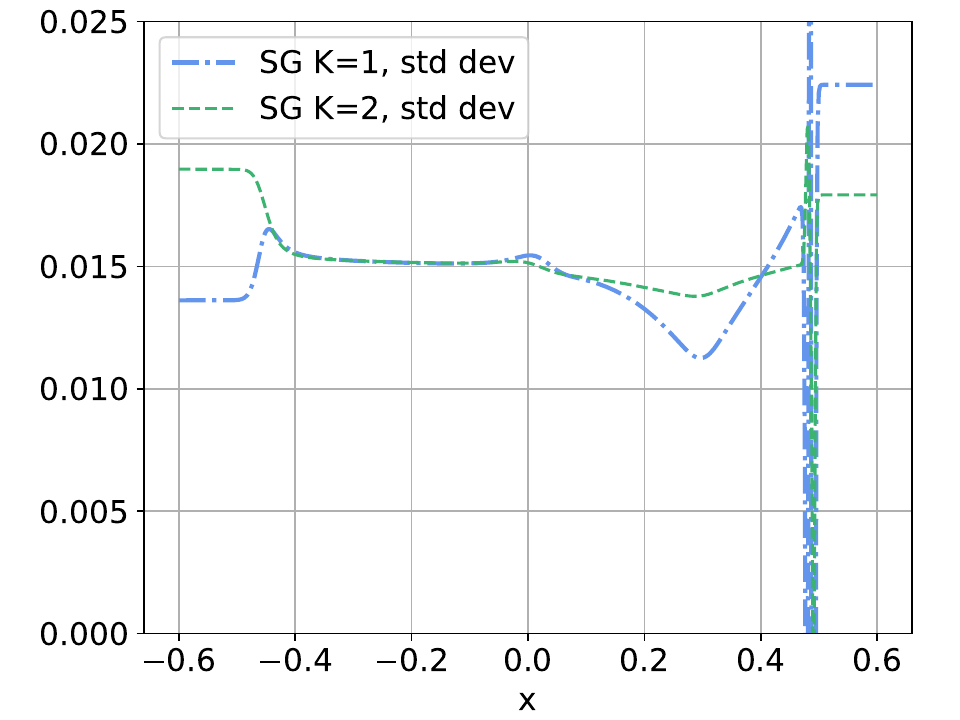}
        \vspace{-2em}
        \caption*{(h) Relative error std. dev. of $u_m$}
    \endminipage
    \minipage{0.33\textwidth}
        \centering
        \includegraphics[width=\linewidth]{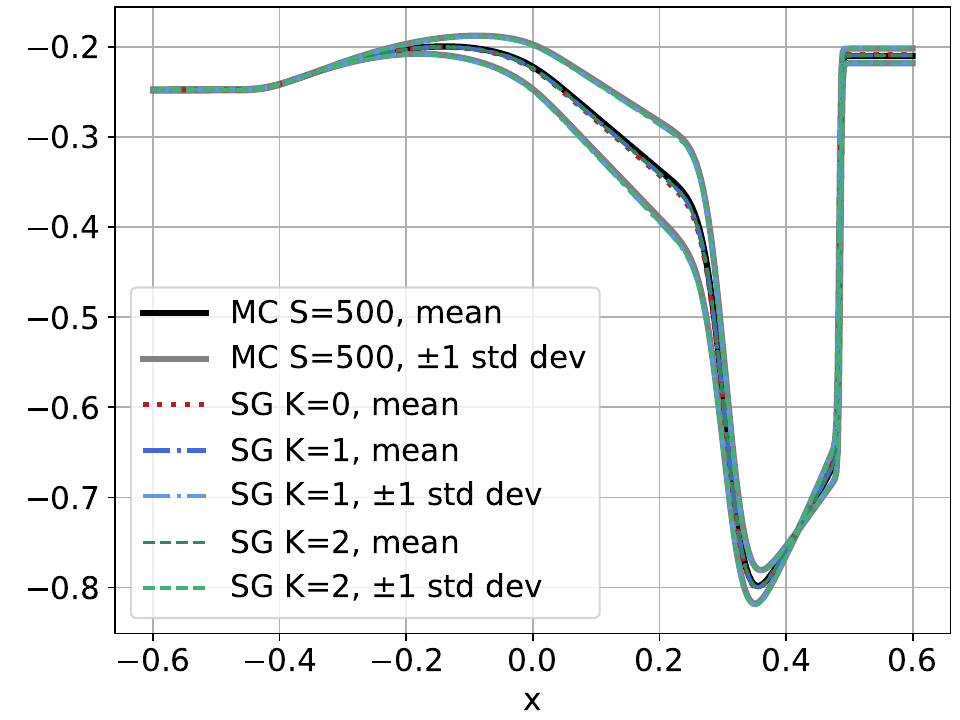}
        \vspace{-2em}
        \caption*{(c) First moment $u_1$}
        \includegraphics[width=\linewidth]{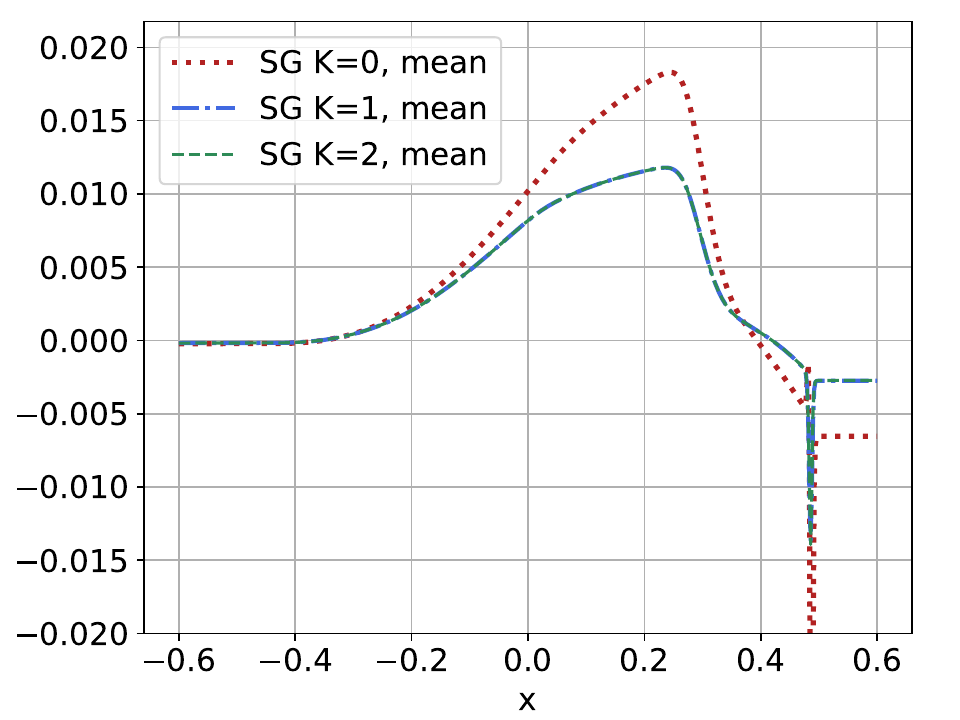}
        \vspace{-2em}
        \caption*{(f) Relative error mean of $u_1$}
        \includegraphics[width=\linewidth]{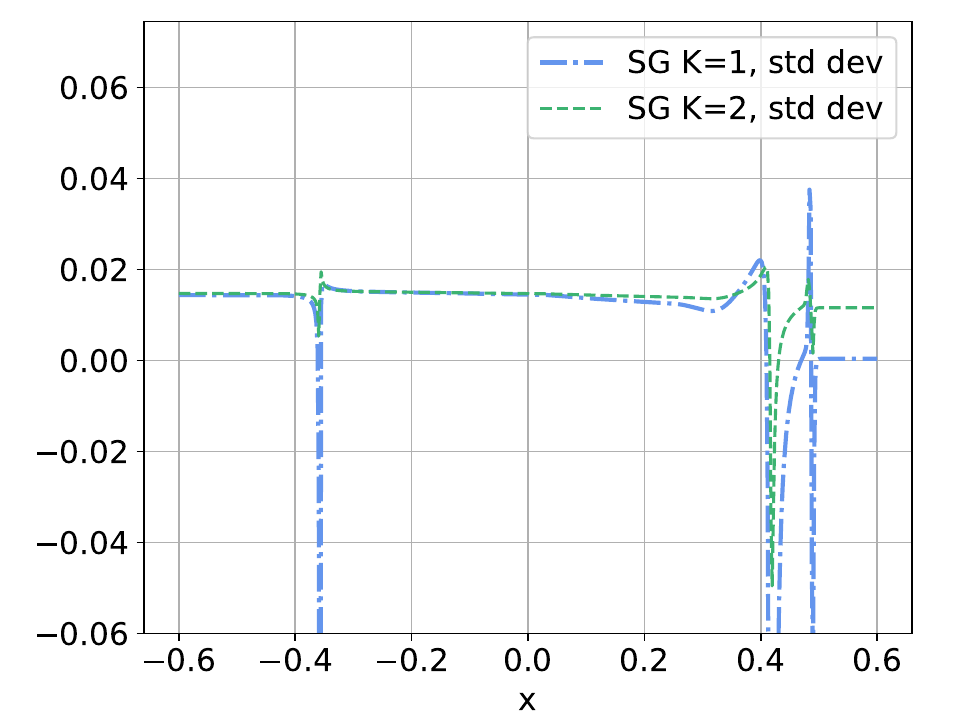}
        \vspace{-2em}
        \caption*{(i) Relative error std. dev. of $u_1$}
    \endminipage
    \caption{Stochastic Galerkin accuracy test for the high dam break test case with (SG)SWLME $N=1$ using different stochastic Galerkin orders $K$. Higher orders improve the accuracy of the solution.}\label{RES-fig:SG_N1_high}
\end{figure*}

For the high dam with $N=2$ systems, as shown in Figure \ref{RES-fig:SG_N2_high}, we observe - just as in Figure \ref{RES-fig:SG_N1_high} - a significant peak in inaccuracy around the shock of the dam break. Additionally, consistent with Figure \ref{RES-fig:SG_N2_low}, the $K=2$ solution does not always produce a more accurate result than the $K=1$ solution, although it does improve upon the deterministic curve ($K=0$).

\begin{figure*}[!htb]
    \centering
    \minipage{0.33\textwidth}
        \centering
        \includegraphics[width=0.9\linewidth]{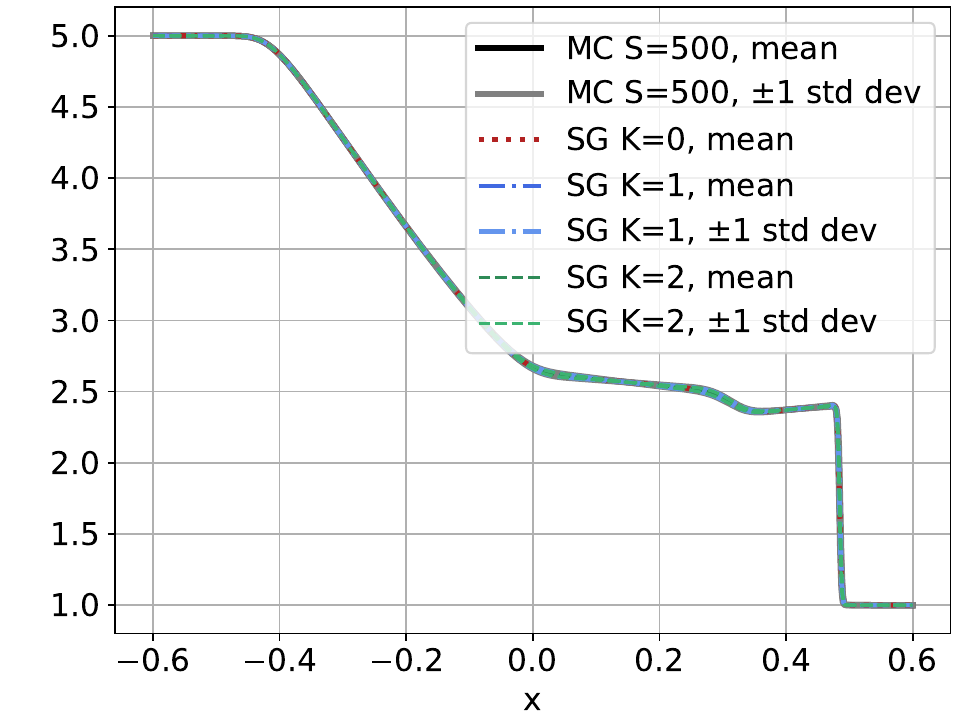}
        \vspace{-1em}
        \caption*{(a) Water height $h$}
        \includegraphics[width=0.9\linewidth]{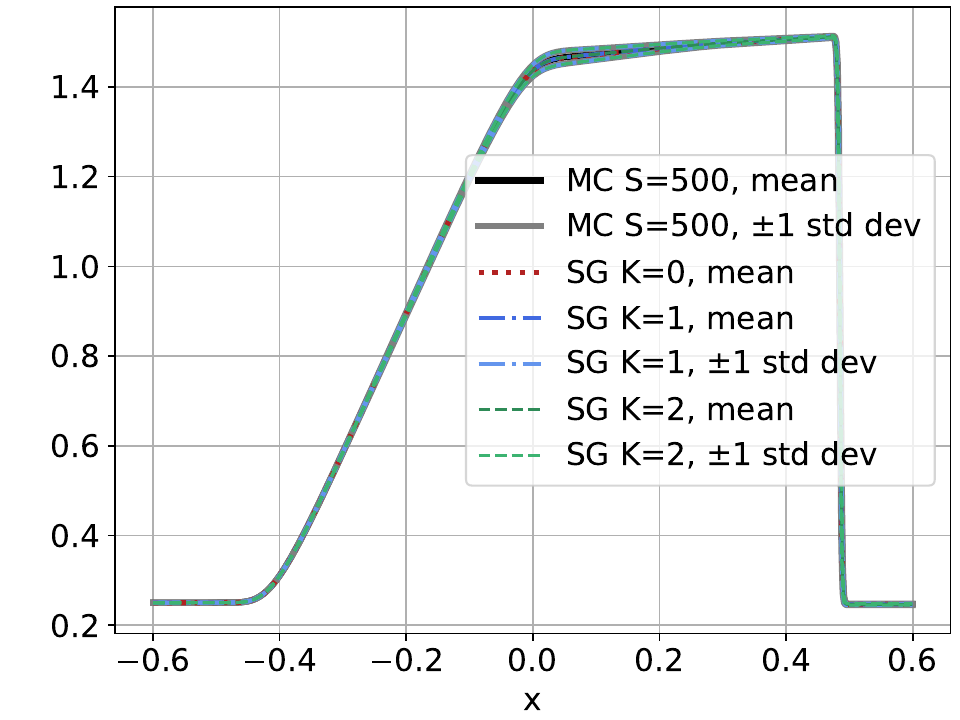}
        \vspace{-1em}
        \caption*{(d) Average velocity $u_m$}
        \includegraphics[width=0.9\linewidth]{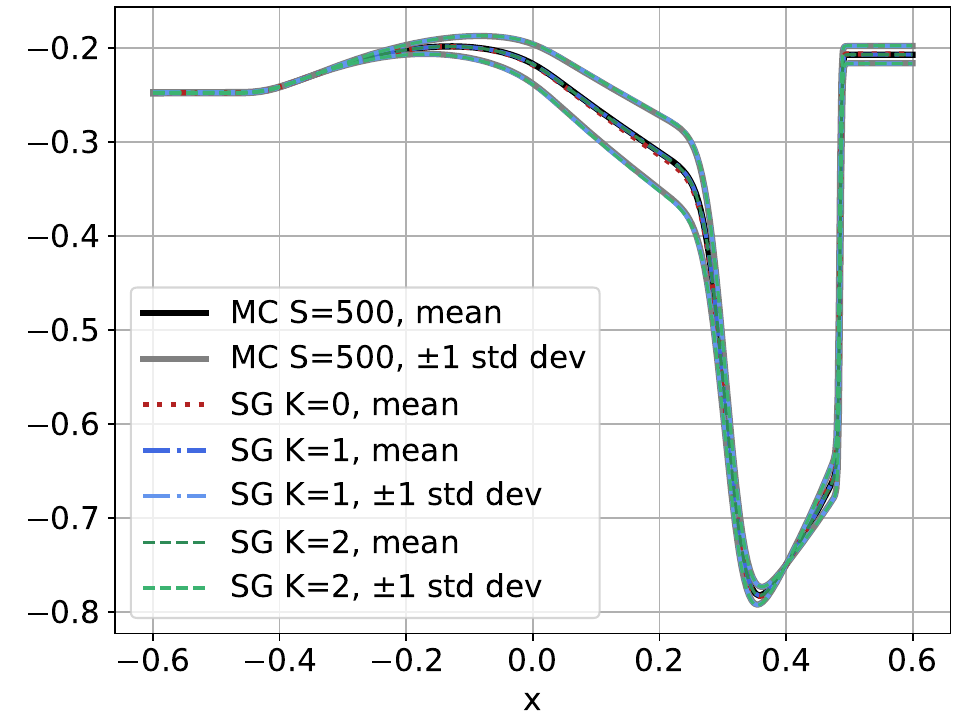}
        \vspace{-1em}
        \caption*{(g) First moment $u_1$}
        \includegraphics[width=0.9\linewidth]{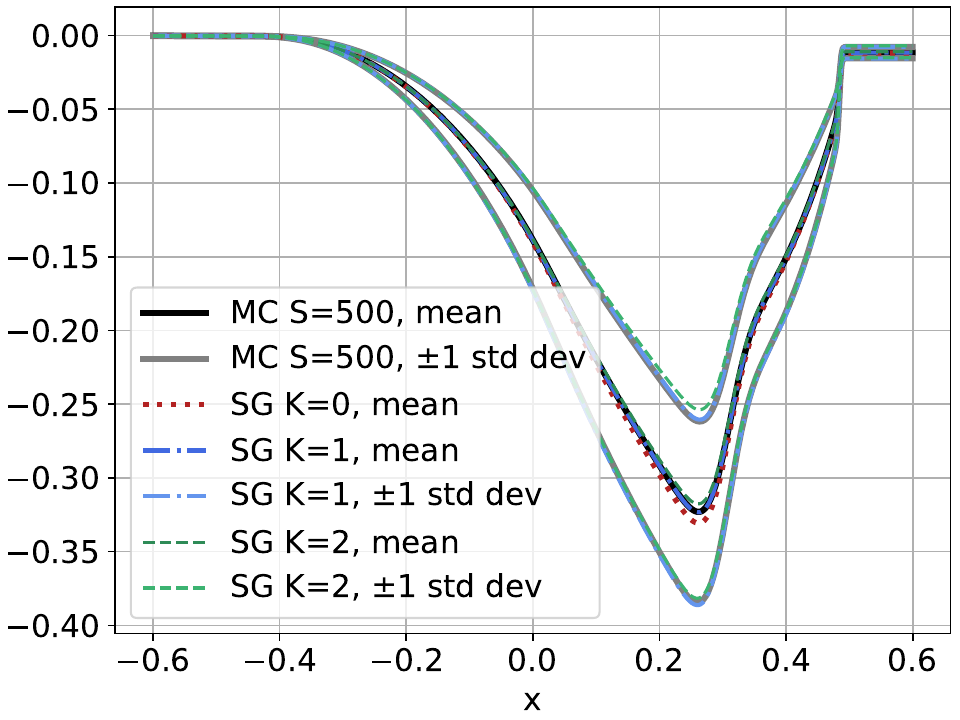}
        \vspace{-1em}
        \caption*{(j) Second moment $u_2$}
    \endminipage
    \minipage{0.33\textwidth}
        \centering
        \includegraphics[width=0.9\linewidth]{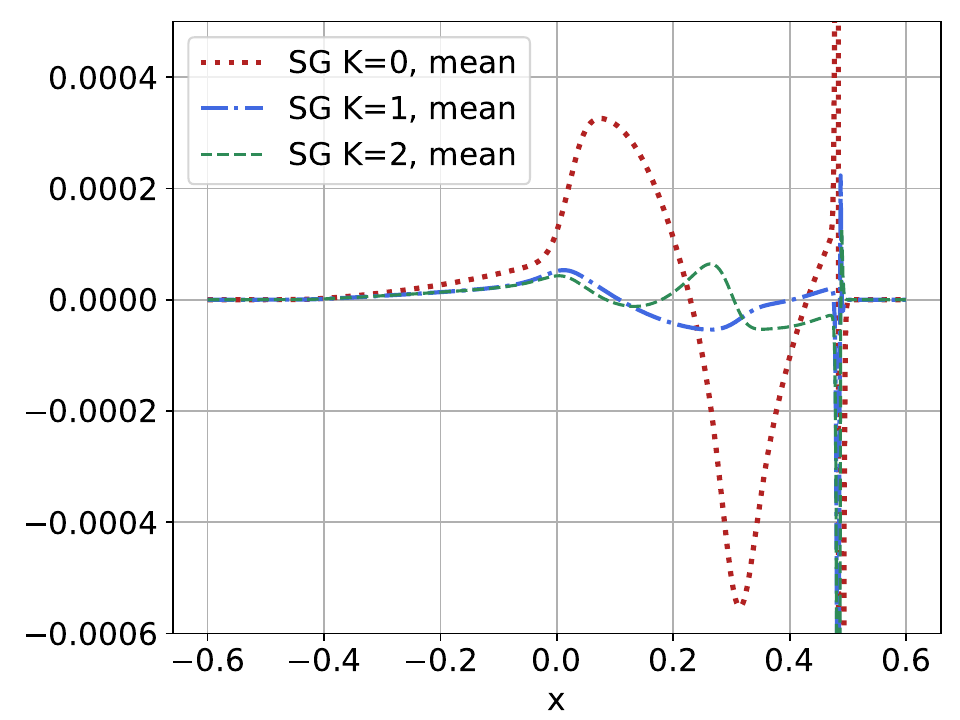}
        \vspace{-1em}
        \caption*{(b) Relative error mean of $h$}
        \includegraphics[width=0.9\linewidth]{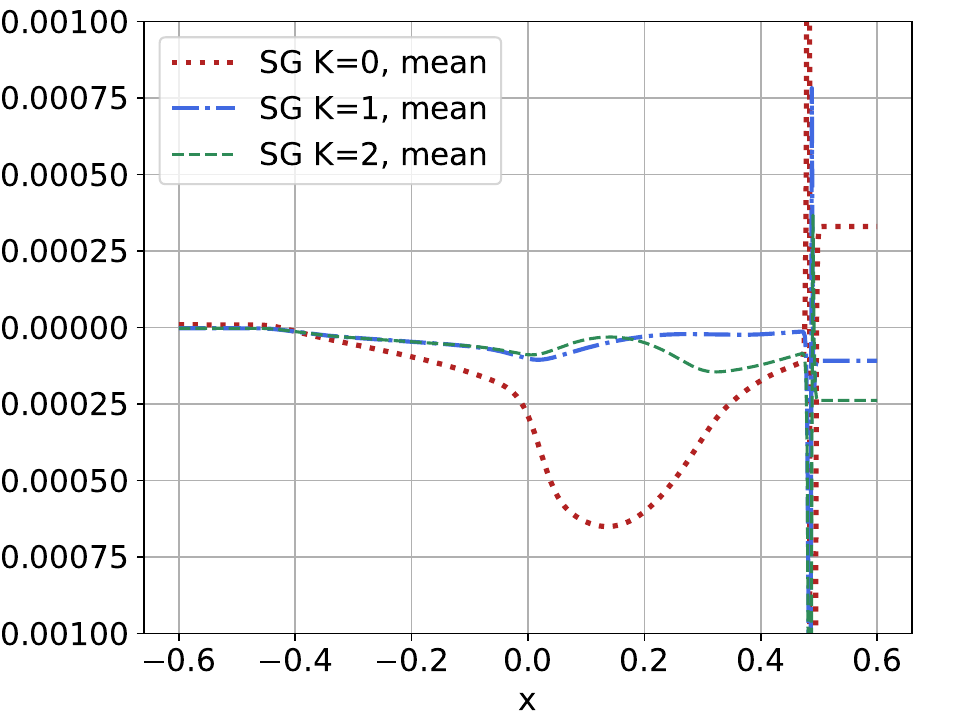}
        \vspace{-1em}
        \caption*{(e) Relative error mean of $u_m$}
        \includegraphics[width=0.9\linewidth]{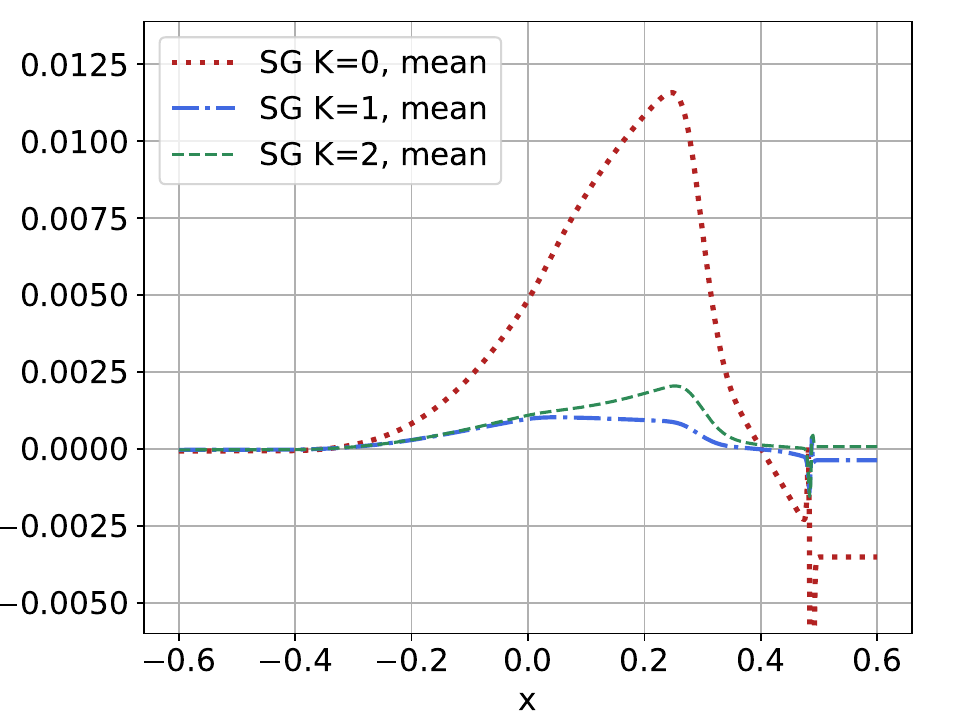}
        \vspace{-1em}
        \caption*{(h) Relative error mean of $u_1$}
        \includegraphics[width=0.9\linewidth]{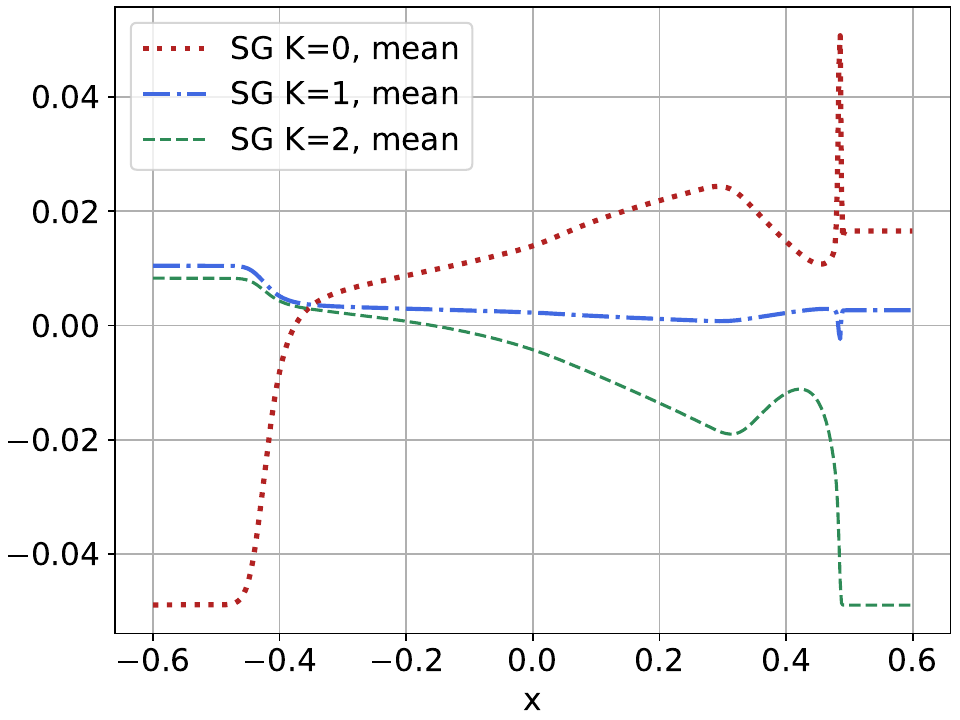}
        \vspace{-1em}
        \caption*{(k) Relative error mean of $u_2$}
    \endminipage
    \minipage{0.33\textwidth}
        \centering
        \includegraphics[width=0.9\linewidth]{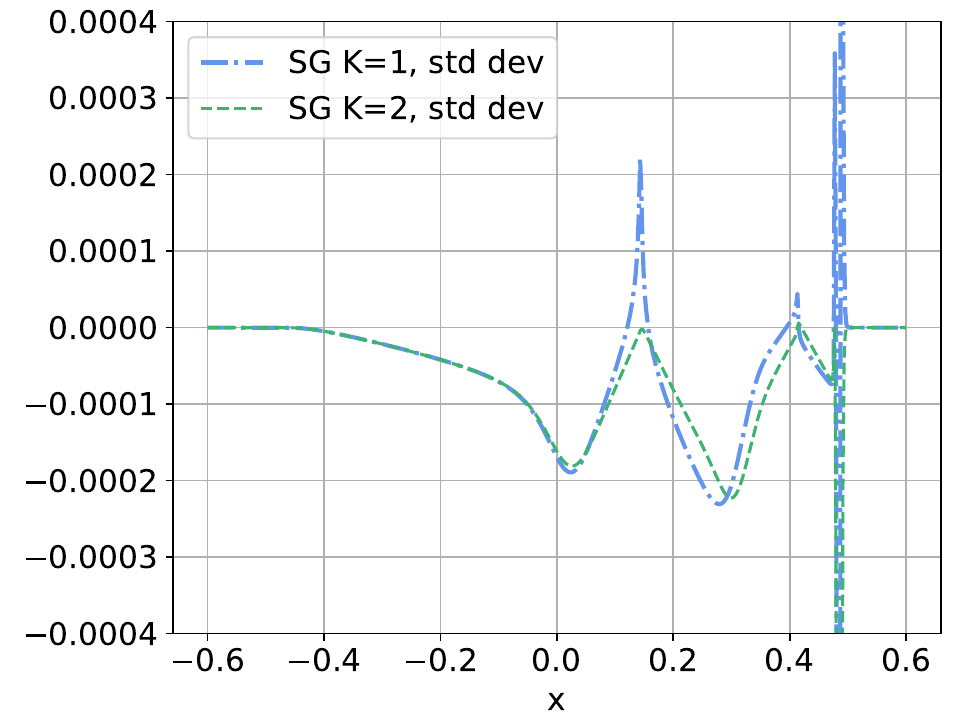}
        \vspace{-1em}
        \caption*{(c) Relative error std. dev. of $h$}
        \includegraphics[width=0.9\linewidth]{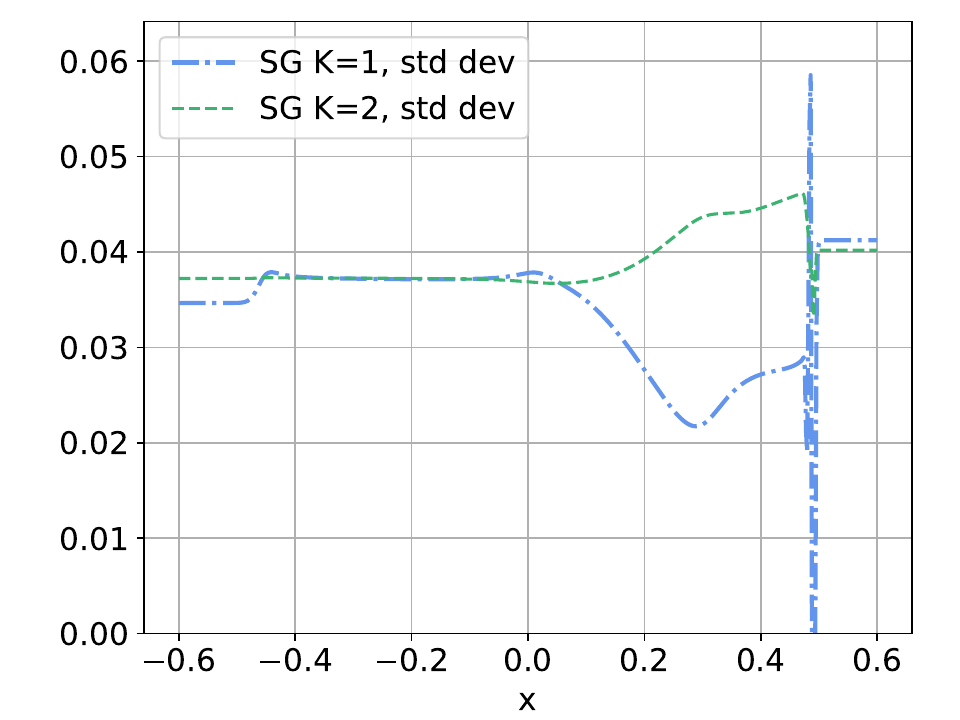}
        \vspace{-1em}
        \caption*{(f) Relative error std. dev. of $u_m$}
        \includegraphics[width=0.9\linewidth]{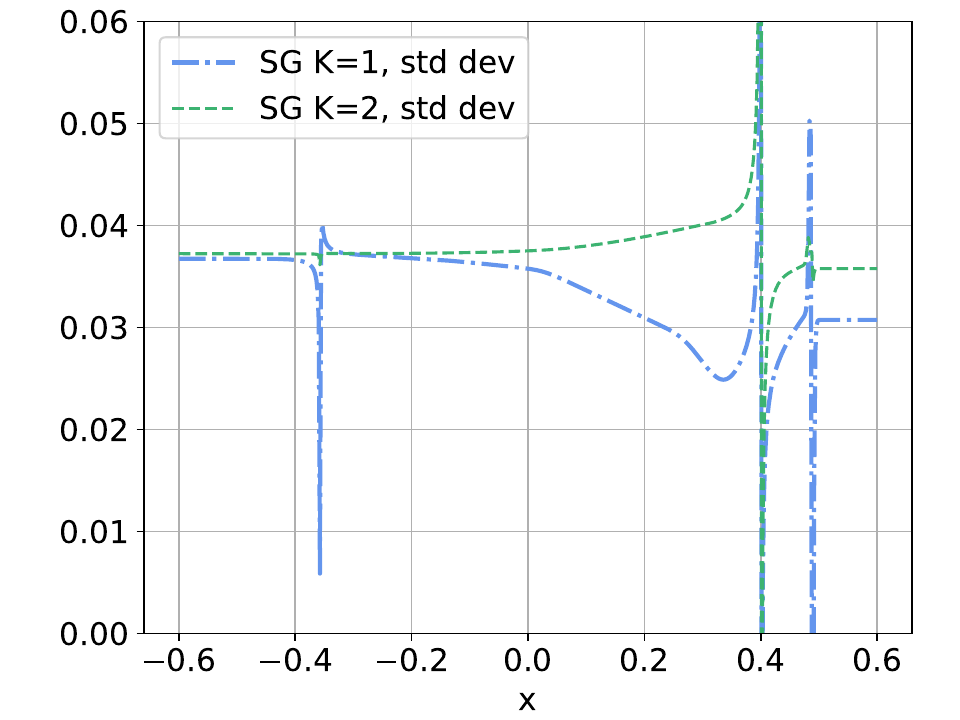}
        \vspace{-1em}
        \caption*{(i) Relative error std. dev. of $u_1$}
        \includegraphics[width=0.9\linewidth]{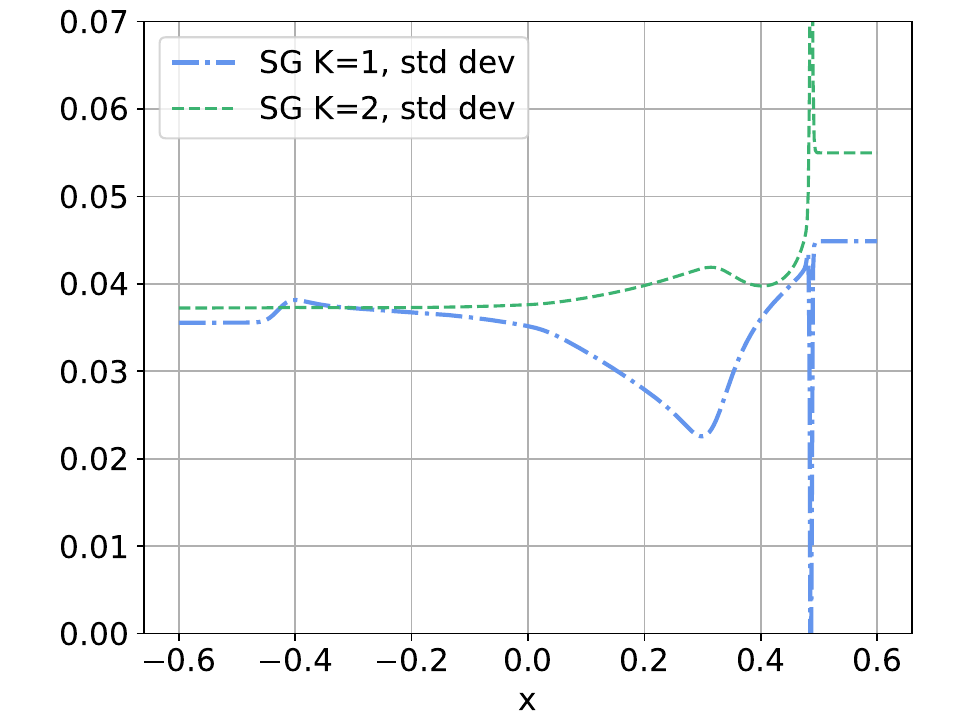}
        \vspace{-1em}
        \caption*{(l) Relative error std. dev. of $u_2$}
    \endminipage
    \caption{Stochastic Galerkin accuracy test for the high dam break test case with (SG)SWLME $N=2$ using different stochastic Galerkin orders $K$. While increasing the order generally enhances the accuracy of the simulation, its effectiveness varies across different variables.}\label{RES-fig:SG_N2_high}
\end{figure*}

\newpage
In summary, the dam break test cases demonstrate the ability of the stochastic Galerkin solution to produce a relatively accurate solution in a single run, compared to the Monte Carlo reference solution, which requires many (albeit faster) runs. However, the accuracy is not uniform across all variables and the relative errors remain large around the discontinuities.

For the smooth wave test case (Table \ref{RES-tab:smooth_wave}) with $N=1$ systems, as shown in Figure \ref{RES-fig:SG_N1_wave}, inaccuracy peaks around $x=0.4$, at the wave crest. For all three variables - $h$, $u_m$ and $u_1$ - the maximal relative error clearly decreases as the order $K$ increases. It is also evident that the mean for $K=1$ and $K=2$ does not coincide with the deterministic result ($K=0$). Additionally, the inaccuracy for $K=1$ and $K=2$ remains larger in the standard deviation than in the mean.

\begin{figure*}[!htb]
    \centering
    \minipage{0.33\textwidth}
        \centering
        \includegraphics[width=\linewidth]{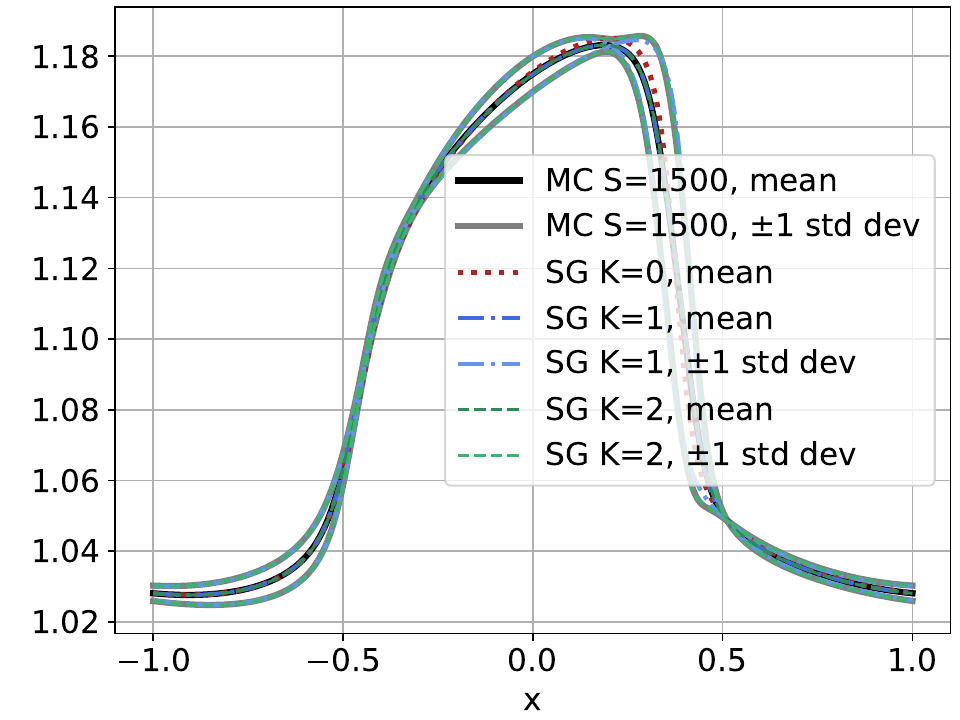}
        \vspace{-2em}
        \caption*{(a) Water height $h$}
        \includegraphics[width=\linewidth]{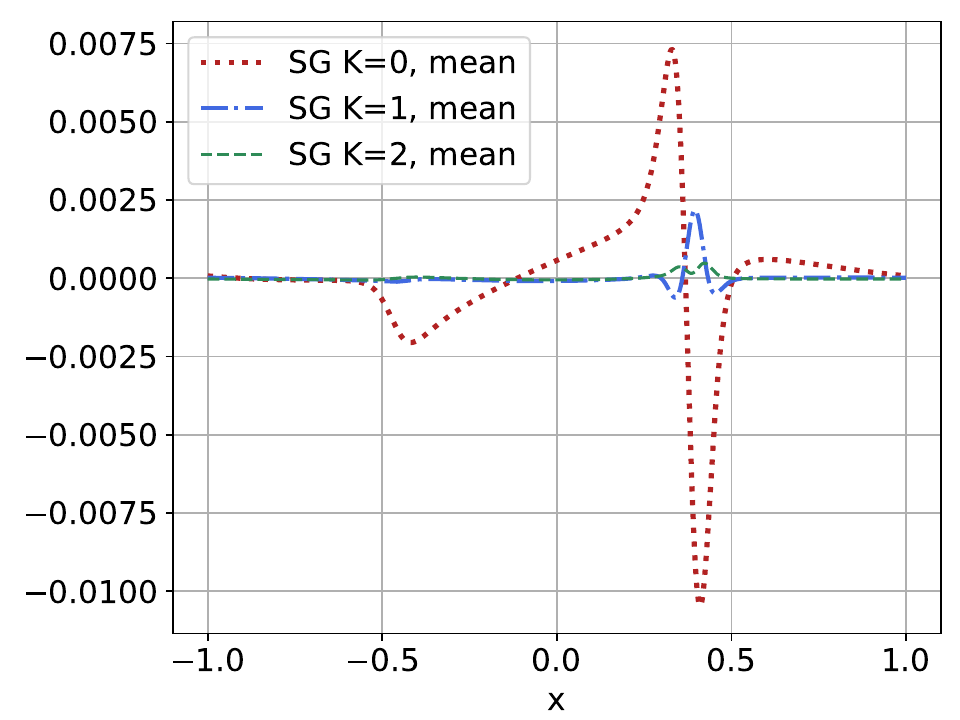}
        \vspace{-2em}
        \caption*{(d) Relative error mean of $h$}
        \includegraphics[width=\linewidth]{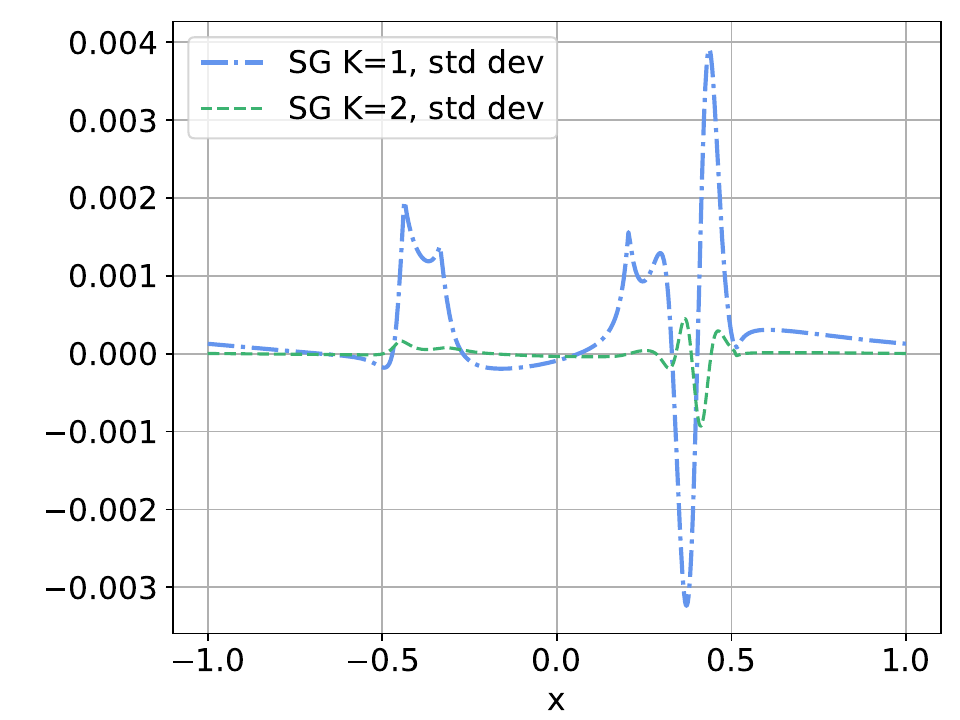}
        \vspace{-2em}
        \caption*{(g) Relative error std. dev. of $h$}
    \endminipage
    \minipage{0.33\textwidth}
        \centering
        \includegraphics[width=\linewidth]{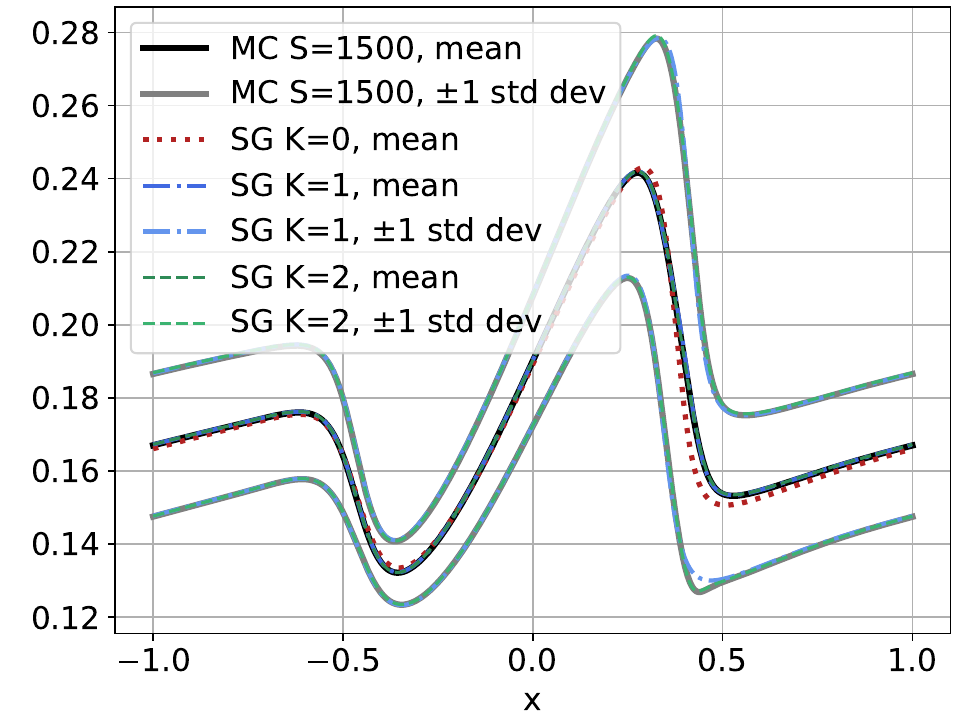}
        \vspace{-2em}
        \caption*{(b) Average velocity $u_m$}
        \includegraphics[width=\linewidth]{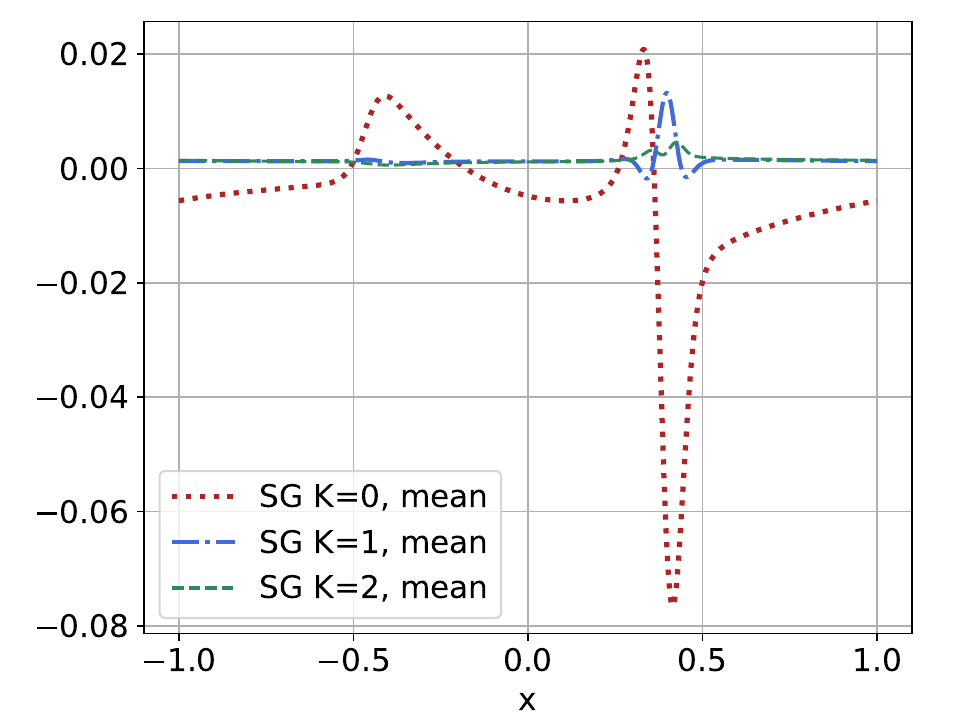}
        \vspace{-2em}
        \caption*{(e) Relative error mean of $u_m$}
        \includegraphics[width=\linewidth]{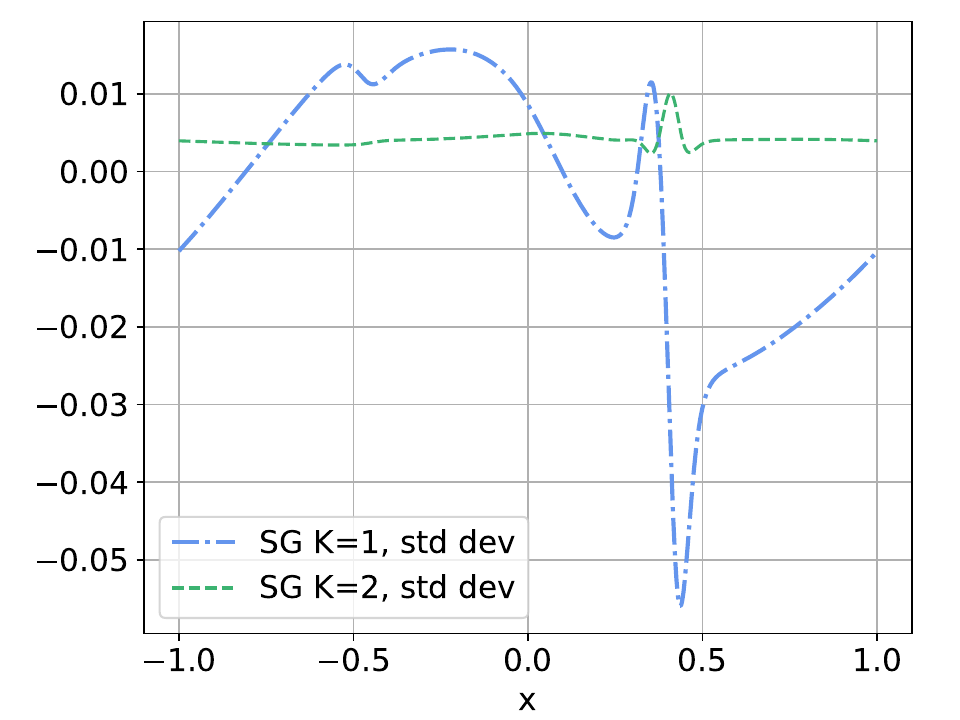}
        \vspace{-2em}
        \caption*{(h) Relative error std. dev. of $u_m$}
    \endminipage
    \minipage{0.33\textwidth}
        \centering
        \includegraphics[width=\linewidth]{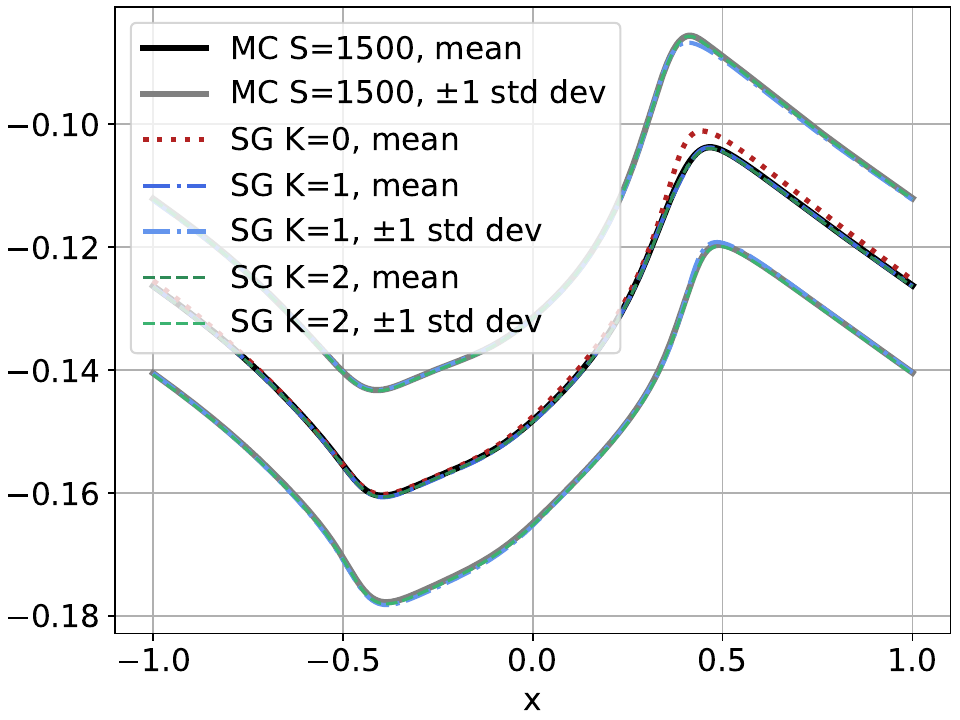}
        \vspace{-2em}
        \caption*{(c) First moment $u_1$}
        \includegraphics[width=\linewidth]{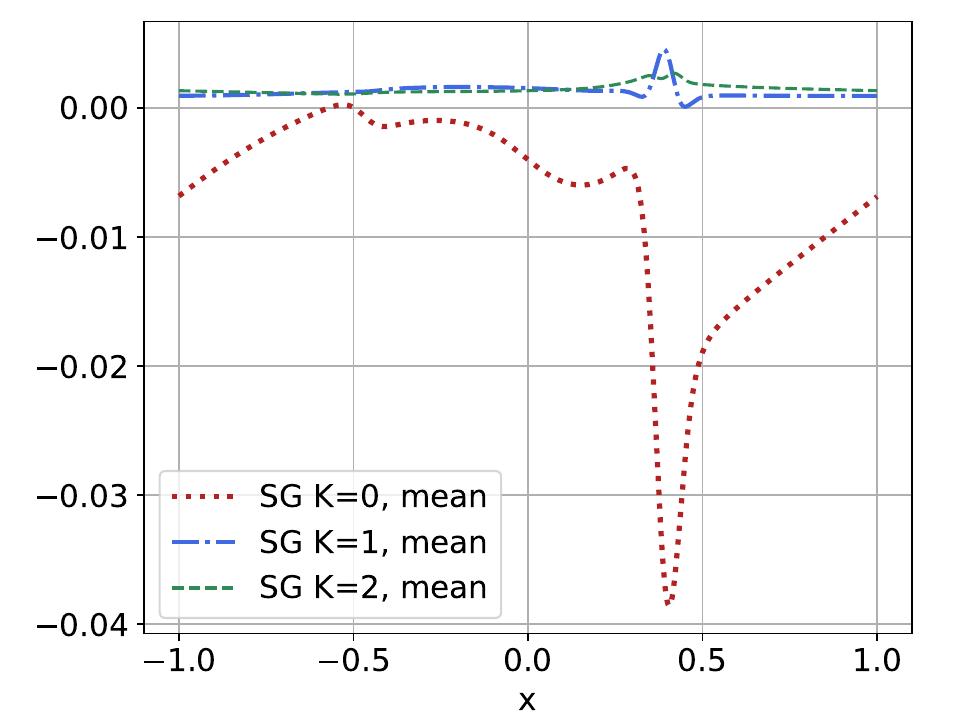}
        \vspace{-2em}
        \caption*{(f) Relative error mean of $u_1$}
        \includegraphics[width=\linewidth]{Figures/SG_N1_wave_u1_difv.pdf}
        \vspace{-2em}
        \caption*{(i) Relative error std. dev. of $u_1$}
    \endminipage
    \caption{Stochastic Galerkin accuracy test for the smooth periodic wave test case with (SG)SWLME $N=1$ using different stochastic Galerkin orders $K$. Higher orders improve the accuracy of the solution.}\label{RES-fig:SG_N1_wave}
\end{figure*}

In the smooth wave test case for $N=2$ systems, as shown in Figure \ref{RES-fig:SG_N2_wave}, we observe a behaviour similar to that in Figure \ref{RES-fig:SG_N1_wave}, with peaks in the relative error of the mean around $x=0.4$, near the wave crest. The $K=1$ solution yields small errors for all functions, both in mean and standard deviation, except for the standard deviation of the second moment $u_2$. However, the second moment $u_2$ is a very small quantity, with a range of only 0.05.

In general, we again observe that the deterministic case ($K=0$) does not always provide a good approximation of the mean of a function. This is due to the non-linear influence of the friction coefficient $\nu$ on the four functions $h$, $u_m$, $u_1$, and $u_2$. Additionally, for the second moment $u_2$, the $K=2$ curve begins to deviate significantly from the Monte Carlo solution, although the absolute scale remains small. Consequently, we conclude that the $K=1$ curve is more reliable than the $K=2$ curve.

\begin{figure*}[!htb]
    \centering
    \minipage{0.33\textwidth}
        \centering
        \includegraphics[width=0.9\linewidth]{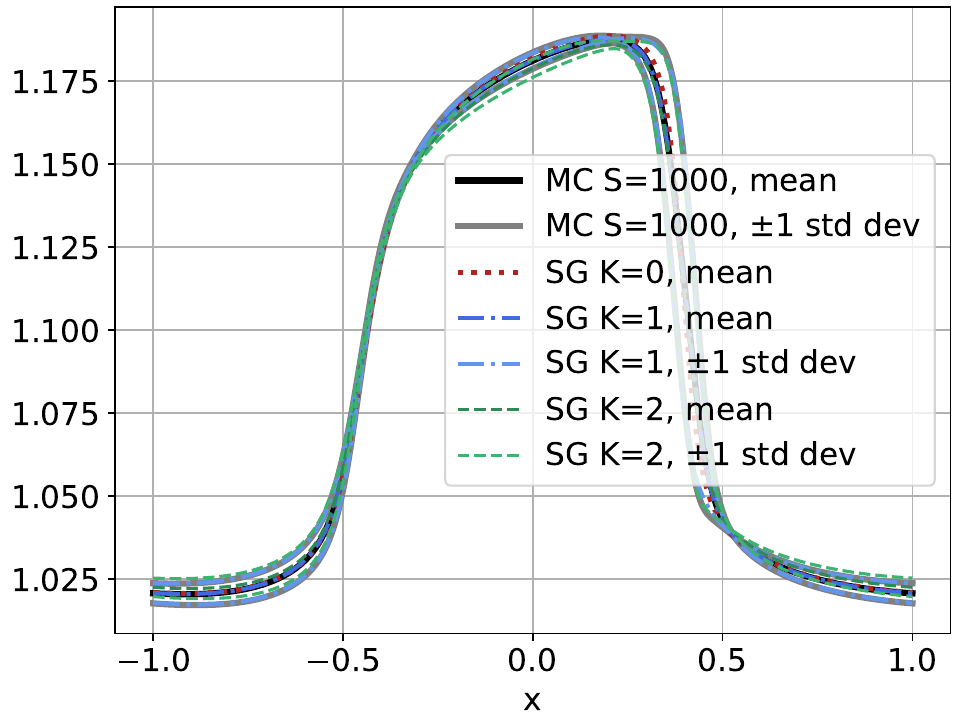}
        \vspace{-1em}
        \caption*{(a) Water height $h$}
        \includegraphics[width=0.9\linewidth]{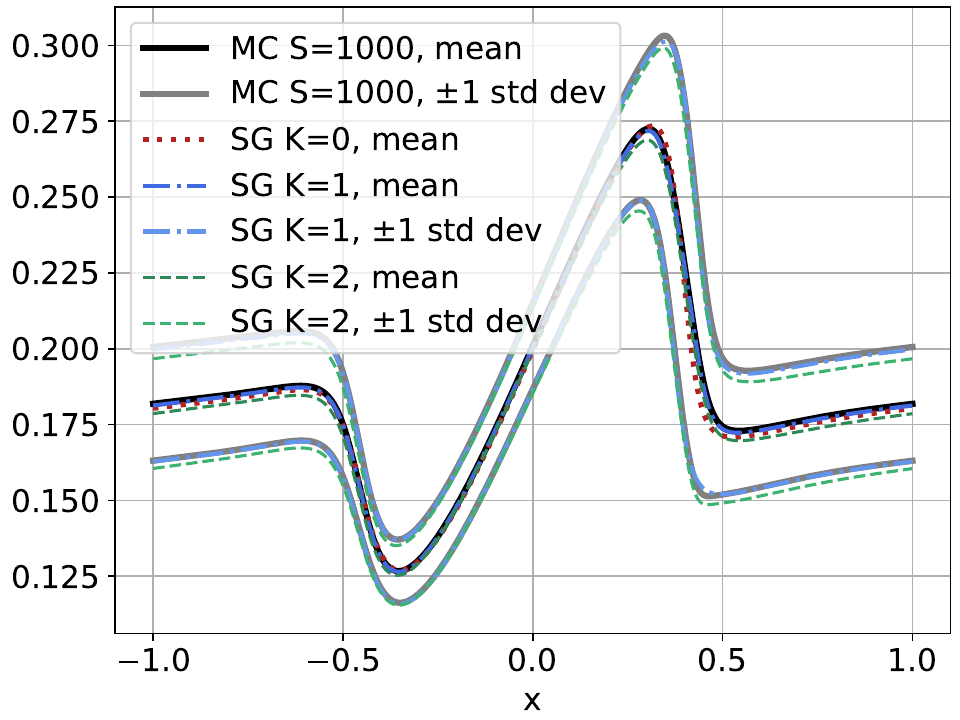}
        \vspace{-1em}
        \caption*{(d) Average velocity $u_m$}
        \includegraphics[width=0.9\linewidth]{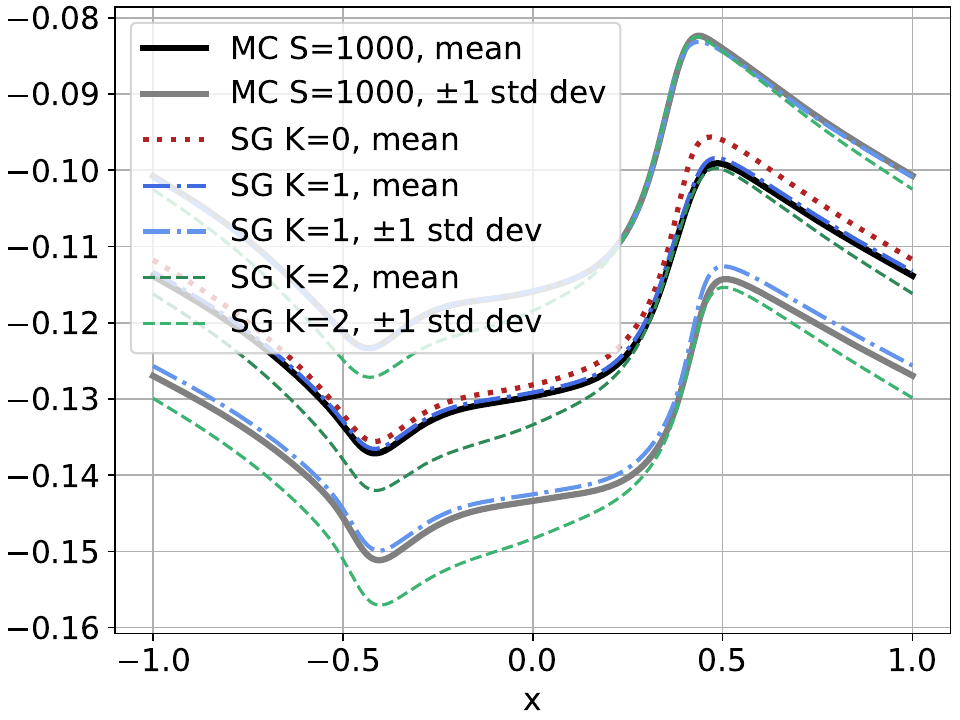}
        \vspace{-1em}
        \caption*{(g) First moment $u_1$}
        \includegraphics[width=0.9\linewidth]{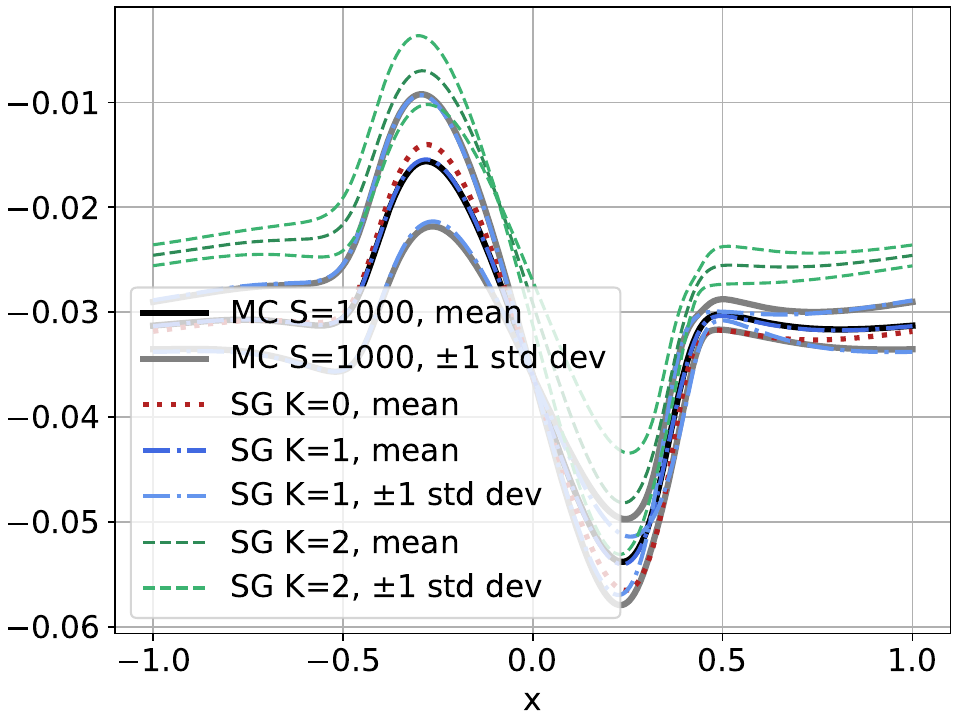}
        \vspace{-1em}
        \caption*{(j) Second moment $u_2$}
    \endminipage
    \minipage{0.33\textwidth}
        \centering
        \includegraphics[width=0.9\linewidth]{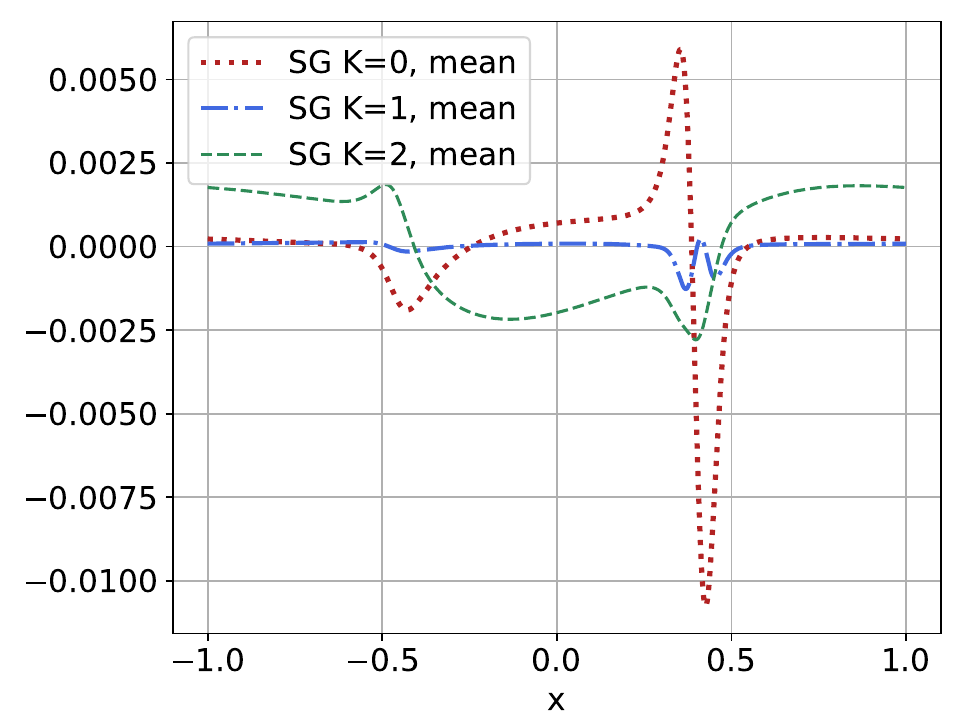}
        \vspace{-1em}
        \caption*{(b) Relative error mean of $h$}
        \includegraphics[width=0.9\linewidth]{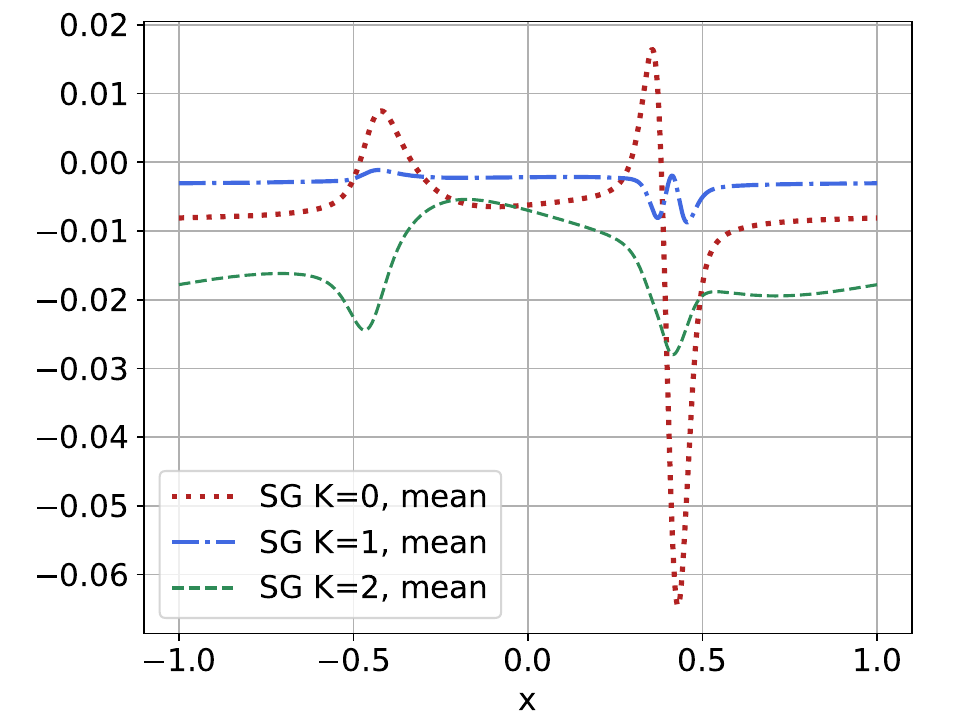}
        \vspace{-1em}
        \caption*{(e) Relative error mean of $u_m$}
        \includegraphics[width=0.9\linewidth]{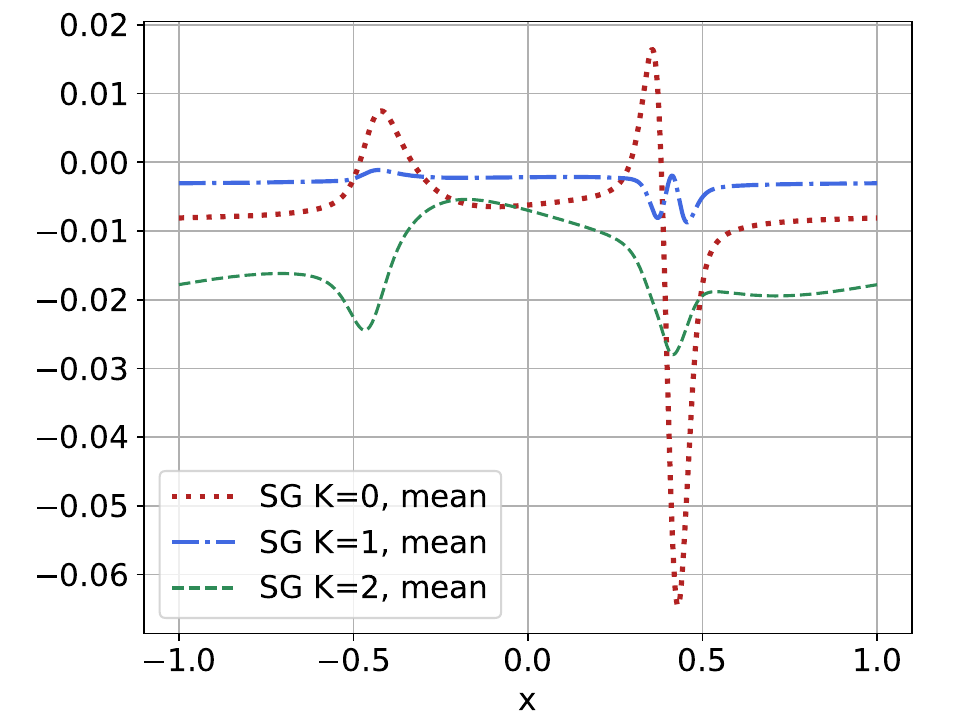}
        \vspace{-1em}
        \caption*{(h) Relative error mean of $u_1$}
        \includegraphics[width=0.9\linewidth]{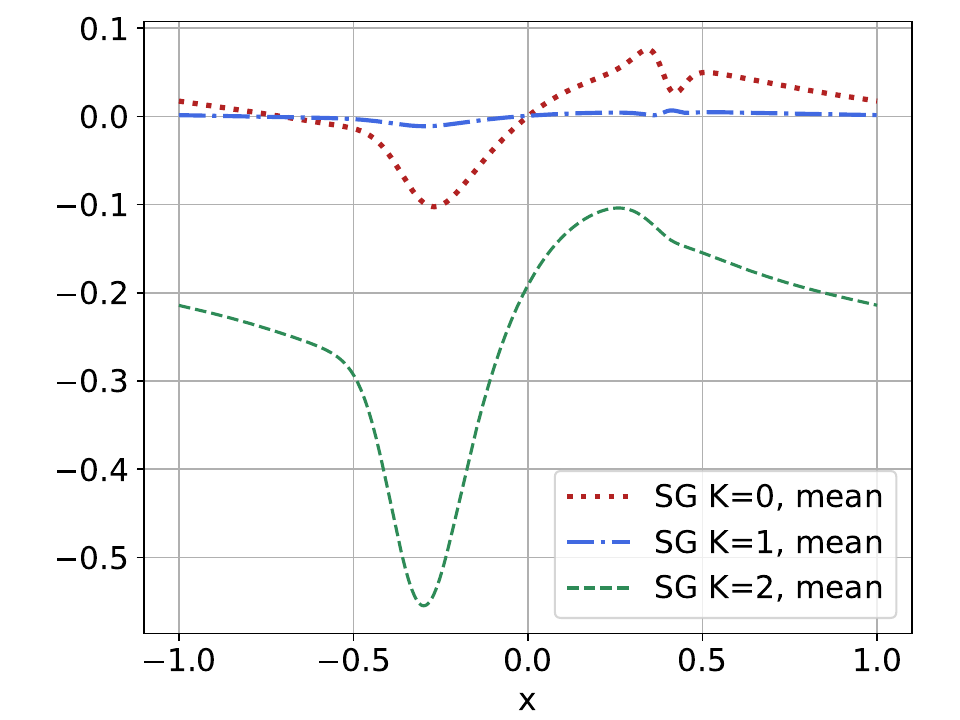}
        \vspace{-1em}
        \caption*{(k) Relative error mean of $u_2$}
    \endminipage
    \minipage{0.33\textwidth}
        \centering
        \includegraphics[width=0.9\linewidth]{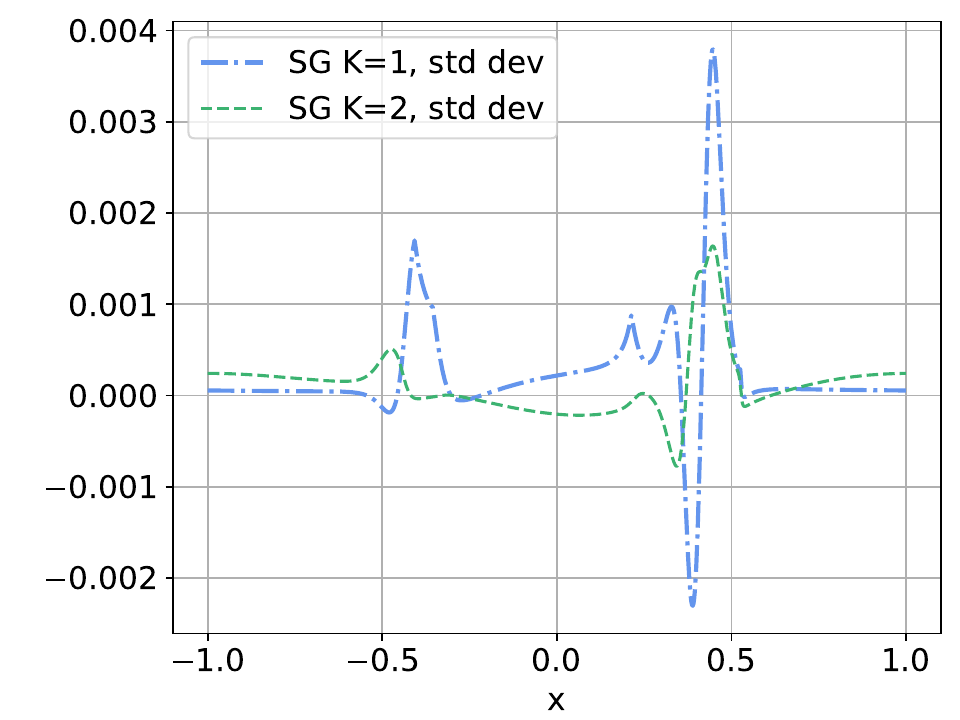}
        \vspace{-1em}
        \caption*{(c) Relative error std. dev. of $h$}
        \includegraphics[width=0.9\linewidth]{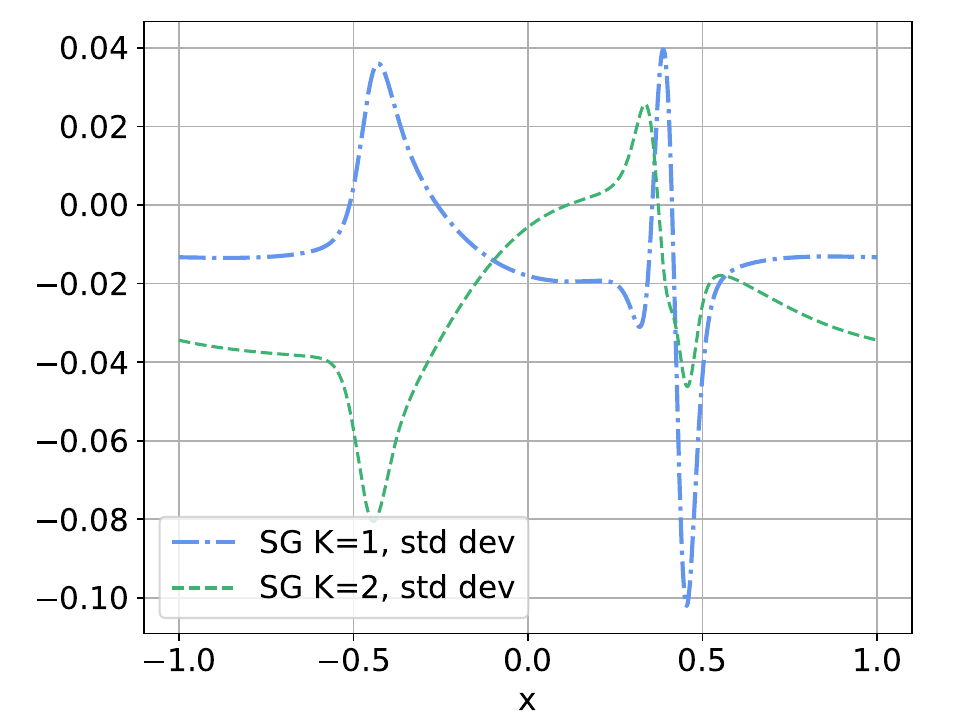}
        \vspace{-1em}
        \caption*{(f) Relative error std. dev. of $u_m$}
        \includegraphics[width=0.9\linewidth]{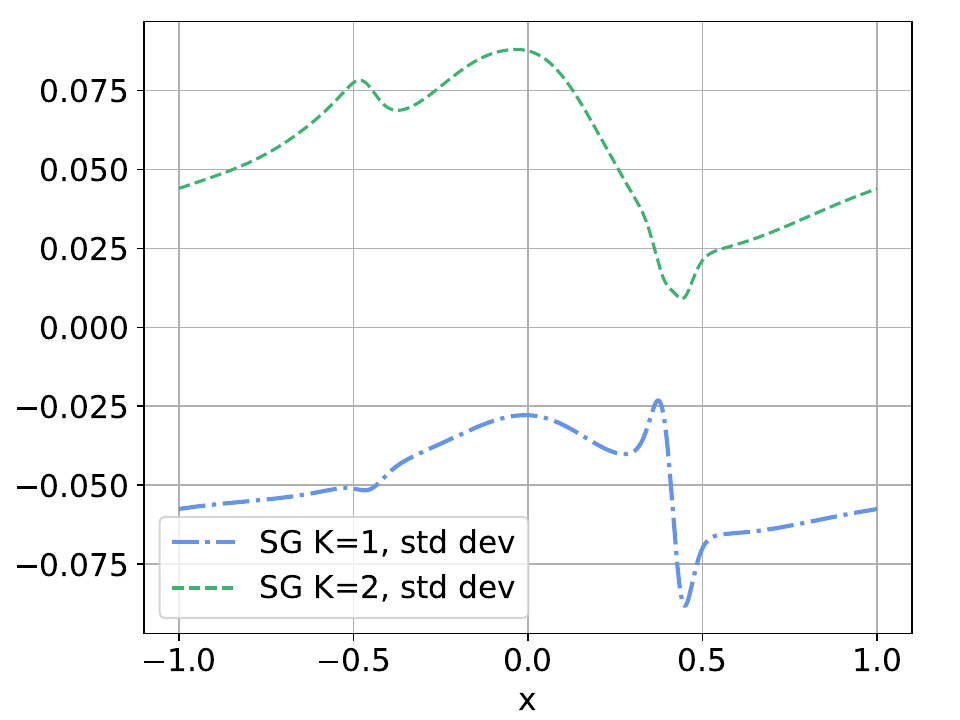}
        \vspace{-1em}
        \caption*{(i) Relative error std. dev. of $u_1$}
        \includegraphics[width=0.9\linewidth]{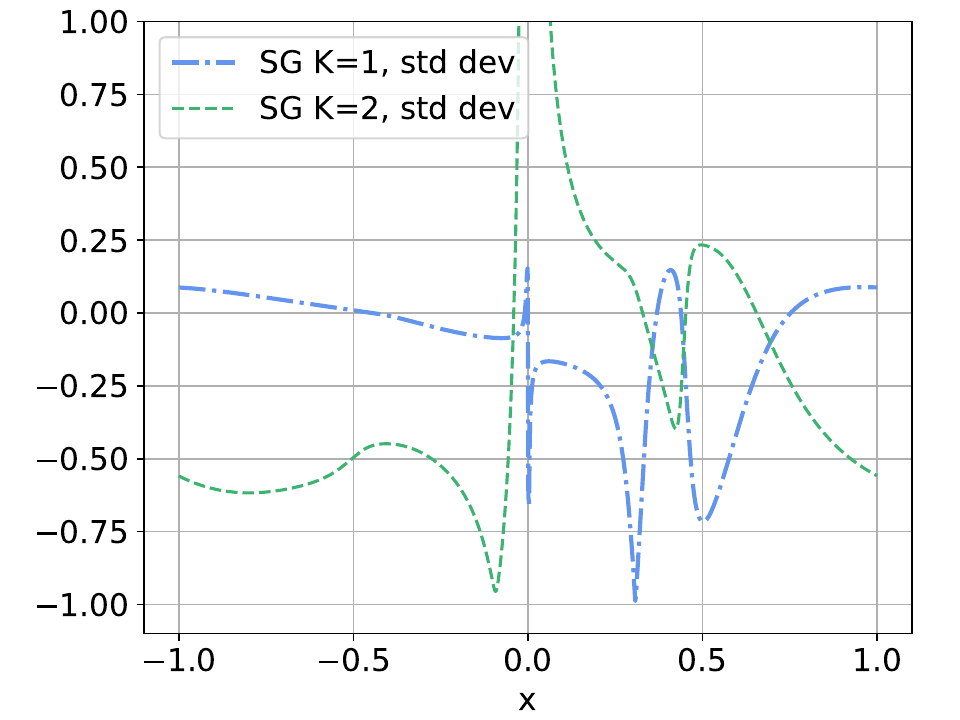}
        \vspace{-1em}
        \caption*{(l) Relative error std. dev. of $u_2$}
    \endminipage
    \caption{Stochastic Galerkin accuracy test for the smooth periodic wave test case with (SG)SWLME $N=2$ using different stochastic Galerkin orders $K$. Not for all variables does increasing the order improve the accuracy of the simulation.}\label{RES-fig:SG_N2_wave}
\end{figure*}


\newpage
In Table \ref{RES-tab:rel_runtimes}, we compare the run time of Monte Carlo and stochastic Galerkin solutions relative to the time required to run a single Monte Carlo sample for $K=0$. Although increasing the stochastic Galerkin order from $K=0$ to $K=1$ results in only a modest increase in run time, increasing the order from $K=1$ to $K=2$ significantly impacts performance. This is due to the larger system size and the greater complexity of the matrix and vector entries involved. Nevertheless, the stochastic Galerkin approach consistently achieves a substantial speedup compared to Monte Carlo with $S=100-350$ samples for the dam break test case and $S=550-900$ samples for the smooth periodic wave test case.

For a global estimate of mean and standard deviation, the $K=1$ system can sometimes offer a better balance between run time and accuracy than $K=2$. For the $N=2$ case (see Figures \ref{RES-fig:SG_N2_low}, \ref{RES-fig:SG_N2_high}, and \ref{RES-fig:SG_N2_wave}), we found that the $K=1$ solutions were at times even more accurate than the $K=2$ solutions.
    
\begin{table}[H]
    \centering
    \begin{tabular}{||c|c|c|c||}
        \hhline{#=|=|=|=#} \textbf{Model} & \textbf{Method} & \textbf{Dam Break} & \textbf{Smooth Wave} \\
        \hhline{#=|=|=|=#} $N=1$ & Deterministic & 1 & 1 \\
        \hline $N=1$ & Monte Carlo & 100 & 900 \\
        \hline $N=1$ & Stochastic Galerkin $K=1$ & 2.0 & 2.0 \\
        \hline $N=1$ & Stochastic Galerkin $K=2$ & 38.1 & 37.9 \\
        \hhline{#=|=|=|=#} $N=2$ & Deterministic & 1 & 1 \\
        \hline $N=2$ & Monte Carlo & 100-350 & 550 \\
        \hline $N=2$ & Stochastic Galerkin $K=1$ & 1.9 & 2.1 \\
        \hline $N=2$ & Stochastic Galerkin $K=2$ & 58.0 & 58.3 \\
        \hhline{#=|=|=|=#}
    \end{tabular}
    \caption{Comparison of relative run times. Stochastic Galerkin yields substantial speedup compared to Monte Carlo.}
    \label{RES-tab:rel_runtimes}
\end{table}

In this particular case, where the uncertainty is limited to a one-dimensional parameter - the friction coefficient $\nu$ - stochastic collocation methods could also have been employed. These methods offer faster convergence than Monte Carlo and allow the use of existing deterministic solvers without modification, much like Monte Carlo. However, stochastic collocation methods are still limited by the curse of dimensionality, as computational costs grow rapidly with the number of uncertain parameters. For scenarios involving uncertain bottom topography, such as those in \cite{dai_hyperbolicity-preserving_2021,dai_hyperbolicity-preserving_2022}, where multiple uncertain parameters are involved, stochastic collocation becomes less suitable.

\subsection{Regularisation of shallow water linearised moment equations \texorpdfstring{$N=1$}{N=1}}
In Theorem \ref{ANAL-th:hyp_N=1}, we proved that the system matrix \eqref{SGSWME-eq:SGSWLME_flux} for $N=1$ is hyperbolic after the regularisation given in equation \eqref{ANAL-eq:linearisation} for $K>1$. In Figures \ref{RES-fig:HYP_N1_low}, \ref{RES-fig:HYP_N1_high}, and \ref{RES-fig:HYP_N1_wave}, we compare simulations using the regularised SGSWLME model with those using the standard SGSWLME model for $K=2$. For the two dam break test cases, the accuracy is almost unaffected by the regularisation. However, for the smooth periodic wave test case, the regularised solution deviates slightly more from the standard solution and exhibits a minor loss of accuracy.

\begin{figure}[!htb]
    \centering
    \minipage{0.33\textwidth}
        \centering
        \includegraphics[width=\linewidth]{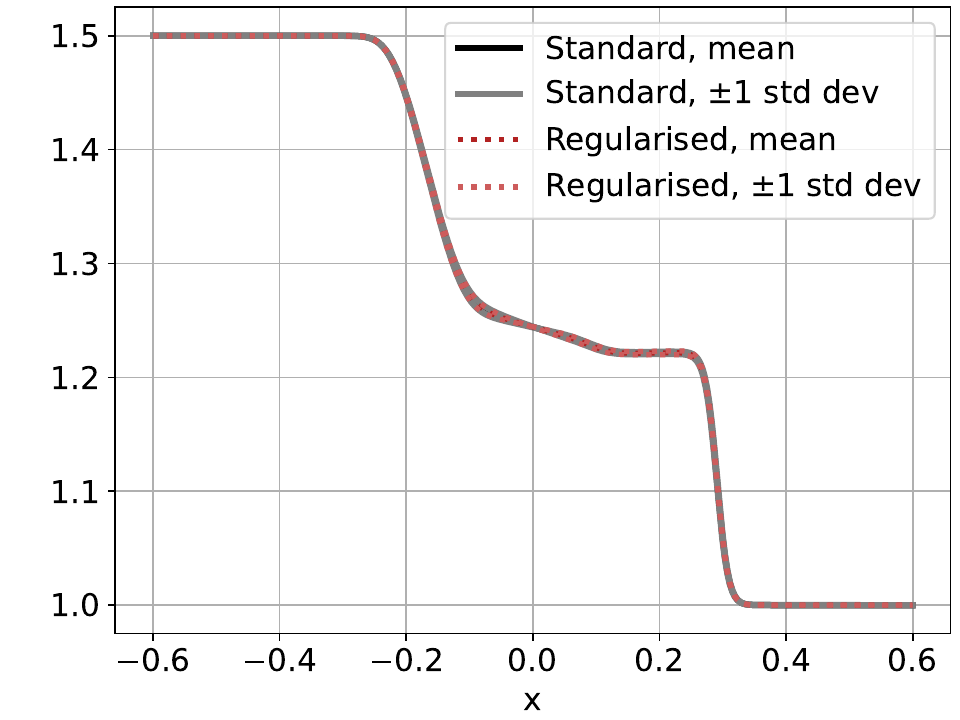}
        \vspace{-2em}
        \caption*{(a) Water height $h$}
    \endminipage\hfill
    \minipage{0.33\textwidth}
        \centering
        \includegraphics[width=\linewidth]{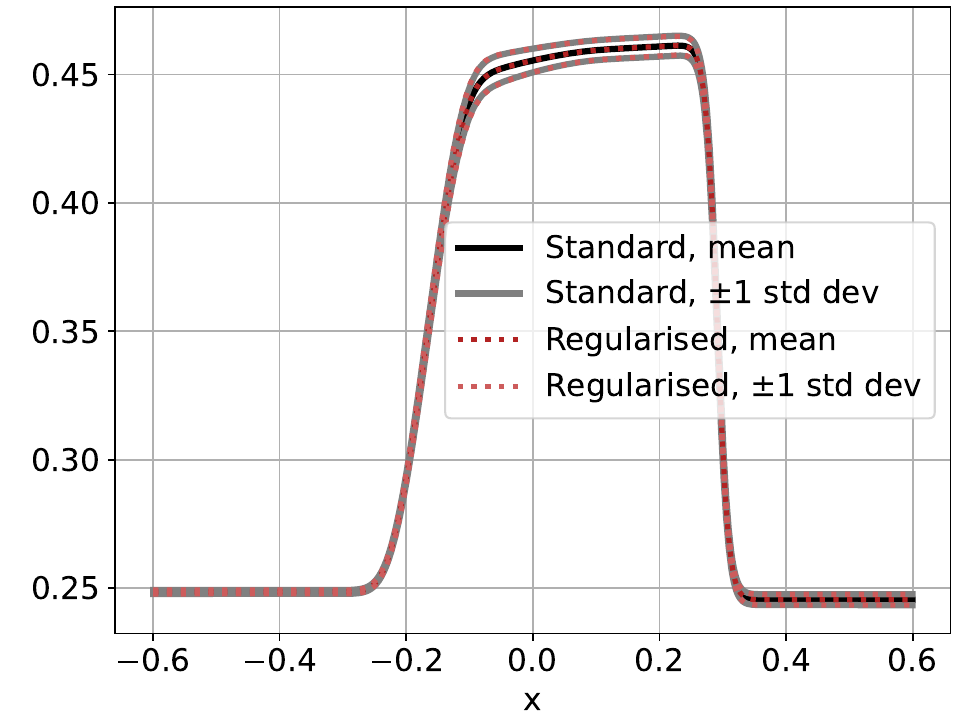}
        \vspace{-2em}
        \caption*{(b) Average velocity $u_m$}
    \endminipage\hfill
    \minipage{0.33\textwidth}
      \includegraphics[width=\linewidth]{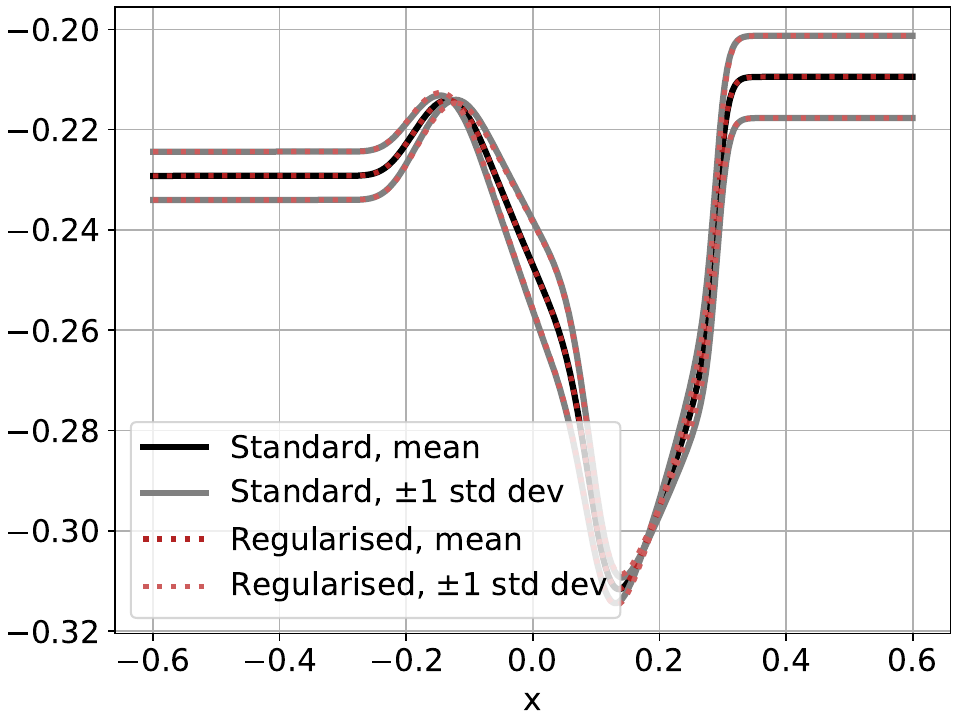}
      \vspace{-2em}
      \caption*{(c) First moment $u_1$}
    \endminipage
    \caption{Regularisation accuracy test for the low dam break test case with SGSWLME $N=1$ and $K=2$. Regularisation does not affect the accuracy of the simulation.}\label{RES-fig:HYP_N1_low}
\end{figure}
\newpage
\begin{figure}[!htb]
    \centering
    \minipage{0.33\textwidth}
        \centering
        \includegraphics[width=\linewidth]{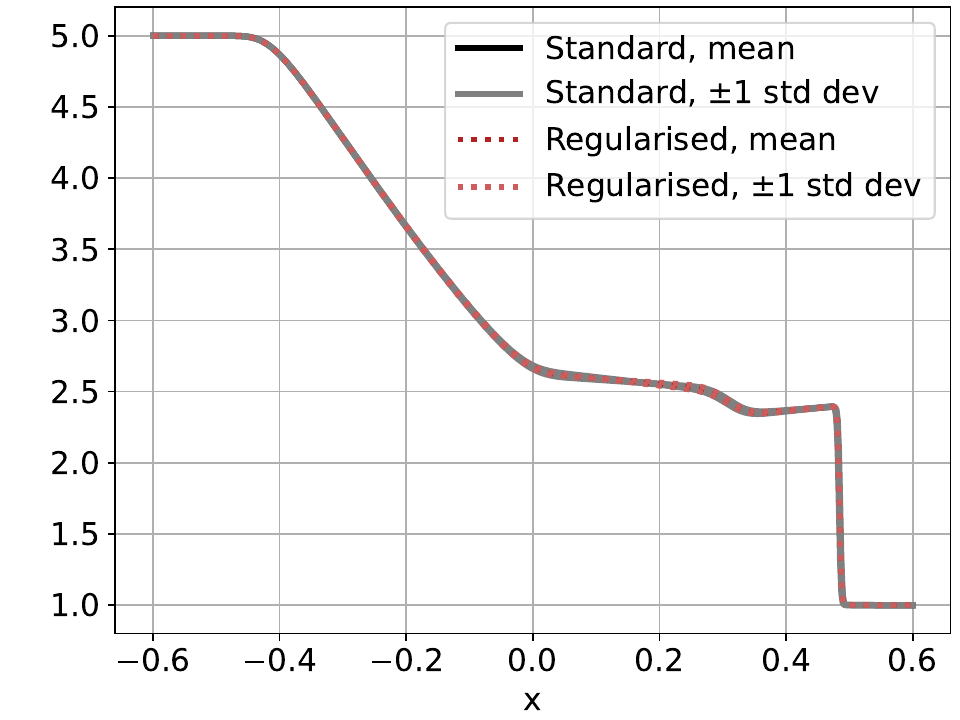}
        \vspace{-2em}
        \caption*{(a) Water height $h$}
    \endminipage\hfill
    \minipage{0.33\textwidth}
        \centering
        \includegraphics[width=\linewidth]{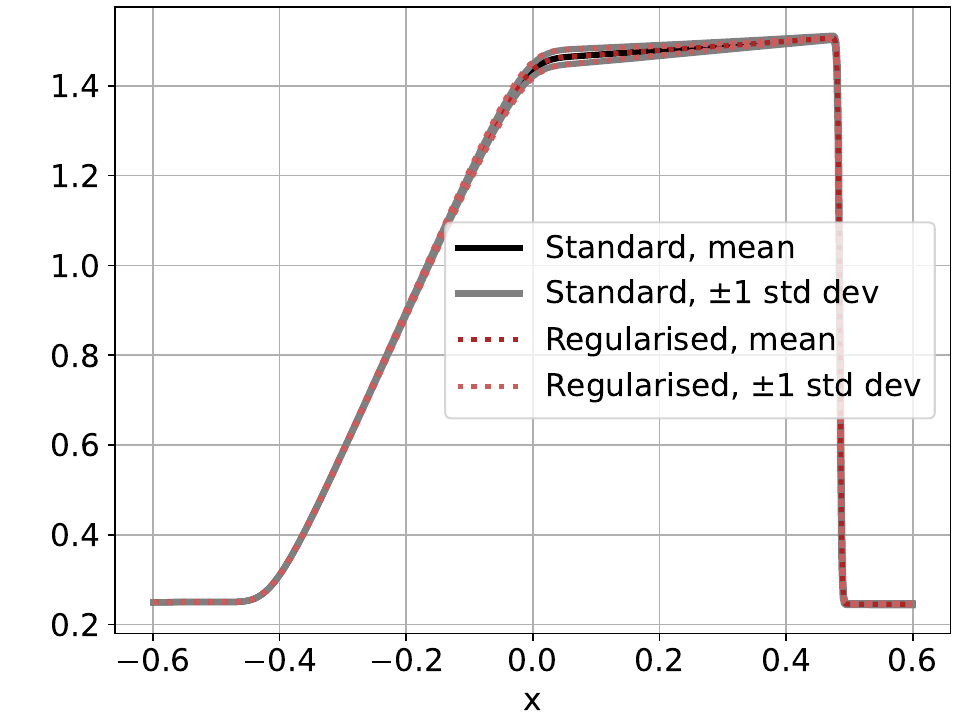}
        \vspace{-2em}
        \caption*{(b) Average velocity $u_m$}
    \endminipage\hfill
    \minipage{0.33\textwidth}
        \centering
        \includegraphics[width=\linewidth]{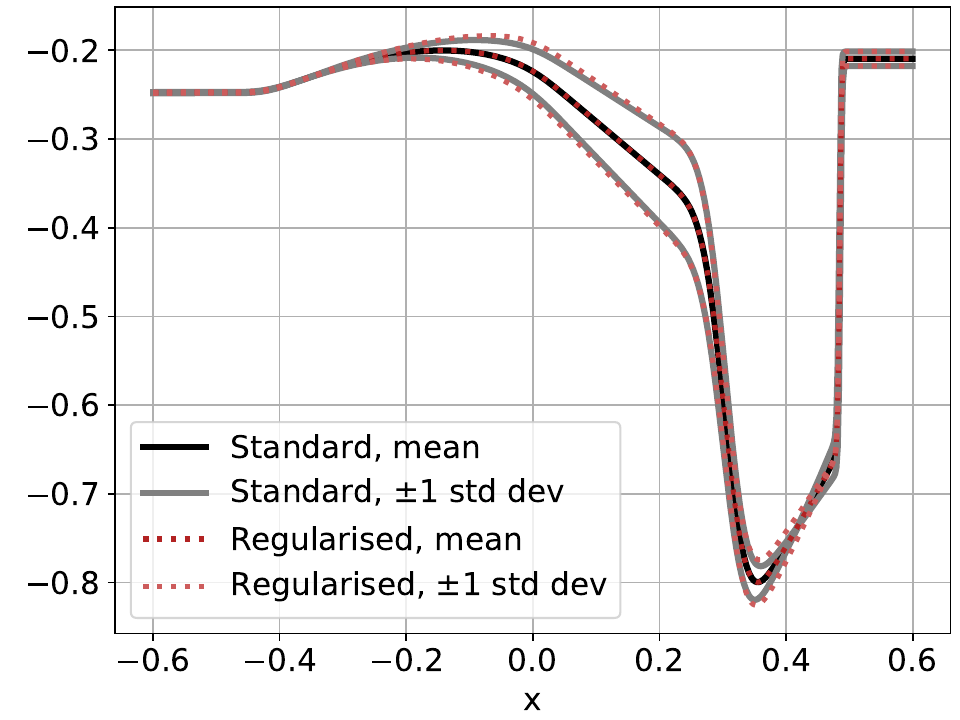}
        \vspace{-2em}
        \caption*{(c) First moment $u_1$}
    \endminipage
    \caption{Regularisation accuracy test for the high dam break test case with SGSWLME $N=1$ and $K=2$. Regularisation slightly affects the accuracy of the simulation.}\label{RES-fig:HYP_N1_high}
\end{figure}
\begin{figure}[!htb]
    \minipage{0.33\textwidth}
      \includegraphics[width=\linewidth]{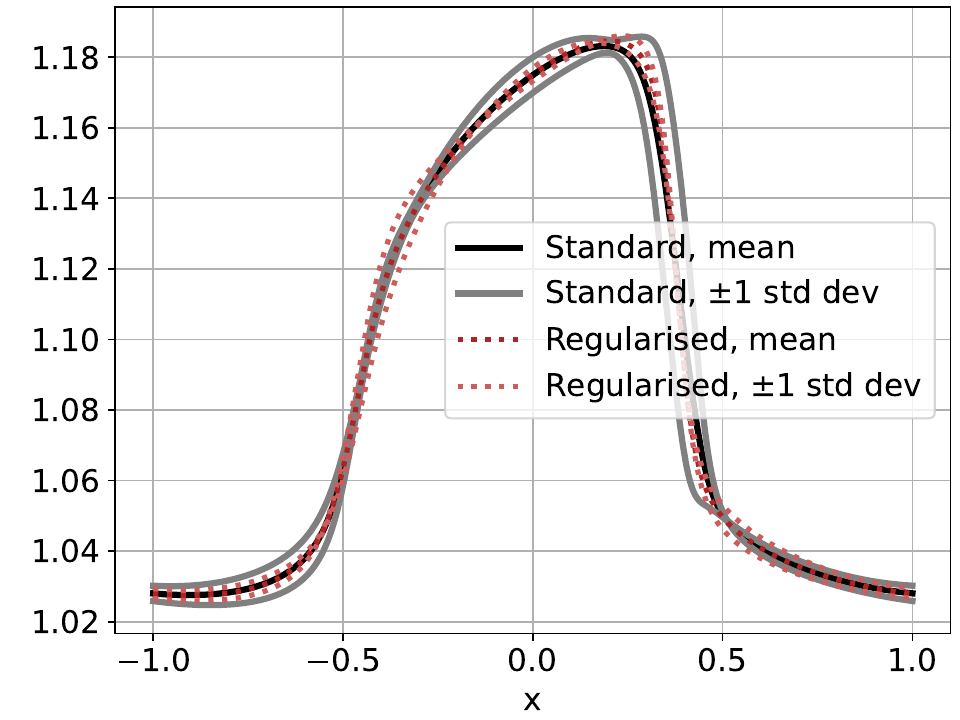}
      \caption*{(a) Water height $h$}
    \endminipage\hfill
    \minipage{0.33\textwidth}
      \includegraphics[width=\linewidth]{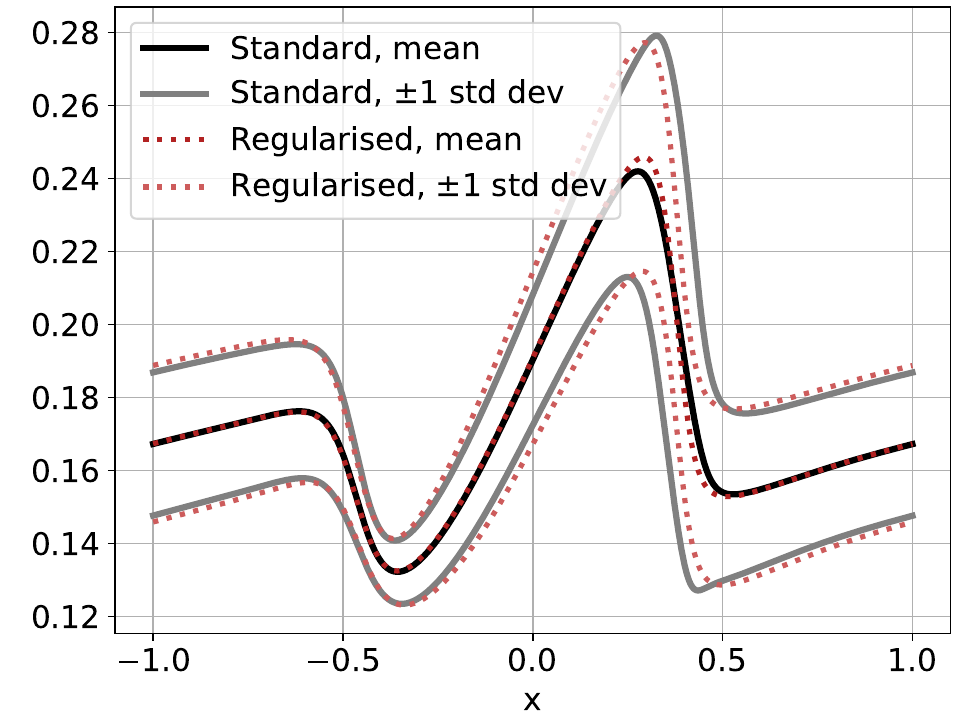}
      \caption*{(b) Average velocity $u_m$}
    \endminipage\hfill
    \minipage{0.33\textwidth}%
      \includegraphics[width=\linewidth]{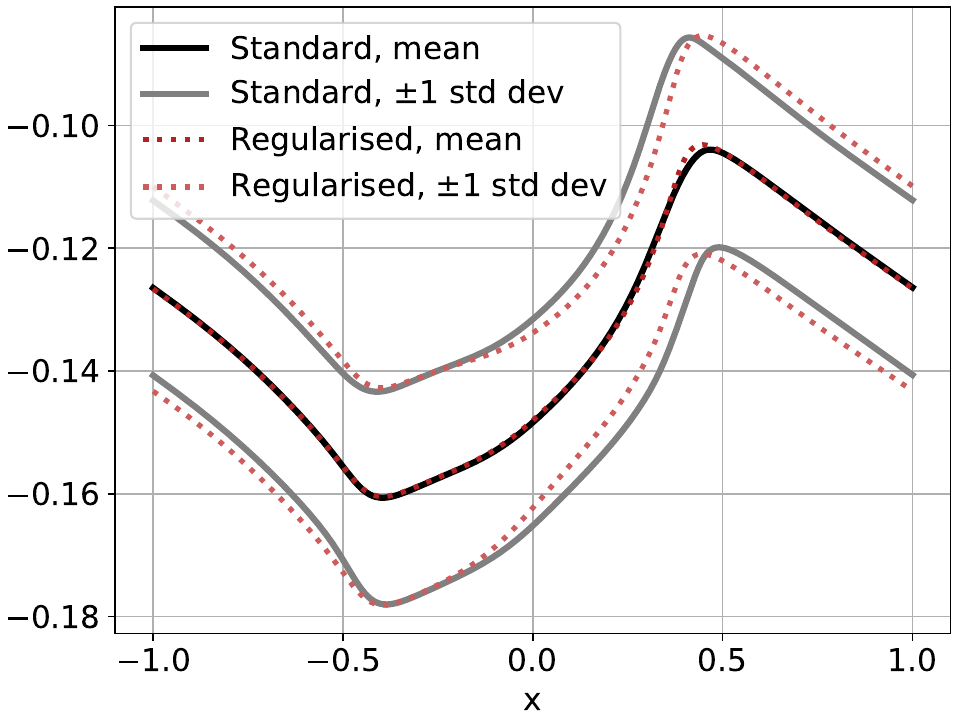}
      \caption*{(c) First moment $u_1$}
    \endminipage
    \caption{Regularisation accuracy test for the smooth periodic wave test case with SGSWLME $N=1$ and $K=2$. Regularisation results in a minor loss of accuracy.}\label{RES-fig:HYP_N1_wave}
\end{figure}

\subsection{Comparison of friction coefficients with varying uncertainty}
In this subsection, we compare different values of $\mu$ and $\sigma$ in the distribution of the friction coefficient $\nu\sim\mathcal{U}([\mu-\sigma,\mu+\sigma])$ for the $N=2$, $K=1$ SGSWLME system of equation \eqref{SGSWME-eq:final}.

In the low dam break case (Figure \ref{RES-fig:MU_N2_low}), we observe that changing the mean $\mu$ of the friction coefficient distribution has little effect on the water height $h$. However, a higher friction coefficient slows the water flow, as evidenced by a lower average velocity $u_m$. The standard deviation of the first moment $u_1$ decreases with increasing $\mu$, and the same applies to the second moment $u_2$. A higher friction coefficient $\nu$ leads to faster relaxation of the first moment $u_1$ to zero, consistent with friction acting on velocity gradients. In contrast, the second moment $u_2$ increases with the mean friction.

\begin{figure*}[!htb]
    \centering
    \minipage{0.33\textwidth}
        \centering
        \includegraphics[width=\linewidth]{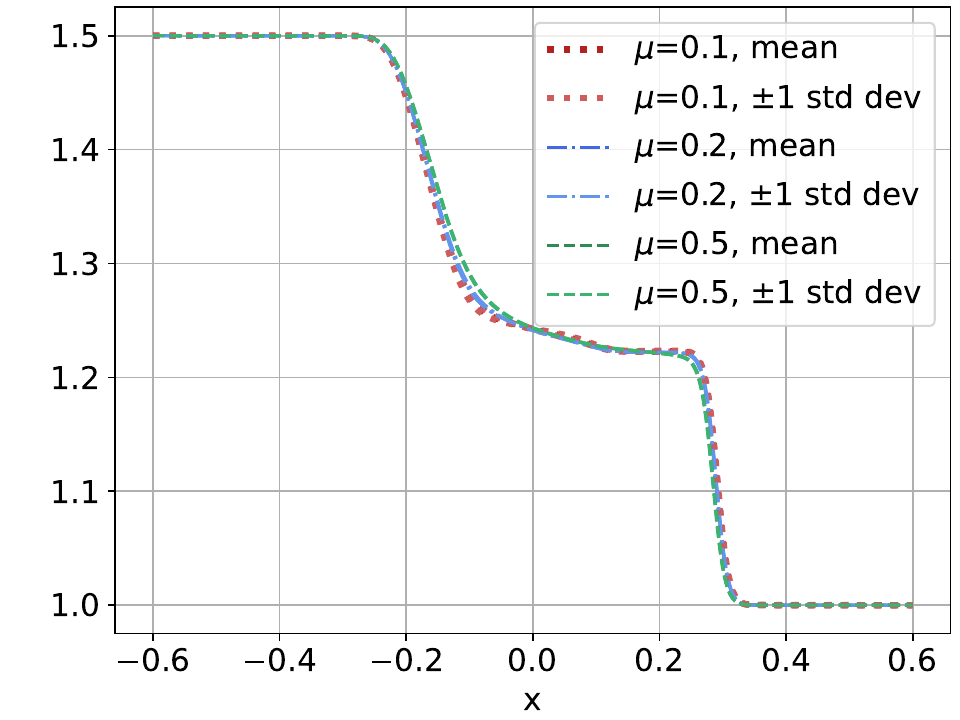}
        \vspace{-2em}
        \caption*{(a) Water height $h$}
        \includegraphics[width=\linewidth]{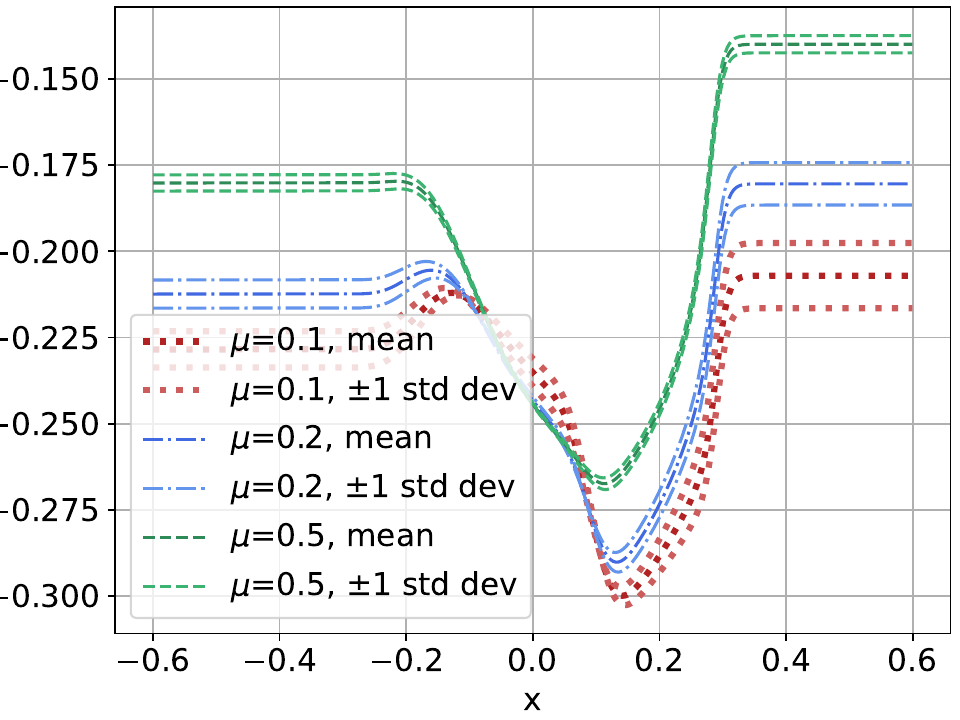}
        \vspace{-2em}
        \caption*{(c) First moment $u_1$}
    \endminipage
    \minipage{0.33\textwidth}
        \centering
        \includegraphics[width=\linewidth]{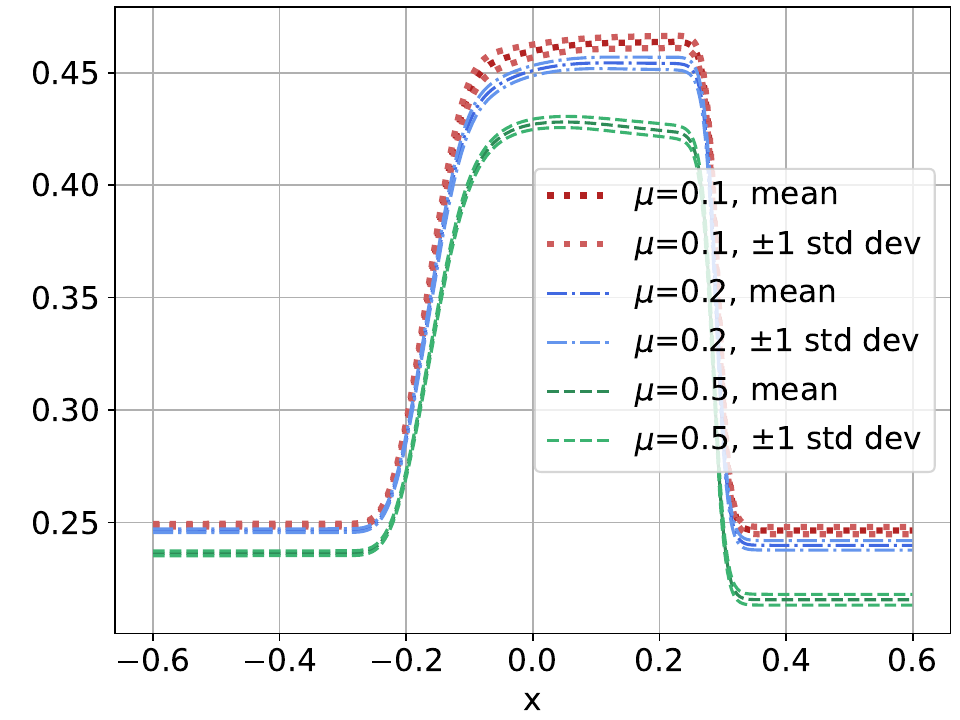}
        \vspace{-2em}
        \caption*{(b) Average velocity $u_m$}
        \includegraphics[width=\linewidth]{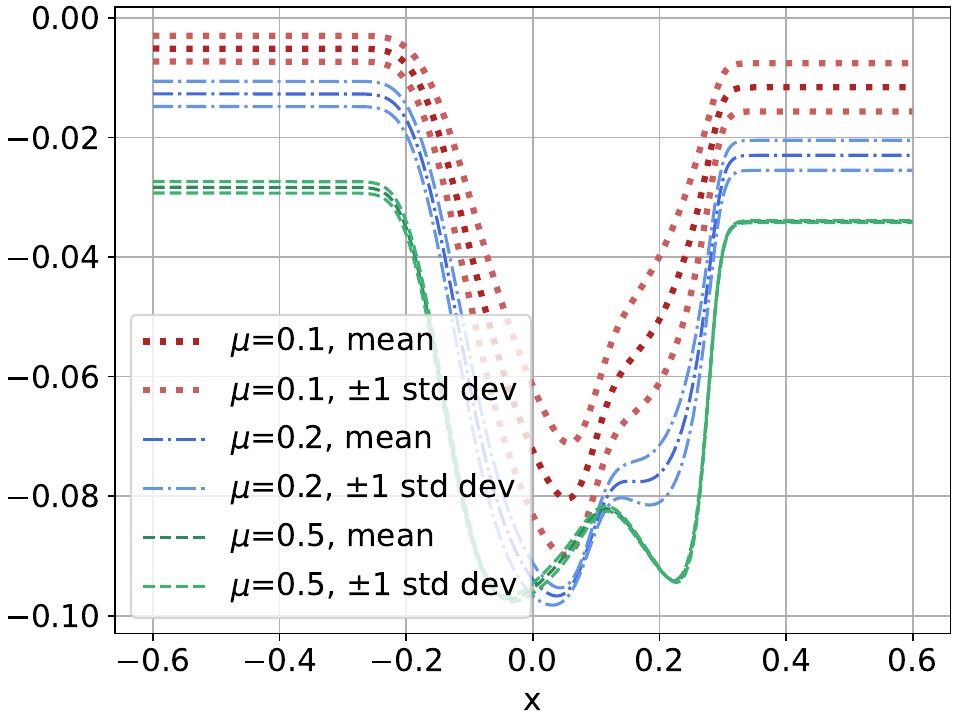}
        \vspace{-2em}
        \caption*{(d) Second moment $u_2$}
    \endminipage
    \caption{Influence of the mean $\mu$ of the friction coefficient $\nu$ test for the low dam break test case with SGSWLME $N=2$ and $K=1$. The used standard deviation is $\sigma=0.05$.}\label{RES-fig:MU_N2_low}
\end{figure*}

For the high dam-break case in Figure \ref{RES-fig:MU_N2_high}, we observe the same characteristics as in Figure \ref{RES-fig:MU_N2_low}: changing the mean of the friction coefficient distribution has little effect on the water height $h$. Similarly, a higher friction coefficient slows the water flow, as reflected in a lower average velocity $u_m$ and the second moment $u_2$.

While the first moment $u_1$ increases with increasing mean friction coefficient in regions away from the shock, the opposite behaviour occurs in the region of higher average velocity $u_m$. In this central region, the first moment $u_1$ decreases with increasing friction coefficient. This indicates that the uncertainty affects areas of equilibrium and non-equilibrium flow differently.

\begin{figure*}[!htb]
    \centering
    \minipage{0.33\textwidth}
        \centering
        \includegraphics[width=\linewidth]{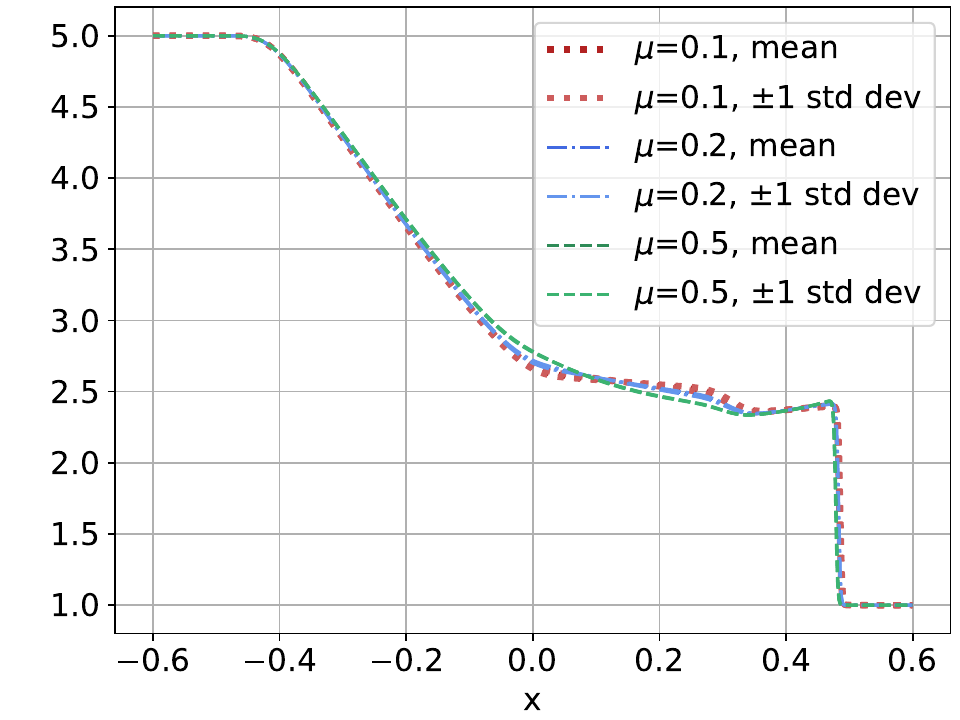}
        \vspace{-2.3em}
        \caption*{(a) Water height $h$}
        \includegraphics[width=\linewidth]{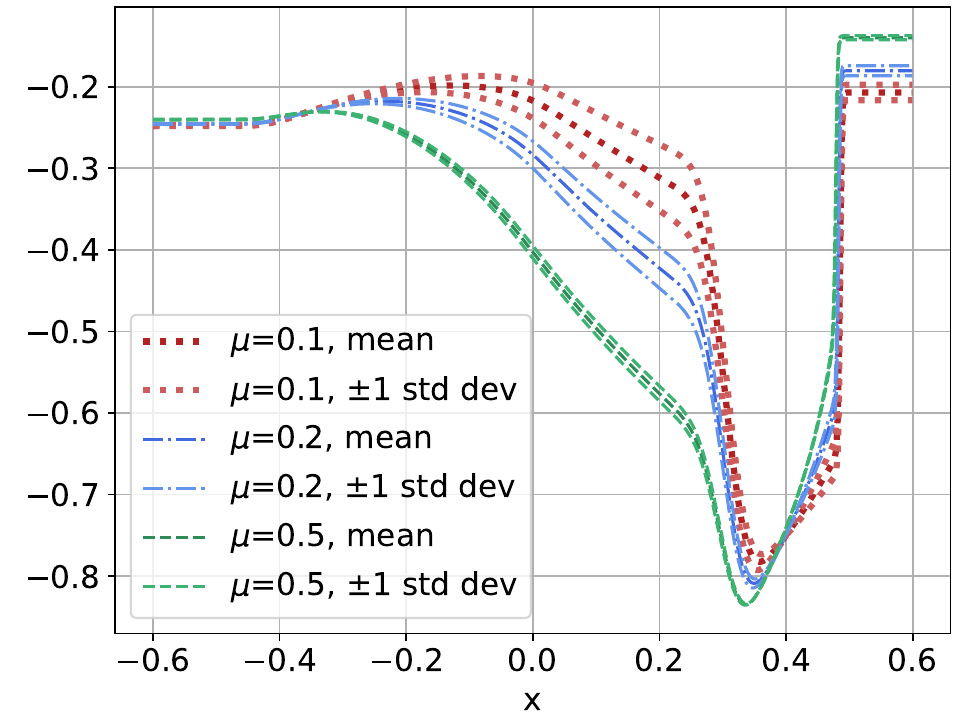}
        \vspace{-2.3em}
        \caption*{(c) First moment $u_1$}
    \endminipage
    \minipage{0.33\textwidth}
        \centering
        \includegraphics[width=\linewidth]{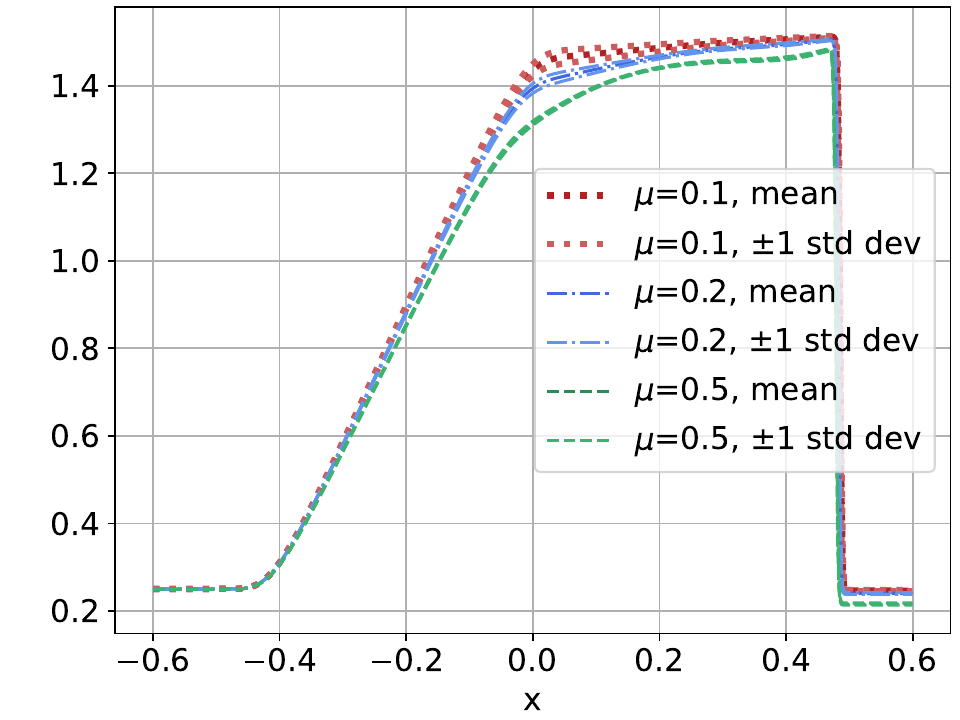}
        \vspace{-2.3em}
        \caption*{(b) Average velocity $u_m$}
        \includegraphics[width=\linewidth]{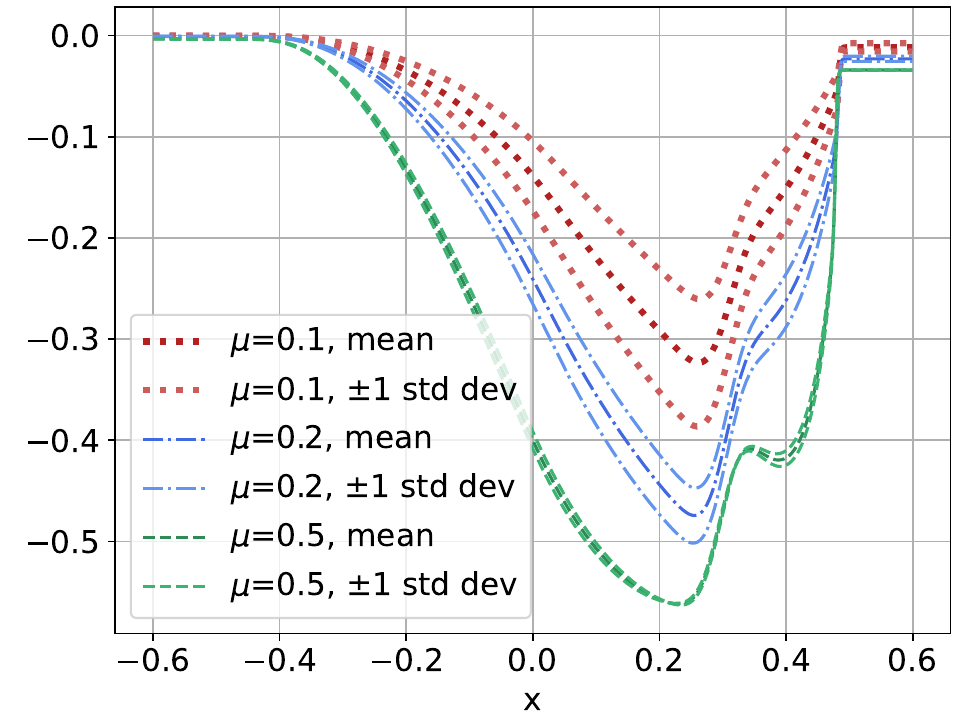}
        \vspace{-2.3em}
        \caption*{(d) Second moment $u_2$}
    \endminipage
    \vspace{-0.9em}
    \caption{Influence of the mean $\mu$ of the friction coefficient $\nu$ test for the high dam break test case with SGSWLME $N=2$ and $K=1$. The used standard deviation is $\sigma=0.05$.}\label{RES-fig:MU_N2_high}
\end{figure*}

\begin{figure*}[!htb]
    \centering
    \minipage{0.33\textwidth}
        \centering
        \includegraphics[width=\linewidth]{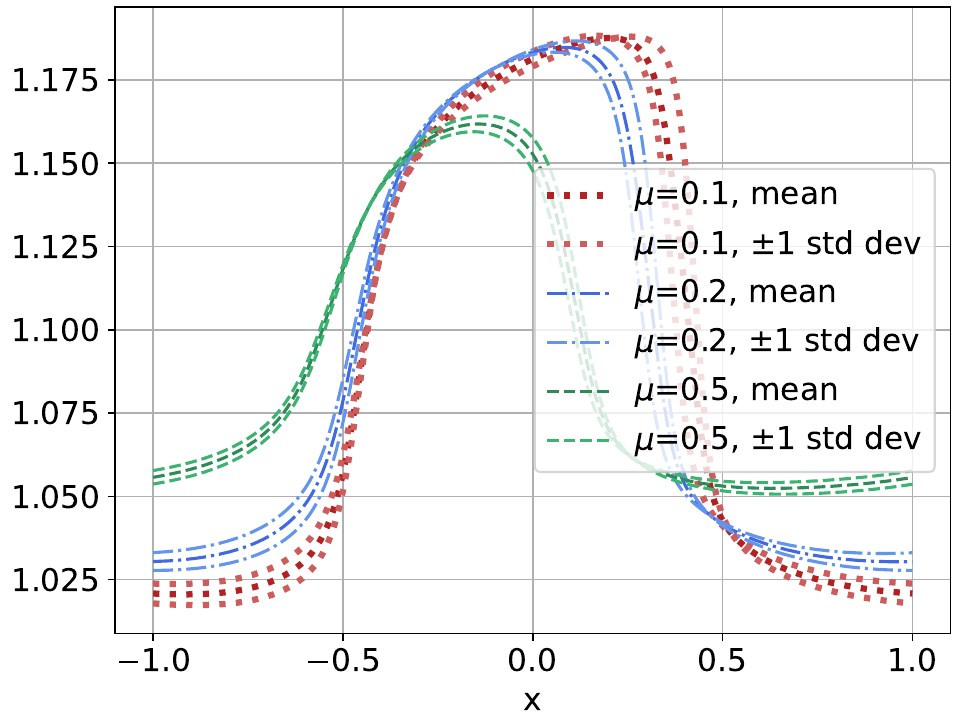}
        \vspace{-2.3em}
        \caption*{(a) Water height $h$}
        \includegraphics[width=\linewidth]{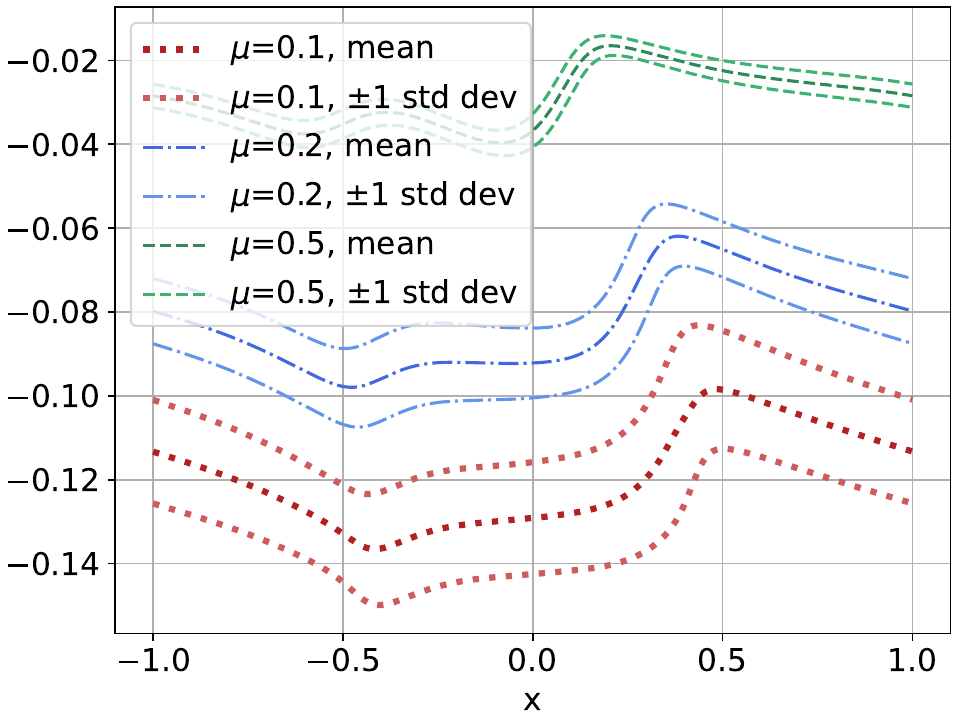}
        \vspace{-2.3em}
        \caption*{(c) First moment $u_1$}
    \endminipage
    \minipage{0.33\textwidth}
        \centering
        \includegraphics[width=\linewidth]{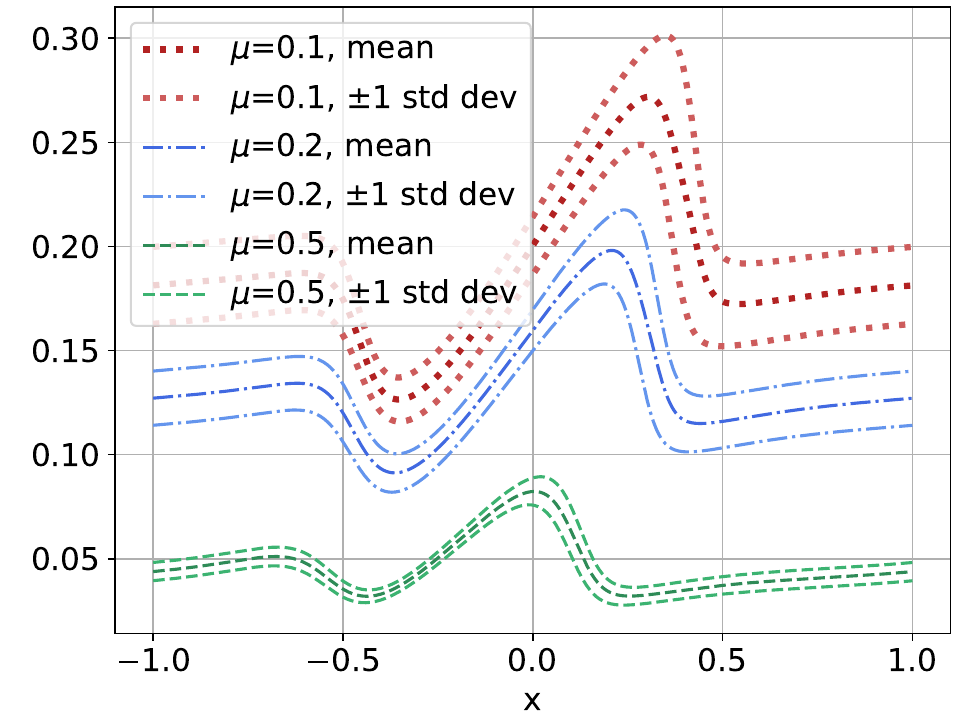}
        \vspace{-2.3em}
        \caption*{(b) Average velocity $u_m$}
        \includegraphics[width=\linewidth]{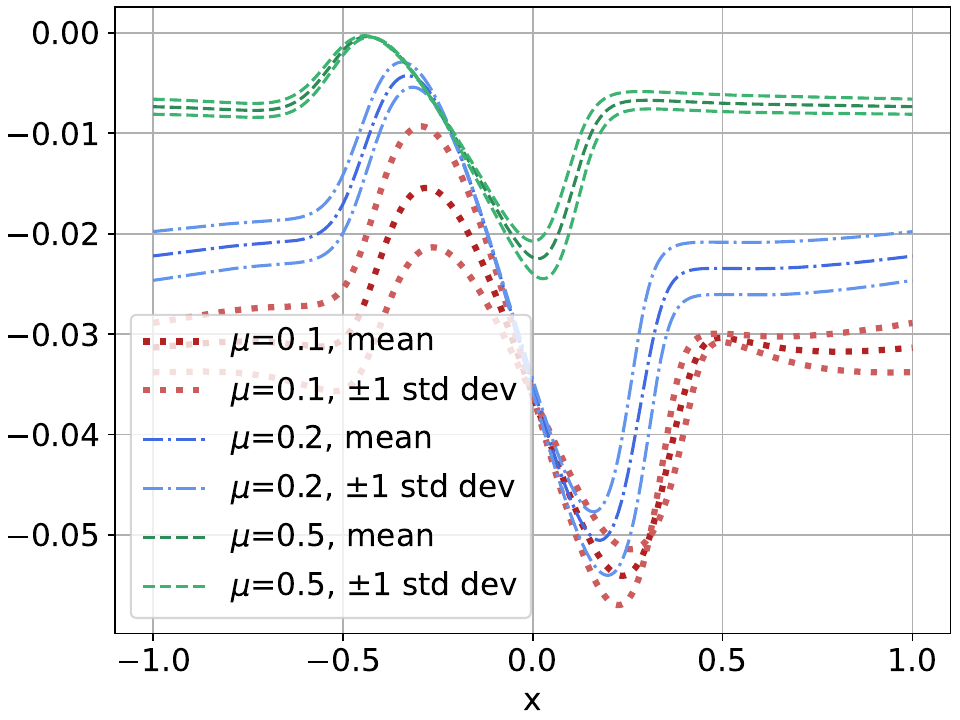}
        \vspace{-2.3em}
        \caption*{(d) Second moment $u_2$}
    \endminipage
    \vspace{-0.9em}
    \caption{Influence of the mean $\mu$ of the friction coefficient $\nu$ test for the smooth periodic wave test case with SGSWLME $N=2$ and $K=1$. The used standard deviation is $\sigma=0.05$.}\label{RES-fig:MU_N2_wave}
\end{figure*}

\newpage
In Figure \ref{RES-fig:MU_N2_wave}, for the smooth periodic wave, we observe that changing the mean of the friction coefficient distribution now significantly affects the water height $h$. Due to the higher friction coefficient, the wave reaches a lower height and travels more slowly, as reflected by the lower average velocity $u_m$. In the absence of shocks, all velocity variables - $u_m$, $u_1$, and $u_2$ - decrease in absolute terms with increasing friction. Additionally, the standard deviation of all three velocity variables decreases with increasing $\mu$.

We repeat the previous test cases, this time varying the standard deviation $\sigma$ of the friction coefficient while keeping its mean $\mu=0.2$ constant. For the low dam break test case in Figure \ref{RES-fig:SIG_N2_low}, we observe that the water height $h$ is largely unaffected by changes in $\sigma$. Similarly, the mean of the average velocity $u_m$, the first moment $u_1$, and the second moment $u_2$ remain almost constant when $\sigma$ is varied. As expected, only the standard deviation increases with increasing $\sigma$.

\begin{figure*}[!htb]
    \centering
    \minipage{0.33\textwidth}
        \centering
        \includegraphics[width=\linewidth]{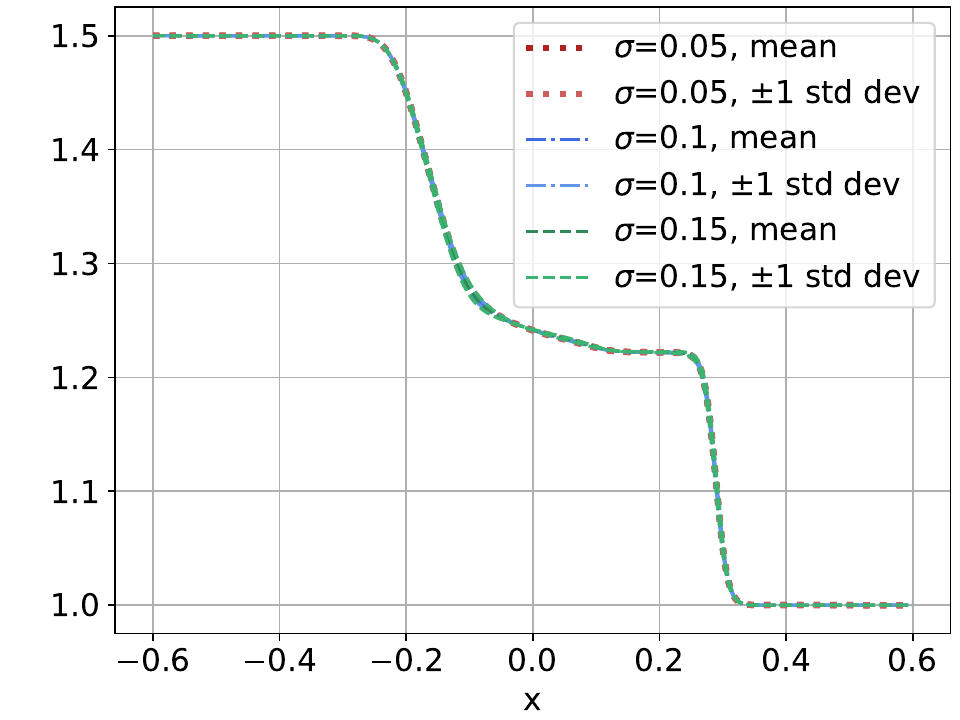}
        \vspace{-2em}
        \caption*{(a) Water height $h$}
        \includegraphics[width=\linewidth]{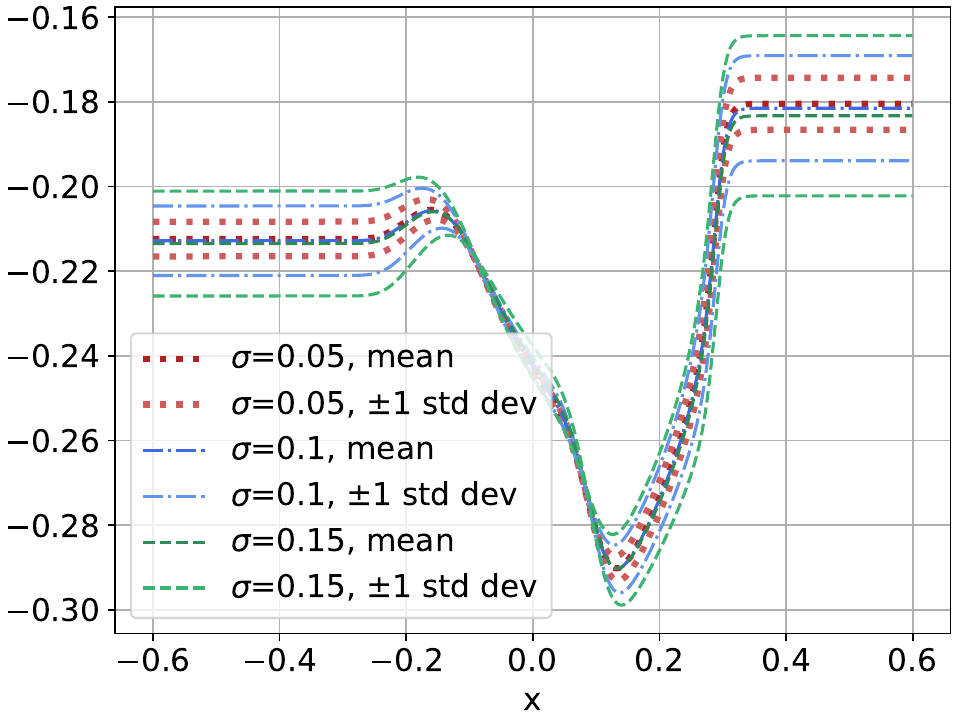}
        \vspace{-2em}
        \caption*{(c) First moment $u_1$}
    \endminipage
    \minipage{0.33\textwidth}
        \centering
        \includegraphics[width=\linewidth]{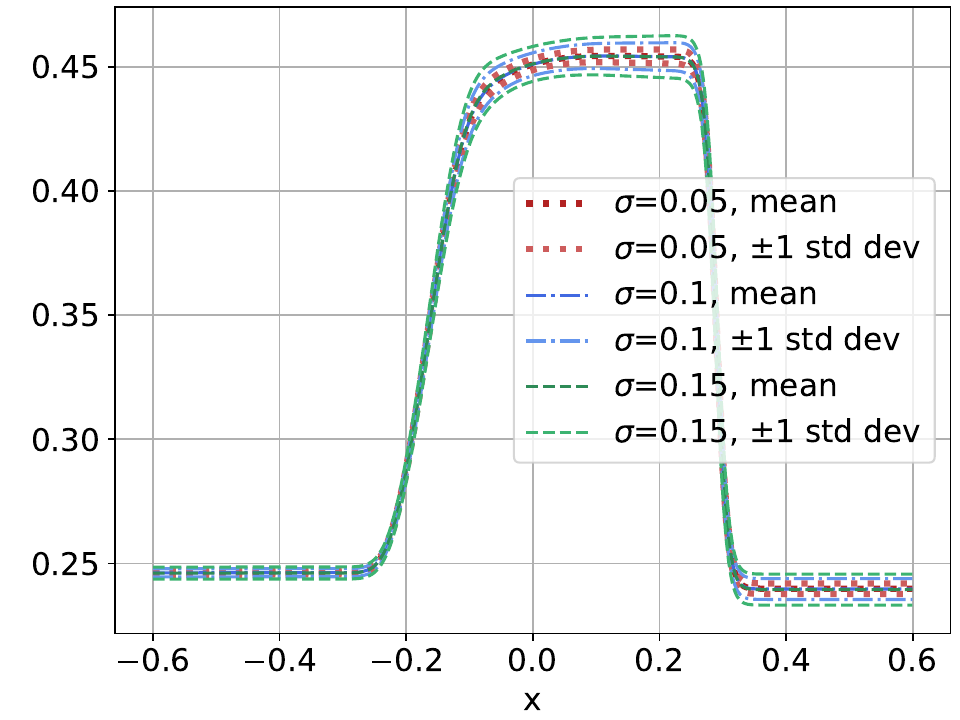}
        \vspace{-2em}
        \caption*{(b) Average velocity $u_m$}
        \includegraphics[width=\linewidth]{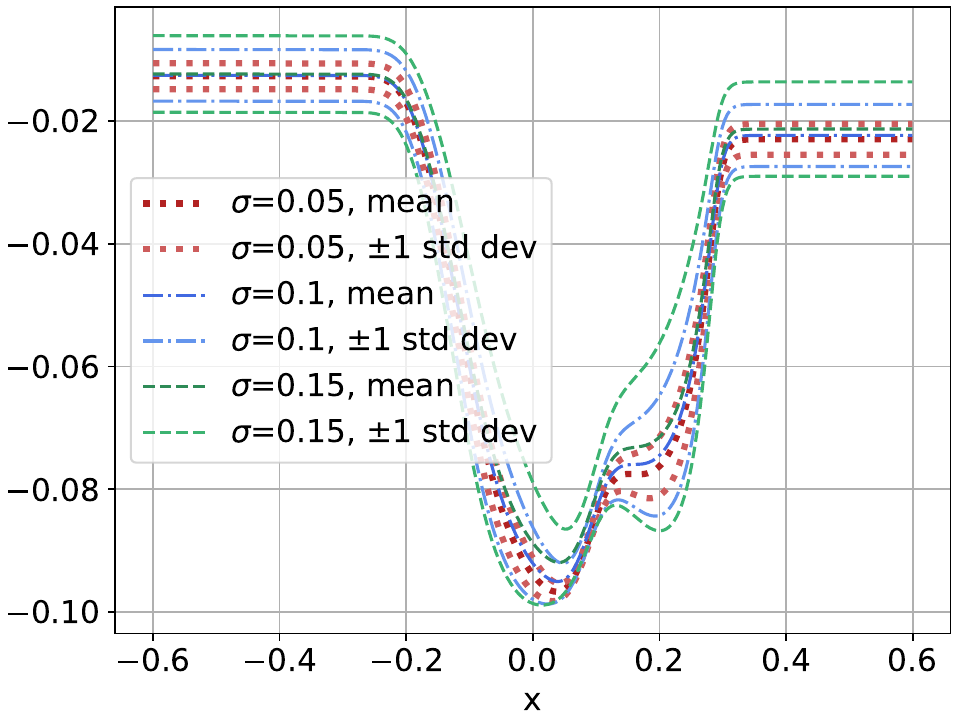}
        \vspace{-2em}
        \caption*{(d) Second moment $u_2$}
    \endminipage
    \caption{Influence of the standard deviation $\sigma$ of the friction coefficient $\nu$ test for the low dam break test case with SGSWLME $N=2$ and $K=1$. The used mean is $\mu=0.2$.}\label{RES-fig:SIG_N2_low}
\end{figure*}

In Figure \ref{RES-fig:SIG_N2_high}, for the high dam break test case, we observe a trend comparable to that in Figure \ref{RES-fig:SIG_N2_low}. The water height $h$ is largely unaffected by changes in $\sigma$, and the mean of the average velocity $u_m$ and the first moment $u_1$ remain almost constant as $\sigma$ varies. A slight increase in the mean with increasing $\sigma$ is noticeable in the second moment $u_2$. As expected, the most significant effect is the increase in standard deviation with increasing $\sigma$.

\begin{figure*}[!htb]
    \centering
    \minipage{0.33\textwidth}
        \centering
        \includegraphics[width=\linewidth]{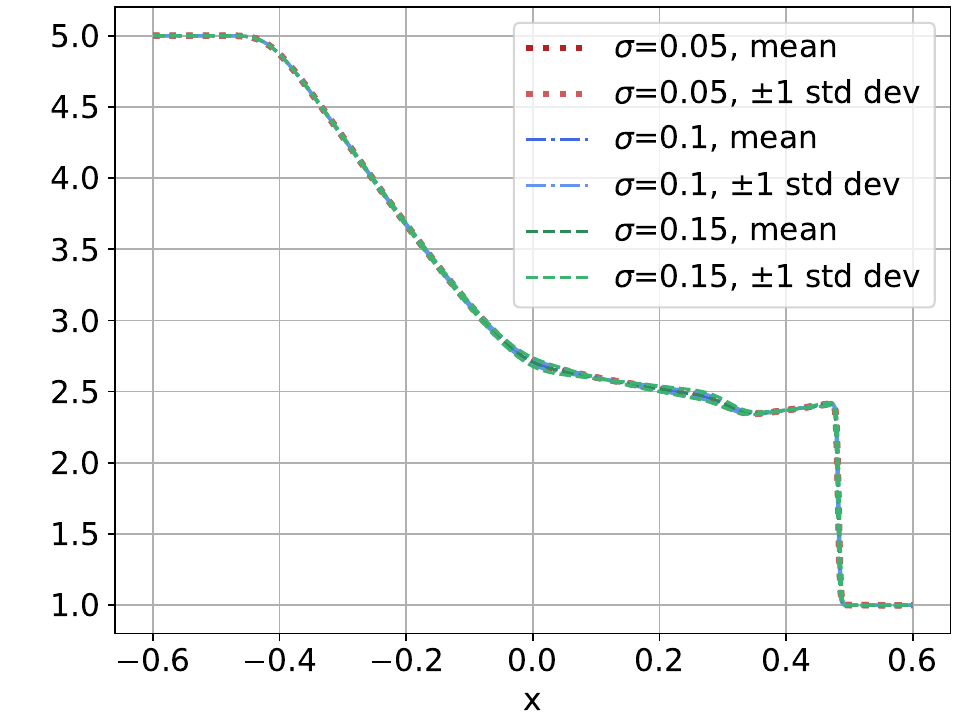}
        \vspace{-2.3em}
        \caption*{(a) Water height $h$}
        \includegraphics[width=\linewidth]{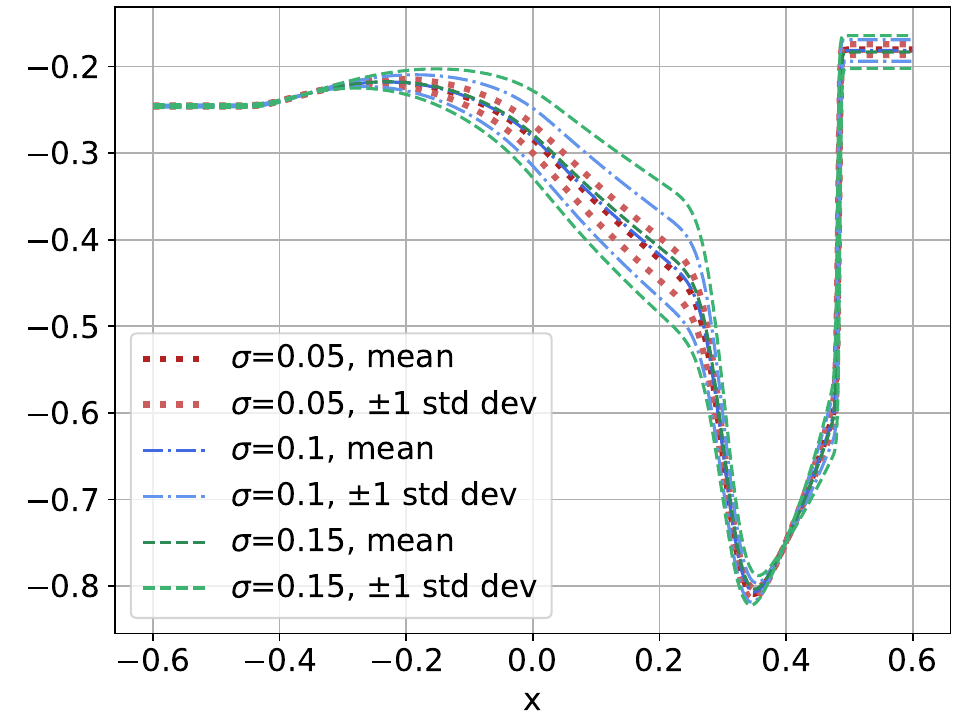}
        \vspace{-2.3em}
        \caption*{(c) First moment $u_1$}
    \endminipage
    \minipage{0.33\textwidth}
        \centering
        \includegraphics[width=\linewidth]{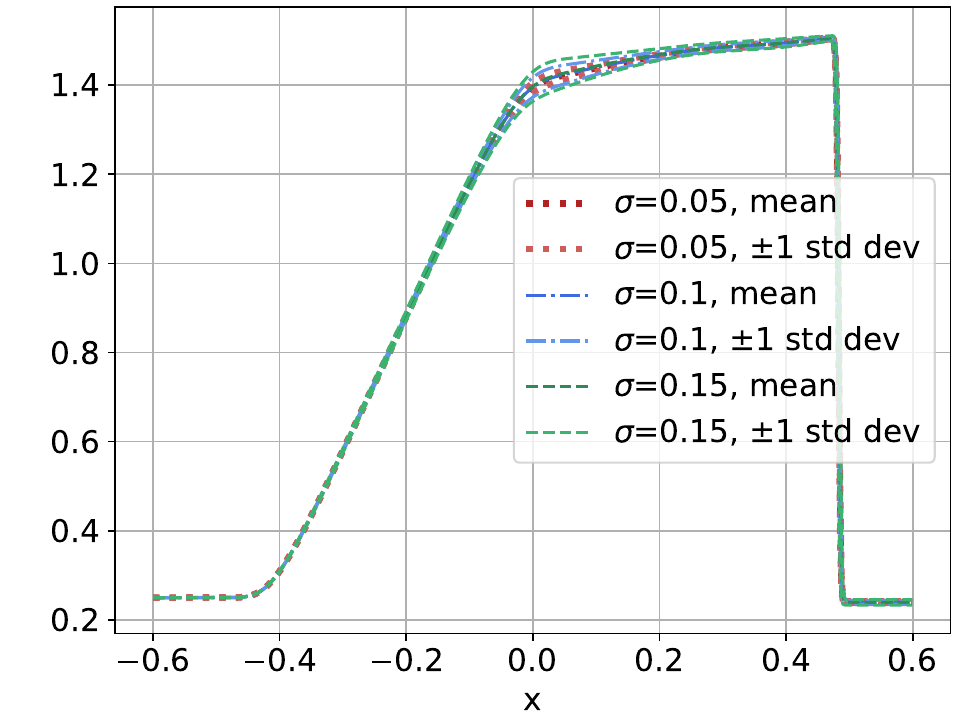}
        \vspace{-2.3em}
        \caption*{(b) Average velocity $u_m$}
        \includegraphics[width=\linewidth]{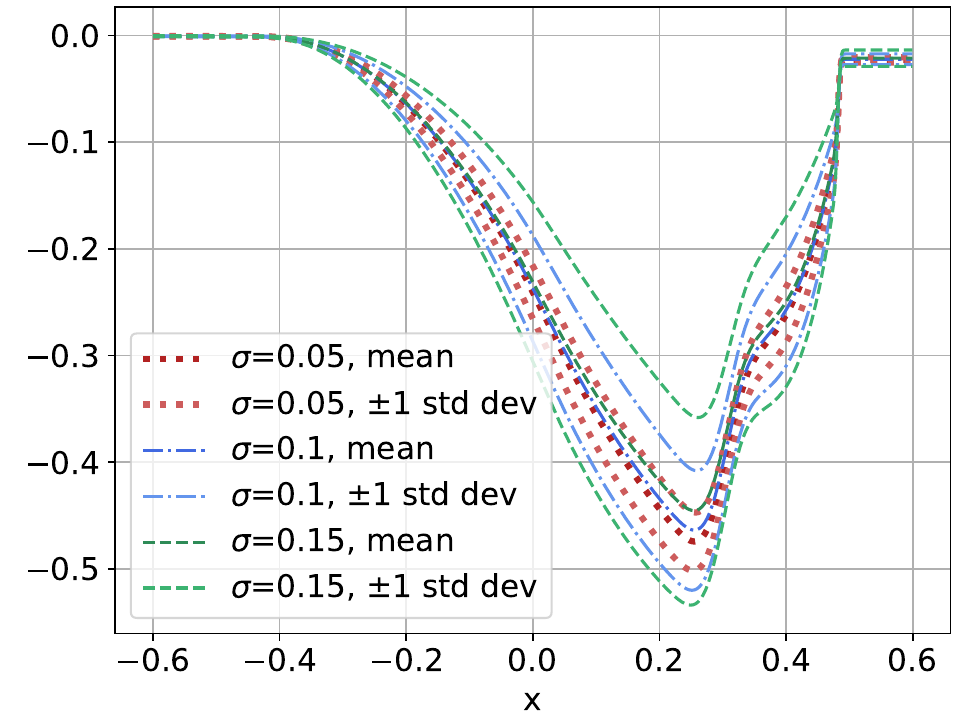}
        \vspace{-2.3em}
        \caption*{(d) Second moment $u_2$}
    \endminipage
    \vspace{-0.9em}
    \caption{Influence of the standard deviation $\sigma$ of the friction coefficient $\nu$ test for the high dam break test case with SGSWLME $N=2$ and $K=1$. The used mean is $\mu=0.2$.}\label{RES-fig:SIG_N2_high}
\end{figure*}

\begin{figure*}[!htb]
    \centering
    \minipage{0.33\textwidth}
        \centering
        \includegraphics[width=\linewidth]{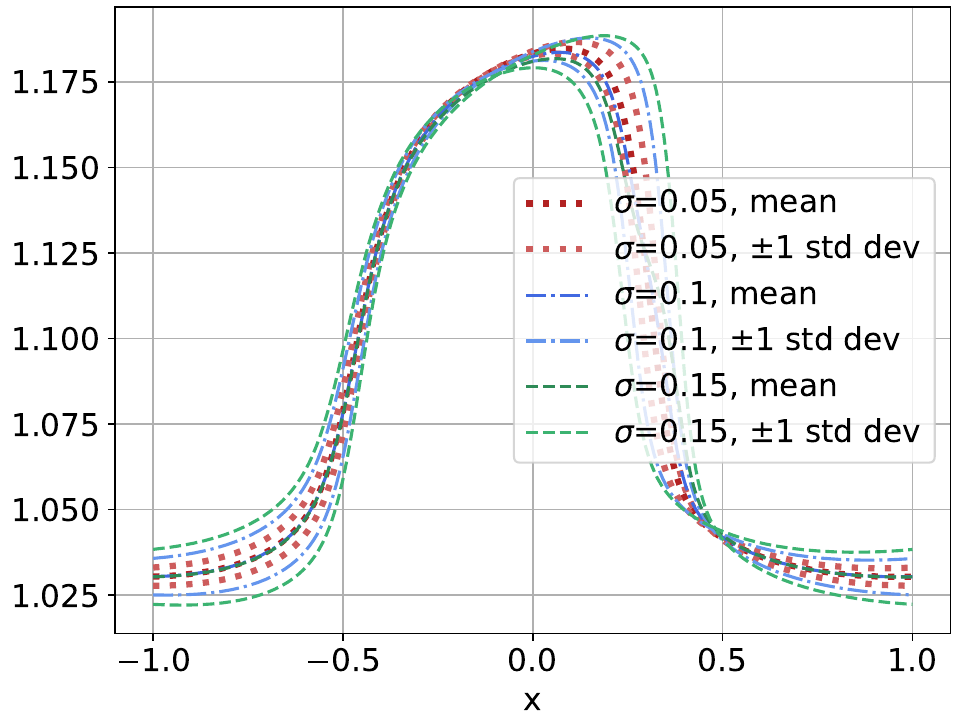}
        \vspace{-2.3em}
        \caption*{(a) Water height $h$}
        \includegraphics[width=\linewidth]{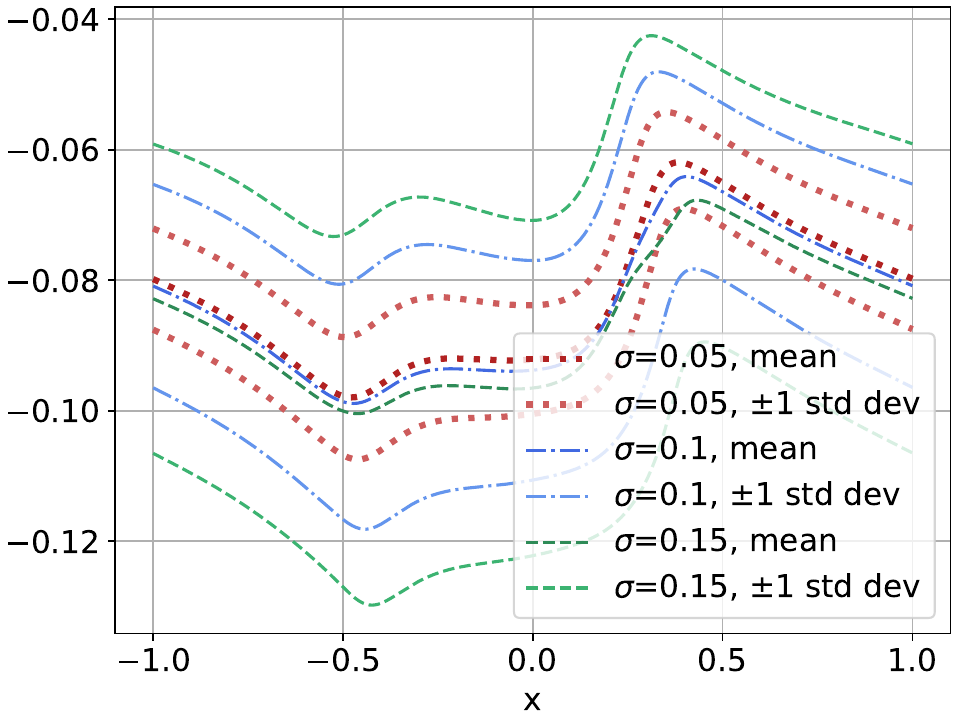}
        \vspace{-2.3em}
        \caption*{(c) First moment $u_1$}
    \endminipage
    \minipage{0.33\textwidth}
        \includegraphics[width=\linewidth]{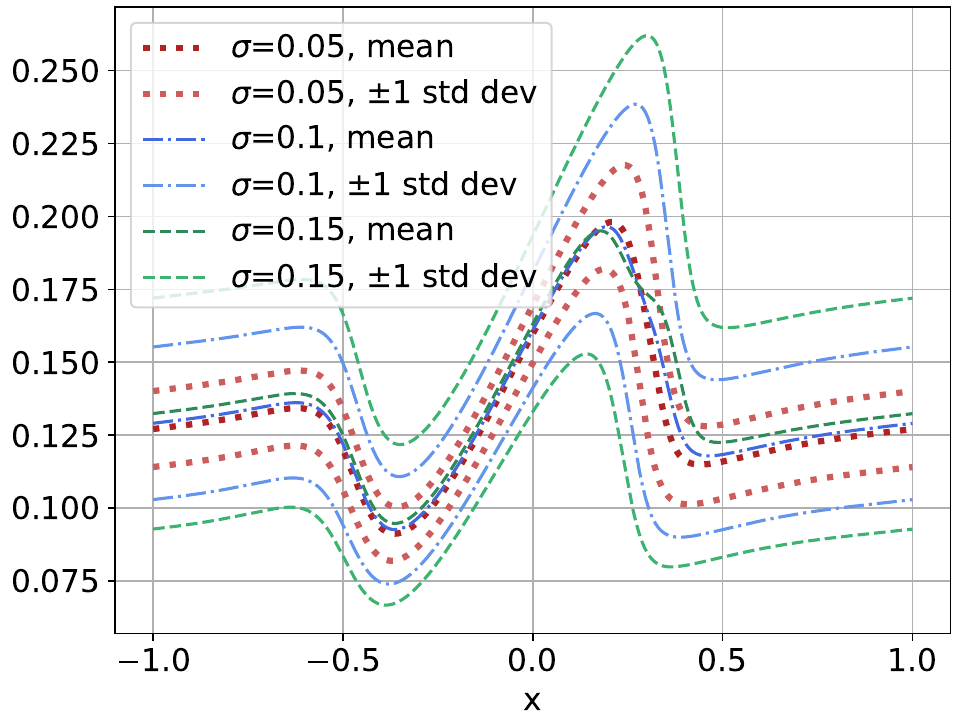}
        \vspace{-2.3em}
        \caption*{(b) Average velocity $u_m$}
        \includegraphics[width=\linewidth]{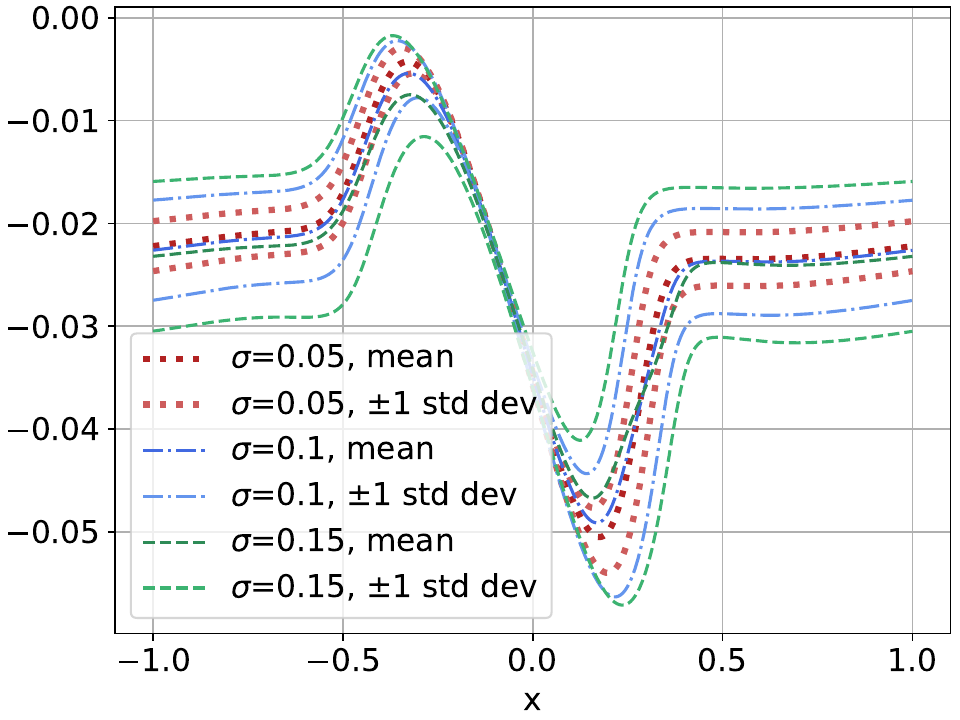}
        \vspace{-2.3em}
        \caption*{(d) Second moment $u_2$}
    \endminipage
    \vspace{-0.9em}
    \caption{Influence of the standard deviation $\sigma$ of the friction coefficient $\nu$ test for the smooth periodic wave test case with SGSWLME $N=2$ and $K=1$. The used mean is $\mu=0.2$.}\label{RES-fig:SIG_N2_wave}
\end{figure*}

\newpage
For the smooth periodic wave test case in Figure \ref{RES-fig:SIG_N2_wave}, we observe that the water height $h$ is more affected by changes in $\sigma$ than in the dam break test cases. Additionally, varying $\sigma$ in this scenario slightly decreases the mean of the water height $h$, the first moment $u_1$, and the second moment $u_2$, while increasing the standard deviation. It also slightly increases the average velocity $u_m$.

In summary, Figures \ref{RES-fig:MU_N2_low}-\ref{RES-fig:SIG_N2_wave} illustrate the nonlinear effects of varying $\mu$ and $\sigma$ in $\nu \sim \mathcal{U}([\mu-\sigma,\mu+\sigma])$ on the water height $h$ and velocity variables $u_m$, $u_1$, and $u_2$. The stochastic Galerkin formulation of the SWLME enables a rapid assessment of the influence of these parameter settings.

%% file: Sections/08_Conclusion.tex
In this paper, we derived, analysed, and numerically solved the stochastic Galerkin formulation of the one-dimensional shallow water linearised moment equations. This was made possible using conservative variables and the definition of the pseudospectral product. We derived an energy equation and analysed hyperbolicity in two distinct ways: for specific moment order $N$ with general stochastic Galerkin order $K$, and for specific stochastic Galerkin order $K$ with general moment order $N$. The simulations already show good results for low orders of $K$ and provide a significant speedup compared to Monte Carlo methods.

The use of stochastic Galerkin methods opens possibilities for future extensions, including the design of adaptivity in stochastic space \cite{meyer_posteriori_2020, burger_hybrid_2017} and the development of numerical schemes that are well-balanced at any location in space \cite{jin_well-balanced_2016}. Furthermore, other methods to overcome the loss of hyperbolicity could be explored, such as: Roe variable transformations combined with Haar wavelets and projection techniques, together with high-order quadrature rules \cite{bender_entropy-conservative_2024}, limiting strategies \cite{schlachter_hyperbolicity-preserving_2018,durrwachter_hyperbolicity-preserving_2020}, filtering techniques \cite{kusch_filtered_2020}, linearisation methods \cite{wu_stochastic_2017}, and entropic variable representations \cite{poette_uncertainty_2009}.

Additionally, the energy equation could be used to derive a tailored numerical method for discrete energy preservation during numerical simulations \cite{tadmor_numerical_1987, bohm_entropy_2020, ersing_entropy_2024, ersing_entropy_2025, fan_well-balanced_2026, gassner_well_2016}. Future work could also explore the hyperbolicity of systems with $N\geq2$ for higher orders of $K$, as well as their accuracy in numerical simulations. Finally, a natural extension of this work involves investigating the two-dimensional case.

%% file: Sections/A_Convergence_MC.tex
This appendix presents figures illustrating the number of Monte Carlo samples $S$ required to compute the reference solutions in Figures \ref{RES-fig:SG_N1_low}, \ref{RES-fig:SG_N1_high}, \ref{RES-fig:SG_N1_wave}, \ref{RES-fig:SG_N2_low}, \ref{RES-fig:SG_N2_high}, and \ref{RES-fig:SG_N2_wave}, as well as the relative run times in Table \ref{RES-tab:rel_runtimes}. Figures \ref{RES-fig:MC_N1_low}, \ref{RES-fig:MC_N1_high}, and \ref{RES-fig:MC_N1_wave} show curves for different numbers of samples $S$ for $N=1$, while Figures \ref{RES-fig:MC_N2_low}, \ref{RES-fig:MC_N2_high}, and \ref{RES-fig:MC_N2_wave} show analogous results for $N=2$. The Monte Carlo samples are independent of each other. Convergence is visually assessed by comparison with a significantly larger $S$.

In Figure \ref{RES-fig:MC_N1_low}, the low dam break test case of Table \ref{RES-tab:dambreak} with $N=1$ demonstrates visual convergence of the solution in approximately $S=100$ Monte Carlo samples. For the considered range of the uncertain friction coefficient $\nu\sim\mathcal{U}([0.05,0.15])$, the standard deviation bounds of the computed functions remain relatively narrow. The water height $h$ and the average velocity $u_m$ exhibit minimal sensitivity to variations in $\nu$, while the first moment $u_1$ shows a noticeable dependence. 

\begin{figure}[H]
    \minipage{0.33\textwidth}
        \includegraphics[width=1\linewidth]{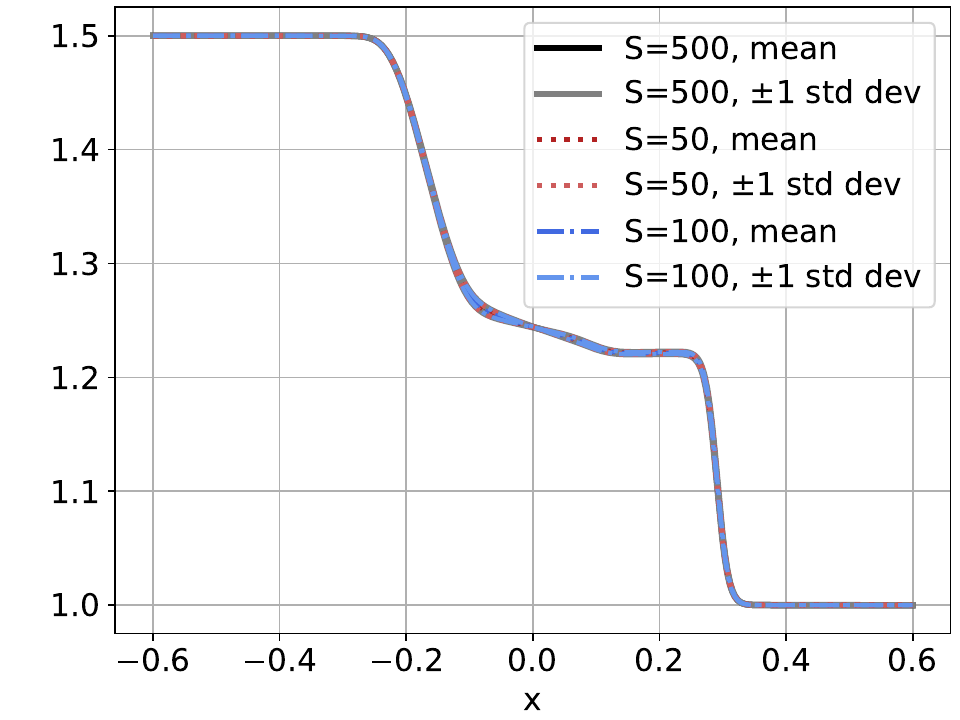}
        \vspace{-2em}
        \caption*{(a) Water height $h$}
    \endminipage\hfill
    \minipage{0.33\textwidth}
        \includegraphics[width=\linewidth]{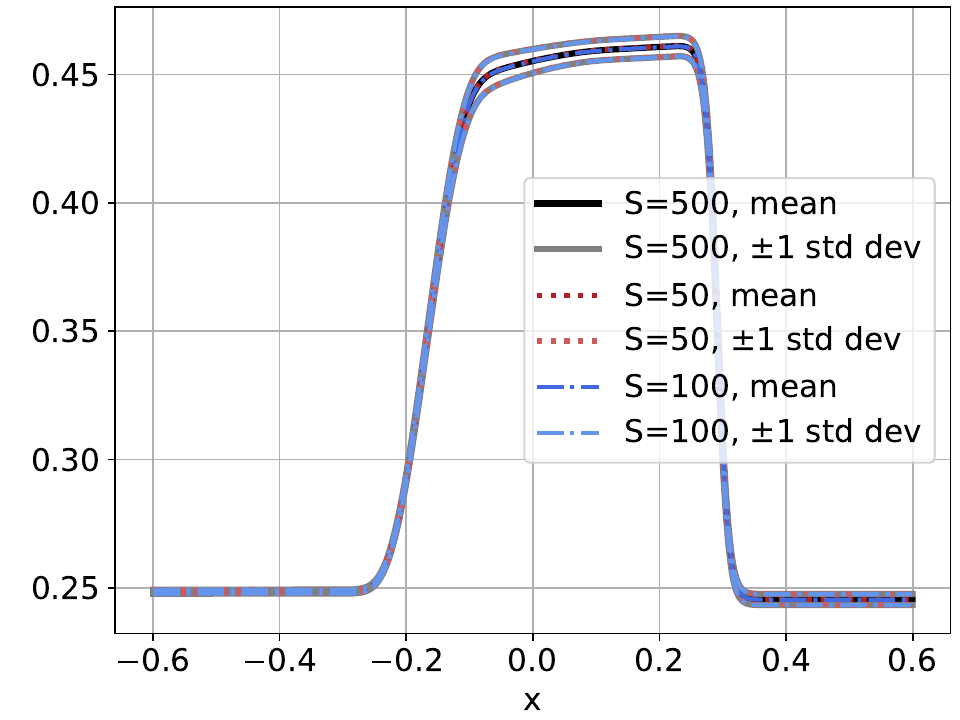}
        \vspace{-2em}
        \caption*{(b) Average velocity $u_m$}
    \endminipage\hfill
    \minipage{0.33\textwidth}%
        \includegraphics[width=\linewidth]{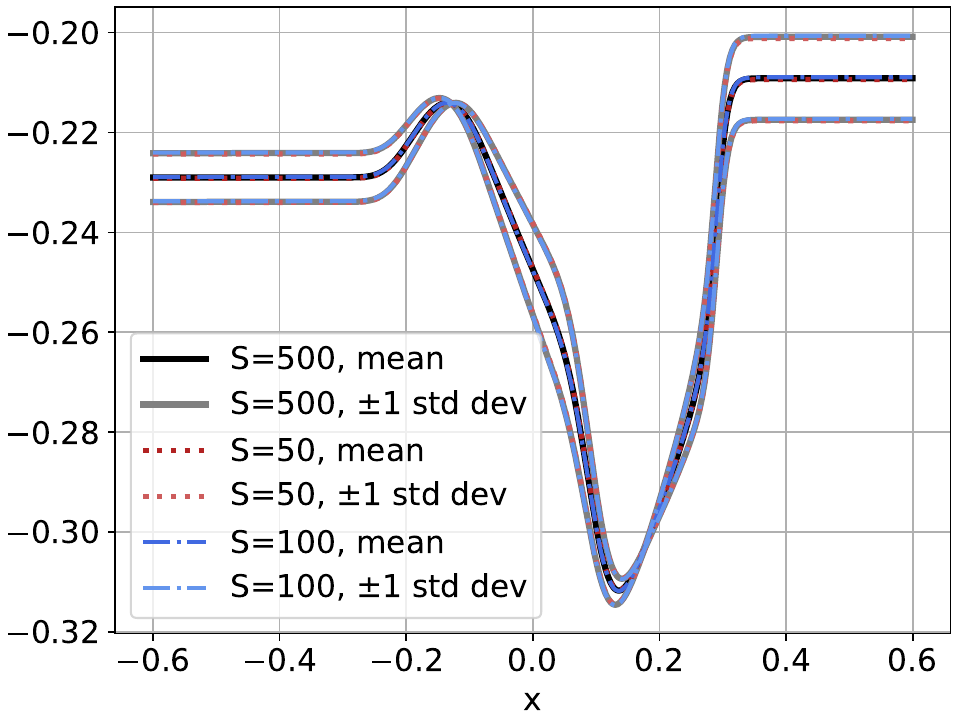}
        \vspace{-2em}
        \caption*{(c) First moment $u_1$}
    \endminipage
    \caption{Monte Carlo convergence test for the low dam break test case with SWLME $N=1$ using different number of samples $S$. At $S=100$ the Monte Carlo solution yields a converged result.}\label{RES-fig:MC_N1_low}
\end{figure}

In Figure \ref{RES-fig:MC_N2_low}, the same test case with $N=2$ yields analogous trends for $h$, $u_m$, and $u_1$, but introduces significant relative uncertainty in the second moment $u_2$. For $N=2$, visual convergence of Monte Carlo samples requires substantially larger $S$, with convergence observed at approximately $S=350$.

\begin{figure*}[!htb]
    \centering
    \minipage{0.33\textwidth}
        \centering
        \includegraphics[width=\linewidth]{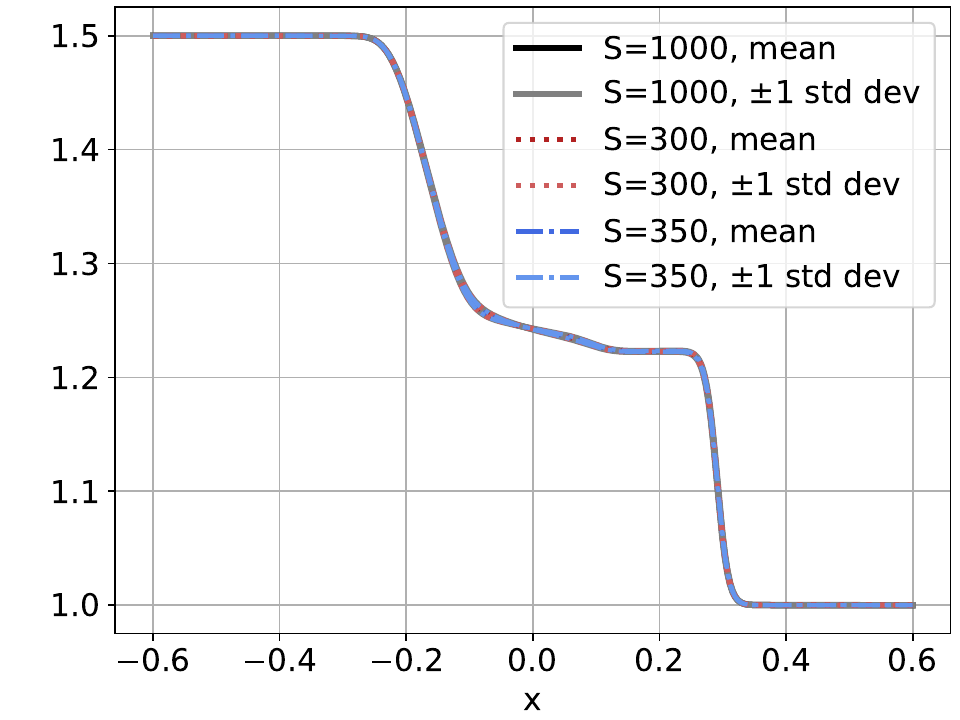}
        \vspace{-2em}
        \caption*{(a) Water height $h$}
        \includegraphics[width=\linewidth]{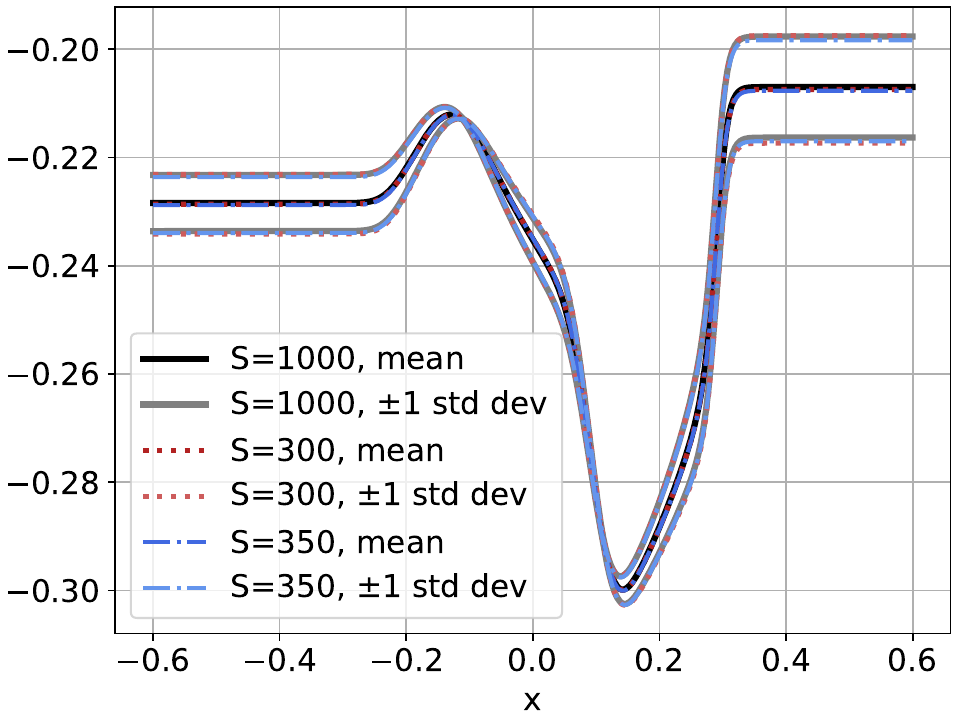}
        \vspace{-2em}
        \caption*{(c) First moment $u_1$}
    \endminipage
    \minipage{0.33\textwidth}
        \centering
        \includegraphics[width=\linewidth]{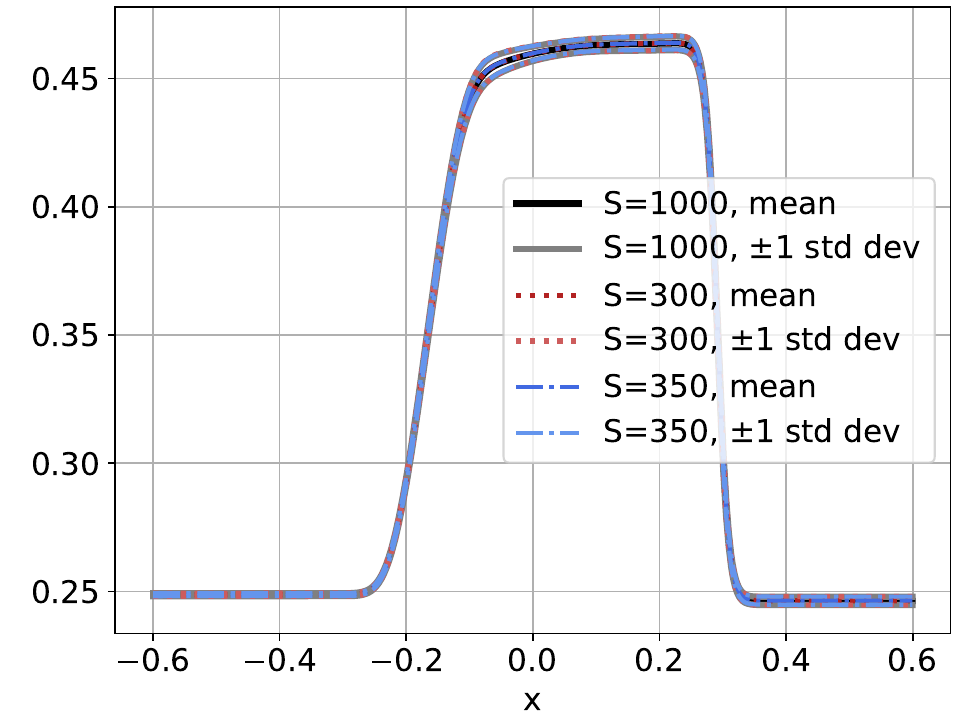}
        \vspace{-2em}
        \caption*{(b) Average velocity $u_m$}
        \includegraphics[width=\linewidth]{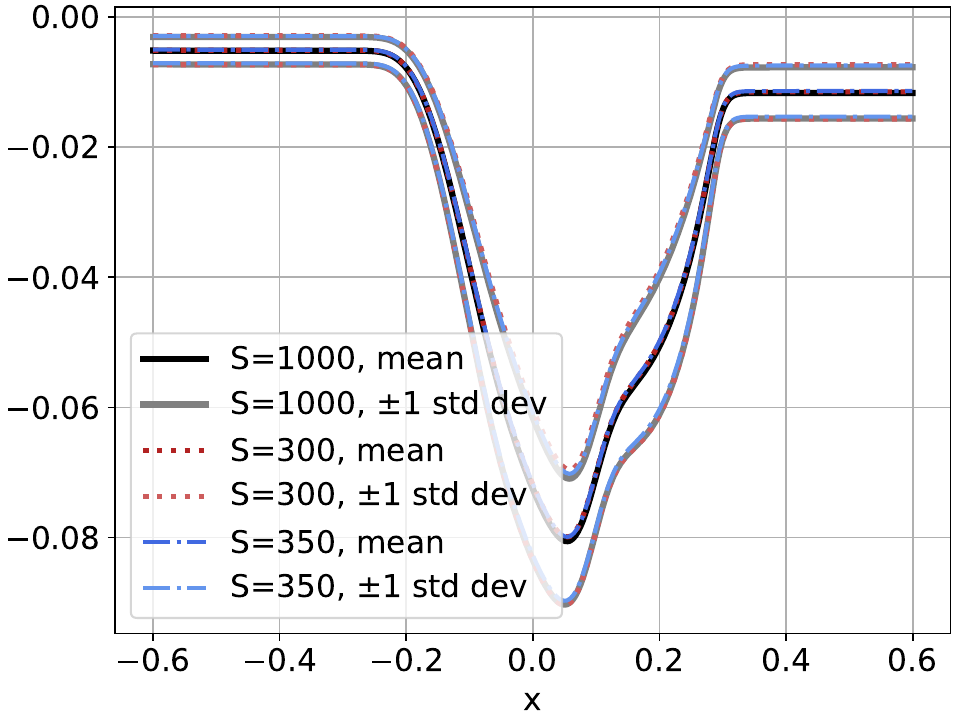}
        \vspace{-2em}
        \caption*{(d) Second moment $u_2$}
    \endminipage
    \caption{Monte Carlo convergence test for the low dam break test case with SWLME $N=2$ using different number of samples $S$. At $S=350$ the Monte Carlo solution yields a converged result.}\label{RES-fig:MC_N2_low}
\end{figure*}

In the high dam break test case with $N=1$ shown in Figure \ref{RES-fig:MC_N1_high}, the uncertainty in the first moment $u_1$ increases compared to the low dam, particularly in regions of higher average velocity $u_m$ and near the leading shock. The water height $h$ and the average velocity $u_m$ remain highly insensitive to variations in the friction coefficient $\nu$. Convergence is achieved rapidly, with visual convergence observed in approximately $S=100$ Monte Carlo samples.
\begin{figure}[H]
    \minipage{0.33\textwidth}
        \includegraphics[width=\linewidth]{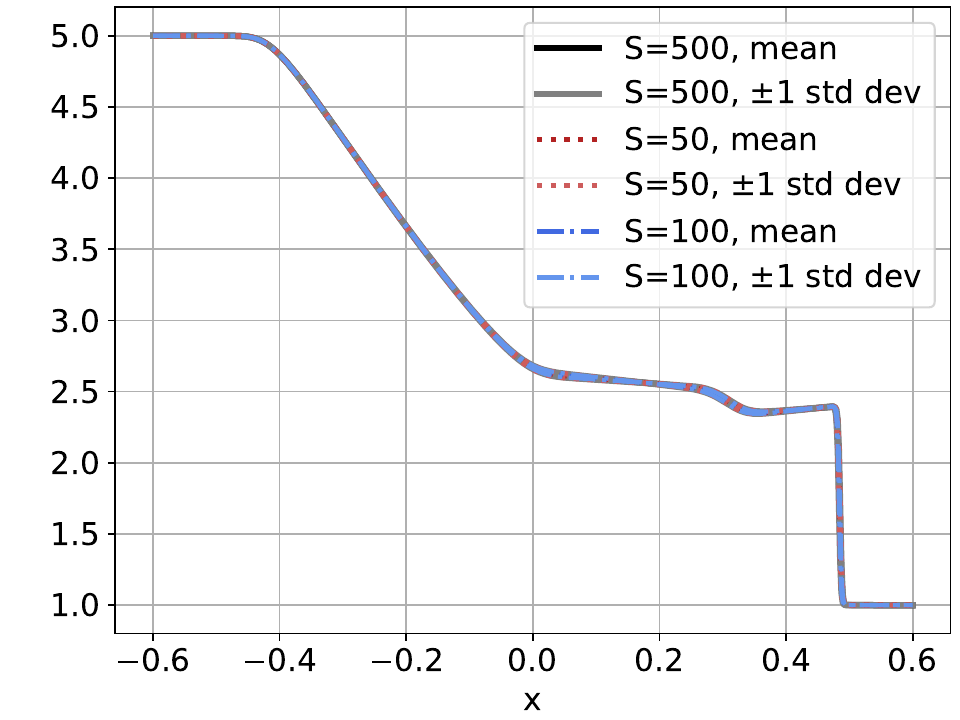}
        \vspace{-2em}
        \caption*{(a) Water height $h$}
    \endminipage\hfill
    \minipage{0.33\textwidth}
        \includegraphics[width=\linewidth]{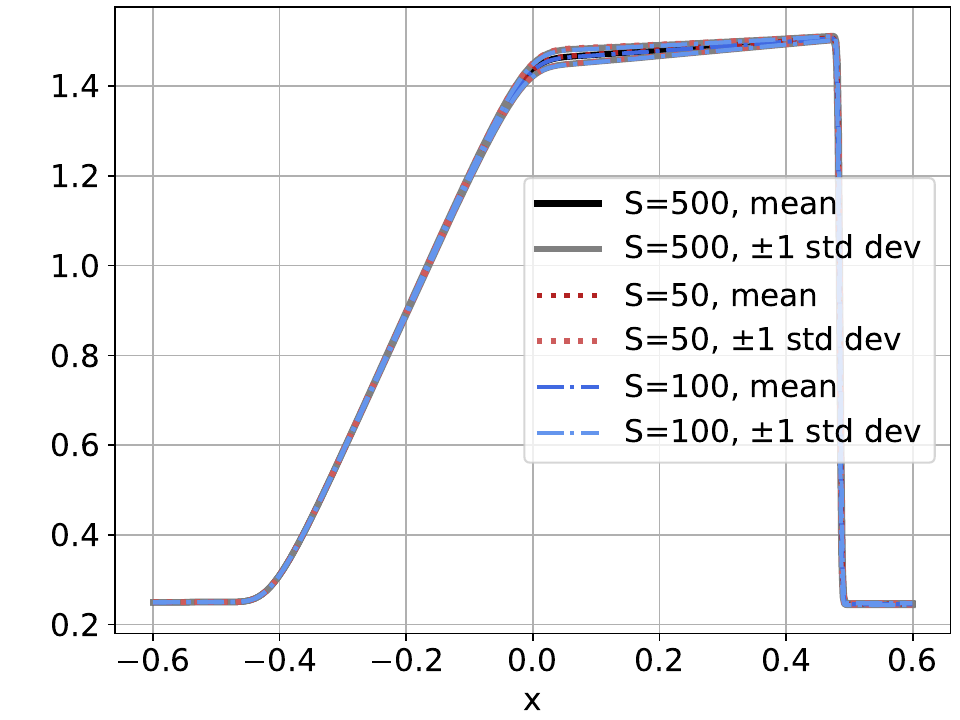}
        \vspace{-2em}
        \caption*{(b) Average velocity $u_m$}
    \endminipage\hfill
    \minipage{0.33\textwidth}%
        \includegraphics[width=\linewidth]{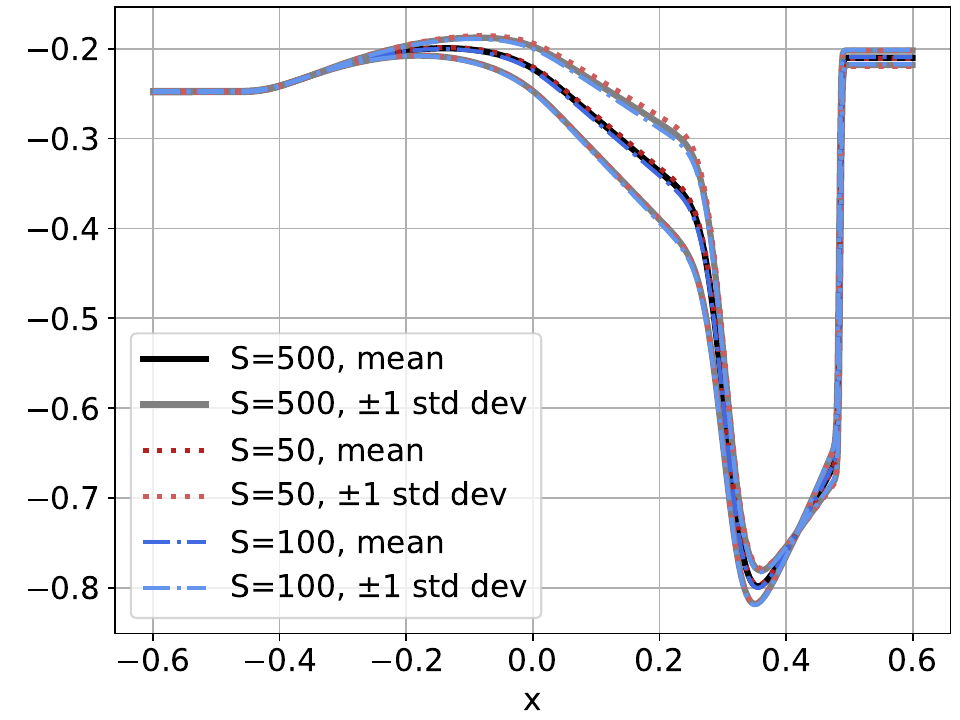}
        \vspace{-2em}
        \caption*{(c) First moment $u_1$}
    \endminipage
    \caption{Monte Carlo convergence test for the high dam break test case with  SWLME $N=1$ using different number of samples $S$. At $S=100$ the Monte Carlo solution yields a converged result.}\label{RES-fig:MC_N1_high}
\end{figure}

Increasing the order to $N=2$ in the high dam break test case, shown in Figure \ref{RES-fig:MC_N2_high}, yields effects analogous to those observed for the low dam break case. The subfigures for the water height $h$, the average velocity $u_m$, and the first moment $u_1$ remain visually unchanged. However, significant relative uncertainty emerges in the second moment $u_2$. Notably, the number of Monte Carlo samples $S$ required for visual convergence does not increase with the order $N$ in this instance.

\begin{figure*}[!htb]
    \centering
    \minipage{0.33\textwidth}
        \centering
        \includegraphics[width=\linewidth]{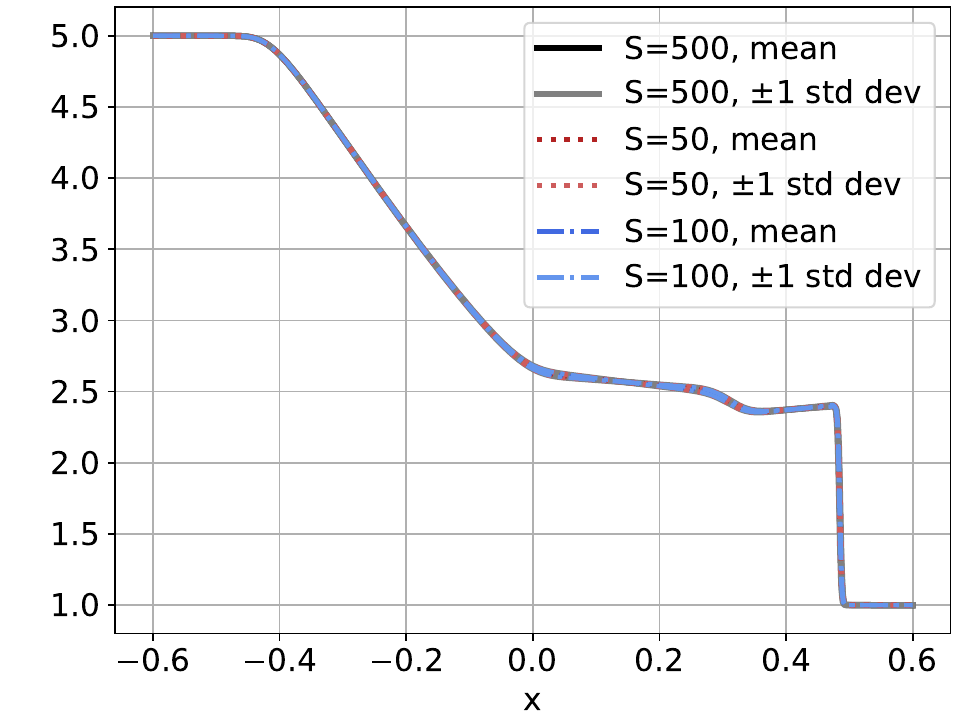}
        \vspace{-2em}
        \caption*{(a) Water height $h$}
        \includegraphics[width=\linewidth]{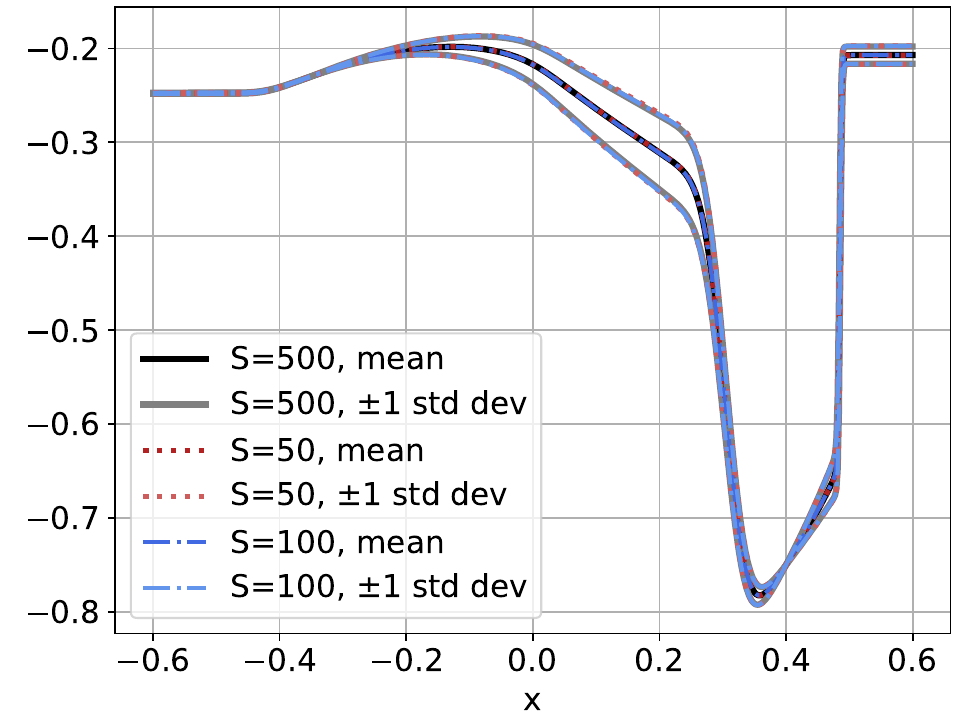}
        \vspace{-2em}
        \caption*{(c) First moment $u_1$}
    \endminipage
    \minipage{0.33\textwidth}
        \centering
        \includegraphics[width=\linewidth]{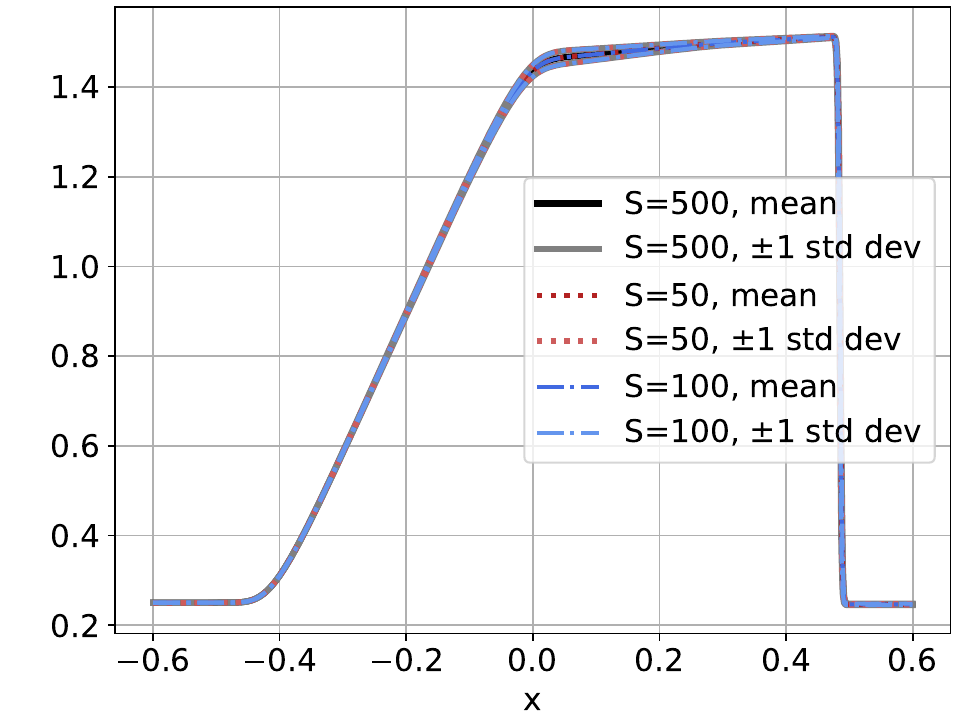}
        \vspace{-2em}
        \caption*{(b) Average velocity $u_m$}
        \includegraphics[width=\linewidth]{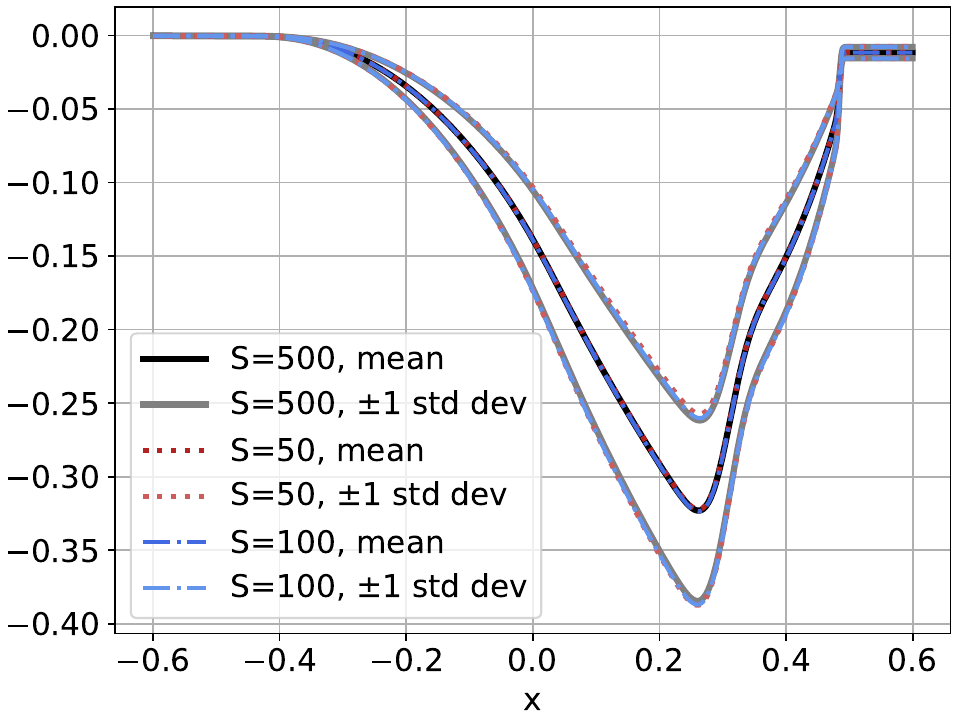}
        \vspace{-2em}
        \caption*{(d) Second moment $u_2$}
    \endminipage
    \caption{Monte Carlo convergence test for the high dam break test case with SWLME $N=2$ model for different number of samples $S$. At $S=100$ the Monte Carlo solution yields a converged result.}\label{RES-fig:MC_N2_high}
\end{figure*}

In Figure \ref{RES-fig:MC_N1_wave}, it is evident that for the smooth periodic wave test case of Table \ref{RES-tab:smooth_wave} with $N=1$, uncertainty in the friction coefficient $\nu$ significantly affects all three functions: the water height $h$, average velocity $u_m$, and first moment $u_1$. Convergence in this scenario requires more Monte Carlo samples than in the dam break cases shown in Figures \ref{RES-fig:MC_N1_low}, \ref{RES-fig:MC_N2_low}, \ref{RES-fig:MC_N1_high}, and \ref{RES-fig:MC_N2_high}. Visual convergence is only achieved at $S=900$ samples. The greater sensitivity to the uncertain friction coefficient and the increase in the sample requirement for convergence, compared to the dam break test cases, can be partially attributed to the longer simulation run time. Over the duration of $t_{\text{end}}=2.0$ seconds, the smooth periodic wave traverses the computational domain twice. This order-of-magnitude longer run time allows the uncertain parameter more time to influence the final solution.

\begin{figure}[H]
    \minipage{0.33\textwidth}
        \includegraphics[width=\linewidth]{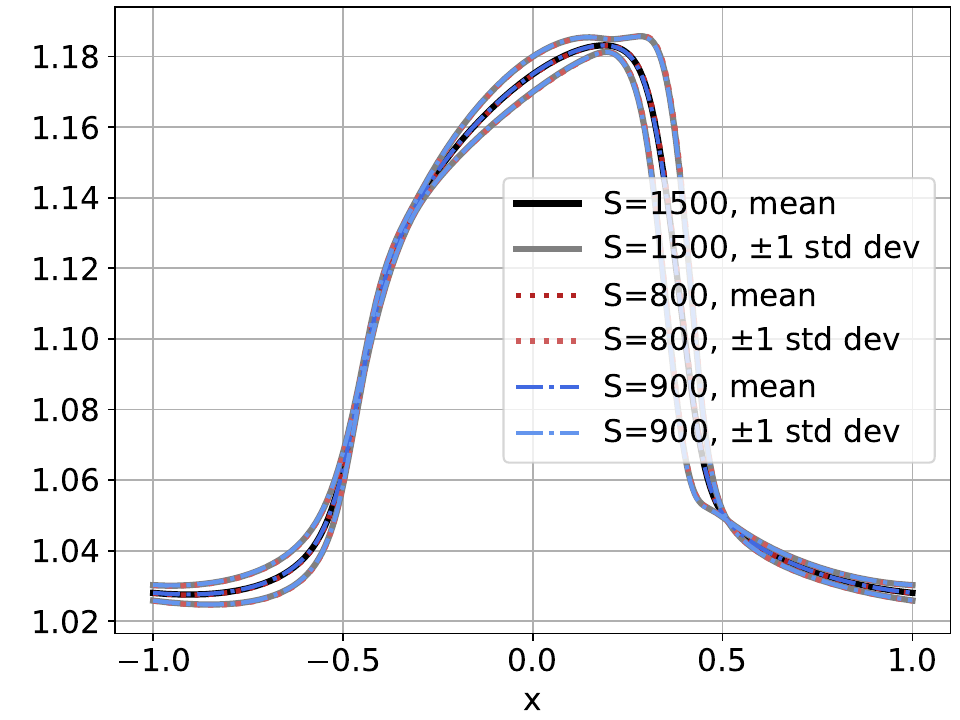}
        \vspace{-2.3em}
        \caption*{(a) Water height $h$}
    \endminipage\hfill
    \minipage{0.33\textwidth}
        \includegraphics[width=\linewidth]{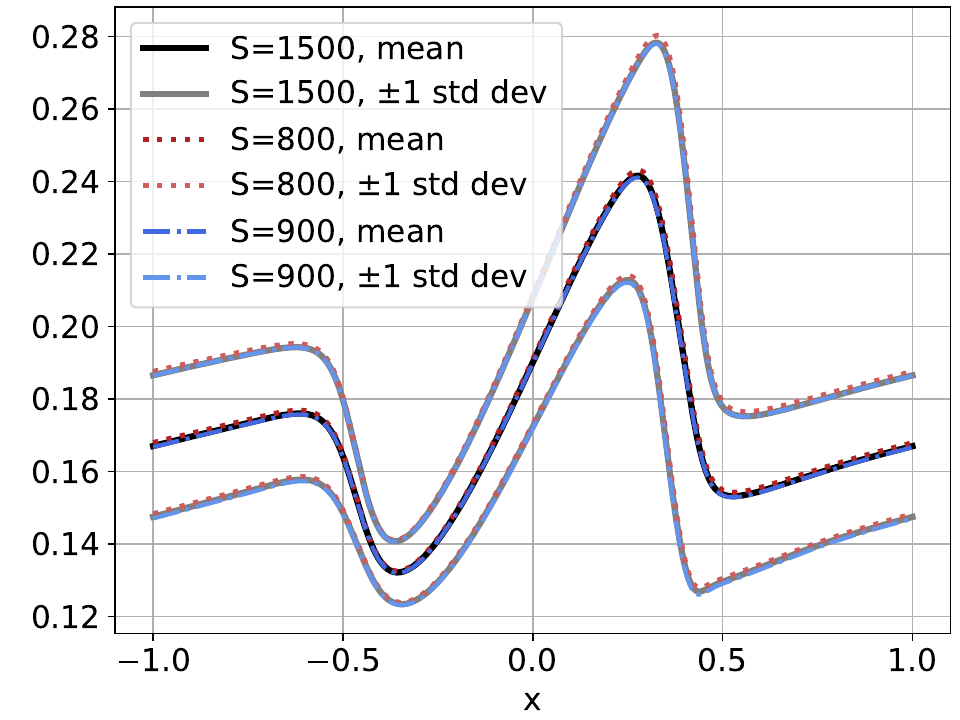}
        \vspace{-2.3em}
        \caption*{(b) Average velocity $u_m$}
    \endminipage\hfill
    \minipage{0.33\textwidth}
        \includegraphics[width=\linewidth]{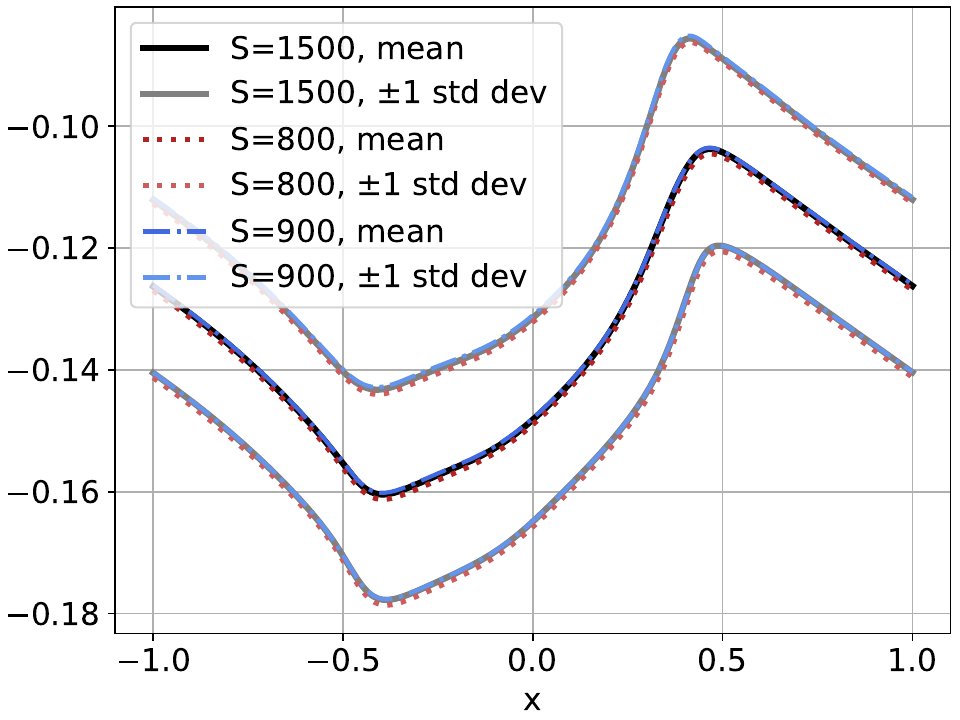}
        \vspace{-2.3em}
        \caption*{(c) First moment $u_1$}
    \endminipage
    \vspace{-0.9em}
    \caption{Monte Carlo convergence test for the smooth periodic wave test case with SWLME $N=1$ using different number of samples $S$. At $S=900$ the Monte Carlo solution yields a converged result.}\label{RES-fig:MC_N1_wave}
\end{figure}

As in Figure \ref{RES-fig:MC_N1_wave}, the uncertainty in the friction coefficient $\nu$ affects all functions in the $N=2$ scenario: the water height $h$, the average velocity $u_m$, the first moment $u_1$, and the second moment $u_2$, as shown in Figure \ref{RES-fig:MC_N2_wave}. Convergence in this case also requires more samples than in the dam break scenarios, though fewer than in the $N=1$ case. Visual convergence is achieved in $S=550$ Monte Carlo samples.

\begin{figure*}[!htb]
    \centering
    \minipage{0.33\textwidth}
        \centering
        \includegraphics[width=\linewidth]{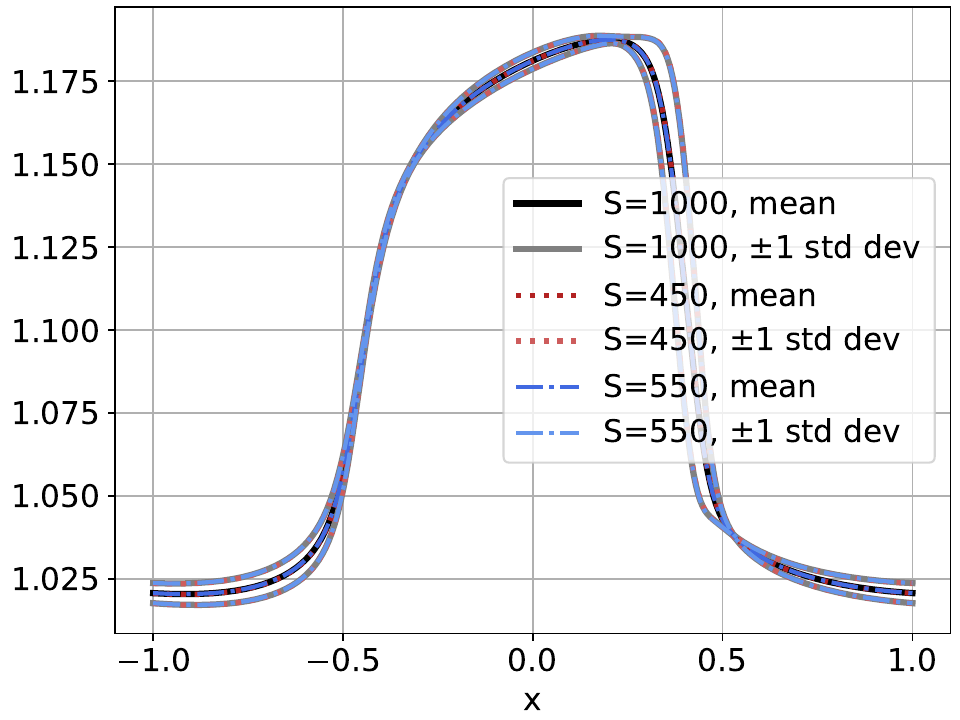}
        \vspace{-2.3em}
        \caption*{(a) Water height $h$}
        \includegraphics[width=\linewidth]{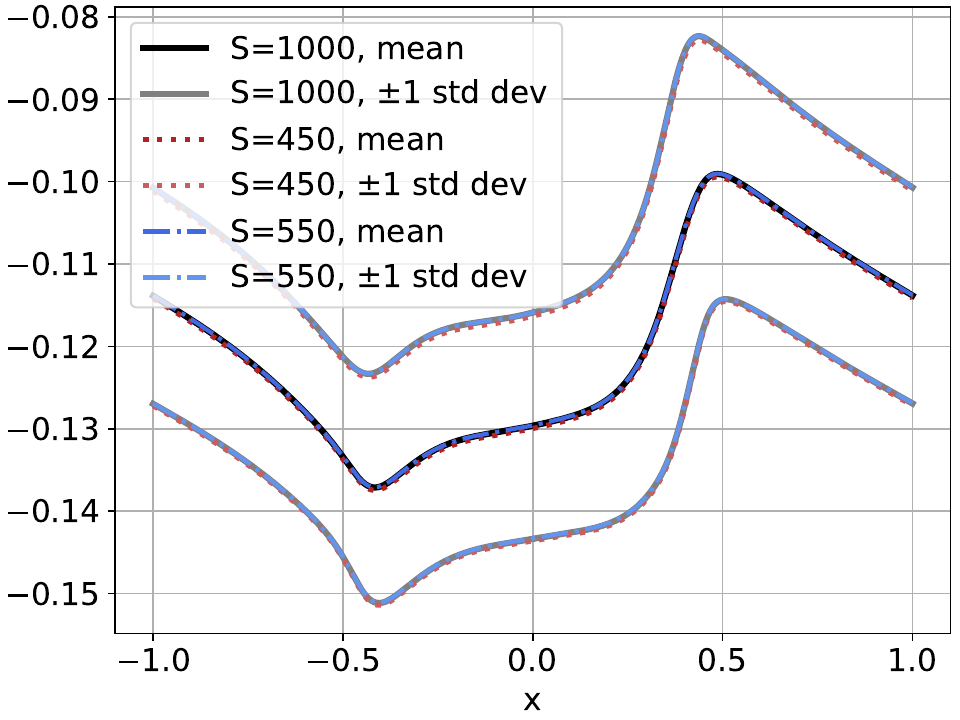}
        \vspace{-2.3em}
        \caption*{(c) First moment $u_1$}
    \endminipage
    \minipage{0.33\textwidth}
        \centering
        \includegraphics[width=\linewidth]{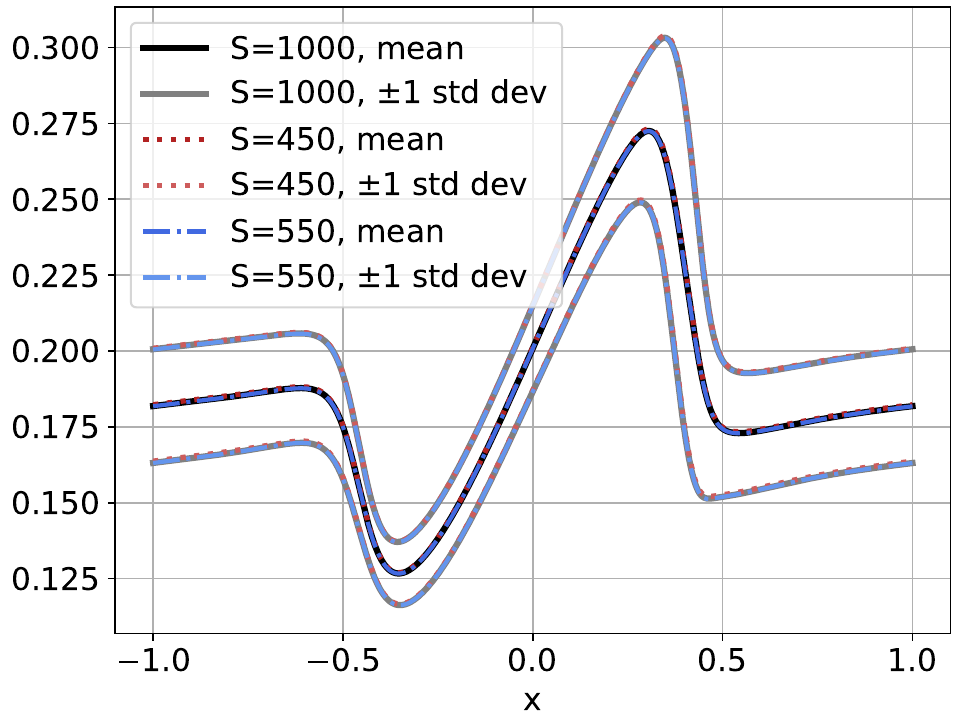}
        \vspace{-2.3em}
        \caption*{(b) Average velocity $u_m$}
        \includegraphics[width=\linewidth]{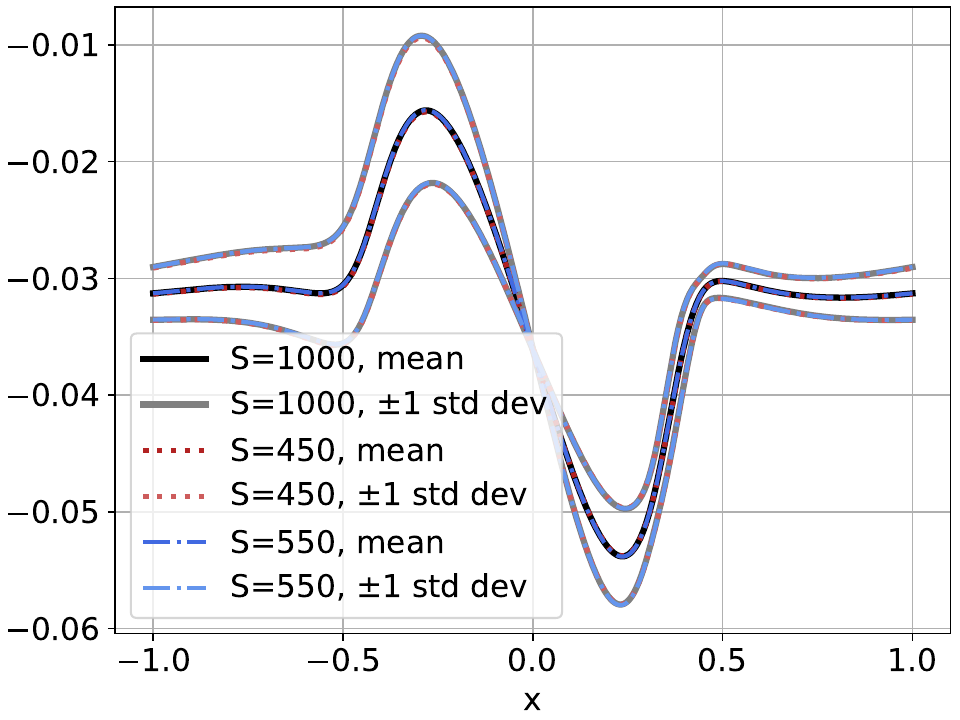}
        \vspace{-2.3em}
        \caption*{(d) Second moment $u_2$}
    \endminipage
    \vspace{-0.9em}
    \caption{Monte Carlo convergence test for the smooth periodic wave test case with SWLME $N=2$ using different number of samples $S$. At $S=550$ the Monte Carlo solution yields a converged result.}\label{RES-fig:MC_N2_wave}
\end{figure*}